\documentclass[12pt]{article}
\pdfoutput=1
\usepackage{e-jc}
\usepackage[latin1]{inputenc}   
\usepackage{subeqnarray}
\usepackage{indentfirst} 
\usepackage{amsfonts,amssymb,amsmath,amsthm}
\usepackage{bm}          
\usepackage{cite}
\usepackage{url}
            
\dateline{Aug 7, 2020}{Apr 1, 2021}{TBD}

\MSC{05A10 (Primary);  05A15, 05A19, 30B70 (Secondary)}

\Copyright{The authors. Released under the CC BY-ND license (International 4.0).}

\title{The Graham--Knuth--Patashnik recurrence: \\
       Symmetries and continued fractions}

\author{Jes\'us Salas\\
\small Departamento de Matem\'aticas\\[-0.8ex]
\small Universidad Carlos III de Madrid \\[-0.8ex]
\small Avda.\  de la Universidad, 30\\[-0.8ex] 
\small 28911 Legan\'es, Spain\\[2mm]
\small Grupo de Teor\'{\i}as de Campos y \\[-0.8ex]
\small F\'{\i}sica Estadística\\[-0.8ex]
\small Associate Unit UC3M--IEM (CSIC), Spain \\ 
\small\tt jsalas@math.uc3m.es\\
\and
Alan D.~Sokal\\ 
\small Department of Mathematics\\[-0.8ex] 
\small University College London\\[-0.8ex]
\small Gower Street\\[-0.8ex]
\small London WC1E 6BT\\[-0.8ex]
\small United Kingdom\\
\small\tt sokal@math.ucl.ac.uk\\
\small Department of Physics\\[-0.8ex]
\small New York University\\[-0.8ex]
\small 726 Broadway\\[-0.8ex]
\small New York, NY 10003, USA\\
\small\tt sokal@nyu.edu}

\numberwithin{equation}{section}  
\numberwithin{theorem}{section}   
\begin{document}
\bibliographystyle{plain}
\maketitle 

\begin{abstract}
We study the triangular array defined by the
Graham--Knuth--Patashnik recurrence
$$
  T(n,k)
  \;=\;
  (\alpha n + \beta k + \gamma)    \, T(n-1,k)
  \:+\:
  (\alpha' n + \beta' k + \gamma') \, T(n-1,k-1)
$$
with initial condition $T(0,k) = \delta_{k0}$
and parameters $\bm{\mu} = (\alpha,\beta,\gamma, \alpha',\beta',\gamma')$.
We show that the family of arrays $T(\bm{\mu})$ is invariant under
a 48-element discrete group isomorphic to $S_3 \times D_4$. 
Our main result is to determine all parameter sets
$\bm{\mu} \in \mathbb{C}^6$
for which the ordinary generating function
$f(x,t) = \sum_{n,k=0}^\infty T(n,k) \, x^k t^n$
is given by a Stieltjes-type continued fraction in $t$
with coefficients that are polynomials in~$x$.
We also exhibit some special cases in which $f(x,t)$
is given by a Thron-type or Jacobi-type continued fraction in $t$
with coefficients that are polynomials in~$x$.
\end{abstract}



%
%
%
%
\newcommand{\be}{\begin{equation}}
\newcommand{\ee}{\end{equation}}
\newcommand{\<}{\langle}
\renewcommand{\>}{\rangle}
\newcommand{\widebar}{\overline}
\def\reff#1{(\protect\ref{#1})}
\def\spose#1{\hbox to 0pt{#1\hss}}
\def\ltapprox{\mathrel{\spose{\lower 3pt\hbox{$\mathchar"218$}}
 \raise 2.0pt\hbox{$\mathchar"13C$}}}
\def\gtapprox{\mathrel{\spose{\lower 3pt\hbox{$\mathchar"218$}}
 \raise 2.0pt\hbox{$\mathchar"13E$}}}
\def\textprime{${}^\prime$}
\def\half{\frac{1}{2}}
\def\third{\frac{1}{3}}
\def\twothird{\frac{2}{3}}
\def\smfrac#1#2{\textstyle \frac{#1}{#2}}
\def\smhalf{\smfrac{1}{2} }

%
%
\newcommand{\restrict}{\upharpoonright}
\newcommand{\drop}{\setminus}
\renewcommand{\emptyset}{\varnothing}
\newcommand{\eqdef}{\stackrel{\rm def}{=}}
\newcommand{\rem}{\textrm{rem}}
\newcommand{\Sym}{{\mathfrak{S}}}
\def\qed{ $\square$ \bigskip}
\newcommand{\myendremark}{ $\blacksquare$ \bigskip}
\def\proofof#1{\bigskip\noindent{\sc Proof of #1.\ }}
\newcommand{\myle}{\preceq}
\newcommand{\myge}{\succeq}
\newcommand{\mygt}{\succ}

%
%
\newcommand{\cyc}{{\rm cyc}}
\newcommand{\exc}{{\rm exc}}
\newcommand{\rec}{{\rm rec}}
\newcommand{\arec}{{\rm arec}}
\newcommand{\erec}{{\rm erec}}

%
%
\newcommand{\C}{{\mathbb C}}
\newcommand{\D}{{\mathbb D}}
\newcommand{\Z}{{\mathbb Z}}
\newcommand{\N}{{\mathbb N}}
\newcommand{\R}{{\mathbb R}}
\newcommand{\Q}{{\mathbb Q}}

%
%
\newcommand{\TT}{{\mathsf T}}
\newcommand{\HH}{{\mathsf H}}
\newcommand{\VV}{{\mathsf V}}
\newcommand{\JJ}{{\mathsf J}}
\newcommand{\PP}{{\mathsf P}}
\newcommand{\DD}{{\mathsf D}}
\newcommand{\QQ}{{\mathsf Q}}
\newcommand{\RR}{{\mathsf R}}

%
%
\newcommand{\bgamma}{{\bm{\gamma}}}
\newcommand{\bdelta}{{\bm{\delta}}}
\newcommand{\bmu}{{\bm{\mu}}}
\newcommand{\bsigma}{{\bm{\sigma}}}
\newcommand{\vecbsigma}{{\vec{\bm{\sigma}}}}
\newcommand{\bpi}{{\bm{\pi}}}
\newcommand{\vecbpi}{{\vec{\bm{\pi}}}}
\newcommand{\btau}{{\bm{\tau}}}
\newcommand{\bphi}{{\bm{\phi}}}
\newcommand{\bvarphi}{{\bm{\varphi}}}
\newcommand{\bGamma}{{\bm{\Gamma}}}

%
%
\newcommand{\psibar}{ {\bar{\psi}} }
\newcommand{\varphibar}{ {\bar{\varphi}} }

%
%
\newcommand{\bfa}{ {\bf a} }
\newcommand{\bfb}{ {\bf b} }
\newcommand{\bfc}{ {\bf c} }
\newcommand{\bfp}{ {\bf p} }
\newcommand{\bfr}{ {\bf r} }
\newcommand{\bfs}{ {\bf s} }
\newcommand{\bft}{ {\bf t} }
\newcommand{\bfu}{ {\bf u} }
\newcommand{\bfv}{ {\bf v} }
\newcommand{\bfw}{ {\bf w} }
\newcommand{\bfx}{ {\bf x} }
\newcommand{\bfy}{ {\bf y} }
\newcommand{\bfz}{ {\bf z} }
\newcommand{\bfT}{ {\bf T} }
\newcommand{\bone}{ {\mathbf 1} }
\newcommand{\bzero}{ {\mathbf 0} }

%
%
\newcommand{\ba}{{\bm{a}}}
\newcommand{\bb}{{\bm{b}}}
\newcommand{\bc}{{\bm{c}}}
\newcommand{\bd}{{\bm{d}}}
\newcommand{\bee}{{\bm{e}}}
\newcommand{\bff}{{\bm{f}}}
\newcommand{\bA}{{\bm{A}}}
\newcommand{\bB}{{\bm{B}}}
\newcommand{\bC}{{\bm{C}}}
\newcommand{\bP}{{\bm{P}}}
\newcommand{\bS}{{\bm{S}}}
\newcommand{\bT}{{\bm{T}}}

%
%
\newcommand{\scra}{{\mathcal{A}}}
\newcommand{\scrb}{{\mathcal{B}}}
\newcommand{\scrc}{{\mathcal{C}}}
\newcommand{\scrd}{{\mathcal{D}}}
\newcommand{\scre}{{\mathcal{E}}}
\newcommand{\scrf}{{\mathcal{F}}}
\newcommand{\scrg}{{\mathcal{G}}}
\newcommand{\scrh}{{\mathcal{H}}}
\newcommand{\scri}{{\mathcal{I}}}
\newcommand{\scrj}{{\mathcal{J}}}
\newcommand{\scrk}{{\mathcal{K}}}
\newcommand{\scrl}{{\mathcal{L}}}
\newcommand{\scrm}{{\mathcal{M}}}
\newcommand{\scrn}{{\mathcal{N}}}
\newcommand{\scro}{{\mathcal{O}}}
\newcommand{\scrp}{{\mathcal{P}}}
\newcommand{\scrq}{{\mathcal{Q}}}
\newcommand{\scrr}{{\mathcal{R}}}
\newcommand{\scrs}{{\mathcal{S}}}
\newcommand{\scrt}{{\mathcal{T}}}
\newcommand{\scru}{{\mathcal{U}}}
\newcommand{\scrv}{{\mathcal{V}}}
\newcommand{\scrw}{{\mathcal{W}}}
\newcommand{\scrx}{{\mathcal{X}}}
\newcommand{\scry}{{\mathcal{Y}}}
\newcommand{\scrz}{{\mathcal{Z}}}

%
%
\newcommand{\ofo}{ {{}_1 \! F_1} }
\newcommand{\tfo}{ {{}_2 F_1} }

%
%
\def\hboxrm#1{ {\hbox{\scriptsize\rm #1}} }
\def\hboxsans#1{ {\hbox{\scriptsize\sf #1}} }
\def\hboxscript#1{ {\hbox{\scriptsize\it #1}} }

%
%

%
%
\newcommand{\stirlingsubset}[2]{\genfrac{\{}{\}}{0pt}{}{#1}{#2}}
\newcommand{\stirlingcycle}[2]{\genfrac{[}{]}{0pt}{}{#1}{#2}}
\newcommand{\associatedstirlingsubset}[2]%
      {\left\{\!\! \stirlingsubset{#1}{#2} \!\! \right\}}
\newcommand{\assocstirlingsubset}[3]%
      {{\genfrac{\{}{\}}{0pt}{}{#1}{#2}}_{\! \ge #3}}
\newcommand{\assocstirlingcycle}[3]{{\genfrac{[}{]}{0pt}{}{#1}{#2}}_{\ge #3}}
\newcommand{\associatedstirlingcycle}[2]{\left[\!\!%
            \stirlingcycle{#1}{#2} \!\! \right]}
\newcommand{\euler}[2]{\genfrac{\langle}{\rangle}{0pt}{}{#1}{#2}}
\newcommand{\eulergen}[3]{{\genfrac{\langle}{\rangle}{0pt}{}{#1}{#2}}_{\! #3}}
\newcommand{\eulersecond}[2]{\left\langle\!\! \euler{#1}{#2} \!\!\right\rangle}
\newcommand{\eulersecondBis}[2]{\big\langle\!\! \euler{#1}{#2} \!\!\big\rangle}
\newcommand{\eulersecondgen}[3]%
     {{\left\langle\!\! \euler{#1}{#2} \!\!\right\rangle}_{\! #3}}
\newcommand{\associatedstirlingcycleBis}[2]{\big[ \!\!%
            \stirlingcycle{#1}{#2} \!\! \big]}
\newcommand{\binomvert}[2]{\genfrac{\vert}{\vert}{0pt}{}{#1}{#2}}
\newcommand{\doublebinom}[2]{\left(\!\! \binom{#1}{#2} \!\!\right)}
\newcommand{\nueuler}[3]{{\genfrac{\langle}{\rangle}{0pt}{}{#1}{#2}}^{\! #3}}
\newcommand{\nueulergen}[4]%
{{\genfrac{\langle}{\rangle}{0pt}{}{#1}{#2}}^{\! #3}_{\! #4}}

%
%
\newenvironment{sarray}{
          \textfont0=\scriptfont0
          \scriptfont0=\scriptscriptfont0
          \textfont1=\scriptfont1
          \scriptfont1=\scriptscriptfont1
          \textfont2=\scriptfont2
          \scriptfont2=\scriptscriptfont2
          \textfont3=\scriptfont3
          \scriptfont3=\scriptscriptfont3
        \renewcommand{\arraystretch}{0.7}
        \begin{array}{l}}{\end{array}}

\newenvironment{scarray}{
          \textfont0=\scriptfont0
          \scriptfont0=\scriptscriptfont0
          \textfont1=\scriptfont1
          \scriptfont1=\scriptscriptfont1
          \textfont2=\scriptfont2
          \scriptfont2=\scriptscriptfont2
          \textfont3=\scriptfont3
          \scriptfont3=\scriptscriptfont3
        \renewcommand{\arraystretch}{0.7}
        \begin{array}{c}}{\end{array}}

\newcommand{\seqnum}[1]{\href{http://oeis.org/#1}{#1}}


\clearpage
\tableofcontents
%
%
\section{Introduction} \label{sec.intro} 

Graham, Knuth and Patashnik (GKP), in their book {\em Concrete Mathematics}\/
\cite{Graham_94}, posed the following ``research problem''
\cite[Problem~6.94, pp.~319 and~564]{Graham_94}:

\begin{problem}
    \label{problem.GKP}
Develop a general theory of the solutions to the recurrence
\begin{equation}
  T(n,k)
  \;=\;
  (\alpha n + \beta k + \gamma)    \, T(n-1,k)
  \:+\:
  (\alpha' n + \beta' k + \gamma') \, T(n-1,k-1)
  \label{eq_binomvert}
\end{equation}
for $n \ge 1$ and $k \in \Z$, with initial condition $T(0,k) = \delta_{k0}$.
(Here and in the following, $\delta_{ab}$ denotes the Kronecker delta.)
\end{problem}

By induction on $n$ we clearly have $T(n,k) = 0$ if $k < 0$ or $k > n$.
Therefore, for each choice of the parameters 
$\bmu = (\alpha,\beta,\gamma, \alpha',\beta',\gamma')$,
we obtain a unique solution $T(n,k) = T(n,k;\bmu)$,
forming a triangular array
$\bT(\bmu) = \bigl( T(n,k;\bmu) \bigr)_{0 \le k \le n}$.
Here the parameters $\bmu$ can be considered to be indeterminates,
in which case the matrix elements $T(n,k;\bmu)$
belong to the polynomial ring $\Z[\bmu]$;
or they can be real or complex numbers,
in which case the matrix elements $T(n,k;\bmu)$
are likewise real or complex numbers.
We shall take each of these two points of view
at appropriate places in this paper.

Given a triangular array $\bT = \bigl( T(n,k) \bigr)_{0 \le k \le n}$,
we define its row-generating polynomials
\be
P_n(x) \;\eqdef\; 
       \sum\limits_{k=0}^n T(n,k) \, x^k \,,
\label{def_Pn}
\ee
its ordinary generating function (ogf)
\be
f(x,t) \;\eqdef\; \sum\limits_{n\ge 0} P_n(x) \, t^n
       \;=\;  \sum\limits_{n,k \ge 0} T(n,k) \: x^k t^n  \,,
\label{def_ogf}
\ee
and its exponential generating function (egf)
\be
F(x,t) \;\eqdef\; \sum\limits_{n\ge 0} P_n(x) \, \frac{t^n}{n!}
       \;=\;  \sum\limits_{n,k \ge 0} T(n,k) \: x^k \,\frac{t^n}{n!}  \,.
\label{def_egf}
\ee
It is straightforward to check that
the row-generating polynomials $P_n(x) = P_n(x;\bmu)$
of the GKP recurrence satisfy the linear differential recurrence
\be
   P_n(x)  \;=\; \bigl[ n(\alpha+ \alpha' x) + \gamma + (\beta' + \gamma') x
                 \bigr] \, P_{n-1}(x)
      \:+\:
      x (\beta + \beta' x) \, \frac{d P_{n-1}(x)}{dx}
 \label{eq.diff_recurrence}
\ee
for $n \ge 1$, with initial condition $P_0(x) = 1$.
Similarly, the exponential generating function $F(x,t) = F(x,t;\bmu)$
satisfies the first-order linear partial differential equation
\be
   (1 - \alpha t - \alpha' xt) F_t
   \;=\;
   (\beta x + \beta' x^2) F_x  \:+\:
     (\alpha + \gamma + (\alpha' + \beta' + \gamma')x) F
 \label{eq.PDE.egf}
\ee
with initial condition $F(x,0) = 1$.

Explicit solutions for the matrix elements $T(n,k)$ or the egf $F(x,t)$
were found for some special cases of the parameters $\bmu$ by
Th\'eor\^et \cite{Theoret_94,Theoret_95a,Theoret_95b},
Neuwirth \cite{Neuwirth}, Spivey \cite{Spivey_11},
and Mansour and Shattuck \cite{Mansour_13}.
In particular, Neuwirth \cite{Neuwirth} solved the case $\alpha' = 0$,
while Spivey \cite{Spivey_11} solved three additional cases:
(S1) $\beta = -\alpha$,
(S2) $\beta = \beta' = 0$,
and (S3) $\alpha/\beta = \alpha'/\beta' + 1$. 
Wilf \cite{Wilf_04} pointed out that the PDE \reff{eq.PDE.egf}
can in principle be solved by the method of characteristics,
and he showed how to obtain the solution (albeit in an unwieldy form)
by using the {\sc Maple} function {\tt pdsolve}.
Finally, Barbero, Salas and Villase\~nor \cite{BSV}
explicitly solved the PDE \reff{eq.PDE.egf}
by the method of characteristics.
In general this solution contains inverse functions
\cite[Theorems~2.1--2.4]{BSV},
but in many combinatorially interesting cases
there exist closed-form expressions in terms of elementary functions
\cite[Appendix~A]{BSV}.
It is worthy of note that the function $\bmu \mapsto \bT(\bmu)$
is not injective:
there are some families of parameters $\bmu$ that produce the same 
triangular array $\bT$
\cite[section~2.4]{Theoret_94} \cite[section~3]{BSV}.
Finally, some additional properties of the triangular arrays 
$\bT = \bigl( T(n,k) \bigr)_{0 \le k \le n}$ have recently been obtained by
Spivey \cite{Spivey_20}.

Our goal in the present paper is twofold:
to study the symmetry group of the GKP recurrence,
and to study some continued-fraction expansions
of the ordinary generating function \reff{def_ogf}.
Let us now discuss these two goals in turn.

By a ``symmetry'' of the GKP recurrence \reff{eq_binomvert},
we mean a map $M \colon \bmu \mapsto \bmu'$
for which the array $\bT(\bmu')$ can be written in a ``simple'' way
in terms of $\bT(\bmu)$.
(See Section~\ref{sec.GKP} for more details of what we mean by ``simple''.)
Here we will show that the symmetry group of the GKP recurrence
is surprisingly large, and includes a 48-element discrete group
that is isomorphic to $S_3 \times D_4$.

It is well known that many combinatorial sequences $\ba = (a_n)_{n\ge 0}$
with $a_0=1$ lead to an ogf that can be expressed as a
continued fraction of Stieltjes type (or {\em S-fraction}\/ for short),
\be
\sum\limits_{n\ge 0} a_n \, t^n  
       \;=\; \cfrac{1}{1 - \cfrac{c_1 t}{ 1 - 
                           \cfrac{c_2 t}{ 1 - \cdots}}}
   \;,
\label{def_Stype.one}
\ee
for some coefficients $\bc=(c_i)_{i\ge 1}$.
(Both sides of this expressions are to be interpreted as
formal power series in the indeterminate $t$.)
This line of investigation goes back at least to
Euler \cite{Euler_1760,Euler_1788},
but it gained impetus following Flajolet's \cite{Flajolet_80}
seminal discovery that the S-fraction \reff{def_Stype.one}
can be interpreted combinatorially as a generating function
for Dyck paths with a weight $c_i$ for each fall from height $i$.
There are now literally dozens of sequences $\ba = (a_n)_{n \ge 0}$
of combinatorial numbers or polynomials for which
a continued-fraction expansion of the type \reff{def_Stype.one}
is explicitly known.

In the generic S-fraction \reff{def_Stype.one},
the Taylor coefficients $a_n$ are polynomials
in the Stieltjes coefficients $\bc$:
these are the {\em Stieltjes--Rogers polynomials}\/ $S_n(c_1,\ldots,c_n)$
\cite{Flajolet_80}.
If, however, one seeks conversely to express the Stieltjes coefficients $\bc$
in terms of the Taylor coefficients $\ba$,
one obtains in general rational functions, not polynomials:
\begin{subeqnarray}
   c_1  & = &  a_1
     \slabel{eq.firstfew_ci.c1}  \\
   c_2  & = &  \frac{a_2 - a_1^2}{a_1}
     \slabel{eq.firstfew_ci.c2}  \\
   c_3  & = &  \frac{a_1 a_3 - a_2^2}{a_1 \, (a_2 - a_1^2)}
     \slabel{eq.firstfew_ci.c3}  \\[-3mm]
        & \vdots &   \nonumber
 \label{eq.firstfew_ci}
\end{subeqnarray}
In particular, if one applies these formulae to the row-generating polynomials
of the GKP recurrence, $a_n = P_n(x;\bmu)$,
one obtains Stieltjes coefficients $c_i$ that, for $i \ge 2$,
are rational functions of $x$ and the parameters $\bmu$.
Nevertheless, for many specific cases of the GKP recurrence
--- including the binomial coefficients,
the Stirling cycle and Stirling subset numbers,
and the Eulerian numbers, among others ---
it is known that the Stieltjes coefficients $c_i$ are polynomials in $x$
(and fairly simple ones at that).
Consequently, the second (and principal) goal of this paper
is to obtain a complete determination of the submanifolds of $\bmu \in \C^6$
where the Stieltjes coefficients $\bc$ are polynomials in $x$.
These coefficients will always be of the form $c_i = c_{i0} + c_{i1} x$;
and while {\em a priori}\/ we allow $c_{i0}$ and $c_{i1}$
to be rational functions of the parameters $\bmu$,
we will find {\em a posteriori}\/ that they are in fact always
polynomials in $\bmu$
(or more precisely, in suitable parameters coordinatizing the given
 submanifold).
Our results are contained in Theorem~\ref{theor.main}
and Propositions~\ref{prop.I}--\ref{prop.VI}.

More generally, we shall briefly consider some continued fractions
of the Thron type (or {\em T-fractions}\/) \cite{Thron_48},
\be
\sum\limits_{n\ge 0} a_n \, t^n  
       \;=\; \cfrac{1}{1 - d_1 t - \cfrac{c_1 t}{ 1 - d_2 t -  
                           \cfrac{c_2 t}{ 1 - \cdots}}}
   \;,
\label{def_Ttype.one}
\ee
and of the Jacobi type (or {\em J-fractions}\/)
\be
\sum\limits_{n\ge 0} a_n \, t^n  
       \;=\; \cfrac{1}{1 - e_0 t - \cfrac{f_1 t^2}{ 1 - e_1 t -  
                           \cfrac{f_2 t^2}{ 1 - \cdots}}}
   \;,
\label{def_Jtype.one}
\ee
once again interpreted as formal power series in $t$.
We will find some recurrences of the GKP type \eqref{eq_binomvert}
whose ogf can be expressed as a T-fraction (or J-fraction)
in which the coefficients $\bc=(c_i)_{i\ge 1}$ and $\bd = (d_i)_{i\ge 1}$
(or $\bee=(e_i)_{i\ge 0}$ and $\bff = (f_i)_{i\ge 1}$)
are polynomials in the indeterminates $x$ and $\bmu$. 
However, this list is almost certainly incomplete;
we have not attempted to make a complete determination,
for reasons that will be explained later.

This paper is organized as follows:
In Section~\ref{sec.GKP} we analyze the symmetry group of the GKP recurrence.
In Section~\ref{sec.main} we completely characterize those GKP recurrences
whose ogf has an S-fraction representation
in which the coefficients are polynomials in~$x$.
In Section~\ref{sec.prelim.TJ}
we review some transformation formulae that will be needed
for our discussion of T-fractions and J-fractions.
In Section~\ref{sec.mainT} we show some examples
(but not a complete characterization) of GKP recurrences
whose ogf has a T-fraction or J-fraction representation
in which the coefficients are polynomials in $x$.

Finally, in Section~\ref{sec.open} we propose some problems for future work.
In particular, one motivation for the continued-fraction expansions
studied here is a conjecture \cite{Sokal_unpub} concerning
the coefficientwise Hankel-total positivity
\cite{Sokal_flajolet,Sokal_totalpos}
of the row-generating polynomials $P_n(x;\bmu)$ of the GKP recurrence.
Section~\ref{sec.hankel} is devoted to introducing
and discussing this conjecture.

Appendix~\ref{app.spivey-zhu} provides some more details
concerning the matrix product of two GKP arrays
(as well as some more general arrays),
while Appendix~\ref{app.inverse} gives a general treatment
of inverse pairs of lower-triangular arrays.
Finally, in the Supplementary Material we provide details
of the computer-assisted search that constitutes the first stage
of the proof of Theorem~\ref{theor.main}.

%
%
\section{Symmetries of the GKP recurrence} \label{sec.GKP}

In this section we will treat some symmetries of the GKP recurrence:
by this we mean maps $M \colon \bmu \mapsto \bmu'$
for which the array $\bT(\bmu')$ can be written in a simple way
in terms of $\bT(\bmu)$.
By ``simple'' we mean that the row-generating polynomials $P_n(x;\bmu')$
can be written as a M\"obius transformation of $P_n(x;\bmu)$:
\be
   P_n(x;\bmu')
   \;=\;
   [c(\bmu) x \,+\, d(\bmu)]^n \:
   P_n \biggl( \frac{a(\bmu) x \,+\, b(\bmu)}{ c(\bmu) x \,+\, d(\bmu) }
       \biggr)
\ee
for some functions $a(\bmu),\, b(\bmu),\, c(\bmu),\, d(\bmu)$.
We do not purport to find {\em all}\/ such transformations,
but we will find a fairly large (48-element) discrete group of them.
Some (but not all) of the transformations we will consider are involutions,
i.e.\ satisfy $M(M(\bmu)) = \bmu$.

\subsection{Some special symmetries} \label{sec.sym}

The most obvious symmetry is:

\medskip
{\bf Scaling.}
For parameters $\kappa$ and $\lambda$, define the map $S_{\kappa,\lambda}$ by
\be
   S_{\kappa,\lambda} (\alpha,\beta,\gamma, \alpha',\beta',\gamma')
   \;=\;
   (\kappa\alpha,\kappa\beta,\kappa\gamma,
    \lambda\alpha',\lambda\beta',\lambda\gamma')
   \;.
\label{def.Skl}
\ee
Then we have the obvious relation
\be
   T(n,k; S_{\kappa,\lambda}\bmu)
   \;=\;
   \kappa^{n-k} \lambda^k \, T(n,k;\bmu)
   \;.
\label{eq.TnkSkl}
\ee
Here the parameters $\kappa$ and $\lambda$ can be indeterminates,
or they can be real or complex numbers.
All the maps $S_{\kappa,\lambda}$ trivially commute:
we have
$S_{\kappa,\lambda} S_{\kappa',\lambda'}
 = S_{\kappa \kappa', \lambda \lambda'}$.
Thus, if we consider $\kappa$ and $\lambda$ to be nonzero complex numbers,
then the maps $S_{\kappa,\lambda}$
form a group isomorphic to $\C^* \times \C^*$,
where $\C^* \eqdef \C \setminus \{0\}$.

Later, we shall use in particular the scaling maps
$S = S_{-1,1}$, $S' = S_{1,-1}$ and $SS' = S_{-1,-1}$.
Of course these are commuting involutions,
which generate a group $\{1,S,S',SS'\}$ isomorphic to $\Z_2 \times \Z_2$.

\bigskip

Let us now turn to some discrete symmetries of the GKP recurrence.
The most important of these is:

\medskip

{\bf Duality} \cite{Theoret_94,Theoret_95a,Theoret_95b,BSV}.
Given a triangular array $\bT = \bigl( T(n,k) \bigr)_{0 \le k \le n}$,
let us define the {\em dual}\/ (or {\em reversed}\/) array
$\bT^* = \bigl( T^*(n,k) \bigr)_{0 \le k \le n}$ by
\be
   T^*(n,k)  \;\eqdef\; T(n,n-k) \,.
 \label{def_duality_binom}
\ee
The map $\bT \mapsto \bT^*$ is obviously an involution,
i.e.\ $(\bT^*)^* = \bT$.
In the GKP recurrence, the duality map $\bT \mapsto \bT^*$
can be implemented by the transformation of parameters
\cite[p.~67, Proposition~3.1.1]{Theoret_94}
\be
   \bmu \;\mapsto\; D\bmu 
   \;\eqdef\; 
   (\alpha' + \beta',-\beta',\gamma', \alpha + \beta, -\beta, \gamma) \,.
 \label{eq.duality.mu}
\ee
That is, we have
\be
   T(n,k; D\bmu)  \;=\; T(n,n-k;\bmu)
   \;.
\ee
Of course, the map $D$ is also an involution, i.e.\ $D(D(\bmu)) = \bmu$.
The corresponding generating functions transform as
\begin{subeqnarray}
\slabel{def_PnDmu}
 P_n(x;D\bmu)   & = &   x^n \, P_n\biggl( \frac{1}{x}; \bmu \biggr) \\ 
 f(x,t;D\bmu)   & = &   f\biggl( \frac{1}{x},xt; \bmu \biggr)  \\
 F(x,t;D\bmu)   & = &   F\biggl( \frac{1}{x},xt; \bmu \biggr)
\end{subeqnarray}

\smallskip

{\bf Remark.}
Neuwirth's \cite{Neuwirth} case $\alpha' = 0$
is dual to Spivey's \cite{Spivey_11} case (S1) $\beta = -\alpha$,
while Spivey's \cite{Spivey_11} cases
(S2) $\beta = \beta' = 0$
and (S3) $\alpha/\beta = \alpha'/\beta' + 1$
are self-dual.
\myendremark

\bigskip

Another important involution is the following:

\medskip

{\bf Zhu involution} \cite{Zhu_unpub}.
Define the map $Z \colon\, \bmu \mapsto Z\bmu$ by
\be
   Z \, (\alpha,\beta,\gamma, \alpha',\beta',\gamma')
   \;\eqdef\;
   \left( \alpha - \frac{\beta}{\beta'} \alpha', -\beta,
         -\beta + \gamma - \frac{\beta}{\beta'} \gamma', \,
             \alpha',\beta',\gamma'
   \right) \,.
 \label{def_transinv}
\ee  
A simple computation shows that $Z$ is an involution,
i.e.\ $Z(Z(\bmu)) = \bmu$.
Here we should either consider $\bmu$ to be indeterminates
and work in the polynomial ring $\Z[\bmu,(\beta')^{-1}]$,
or else consider $\bmu$ to be real or complex numbers
subject to the condition $\beta' \neq 0$.
The triangular array $\bT$ transforms as
\begin{subeqnarray}
    T(n,k; Z\bmu)  & = &  \sum_{j=k}^n T(n,j;\bmu) \, \binom{j}{k} \,
              \left( -\frac{\beta}{\beta'} \right)^{j-k}
          \slabel{eq.transinv.a} \\[2mm]
    T(n,k; \bmu)  & = &  \sum_{j=k}^n T(n,j; Z\bmu) \, \binom{j}{k} \,
              \left( \frac{\beta}{\beta'} \right)^{j-k}
          \slabel{eq.transinv.b}
          \label{eq.transinv}
\end{subeqnarray}
This has a simple expression in terms of the row-generating polynomials:
\be
   P_n(x;Z\bmu)
   \;=\;
   P_n \biggl(  x - \frac{\beta}{\beta'} \,; \, \bmu \biggr)
   \;.
 \label{eq.transinv_Pn}
\ee
Equivalently, the generating functions transform as
\begin{subeqnarray}
   f(x,t; Z\bmu)
   & = &
   f\biggl( x - \frac{\beta}{\beta'} , t;  \bmu \biggr)
       \\[2mm]
   F(x,t; Z\bmu)
   & = &
   F\biggl( x - \frac{\beta}{\beta'} , t;  \bmu \biggr)
\end{subeqnarray}
The identity \reff{eq.transinv.a} is a special case of
Corollary~\ref{cor6.prop.spivey.corollary5} in Appendix~\ref{app.spivey-zhu}.

Let us observe that that \eqref{eq.transinv.a} can be written
as a matrix product
\be
    \bT(Z\bmu) 
    \;=\;
    \bT(\bmu) \, B_{-\beta/\beta'}
\ee
where $B_\xi$ denotes the $\xi$-binomial matrix
\be
    B_\xi(n,k)  \;=\;  \binom{n}{k} \, \xi^{n-k} 
   \;.
 \label{def_Bxi}   
\ee
Equivalently, \eqref{eq.transinv.b} can be written as 
$\bT(\bmu) = \bT(Z\bmu) \, B_{\beta/\beta'}$.
Please note that the $\xi$-binomial matrix satisfies
$B_\xi B_{\xi'} = B_{\xi'} B_\xi = B_{\xi+\xi'}$
and hence in particular $(B_\xi)^{-1} = B_{-\xi}$.

A brute-force computation using $n=0,1,2,3$
shows that the {\em only}\/ solutions to the equations
$P_n(x;\bmu') = P_n(x + \xi; \bmu)$
valid for {\em generic}\/ parameters $\bmu$
are the identity map ($\xi=0$, $\bmu' = \bmu$)
and the map \reff{def_transinv}/\reff{eq.transinv_Pn}.

\bigskip

The dual of the Zhu involution is:

\medskip

{\bf ``Riordan'' involution} $\bm{R = DZD}$ \cite[Proposition~17]{BSV2}.
We have
\be
   R \, (\alpha,\beta,\gamma, \alpha',\beta',\gamma')
   \;\eqdef\;
   \left( \alpha,\beta,\gamma, \alpha'+\beta'-\frac{\beta'}{\beta} \alpha,
             -\beta',\gamma'+\beta'- \frac{\beta'}{\beta} \gamma
   \right) \,.
 \label{def_muhat}
\ee  
Here we should either consider $\bmu$ to be indeterminates
and work in the polynomial ring $\Z[\bmu,\beta^{-1}]$,
or else consider $\bmu$ to be real or complex numbers
subject to the condition $\beta \neq 0$.
The triangular array $\bT$ transforms as
\begin{subeqnarray}
    T(n,k; R\bmu)  & = &  \sum\limits_{j=0}^k T(n,j;\bmu) \,
     \binom{n-j}{n-k} \, \left( -\frac{\beta'}{\beta} \right)^{k-j}
          \slabel{eq.riordan.a} \\[2mm]
    T(n,k; \bmu)   & = &  \sum\limits_{j=0}^k T(n,j; R\bmu) \,
     \binom{n-j}{n-k} \, \left( \frac{\beta'}{\beta} \right)^{k-j}
          \slabel{eq.riordan.b}
          \label{eq.riordan}
\end{subeqnarray}
This has a simple expression in terms of the row-generating polynomials:
\be
   P_n(x;R\bmu)
   \;=\;
   \left(\frac{\beta - \beta' \, x}{\beta}\right)^n
   \, P_n \left(  \frac{ \beta \, x}{\beta - \beta' \, x} \,; \, \bmu
          \right)
   \;.
\ee
We say that $\bT(\bmu)$ and $\bT(R\bmu)$ form
an \textit{inverse pair} of lower-triangular arrays:
see Appendix~\ref{app.inverse}.

\subsection{Determination of the symmetry group}

Let us now consider the discrete group $\scrg$
generated by the three involutions $S$, $D$ and $Z$.
(This group also includes $S' = DSD$ and $R = DZD$.)
Of course $S^2 = D^2 = Z^2 = (S')^2 = 1$.
It is easy to see that
\begin{subeqnarray}
   & &  SZ \;=\; ZS   \\
   & &  S' Z \;=\; Z S'   \\
   & &  S S' \;=\; S' S  \\
   & &  (DS)^2 \;=\; (SD)^2 \;=\; SS'
\end{subeqnarray}
and hence $(DS)^4 = (SD)^4 = 1$.
A slightly more involved computation shows that
\be
   (DZ)^6 \;=\; (ZD)^6 \;=\; SS'
\ee
and hence $(DZ)^{12} = (ZD)^{12} = 1$.
This motivates defining the group element
\be
   X  \;\eqdef\; DZ  \;,
\ee
which acts on $\bmu$ as
\be
   X \, (\alpha,\beta,\gamma, \alpha',\beta',\gamma')
   \;=\;
   \left( \alpha' + \beta', -\beta', \gamma', \,
          \alpha-\beta- \frac{\beta}{\beta'} \alpha' ,
          \beta, \gamma - \beta - \frac{\beta}{\beta'} \gamma' 
   \right)
 \label{def_X}
\ee  
and on the row-generating polynomials as
\be
   P_n(x; X\bmu)
   \;=\;
   x^n \,  P_n \left(  \frac{1}{x} - \frac{\beta}{\beta'} \,; \, \bmu
               \right)
   \;.
\ee
We therefore consider $\scrg$ as generated by $X$
together with the two commuting involutions $S$ and $Z$.
These generators satisfy the relations
\begin{subeqnarray}
   X^{12} & = & 1  \\
   S^2    & = & 1  \\
   Z^2    & = & 1  \\
   SZ     & = & ZS   \\
   S X S  & = & X^7  \\
   Z X Z  & = & X^{11} \;=\; X^{-1}
 \label{eq.XSZ.relations}
\end{subeqnarray}
and we shall see shortly that there are no other independent relations.
Brute-force computation shows that the group $\scrg$ has 48 elements,
each of which corresponds to a distinct transformation $\bmu \mapsto \bmu'$.
There are two central elements:
the identity element~$1$, and the nontrivial central element $X^6 = SS'$.
The conjugacy classes of $\scrg$ are as follows
(here the integer labeling a conjugacy class denotes the order of its elements):
\begin{itemize}
   \item Class 1 (central, order 1, size 1):  $\{1\}$
   \item Class 2a (central, order 2, size 1): $\{ X^6 \}$
   \item Class 2b (order 2, size 2): $\{ S, SX^6 \}$
   \item Class 2c (order 2, size 2): $\{ SX^3, SX^9 \}$
   \item Class 2d (order 2, size 3): $\{ SZ, SZX^4, SZX^8 \}$
   \item Class 2e (order 2, size 3): $\{ SZX^2, SZX^6, SZX^{10} \}$
   \item Class 2f (order 2, size 6): $\{ Z, ZX^2, ZX^4, ZX^6, ZX^8, ZX^{10} \}$
   \item Class 2g (order 2, size 6): $\{ ZX, ZX^3, ZX^5, ZX^7, ZX^9, ZX^{11}\}$
   \item Class 3 (order 3, size 2): $\{ X^4, X^8  \}$
   \item Class 4a (order 4, size 2): $\{ X^3, X^9  \}$
   \item Class 4b (order 4, size 6): $\{ SZX, SZX^3, SZX^5, SZX^7, SZX^9, SZX^{11} \}$
   \item Class 6a (order 6, size 2): $\{ X^2, X^{10}  \}$
   \item Class 6b (order 6, size 4): $\{ SX, SX^5, SX^7, SX^{11}  \}$
   \item Class 6c (order 6, size 4): $\{ SX^2, SX^4, SX^8, SX^{10} \}$
   \item Class 12 (order 12, size 4): $\{ X, X^5, X^7, X^{11}  \}$
\end{itemize}

Let us now show that the group $\scrg$ can be identified
as the direct product $S_3 \times D_4$,
where $S_3$ is the permutation group on three letters
(or equivalently, the dihedral group $D_3$ of symmetries of an
 equilateral triangle)
and $D_4$ is the dihedral group of symmetries of the square:
\begin{subeqnarray}
   S_3 & = & \< a,b \colon\: a^3 = b^2 = 1, \, bab = a^{-1} \>
   \\[2mm]
   D_4 & = & \< c,d \colon\: c^4 = d^2 = 1, \, dcd = c^{-1} \>
 \slabel{def.D4}
\end{subeqnarray}
and hence
\begin{eqnarray}
   S_3 \times D_4
   & = &
   \< a,b,c,d \colon\: a^3 = b^2 = c^4 = d^2 = 1,
                           \, bab = a^{-1}, \, dcd = c^{-1},
    \qquad  \nonumber \\
   & & \qquad\qquad\;\; ac=ca, \, ad=da, \, bc=cb, \, bd=db \>
        \;.
 \label{def.S3xD4}
\end{eqnarray}
Assuming temporarily that $\scrg \simeq S_3 \times D_4$,
let us find the presentation.
Since $a$ is an element of order~3,
we must have $a = X^4$ or its inverse $X^8$;
the two choices are equivalent, so let us choose $a = X^4$.
Since $c$ is an element of order~4 that commutes with $a$,
we must have $c = X^3$ or its inverse $X^9$
(the other order-4 elements are ruled out
 because $SZ$ does not commute with $X^4$:
 from (\ref{eq.XSZ.relations}e,f) we see that
 $SZ X^4 SZ = X^{-4}$);
the two choices are equivalent, so let us choose $c = X^3$.
Note that $c^2 = X^6$ is a central element, as it should be,
and that $a$ commutes with $c$.
Then $b$ has to be a non-central order-2 element that commutes with $X^3$,
and $d$ has to be a non-central order-2 element that commutes with $X^4$.
Looking at classes 2b--2g and using (\ref{eq.XSZ.relations}d,e,f),
we see that the only choices
are $b = SZ X^{2m}$ for $m \in \{0,1,2,3,4,5\}$,
and $d = S X^{3n}$ for $n \in \{0,1,2,3\}$.
It is easily checked that all of these choices satisfy
$bd=db$, $bab = a^{-1}$ and $dcd = c^{-1}$.
For $b$ there are two inequivalent choices:
\begin{itemize}
   \item Choosing $m$ even means that the group $S_3$ generated by $a$ and $b$
      is $\{1, X^4, X^8,$ $SZ, SZX^4, SZX^8 \}$.
      We might as well choose $m=0$, hence $b = SZ$.
   \item Choosing $m$ odd means that the group $S_3$ generated by $a$ and $b$
      is $\{1, X^4, X^8,$ $SZX^2, SZX^6, SZX^{10} \} =
          \{1, X^4, X^8, S' Z, S' ZX^4, S' ZX^8 \}$
      We might as well choose $m=3$, hence $b = SZ X^6 = S' Z$.
\end{itemize}
The two inequivalent choices for $b$ are related by duality $S' = DSD$,
so there is no loss of generality in taking $b = SZ$.
On the other hand, all four choices of $d$
give rise to the same group $D_4$ generated by $c$ and $d$,
namely $\{1, X^3, X^6, X^9, S, SX^3, SX^6, SX^9 \}$.
So we might as well choose $n=0$, hence $d = S$.
We are therefore led to conjecture that $\scrg \simeq S_3 \times D_4$
with the identifications
\be
   a \,=\, X^4 ,\quad
   b \,=\, SZ ,\quad
   c \,=\, X^3 ,\quad
   d \,=\, S
   \;.
 \label{eq.S3xD4.abcd}
\ee
Since we have already checked all the relations \reff{def.S3xD4},
and we have verified that $\scrg$ has 48 elements ---
the same as $S_3 \times D_4$ ---
the conjecture is proven.

Within the group $\scrg$,
a special role is played by the subgroup $\scrg_0$
generated by $S$ and $D$ (or equivalently by $S$, $S'$ and $D$):
\be
   \scrg_0  \;=\;  \{ 1, S, S', SS', D, SD, S' D, SS' D \}
   \;.
\ee
This is a dihedral group $D_4$:
in the presentation \reff{def.D4} we can take $c = SD$ and $d = S$.
Please note, however, that this subgroup $D_4$ is {\em not}\/
the ``canonical'' subgroup $D_4$
that arises in the direct product $\scrg \simeq S_3 \times D_4$ with
the identifications \reff{eq.S3xD4.abcd}.
It is also useful to consider the conjugate subgroups
\begin{subeqnarray}
   \scrg_k  \;\eqdef\; X^{-k} \scrg_0 X^k
   & = &
   \{ 1, S, S', SS', D X^{2k}, SD X^{2k}, S' D X^{2k}, SS' D X^{2k} \}
         \\[2mm]
   & = &
   \{ 1, S, S', SS', Z X^{2k-1}, SZ X^{2k-1}, S' Z X^{2k-1}, SS' Z X^{2k-1} \}
           \nonumber \\
\end{subeqnarray}
for $k \in \Z$.
Since $\scrg_k$ involves $X^{2k}$,
and since the central element $X^6 = SS'$ belongs to $\scrg_0$,
the subgroup $\scrg_k$ depends only on $k \bmod 3$,
so it suffices to consider $k=0,1,2$.
Note that all three subgroups $\scrg_k$
contain the abelian subgroup $\{ 1, S, S', SS' \}$
and in particular the central element $SS' = X^6$.

The special role played by the subgroup $\scrg_0$ is the following:
If we examine the formulae $M \bmu$ for all elements $M \in \scrg$,
we find that $M \bmu$ is a polynomial in $\bmu$ if and only~if $M \in \scrg_0$.
In the other cases:
\begin{itemize}
   \item $M \bmu$ is a polynomial in $\bmu$ and $(\beta')^{-1}$ when
      $M \in \{X, X^7, SX, SX^7, Z, ZX^6,$ $SZ$, $SZX^6 \}$.
      This is the right coset $\scrg_0 X$,
      or equivalently the left coset $X \scrg_1$.
   \item $M \bmu$ is a polynomial in $\bmu$ and $\beta^{-1}$ when
      $M \in \{ X^5, X^{11}, SX^5, SX^{11}, ZX^4, ZX^{10},$ $SZX^4, SZX^{10} \}
       = \{X, X^7, SX, SX^7, R, RX^6, SR, SRX^6 \}$.
      This is the right coset $\scrg_0 X^5$,
      or equivalently the left coset $X^5 \scrg_5 = X^5 \scrg_2$.
   \item $M \bmu$ is a polynomial in $\bmu$, $\beta^{-1}$ and $(\beta')^{-1}$
      when $M$ is any of the other 24 elements of $\scrg$.
      These are the right cosets $\scrg_0 X^k$,
      or equivalently the left cosets $X^k \scrg_k$, for $k=2,3,4$.
\end{itemize}

\bigskip

{\bf Remark.}
The group $S_3 \times D_4$ has many other aliases:
see e.g.\ \cite{Dokchitser_GroupNames}.
\myendremark

\subsection{Further scaling properties of the GKP recurrence}
   \label{sec.scaling}

In Section~\ref{sec.sym} we studied several symmetries of the GKP recurrence,
and in particular the behavior of the triangular-array entries $T(n,k)$ 
under the trivial scaling map $S_{\kappa,\lambda}$
defined in \eqref{def.Skl}/\eqref{eq.TnkSkl}.  
We would now like to point out some less obvious scaling properties.
These properties are most naturally formulated in the first instance
for the more general ``binomial-like'' recurrence
\be
   T(n,k)
   \;=\;
   a_{n,k} \, T(n-1,k) \:+\: a'_{n,k} \, T(n-1,k-1)
 \label{eq.recursion}
\ee
for $n \ge 1$, again with initial condition $T(0,k) = \delta_{k0}$.\footnote{
   A detailed study of this type of recurrence
   can be found in the thesis of Th\'eor\^et \cite{Theoret_94}.
   He uses the term ``hyperbinomial''
   to denote the matrix $\bT = (T(n,k))_{n,k \ge 0}$
   satisfying such a recurrence.
}
Spivey \cite[Theorem~1]{Spivey_11} showed the following:

\begin{lemma}[Rescaling of a binomial-like recurrence]
   \label{lemma.recurrence.rescaling}
Let $(a_{n,k})$, $(a'_{n,k})$, $(c_n)$, $(d_k)$ and $(e_n)$
be indeterminates,
and define triangular arrays
$( T(n,k) )_{0 \le k \le n}$
and
$( T'(n,k) )_{0 \le k \le n}$
by the recurrences
\begin{eqnarray}
   T(n,k)  & = &  a_{n,k} \, T(n-1,k) \:+\: a'_{n,k} \, T(n-1,k-1)
      \label{eq.recurrence.T}   \\[2mm]
   T'(n,k) & = &  c_{n-k} e_n a_{n,k} \, T'(n-1,k) \:+\:
                  d_k e_n a'_{n,k} \, T'(n-1,k-1)
      \label{eq.recurrence.Tprime}
\end{eqnarray}
for $n \ge 1$, with initial conditions $T(0,k) = T'(0,k) = \delta_{k0}$.
Then
\be
   T'(n,k)  \;=\; c_1 \cdots c_{n-k} \:
                  d_1 \cdots d_k \:
                  e_1 \cdots e_n \:   T(n,k)
   \;.
 \label{eq.lemma.recurrence.rescaling}
\ee
\end{lemma}

\begin{proof}[Sketch of proof]
Represent $T(N,K)$
as a sum over walks in $\N \times \N$ from $(0,0)$ to $(N,K)$,
where a ``level step'' $(n-1,k) \to (n,k)$ gets a weight $a_{n,k}$
and a ``rise'' $(n-1,k-1) \to (n,k)$ gets a weight $a'_{n,k}$,
and the weight of a walk is the product of the weights of its steps;
and analogously for $T'(N,K)$.
The result \reff{eq.lemma.recurrence.rescaling}
then follows by simple arguments \cite{Sokal_totalpos}.
Alternatively \cite{Spivey_11,Spivey_16}
one can show that the quantities \reff{eq.lemma.recurrence.rescaling}
satisfy the recurrence \reff{eq.recurrence.Tprime}.
\end{proof}

Let us now specialize Lemma~\ref{lemma.recurrence.rescaling}
to the GKP recurrence
\be
   a_{n,k} \:=\: \alpha n + \beta k + \gamma \;,\qquad
   a'_{n,k} \:=\: \alpha' n + \beta' k + \gamma' \;.
\ee
It is easy to see that there are precisely three cases in which
the recurrence \reff{eq.recurrence.Tprime} for $T'(n,k)$
is also of the GKP form:
\begin{itemize}
   \item[(a)]  $\alpha=\beta=0$ and
               $c_n = \kappa n + \lambda$, $d_k = 1$, $e_n = 1$
   \item[(b)]  $\alpha'=\beta'=0$ and
               $d_k = \kappa k + \lambda$, $c_n = e_n = 1$
   \item[(c)]  $\alpha=\beta=\alpha'=\beta'=0$ and
               $e_n = \kappa n + \lambda$, $c_n = 1$, $d_k = 1$
\end{itemize}
We thus obtain:

\begin{corollary}[Rescaling of a GKP recurrence]
   \label{cor.scaling}
Let $\alpha,\beta,\gamma,\alpha',\beta',\gamma'$ and $\kappa,\lambda$ 
be indeterminates.  Then:
\begin{itemize}
   \item[(a)]  When $\alpha=\beta=0$ we have
\begin{multline}
   T(n,k;\, \kappa\gamma,0,\lambda\gamma,\alpha',\beta',\gamma')
   \;=\; \\
   (\kappa +\lambda)(2\kappa +\lambda)\cdots [(n-k)\kappa +\lambda] \:
   T(n,k;\, 0,0,\gamma,\alpha',\beta',\gamma')
   \;.
\end{multline} 
   \item[(b)]  When $\alpha'=\beta'=0$ we have
\begin{multline}
   T(n,k;\, \alpha,\beta,\gamma,0,\kappa\gamma',\lambda\gamma')
   \;=\; \\
   (\kappa +\lambda)(2\kappa +\lambda)\cdots (k\kappa +\lambda) \:
   T(n,k;\, \alpha,\beta,\gamma,0,0,\gamma')
   \;.
\end{multline} 
   \item[(c)]  When $\alpha=\beta=\alpha'=\beta'=0$ we have
\begin{multline}
   T(n,k;\, \kappa\gamma,0,\lambda\gamma,\kappa\gamma',0,\lambda\gamma')
   \;=\; \\
   (\kappa +\lambda)(2\kappa +\lambda)\cdots (n\kappa +\lambda) \:
   T(n,k;\, 0,0,\gamma,0,0,\gamma')
   \;.
\end{multline} 
\end{itemize}
\end{corollary}

%
%
\section{Characterization of the GKP recurrences with an S-fraction
    representation} \label{sec.main}

Our main goal in this paper is to determine all families of parameters 
$\bmu \in \C^6$ such that the corresponding ogf \eqref{def_ogf} has an
S-type continued fraction \reff{def_Stype.one}
with coefficients $c_1,c_2,\ldots$
that are {\em polynomials}\/ in~$x$ (rather than rational functions).
Our findings can be summarized in the following:

\begin{theorem}
   \label{theor.main} 
Let $\bmu=(\alpha,\beta,\gamma,\alpha',\beta',\gamma')\in \C^6$,
and let $\bT(\bmu) = \bigl( T(n,k;\bmu) \bigr)_{0 \le k \le n}$
be the triangular array determined via the recurrence \eqref{eq_binomvert}.
Then the corresponding ogf $f(x,t;\bmu)$
has a \emph{nonterminating} S-type continued fraction representation
(in the sense of formal power series in the indeterminate $t$) 
\be
f(x,t;\bmu) \;=\; \sum\limits_{n=0}^\infty P_n(x;\bmu) \, t^n
       \;=\; \cfrac{1}{1 - \cfrac{c_1 t}{ 1 -
                             \cfrac{c_2 t}{ 1 - \cdots}}}
\label{def_Stype.three}
\ee
with coefficients $c_1,c_2,\ldots$ that are \emph{polynomials} in $x$
only~if the parameter $\bmu$ belongs to
one or more of the following families: 
\begin{quote}
\begin{itemize} 
\item[F1a.\ ] $\bmu=(0,\beta,0,\alpha',-\alpha',\gamma')$
\item[F1b.\ ] $\bmu=(0,\beta,\gamma,\alpha',-\alpha',0)$
\item[F2a.\ ] $\bmu=(\alpha,-\alpha,-\alpha,\alpha',\beta',\gamma')$
\item[F2b.\ ] $\bmu=(\alpha,\beta,\gamma, 0,\beta',-\beta')$
\item[F3a.\ ] $\bmu=(0,\beta,0, 0,\beta',\gamma')$
\item[F3b.\ ] $\bmu=(\alpha,-\alpha,\gamma, \alpha',-\alpha',0)$
\item[F4a.\ ] $\bmu=(0,\kappa\beta',\kappa(\beta'+\gamma'), 0,\beta', \gamma')$
\item[F4b.\ ] $\bmu=(\alpha,-\alpha,\gamma,
                       \kappa\alpha,-\kappa\alpha,\kappa(\alpha+\gamma))$
\item[F5.\ ] $\bmu=(\alpha,0,\gamma, \alpha',0,\gamma')$
\item[F6.\ ] $\bmu=(\kappa(\alpha'+\beta'),\kappa\beta',
                       \kappa\gamma',\alpha',\beta',\gamma')$
\end{itemize} 
\end{quote}
Moreover, in all these cases $f(x,t;\bmu)$ does have a representation
\reff{def_Stype.three} with coefficients $c_1,c_2,\ldots$ that are
polynomials in $x$
(though in some degenerate cases this continued fraction might be terminating).
\end{theorem}

Here families 1a and 1b are duals of each other,
and likewise for the pairs (2a,2b), (3a,3b) and (4a,4b);
families 5 and 6 are self-dual.
Note that families 1a, 1b, 3a, 3b, 4a, 4b have three free parameters,
while families 2a, 2b, 5, 6 have four free parameters;
that is, the submanifolds in $\bmu$-space are of
(complex) dimension 3 and 4, respectively.
However, we shall see later that in families 2a, 2b and 6
some of the parameters are redundant
in the sense that they have no effect on the matrix $\bT(\bmu)$
[or therefore on the coefficients $c_i$].
In these cases the manifold of matrices $\bT(\bmu)$
is of lower dimension than the manifold of parameters $\bmu$:
it is of (complex) dimension~2 for families 2a and 2b,
and (complex) dimension~3 for family 6.
Note, finally, that these families have some overlaps:
for instance, family 1a at $\alpha' = 0$
coincides with family 3a at $\beta' = 0$;
and family 2a at $\alpha=\alpha'=0$
coincides with family 3a at $\beta = 0$.

Having already observed how duality ($D$) acts on our ten families,
we can also work out how the rest of the group $\scrg$ acts on them.
Obviously $S$ and $S'$ map each family into itself.
The Zhu involution $Z$ maps
\hbox{1a $\leftrightarrow$ 4b,}
\hbox{1b $\leftrightarrow$ 3b,}
\hbox{2a $\leftrightarrow$ 6,}
\hbox{3a $\leftrightarrow$ 4a,}
and 2b into itself;
$Z$ cannot be defined on family~5 because it requires $\beta' \neq 0$.
Equivalently, the Riordan involution $R = DZD$ maps
\hbox{1a $\leftrightarrow$ 3a,}
\hbox{1b $\leftrightarrow$ 4a,}
\hbox{2b $\leftrightarrow$ 6,}
\hbox{3b $\leftrightarrow$ 4b,}
and 2a into itself;
$R$ cannot be defined on family~5 because it requires $\beta \neq 0$.
Finally, the map $X$ (which is not an involution) maps cyclically
\hbox{1a $\to$ 4a $\to$ 3b $\to$ 1a,}
\hbox{1b $\to$ 3a $\to$ 4b $\to$ 1b,}
\hbox{2a $\to$ 6 $\to$ 2b $\to$ 2a;}
$X$ cannot be defined on family~5 because it requires $\beta' \neq 0$.
It follows from this that the group $\scrg$
has one orbit $\{ \textrm{1a,1b,3a,3b,4a,4b} \}$
and another orbit $\{ \textrm{2a,2b,6} \}$.
On family~5 only the subgroup $\scrg_0 \subset \scrg$
generated by $S,S',D$ acts,
and it leaves family~5 invariant.

The details of the continued fractions
for each of the above families will be presented
in Propositions~\ref{prop.I}--\ref{prop.VI} below.
We will find that the coefficients $\bc=(c_i)_{i\ge 1}$
in each S-fraction are not only polynomials in $x$
(with coefficients that might be rational functions of the parameters $\bmu$);
in fact, they are polynomials  (with nonnegative integer coefficients)
simultaneously in {\em all}\/ the relevant variables
(in each case some subset of
 $\alpha,\beta,\gamma,\alpha',\beta',\gamma'$ and $\kappa$ as~well as~$x$)
when those variables are treated as indeterminates.
The restriction to {\em nonterminating}\/ S-fractions
will be explained in Section~\ref{sec.search}.

\bigskip

The strategy for our proof is the following: 
\begin{enumerate}
\item Make a computer-assisted search to find \emph{all} viable candidates to 
      have such an S-fraction representation (see Section~\ref{sec.search}). 
      See the Supplementary Material for a full account of this procedure. 
\item Prove (by specializing a continued fraction to be
      explained in Section~\ref{sec.SZ})
      that the ogf of each candidate found in step~1
      indeed has the conjectured S-fraction representation
      (see Section~\ref{sec.main2}). 
\end{enumerate}

%
%
\subsection{Computer-assisted search} \label{sec.search} 

Let $\ba = (a_n)_{n \ge 0}$ be a sequence in which $a_0 = 1$
and $a_1,a_2,\ldots$ are indeterminates.
Then the ogf $\sum\limits_{n \ge 0} a_n t^n$
can be expressed as an S-fraction \reff{def_Stype.one}
with coefficients $c_1,c_2,\ldots$
that are {\em rational functions}\/ of $a_1,a_2,\ldots\:$:
that is, $c_n \in \Q(a_1,\ldots,a_n)$.
The first few $c_i$ are given in \reff{eq.firstfew_ci}.

If we then substitute $a_n = P_n(x)$
where the $P_n$ are polynomials with complex coefficients in $x$
(and $P_0 = 1$),
then the coefficients $c_i$ are {\em generically}\/
rational functions of $x$.
We say ``generically'' because it is possible that
some coefficient $c_i$ might vanish,
in which case the coefficients $c_j$ for $j > i$ are ill-defined
because their denominator vanishes.
Here we restrict attention to {\em nonterminating}\/ S-fractions,
i.e.\ we insist that no $c_i$ vanishes identically.
Under this assumption, all the coefficients $c_i$
are well-defined and are rational functions of $x$ (with complex coefficients):
that is, $c_i \in \C(x)$.

We now specialize to the case in which the polynomials $P_n$
are those coming from the GKP recurrence \reff{eq_binomvert},
i.e.\ $a_n = P_n(x;\bmu)$ for some $\bmu \in \C^6$.
Our goal is to determine the submanifolds in $\C^6$
where the coefficients $c_i$ are {\em polynomials}\/
--- and not just rational functions --- in $x$.
We do this by making a computer-assisted search using {\sc Mathematica}
to find a finite set of viable candidates for such an S-fraction;
then we prove, case by case, that all of these candidates work.
It is convenient to explain this computer-assisted search
by showing explicitly its first few steps.

The row-generating polynomial $P_n(x;\bmu)$
is a polynomial of degree~$n$ in~$x$
with coefficients that are (fairly complicated) polynomials
with integer coefficients in the parameters $\bmu$.
For instance, the first few $P_n$ are
\begin{subeqnarray}
   P_0(x)   & = &   1   \\[2mm]
   P_1(x)   & = &   (\alpha+\gamma) \,+\, (\alpha'+\beta'+\gamma') x   \\[2mm]
   P_2(x)   & = &   (\alpha + \gamma) (2 \alpha + \gamma)
      \nonumber \\
    & & \;
    \,+\,
    (4 \alpha \alpha' + \beta \alpha' + 3 \alpha \beta' + \beta \beta' + 3 \gamma \alpha' + 2 \gamma \beta' + 3 \alpha \gamma' + \beta \gamma' + 2 \gamma \gamma') x   \qquad
     \nonumber \\
    & & \;
    \,+\,
   [(\alpha' + \beta' + \gamma') (2 \alpha' + 2 \beta' + \gamma')] x^2
\end{subeqnarray}
We now insert these $a_n = P_n(x;\bmu)$
into the equations \reff{eq.firstfew_ci} giving the coefficients $c_i$.

\bigskip

{\bf Coefficient $\bm{c_1}$.}
The first coefficient $c_1 = P_1(x)$ is of course a polynomial in $x$,
which is of degree at most~1.
More precisely, it is of degree~1 if $\alpha'+\beta'+\gamma' \neq 0$,
of degree~0 if $\alpha'+\beta'+\gamma' = 0$ and $\alpha+\gamma \neq 0$,
and identically zero if $\alpha'+\beta'+\gamma' = \alpha+\gamma = 0$.
This third case is now discarded because of our restriction
to nonterminating S-fractions;
and as will be seen shortly, we will need to consider
the first two cases separately.

\bigskip

{\bf Coefficient $\bm{c_2}$.}
The next coefficient $c_2$ is
\be
   c_2(x) \;=\; \frac{P_2(x)-P_1(x)^2}{P_1(x)} \;\eqdef\; 
             \frac{Q(x)}{R(x)} \,,
 \label{eq.c2}
\ee
where the denominator $R(x)$ is a polynomial of degree $\le 1$ in $x$
and the numerator $Q(x)$ is a polynomial of degree $\le 2$ in $x$,
and the coefficients in these polynomials are polynomial expressions in $\bmu$.
We need to determine the conditions under which such a rational function
$Q(x)/R(x)$ is in fact a polynomial in~$x$.
Let us explain the idea in general.

Let $Q(x) = \sum\limits_{k=0}^m a_k(\bmu) \, x^k$
and $R(x) = \sum\limits_{k=0}^n b_k(\bmu) \, x^k$
be polynomials in $x$ in which the coefficients $a_k$ and $b_k$
are polynomial expressions in some set of indeterminates $\bmu$.
The polynomial remainder $\rem(Q(x),R(x))$ is then well-defined
whenever the leading coefficient $b_n$
of the denominator polynomial is nonzero,
and in this case it is a polynomial in $x$ of degree $< n$.
On the other hand, if $b_n = 0$ and $b_{n-1} \neq 0$,
then we must consider $R$ as a polynomial of degree $n-1$,
yielding a remainder that is of degree $< n-1$;
or if $b_n = b_{n-1} = 0$ and $b_{n-2} \neq 0$,
then we must consider $R$ as a polynomial of degree $n-2$,
yielding a remainder that is of degree $< n-2$;
and so on.
The last nontrivial case is $b_n = \ldots = b_1 = 0$ and $b_0 \neq 0$:
then $R$ is a nonzero constant function and $Q/R$ is automatically a polynomial.
The completely degenerate case $b_n = \ldots = b_0 = 0$
(i.e.\ $R=0$) is excluded by our restriction to nonterminating S-fractions.
There are therefore $n+1$ nontrivial cases,
corresponding to the order of the highest nonvanishing coefficient $b_i$
($0 \le i \le n$).
For each of these cases, we obtain a set of polynomial equations
and inequations in the indeterminates $\bmu$:
namely, the equations $b_n = \ldots = b_{i+1} = 0$,
the inequation $b_i \neq 0$,
and then the equations asserting that
all the coefficients of the remainder polynomial are zero
(since this remainder polynomial has degree at most $i-1$,
 there are $i$ such equations).\footnote{
   To compute this remainder polynomial,
   we work in the ring $\C(\bmu)[x]$ of polynomials in $x$
   whose coefficients are rational functions (with complex coefficients)
   in the indeterminates $\bmu$.
}
This system of equations and inequations must then be solved,
yielding some algebraic variety in the space of parameters $\bmu$
(in our case $\C^6$).

Let us examine from this point of view
the coefficient $c_2$ given by \reff{eq.c2}.
The numerator and denominator polynomials are
\begin{subeqnarray}
   Q(x) \;=\; P_2(x)-P_1(x)^2
   & = &
   \alpha (\alpha + \gamma)
             \nonumber \\
   &   &
   \;+\,
   (2 \alpha \alpha' + \alpha' \beta + \alpha \beta' + \beta \beta' + \alpha' \gamma + \alpha \gamma' + \beta \gamma') x
             \nonumber \\
   &   &
   \;+\, (\alpha' + \beta') (\alpha' + \beta' + \gamma') x^2
       \\[2mm]
   R(x) \;=\; P_1(x)
   & = &
   (\alpha+\gamma) \,+\, (\alpha'+\beta'+\gamma') x
\end{subeqnarray}
Note that the polynomial $R(x)$ cannot be identically zero,
 because $R(x) = c_1(x)$ and we have already assumed that
$c_1(x)$ is not identically zero:
that is, $\alpha+\gamma \neq 0$ or $\alpha'+\beta'+\gamma' \neq 0$ or both.
So we have either $\deg R = 1$ or $\deg R = 0$.
Then $c_2(x) = Q(x)/R(x)$ is a polynomial in $x$ if and only~if either
\begin{itemize}
   \item[(a)]  {\bf $\bm{\deg R = 1}$:}
      The leading coefficient $\alpha'+\beta'+\gamma'$
      of the denominator polynomial $R(x)$ is nonzero,
      and the remainder
\be
   \rem(Q(x),R(x)) \;=\; \frac{\alpha+\gamma}{\alpha'+\beta'+\gamma'}
   \,
  \bigl[ (\alpha+\gamma)\beta' \,-\, (\alpha'+\beta'+\gamma')\beta \bigr]
\ee
      is zero;  or else
   \item[(b)]  {\bf $\bm{\deg R = 0}$:}
      The leading coefficient $\alpha'+\beta'+\gamma'$
      of the denominator polynomial $R(x)$ is zero,
      and the constant term $\alpha + \gamma$ is nonzero.
      In this case no constraint is imposed on the
      numerator polynomial $Q(x)$.
\end{itemize}
Therefore, we should consider three distinct cases: 
\begin{itemize}
   \item[(a${}_1$)] $\alpha'+\beta'+\gamma' \neq 0$ and $\alpha + \gamma=0$: 
           Then $\bmu =
           (\alpha,\beta,-\alpha,\alpha',\beta',\gamma')$ 
           with $\alpha'+\beta'+\gamma'\neq 0$,
           and $c_2(x) = (\alpha+\beta) + (\alpha' + \beta')x$.
   \item[(a${}_2$)] $\alpha'+\beta'+\gamma' \neq 0$ and
           $(\alpha+\gamma)\beta' = (\alpha'+\beta'+\gamma')\beta$:
           Then $\bmu =
           \Bigl(\alpha,\displaystyle\frac{(\alpha+\gamma) \beta'}{\alpha'+\beta'+\gamma'},
           \gamma,\alpha',\beta',\gamma'\Bigr)$
           with $\alpha'+\beta'+\gamma' \neq 0$,
           and $c_2(x) = \alpha + (\alpha' + \beta')x$.
   \item[(b)] $\alpha'+\beta'+\gamma'=0$ and $\alpha+\gamma \neq 0$.
           Then $\bmu = 
           (\alpha,\beta,\gamma,\alpha',\beta',-(\alpha'+\beta'))$,
           and $c_2(x) = \alpha + \alpha' x$.
\end{itemize}
In each family, the number of independent parameters in $\bmu$ is five.
And in each case we will impose, going forward, a disjunction of inequations
to guarantee that the polynomial $c_2(x)$ is not identically zero:
for instance, in case (a${}_1$) we will impose
$\alpha+\beta \neq 0$ or $\alpha' + \beta' \neq 0$.

\bigskip

{\bf Coefficient $\bm{c_3}$.}
For each of the three cases encountered in the analysis of $c_2$,
we consider separately the formula \reff{eq.firstfew_ci.c3} for $c_3$,
which becomes
\be
c_3(x) \;=\;\frac{P_1(x)P_3(x)-P_2(x)^2}{P_1(x)[P_1(x)^2-P_2(x)]} 
       \;\eqdef\; \frac{Q_3(x)}{R_3(x)}
\label{eq_a3}
\ee 
for some polynomials $Q_3(x)$ and $R_3(x)$.
If $P_1,P_2,P_3$ were arbitrary polynomials of degree 1,2,3, respectively,
one would expect generically that
the numerator polynomial $Q_3$ has degree~4
and the denominator polynomial $R_3$ has degree~3.
And indeed this is what happens if one inserts the
general expressions for $P_n(x;\bmu)$.
However, in the three special cases (a${}_1$), (a${}_2$) and (b),
the equality conditions cause the degrees of $Q_3$ and $R_3$
to be reduced to $\le 2$ and $\le 1$, respectively.
Furthermore, in all cases the denominator $R_3(x)$
turns out to be a nonzero multiple of $c_2(x)$;
so as a consequence of the inequations imposed at the preceding stage,
the polynomial $R_3(x)$ cannot be the zero polynomial.
Therefore, in each of these three cases
the remainder $\rem(Q_3(x),R_3(x))$ is simply a constant,
just as it was for $c_2$.
One thus repeats, for each of these three cases,
an analysis along the same lines as was done for $c_2$:
firstly distinguishing whether the leading coefficient of $R_3$
is nonzero or zero
(i.e.\ whether $\deg R_3 = 1$ or 0),
and then dividing the first case according to the different factors
in the remainder.
In each sub-case it turns out that we must impose one additional equality
beyond those imposed at the preceding stage;
the equation to be solved is always of degree~1
in at least one of the variables,
so we can solve for this variable as a (possibly rational)
function of the others.
At the end one finds that case (a${}_1$) divides into four sub-cases,
case (a${}_2$) divides into four sub-cases,
and case (b) divides into three sub-cases.
In this way we find 11 subfamilies in which $c_i(x)$ is a polynomial in $x$
for all $i \le 3$.
In ten of the 11 subfamilies
there are now four independent parameters in $\bmu$,
while in one there are only three.

\bigskip

{\bf Coefficients $\bm{c_4,c_5,\ldots\;}$.}
The procedure should by now be clear:
at each stage we divide each previously found case into sub-cases,
thus forming a decision tree.
Some branches of this tree terminate,
i.e.\ the solution set of the system of equations is empty.
It was necessary to go up to level $c_7$
in order to find all the terminating branches.
For each of the branches that survived to level $c_7$,
we continued this procedure through level $c_{10}$,
just to convince ourselves that the candidate was indeed viable.
Thus, we declared a branch to be viable
if, for the corresponding family $\bmu$,
the coefficients $c_i$ for $1\le i\le 10$
are polynomials in $x$ (always of degree $\le 1$)
that are not identically vanishing.
For each such branch, the coefficients $c_i$
turned out to have very simple forms
(namely, affine in $i$ separately odd and even $i$),
from which we were able to conjecture
a precise formula for all the coefficients $c_i(x;\bmu)$.

We also streamlined this procedure as follows:
whenever we obtained a branch of the decision tree
that is either identical to or a subset of a previously found branch,
then we dropped that branch as redundant.

Along the way, we sometimes obtained branches in which the most recent
coefficient $c_i(x)$ is identically vanishing.
These branches correspond to terminating S-fractions
and were therefore dropped.
Some of these terminating S-fractions are simply special cases
of the families 1a--6 of Theorem~\ref{theor.main},
in which one of the coefficients $c_i(x)$ happens to vanish.
(Indeed, there are infinitely many such special cases,
 corresponding to the choice of the index $i$ of the vanishing coefficient.)
On the other hand, a few of the terminating S-fractions are non-trivial
in the sense that they are {\em not}\/ simply specializations of
the families 1a--6.  Since these may be of some interest in their own right,
we compile them in Section~4 of the Supplementary Material.

It is also worthy of note that, at every stage in this procedure,
the simplification noted above for $c_3(x)$ [cf.\ \reff{eq_a3}] occurred:
namely, the polynomials $Q_i(x)$ and $R_i(x)$,
which one would expect generically to be
of degrees $\binom{i}{2} + 1$ and $\binom{i}{2}$,
respectively, turned out to have degrees~2 and 1 (or in a few cases even less),
as a result of the equality conditions imposed at earlier stages.
We leave it as an open problem to find a general explanation
of this remarkable simplification.

The details of this computer-assisted search are given in the
Supplementary Material.
Let us stress that our search was ``computer-assisted''
{\em only}\/ in the sense that it used {\sc Mathematica}
to perform elementary algebraic operations such as manipulation of
polynomials and rational functions, series expansions,
and {\tt PolynomialRemainder}.
We did {\em not}\/ need to use any more advanced algebraic functions
such as {\tt Solve} or {\tt Reduce}.
If one assumes the correctness of {\sc Mathematica}'s algebraic manipulations,
the rest of the proof is easily human-verifiable (though tedious)
and is explained in detail in the Supplementary Material.

The final list of candidates contains 10 members. 
Seven of them are families 1a, 1b, 2a, 2b, 3a, 3b and 5
of Theorem~\ref{theor.main}.
In these cases, the entries of $\bmu$ are linear combinations
of a subset of the parameters $(\alpha,\beta,\gamma,\alpha',\beta',\gamma')$,
and the conjectured coefficients $c_i(x;\bmu)$ are polynomials
(with nonnegative integer coefficients)
jointly in $x$ and the parameters.
Their expressions are:
\begin{quote}
\begin{itemize}
\item[F1a.] $c_{2k-1} = [\gamma' + (k-1)\alpha'] x$ and 
            $c_{2k}   = k \beta$.  
\item[F1b.] $c_{2k-1} = \gamma + (k-1)\beta$ and 
            $c_{2k}   = k \alpha' x$.  
\item[F2a.] $c_{2k-1} = [\gamma' + k(\alpha' + \beta')] x$ and 
            $c_{2k}   = k (\alpha'+\beta') x$.  
\item[F2b.] $c_{2k-1} = k \alpha + \gamma$ and 
            $c_{2k}   = k \alpha$.  
\item[F3a.] $c_{2k-1} = (\gamma' + k \beta') x$ and 
            $c_{2k}   = k (\beta+\beta' x)$.  
\item[F3b.] $c_{2k-1} = k \alpha + \gamma$ and 
            $c_{2k}   = k (\alpha + \alpha' x)$.  
\item[F5.]  $c_{2k-1} = (\gamma + \gamma' x) + k(\alpha + \alpha' x)$ 
        and $c_{2k}   = k (\alpha + \alpha' x)$.  
\end{itemize}
\end{quote}
For these families we therefore find \emph{a posteriori} that the
S-fraction \eqref{def_Stype.three} might also be valid when the
$\bmu$ (or rather, some subset of them) are treated as indeterminates. 

The other three cases require additional discussion.
For instance, family~6 comes originally from 
\be
\bmu \;=\; \left( \alpha, \frac{\alpha\, \beta'}{\alpha' + \beta'},
   \frac{\alpha\, \gamma'}{\alpha' + \beta'}, 
   \alpha', \beta', \gamma' \right) \,,
\label{def_F6_old}
\ee
leading to the coefficients
\be
c_{2k-1} \;=\; [\gamma' + k(\alpha'+\beta')] \left(\frac{\alpha}{\alpha' + \beta'} 
                + x\right) \,, \quad 
c_{2k} \;=\; k(\alpha + (\alpha' + \beta')x) \,.
\label{def_c_F6_old}
\ee
These coefficients are clearly \emph{not} polynomials in the parameters
$\alpha',\beta'$. However, we can remedy this problem by making the 
change of parameters $\alpha \eqdef \kappa (\alpha' + \beta')$. Then
\eqref{def_F6_old} reduces to the form given in Theorem~\ref{theor.main} for
family~6, and the coefficients \eqref{def_c_F6_old} become polynomials
(with nonnegative integer coefficients)
jointly in $x$ and $\alpha',\beta',\gamma',\kappa$.
A similar procedure handles families 4a and 4b.
Once these changes of parameters have been performed,
the coefficients $c_i$ are polynomials in $x$ and the chosen parameters:
\begin{quote}
\begin{itemize}
\item[F4a.] $c_{2k-1} = (\gamma' + k \beta') (\kappa+x)$ and 
            $c_{2k}   = k \beta' x$.  
\item[F4b.] $c_{2k-1} = (\gamma + k \alpha)(1 + \kappa x)$ and 
            $c_{2k}   = k \alpha$.  
\item[F6.]  $c_{2k-1} = [\gamma' + k (\alpha' + \beta')] (\kappa+x)$ 
        and $c_{2k}   = k (\alpha' + \beta')(\kappa + x)$.  
\end{itemize}
\end{quote}
We thus again find \emph{a posteriori} that the
S-fraction \eqref{def_Stype.three} might also be valid when the
chosen parameters (now including $\kappa$) are treated as indeterminates. 

Of course, all the candidate families and their S-fractions
are still at this point merely conjectures.
We will prove them in Section~\ref{sec.main2},
using some general results to be presented in Section~\ref{sec.SZ}.

%
%
\subsection{A ``master S-fraction'' for permutations} \label{sec.SZ} 

Euler \cite[section~21]{Euler_1760}\footnote{
   The paper \cite{Euler_1760},
   which is E247 in Enestr\"om's \cite{Enestrom_13} catalogue,
   was probably written circa 1746;
   it~was presented to the St.~Petersburg Academy in 1753,
   and published in 1760.
}
showed that the generating function of the factorials
can be represented as a beautiful S-fraction,
\be
   \sum_{n=0}^\infty n! \: t^n
   \;=\;
   \cfrac{1}{1 - \cfrac{1t}{1 - \cfrac{1t}{1 - \cfrac{2t}{1- \cfrac{2t}{1-\cdots}}}}}
   \;,
 \label{eq.nfact.contfrac}
\ee
with coefficients $c_{2k-1} = c_{2k} = k$.
Inspired by \reff{eq.nfact.contfrac},
Sokal and Zeng \cite{SZ} introduced the polynomials $\scrp_n(w,y,u,v)$ defined
by the continued fraction
\be
   \sum_{n=0}^\infty \scrp_n(w,y,u,v) \: t^n
   \;=\;
   \cfrac{1}{1 - \cfrac{wt}{1 - \cfrac{yt}{1 - \cfrac{(w+u)t}{1- \cfrac{(y+v)t}{1 - \cdots}}}}}
 \label{eq.eulerian.fourvar.contfrac}
\ee
with coefficients
\begin{subeqnarray}
   c_{2k-1}  & = &  w + (k-1) u \\
   c_{2k}    & = &  y + (k-1) v
 \label{def.weights.eulerian.fourvar}
\end{subeqnarray}
Clearly $\scrp_n(w,y,u,v)$ is a homogeneous polynomial of degree $n$;
it therefore has three ``truly independent'' variables.
Since $\scrp_n(1,1,1,1) = n!$, which enumerates permutations of an $n$-element set,
it is natural to expect that $\scrp_n(w,y,u,v)$ enumerates permutations of $[n]$
according to some natural trivariate statistic.
Sokal and Zeng \cite{SZ} gave
two alternative versions of this trivariate statistic;
here is one of them:

\begin{theorem}[Sokal--Zeng]
   \label{thm.perms.S}
The polynomials $\scrp_n(w,y,u,v)$ defined by
\reff{eq.eulerian.fourvar.contfrac}/\reff{def.weights.eulerian.fourvar}
have the combinatorial interpretation
\be
   \scrp_n(w,y,u,v)
   \;=\;
   \sum_{\sigma \in \Sym_n}
      w^{\cyc(\sigma)} y^{\erec(\sigma)}
         u^{n - \exc(\sigma) - \cyc(\sigma)} v^{\exc(\sigma) - \erec(\sigma)}
   \;,
 \label{eq.eulerian.fourvar.cyc}
\ee
where the sum runs over the set $\Sym_n$ of permutations of $[n]$.
\end{theorem}

\noindent
The permutation statistics appearing in \reff{eq.eulerian.fourvar.cyc}
are defined as follows.
Given a permutation $\sigma \in \Sym_n$, an index $i\in [n]$ is called a
\begin{itemize}
\item {\em record}\/ (rec) (or {\em left-to-right maximum}\/)
         if $\sigma(j) < \sigma(i)$ for all $j < i$
      [note that index 1 is always a record]; 
\item {\em antirecord}\/ (arec) (or {\em right-to-left minimum}\/)
         if $\sigma(j) > \sigma(i)$ for all $j > i$
      [note that index $n$ is always an antirecord];
\item {\em exclusive record}\/ (erec) if it is a record and not also
         an antirecord;
\item {\em exclusive antirecord}\/ (earec) if it is an antirecord and not also
      a record.
\end{itemize}
Also, an index $i\in [n]$ is called an \emph{excedance} if $\sigma(i) > i$.
The number of exclusive records, excedances and cycles of $\sigma$
are denoted by $\erec(\sigma)$, $\exc(\sigma)$ and $\cyc(\sigma)$,
respectively. 

\medskip

{\bf Remark.} The foregoing quantities are a proper subset of a more elaborate
classification of the permutations $\sigma\in \Sym_n$
--- the ``record-and-cycle classification'' into ten disjoint categories
\cite{SZ} ---
but we do not need this refinement here.
\myendremark

\bigskip

For all the candidate families in Theorem~\ref{theor.main},
their conjectured S-fraction coefficients (see Section~\ref{sec.search})
belong to a very special subclass of Theorem~\ref{thm.perms.S},
namely $c_{2k} \propto k$, so that $v=y$:
\be
   \scrp_n(w,y,u,y)
   \;=\; \sum_{\sigma \in \Sym_n} w^{\cyc(\sigma)} y^{\exc(\sigma)}
                 u^{n - \exc(\sigma) - \cyc(\sigma)}
   \;.
 \label{eq.eulerian.fourvar.cyc.v=y}
\ee
These latter polynomials have a nice explicit formula, as follows:

\begin{proposition}
   \label{prop.twovar.Ma.Salas}
The polynomial $\scrp_n(w,y,u,y)$ defined by \reff{eq.eulerian.fourvar.cyc.v=y}
can be written as
\begin{subeqnarray}
   \scrp_n(w,y,u,y)
   & = &
   \sum_{r=0}^n \stirlingsubset{n}{r} \, (y-u)^{n-r} \,
      \sum_{i=0}^r \stirlingcycle{r}{i} \, w^i \, u^{r-i}
          \\[2mm]
   & = &
   \sum_{r=0}^n \stirlingsubset{n}{r} \, (y-u)^{n-r} \,
      \prod\limits_{k=0}^{r-1} (w+ku)
   \;,
 \slabel{eq.twovar.Ma.Salas2}
 \label{eq.twovar.Ma.Salas}
\end{subeqnarray}
where $\stirlingsubset{n}{k}$ denotes the number of partitions
of an $n$-element set into $k$ nonempty blocks,
and $\stirlingcycle{n}{k}$ denotes the number of permutations
of an $n$-element set with $k$ cycles.

Furthermore, the polynomials $\scrp_n(w,y,u,y)$
have the exponential generating function
\be
   \sum_{n=0}^\infty \scrp_n(w,y,u,y) \: \frac{t^n}{n!}
   \;=\;
   \biggl( \frac{y-u}{y - u e^{(y-u)t}} \biggr) ^{\! w/u}
   \;.
 \label{eq.twovar.Ma.Salas.egf}
\ee
\end{proposition}

We shall prove the formulae (\ref{eq.twovar.Ma.Salas}a,b)
by making use of a result due to Ma \cite[Corollary~2.3]{Ma}:

\begin{lemma}[Ma]
   \label{lemma.Ma} 
For $n,i,\ell \ge 0$ with $i+\ell \le n$, define
\be
   A_n(i,\ell)
   \;\eqdef\;
   \bigl|  \{ \sigma \in \Sym_n \colon\: \cyc(\sigma)=i
                   \hbox{ and } \exc(\sigma) = \ell   \}
   \bigr|
   \;.
\ee
Then
\be
   A_n(i,\ell)
   \;=\;
   \sum\limits_{r=i}^{n-\ell} (-1)^{n-r-\ell} \, \stirlingsubset{n}{r} \,
       \stirlingcycle{r}{i} \, \binom{n-r}{\ell}
   \;.
 \label{eq.lemma.Ma}
\ee
\end{lemma} 

\smallskip

And to get the exponential generating function \reff{eq.twovar.Ma.Salas.egf},
we shall use the following general result:

\begin{lemma}[egf of a Stirling subset transform]
   \label{lemma.stirlingtransform.egf}
Let $R$ be a commutative ring containing the rationals,
let $\ba = (a_n)_{n \ge 0}$ be a sequence in $R$,
and let $x$ be an indeterminate;
and define the sequence $\bb = (b_n)_{n \ge 0}$ in $R[x]$ by
\be
   b_n  \;=\;  \sum_{k=0}^n \stirlingsubset{n}{k} \, a_k x^{n-k}
   \;.
\ee
(We call this the ``$x$-Stirling subset transform''.)
Then the exponential generating functions
$A(t) = \sum\limits_{n=0}^\infty a_n t^n/n!$ and
$B(t) = \sum\limits_{n=0}^\infty b_n t^n/n!$
are related by
\be
   B(t)  \;=\;  A\biggl( \frac{e^{xt} -1}{x} \biggr)
   \;.
\ee
\end{lemma}

\begin{proof}
We have
\begin{subeqnarray}
   \sum\limits_{n=0}^\infty a_n \, \frac{t^n}{n!}
   & = &
   \sum\limits_{n=0}^\infty \, \sum_{k=0}^n
       \stirlingsubset{n}{k} \, a_k \, x^{n-k} \, \frac{t^n}{n!}
     \\[2mm]
   & = &
   \sum\limits_{k=0}^\infty  a_k \, x^{-k}
       \sum_{n=k}^\infty \stirlingsubset{n}{k} \, \frac{(xt)^n}{n!}
     \\[2mm]
   & = &
   \sum\limits_{k=0}^\infty  a_k \, x^{-k} \, \frac{(e^{xt} -1)^k}{k!}
       \qquad\hbox{by \cite[eq.~(7.49)]{Graham_94}}
     \\[2mm]
   & = &
   A\biggl( \frac{e^{xt} -1}{x} \biggr)
   \;.
\end{subeqnarray}
\end{proof}
 
\begin{proof}[Proof of Proposition~\ref{prop.twovar.Ma.Salas}]
Multiply \reff{eq.lemma.Ma} by $w^i y^\ell u^{n-i-\ell}$
and sum over $i$ and $\ell$.
Using the well-known identities \cite[eqs.~(7.48)/(5.12)]{Graham_94}
\begin{eqnarray}
   \sum\limits_i \stirlingcycle{r}{i} \, w^i \, u^{r-i}
   & = &
   \prod\limits_{k=0}^{r-1} (w + ku)
        \qquad
        \\[2mm]
   \sum\limits_\ell \binom{n-r}{\ell} \, y^\ell \, (-u)^{n-r-\ell}
   & = &
   (y-u)^{n-r}
\end{eqnarray}
gives (\ref{eq.twovar.Ma.Salas}a,b).
Now use Lemma~\ref{lemma.stirlingtransform.egf} with
\be
   \sum_{r=0}^\infty \sum_{i=0}^r \stirlingcycle{r}{i} \, w^i \, u^{r-i} \,
       \frac{t^r}{r!}
   \;=\;
   (1 - ut)^{-w/u}
   \;\eqdef\; A(t)
 \label{eq.egf.stirlingcycle}
\ee
\hspace*{-1.3mm}\cite[eq.~(7.55)]{Graham_94}
and $x = y-u$ to obtain \reff{eq.twovar.Ma.Salas.egf}.
\end{proof}

\smallskip

{\bf Remarks.}
1.  Lemma~\ref{lemma.stirlingtransform.egf} is a special case
of the fundamental theorem of exponential Riordan arrays (FTERA)
\cite{Barry_16,Barry_18}.

2. Let us stress that this proof of Proposition~\ref{prop.twovar.Ma.Salas}
is based on the combinatorial definition \reff{eq.eulerian.fourvar.cyc.v=y}
of the polynomials $\scrp_n(w,y,u,y)$;
it makes no use of continued fractions
or the Sokal--Zeng Theorem~\ref{thm.perms.S}.
On the other hand, Theorem~\ref{thm.perms.S}
then tells us that the ogf of these polynomials
has an S-fraction representation of the form
\reff{eq.eulerian.fourvar.contfrac} with $v=y$.
Zeng \cite{Zeng_93} has shown that the J-fraction
associated by contraction to this S-fraction
(see Proposition~\ref{prop.contraction_even} below)
can alternatively be obtained
directly from the egf \eqref{eq.twovar.Ma.Salas.egf}
by the Stieltjes--Rogers addition-formula method.
\myendremark

\medskip

Two particular cases of \reff{eq.twovar.Ma.Salas}/\reff{eq.twovar.Ma.Salas.egf}
are of especial interest:

\medskip

1) When $u=y$, we obtain the homogenized Stirling cycle polynomials:
\reff{eq.eulerian.fourvar.cyc.v=y}/\reff{eq.twovar.Ma.Salas} reduce to
\be
   \scrp_n(w,y,y,y)
   \;=\;
   \sum\limits_{r=0}^n \stirlingcycle{n}{r} \, w^r \, y^{n-r}
   \;=\;
   \prod_{k=0}^{n-1} (w + ky)
   \;,
 \label{def_SZ_Pn_formula_v=y=u}
\ee
the limit $u \to y$ of \reff{eq.twovar.Ma.Salas.egf} gives the egf
\cite[eq.~(7.55)]{Graham_94}
\be
   \sum_{n=0}^\infty \scrp_n(w,y,y,y) \: \frac{t^n}{n!}
   \;=\;
   (1 - yt)^{-w/y}
 \label{eq.twovar.Ma.Salas.egf_u=y}
\ee
[cf.\ \reff{eq.egf.stirlingcycle}],
and the ogf has the S-fraction
\be
   \sum_{n=0}^\infty \scrp_n(w,y,y,y) \: t^n
   \;=\;
   \cfrac{1}{1 - \cfrac{wt}{1 - \cfrac{yt}{1 - \cfrac{(w+y)t}{1- \cfrac{2yt}{1 - \cdots}}}}}
 \label{eq.stirlingcycle.Sfrac}
\ee
with coefficients $c_{2k-1} = w + (k-1)y$, $c_{2k} = ky$.
This S-fraction was found by Euler
\cite[section~26]{Euler_1760} \cite{Euler_1788}.\footnote{
   The paper \cite{Euler_1788},
   which is E616 in Enestr\"om's \cite{Enestrom_13} catalogue,
   was apparently presented to the St.~Petersburg Academy in 1776,
   and published posthumously in 1788.
}

\medskip

2) When $u=w$, we obtain the homogenized Eulerian polynomials:
\reff{eq.eulerian.fourvar.cyc.v=y}/\reff{eq.twovar.Ma.Salas} reduce to
\be
   \scrp_n(w,y,w,y)
   \;=\;
   \sum\limits_{k=0}^n \euler{n}{k} \,  y^k \, w^{n-k}
   \;=\;
   \sum_{r=0}^n r! \, \stirlingsubset{n}{r} \, (y-w)^{n-r} \, w^r
 \label{def_SZ_Pn_formula_v=y_w=u}
\ee
where $\euler{n}{k}$ denotes the number of permutations of $[n]$
with $k$ excedances,
\reff{eq.twovar.Ma.Salas.egf} gives the egf \cite[eq.~(7.56)]{Graham_94}
\be
   \sum_{n=0}^\infty \scrp_n(w,y,w,y) \: \frac{t^n}{n!}
   \;=\;
   \frac{y-w}{y - w e^{(y-w)t}}
   \;,
 \label{eq.twovar.Ma.Salas.egf_u=w}
\ee
and the ogf has the S-fraction
\be
   \sum_{n=0}^\infty \scrp_n(w,y,y,y) \: t^n
   \;=\;
   \cfrac{1}{1 - \cfrac{wt}{1 - \cfrac{yt}{1 - \cfrac{2wt}{1- \cfrac{2yt}{1 - \cdots}}}}}
 \label{eq.eulerian.Sfrac}
\ee
with coefficients $c_{2k-1} = kw$, $c_{2k} = ky$.
This S-fraction was found by Stieltjes
\cite[section~79]{Stieltjes_1894}.\footnote{
   Stieltjes does not specifically mention the Eulerian polynomials,
   but he does state that the continued fraction
   is the formal Laplace transform of
   $(1-y) / (e^{t(y-1)} - y)$,
   which is well known to be the exponential generating function
   of the Eulerian polynomials
   [cf.\ \reff{eq.twovar.Ma.Salas.egf_u=w} with $w=1$].
   Stieltjes also refrains from showing the proof:
   ``Pour abr\'eger, je supprime toujours les artifices qu'il faut employer
     pour obtenir la transformation de l'int\'egrale d\'efinie
     en fraction continue'' (!).
   But a proof is sketched, albeit also without much explanation,
   in the book of Wall \cite[pp.~207--208]{Wall_48}.
   The J-fraction corresponding to the contraction
   (see Proposition~\ref{prop.contraction_even} below)
   of this S-fraction
   was proven, by combinatorial methods, by Flajolet
   \cite[Theorem~3B(ii) with a slight typographical error]{Flajolet_80}.
   Also, Dumont \cite[Propositions~2 and 7]{Dumont_86}
   gave a direct combinatorial proof of the S-fraction,
   based on an interpretation of the Eulerian polynomials
   in terms of ``bipartite involutions of $[2n]$''
   and a bijection of these onto Dyck paths.
 \label{footnote.stieltjes}
}
The equality of the last two terms in \reff{def_SZ_Pn_formula_v=y_w=u}
is a well-known identity \cite[eqns.~(6.39)/(6.40)]{Graham_94}
that relates the Eulerian polynomials to the ordered Bell polynomials.%
\footnote{
   The ordered Bell polynomials appear already
   (albeit without the combinatorial interpretation)
   in Euler's book {\em Foundations of Differential Calculus,
    with Applications to Finite Analysis and Series}\/,
   first published in 1755 \cite[paragraph~172]{Euler_1755}.
   This book is E212 in Enestr\"om's \cite{Enestrom_13} catalogue.
   Furthermore, the identity \reff{def_SZ_Pn_formula_v=y_w=u}
   appears already there \cite[paragraphs~172 and 173]{Euler_1755};
   it was rediscovered a century-and-a-half later by Frobenius \cite{Frobenius_10}.
   See also \cite[pp.~150--151]{Flajolet_82} for a simple bijective proof.
}

\bigskip

{\bf Remarks.}
The polynomials \reff{eq.twovar.Ma.Salas} apparently first appeared
in the work of Carlitz \cite{Carlitz_60}.
(He considered the case $u=1$,
 but this is equivalent by homogenization to the general case.)
More specifically, Carlitz \cite[p.~422]{Carlitz_60}
(see also Dillon and Roselle \cite{Dillon_68})
showed that the generalized Eulerian polynomials $A_n(y,w)$ defined by
the exponential generating function
\be
   \biggl( \frac{1-y}{e^{(y-1)t} - y} \biggr) ^{\! w}
   \;=\;
   \sum_{n=0}^\infty A_n(y,w) \, \frac{t^n}{n!}
 \label{def.Anyw}
\ee
have the explicit expression
\be
   A_n(y,w)
   \;=\;  
   \sum_{r=0}^n \stirlingsubset{n}{r} \, (y-1)^{n-r} \,
      \prod\limits_{k=0}^{r-1} (w+k)
   \;.
 \label{eq.Anyw.sum}
\ee
(We proved this in reverse by using Lemma~\ref{lemma.stirlingtransform.egf}.)
The formula \reff{eq.lemma.Ma}
for the coefficients of the polynomials \reff{eq.Anyw.sum}
can also be found in
Dillon and Roselle \cite[eqns.~(1.3)/(3.3)/(3.5)]{Dillon_68}.

Furthermore, Dillon and Roselle \cite[section~5]{Dillon_68}
showed that $A_n(y,w)$ enumerates permutations of $[n]$
with a weight $y$ for each descent
[i.e.\ each index $i \in [n-1]$ such that $\sigma(i) > \sigma(i+1)$]
and a weight $w$ for each record.
And by Foata's fundamental transformation,
this is equivalent to enumerating permutations of $[n]$
with a weight $y$ for each excedance and $w$ for each cycle.\footnote{
   We use Foata's fundamental transformation \cite[section~I.3]{Foata_70}
   in the following form \cite[pp.~17--18]{Stanley_86}:
   Given a permutation $\sigma \in \Sym_n$,
   we write $\sigma$ in disjoint cycle notation
   with the convention that (a)~each cycle is written with its largest element
   (the {\em cycle maximum}\/) first, and
   (b)~the cycles are written in increasing order of their largest element;
   we then erase the parentheses and call the resulting word $\widehat{\sigma}$.
   The map $\sigma \mapsto \widehat{\sigma}$ is a bijection,
   because the permutation $\sigma$ can be uniquely recovered
   from $\widehat{\sigma}$
   by inserting a left parenthesis preceding each record
   and a right parenthesis where appropriate
   (i.e.\ before every internal left parenthesis and at the end).
   There is now a one-to-one correspondence between
   cycle maxima in $\sigma$ and records in $\widehat{\sigma}$,
   and between anti-excedances in $\sigma$ and descents in $\widehat{\sigma}$.
   Finally, the bijection $\sigma \mapsto R \circ \sigma \circ R$
   where $R(i) = n+1-i$ interchanges anti-excedances and excedances
   while preserving the number of cycles.
   This proves the equivalence asserted in the text.
}
So this chain of reasoning gives an alternate proof of the equivalences between
\reff{eq.eulerian.fourvar.cyc.v=y}, \reff{eq.twovar.Ma.Salas}
and \reff{eq.twovar.Ma.Salas.egf}.

On the other hand, Stieltjes \cite[section~81]{Stieltjes_1894}
observed that the continued fraction
\reff{eq.eulerian.fourvar.contfrac} with $v=y$ and $u=1$
is the formal Laplace transform of
the exponential generating function \reff{def.Anyw}.\footnote{
   Here too, Stieltjes refrained from showing the proof:
   see footnote~\ref{footnote.stieltjes}.
}
Combining this result with the just-mentioned combinatorial interpretation
of $A_n(y,w)$ in terms of excedances and cycles \cite[section~IV.3]{Foata_70}
proves the $v=y$ special case of the Sokal--Zeng Theorem~\ref{thm.perms.S}.
(By contrast, the proof of Sokal and Zeng \cite{SZ} is purely combinatorial,
and makes no use of exponential generating functions.)

See also \cite{Carlitz_77,Brenti_00,Ksavrelof_03,Savage_12,Ma}
for some later work on these polynomials.
\myendremark

%
%
\subsection{Proof that the candidate families have an S-fraction representation}
  \label{sec.main2}

We will now complete the proof of Theorem~\ref{theor.main} by showing
that each of the ten families indeed has an S-fraction
representation with the conjectured coefficients $\bc$. 
This section is organized as follows:
For each family, we will state the main result and give its proof;
then we will describe some additional properties of the given family
and mention some special cases of combinatorial interest. 
The families are grouped by duality;
two families related by duality will have essentially identical proofs,
with $c_i(x) \to x \, c_i(1/x)$.
Each proof will be based on identifying
the exponential generating function \reff{def_egf}
as a special case of \reff{eq.twovar.Ma.Salas.egf};
it will then follow from Theorem~\ref{thm.perms.S}
and Proposition~\ref{prop.twovar.Ma.Salas}
that the corresponding ordinary generating function \reff{def_ogf}
is given by the S-fraction \reff{eq.eulerian.fourvar.cyc} with $v=y$
and suitable values for $w,u,y$.

%
%
\subsubsection{Families 1a and 1b} \label{sec.family.I}

\begin{proposition}[S-fraction for families 1a and 1b]
    \label{prop.I}
The ogf $f(x,t;\bmu)$ for the recurrence \eqref{eq_binomvert} with 
$\bmu=(0,\beta,0$, $\alpha',-\alpha',\gamma')$ has an S-type continued 
fraction representation in the ring $\Z[x,\beta,\alpha',\gamma'][[t]]$ 
with coefficients 
\be
c_{2k-1} \;=\; [\gamma' \,+\, (k-1)\alpha']\, x\,, \quad 
c_{2k}   \;=\; k \beta   \;.
\label{eq.ak.Ia}
\ee
Similarly, the ogf $f(x,t;\bmu)$ for the recurrence \eqref{eq_binomvert} with
$\bmu=(0,\beta,\gamma, \alpha',-\alpha',0)$
has an S-type continued fraction representation in the ring
$\Z[x,\beta,\gamma,\alpha'][[t]]$ with coefficients
\be
c_{2k-1} \;=\; \gamma \,+\, (k-1)\beta \,, \quad
c_{2k}   \;=\; k \alpha' x   \;.
\label{eq.ak.Ib}
\ee
\end{proposition}   

\subsubsection*{Family 1a: $\bm{\mu=(0,\beta,0; \alpha',-\alpha',\gamma')}$}

This is Spivey's \cite{Spivey_11} case (S3)
$\alpha/\beta = \alpha'/\beta' + 1$
specialized to $\alpha=\gamma=0$.
Its egf can be computed from \cite[eq.~(A2)]{BSV}: 
\be
  F(x,t) \;=\; \left[ \frac{\beta - \alpha'\, x\, e^{(\beta-\alpha'\, x) t}}
                           {\beta - \alpha'\, x} \right]^{-\gamma'/\alpha'}
  \,.
 \label{eq.egf.Ia} 
\ee 
This is the special case of \reff{eq.twovar.Ma.Salas.egf}
with $w = \gamma' x$, $u = \alpha' x$, $y = \beta$,
so the ogf is given by the S-fraction \reff{def.weights.eulerian.fourvar} with
$w = \gamma' x$, $u = \alpha' x$, $y = v = \beta$,
exactly as stated in \eqref{eq.ak.Ia}.
This completes the proof for family~1a.

\subsubsection*{Family 1b: $\bm{\mu=(0,\beta,\gamma, \alpha',-\alpha',0)}$}

This is Spivey's \cite{Spivey_11} case (S3)
$\alpha/\beta = \alpha'/\beta' + 1$
specialized now to $\alpha=\gamma'=0$.
Its egf is \cite[eq.~(A2)]{BSV}
\be
F(x,t) \;=\; \left[ \frac{\alpha' x \,-\, \beta e^{(\alpha'x - \beta) t}}
                         {\alpha' x - \beta} \right]^{\gamma/\beta}
\,.
\label{eq.egf.Ib} 
\ee
This is the special case of \reff{eq.twovar.Ma.Salas.egf}
with $w = \gamma$, $u = \beta$, $y = \alpha' x$,
so the ogf is given by the S-fraction \reff{def.weights.eulerian.fourvar} with
$w = \gamma$, $u = \beta$, $y = v = \alpha' x$,
exactly as stated in \eqref{eq.ak.Ib}.
This completes the proof for family~1b.

Of course, this result can also be obtained by applying duality to family~1a.
Since the dual of $\bmu=(0,\beta,0, \alpha',-\alpha',\gamma')$ is 
$D\bmu=(0,\alpha',\gamma',\beta,-\beta,0)$,
we obtain family~1b from the dual of family~1a
by applying the map
$(\alpha',\gamma',\beta)\mapsto (\beta,\gamma,\alpha')$.

\subsubsection*{Particular cases}

These two families contain several specific cases of combinatorial interest:
\begin{itemize}
\item The Stirling subset numbers $\stirlingsubset{n}{k}$ have
      $\bmu=(0,1,0, 0,0,1)$ and belong to family~1a.
      The S-fraction for the ogf of the Bell polynomials
      $B_n(x) = \sum\limits_{k=0}^n \stirlingsubset{n}{k} x^k$,
      with coefficients $c_{2k-1} = x$ and $c_{2k} = k$,
      is well known.\footnote{
   Flajolet \cite[Theorem~2(ib)]{Flajolet_80} found a J-type continued fraction
   that is equivalent by contraction
   (see Proposition~\ref{prop.contraction_even} below)
   to this S-fraction.
   Later, Dumont \cite{Dumont_89} found the S-fraction directly
   by a functional-equation method,
   and Zeng \cite[Lemma~3]{Zeng_95} used this same method to find
   two $q$-generalizations of the S-fraction.
}
\item The generalized $(s,0)$-Eulerian numbers \cite{BSV2} 
      $\euler{n}{k}_{(s,0)}$ have $\bmu=(0,1,s$, $1,-1,0)$ with $s\in\N$
      and belong to family~1b.
      The case $s=1$ corresponds to the Eulerian numbers $\euler{n}{k}$
      with the Graham--Knuth--Patashnik indexing \cite[section~6.2]{Graham_94}
      \cite[\seqnum{A173018}]{OEIS}.
\item The generalized $(0,t)$-Eulerian numbers \cite{BSV2}
      $\euler{n}{k}_{(0,t)}$ have $\bmu=(0,1,0$, $1,-1,t)$ with $t\in\N$
      and belong to family~1a.
      The case $t=1$ corresponds to the Eulerian numbers with the traditional 
      indexing $\euler{n}{k}_{(0,1)} = \euler{n}{k-1}$ for $n \ge 1$ 
      \cite[\seqnum{A008292}]{OEIS}.
\end{itemize}

%
%
\subsubsection{Families 2a and 2b} \label{sec.family.II}

\begin{proposition}[S-fraction for families 2a and 2b]
    \label{prop.II}
The ogf $f(x,t;\bmu)$ for the recurrence \eqref{eq_binomvert} with 
$\bmu=(\alpha,-\alpha,-\alpha$, $\alpha',\beta',\gamma')$
has an S-type continued fraction representation in the ring
$\Z[x;\alpha,\alpha',\beta',\gamma'][[t]]$ 
with coefficients 
\be
c_{2k-1} \;=\; [\gamma' \,+\, k(\alpha'+\beta')]\, x\,, \quad
c_{2k}   \;=\; k (\alpha'+\beta')\, x   \;.
\label{eq.ak.IIa}
\ee
Similarly, the ogf $f(x,t;\bmu)$ for the recurrence \eqref{eq_binomvert} with
$\bmu=(\alpha,\beta,\gamma,0,\beta',-\beta')$ 
has an S-type continued fraction representation in the ring
$\Z[x;\alpha,\beta,\gamma,\beta'][[t]]$
with coefficients 
\be
c_{2k-1} \;=\; \gamma \,+\, k\alpha \,, \quad
c_{2k}   \;=\; k\alpha   \;.
\label{eq.ak.IIb}
\ee
\end{proposition}   

Before proceeding further, let us observe that
the coefficients \reff{eq.ak.IIa} for family~2a do not depend on $\alpha$,
and moreover they depend on $\alpha'$ and $\beta'$ only via their sum.
Similarly, the coefficients \reff{eq.ak.IIb}
for family~2b do not depend on $\beta$ or $\beta'$.
These cases illustrate the parametric ambiguities
discussed in \cite[Section~3]{BSV},
in which the map $\bmu \mapsto \bT(\bmu)$ fails to be injective.
So families 2a and 2b, which appear to be four-dimensional,
are in fact only two-dimensional.

This degeneracy also implies that family~2a,
which is defined by $\bmu=(\alpha,-\alpha,-\alpha,$ $\alpha',\beta',\gamma')$,
in fact has the same matrix $\bT(\bmu)$
as $\bmu' = (0,0,0,0,\alpha'+\beta',\gamma')$,
which is a special case of family~3a.
The S-fraction \reff{eq.ak.IIa} for family~2a
is thus obtained from the S-fraction \reff{eq.ak.IIIa} for family~3a
by making the replacements $\beta \to 0$ and $\beta' \to \alpha'+\beta'$.
Similarly, family~2b,
which is defined by $\bmu=(\alpha,\beta,\gamma,0,\beta',-\beta')$,
has the same matrix $\bT(\bmu)$
as $\bmu' = (\alpha,-\alpha,\gamma,0,0,0)$,
which is a special case of family~3b.
So the S-fraction \reff{eq.ak.IIb} for family~2b
is obtained from the S-fraction \reff{eq.ak.IIIb} for family~3b
by making the specialization $\alpha'=0$.
Consequently, families 2a and 2b are redundant
if we work at the level of the matrices $\bT(\bmu)$
rather than the parameters $\bmu$.

In fact, both of these families are also degenerate in a further sense:
in family~2a we have $T(n,k) = 0$ whenever $k \neq n$,
and in family~2b we have $T(n,k) = 0$ whenever $k \neq 0$.
{}From the recurrence \eqref{eq_binomvert}
one easily gets for family 2a
\be
   T(n,k) \;=\; \delta_{kn}\, \prod_{j=1}^{n} [\gamma' + j (\alpha'+\beta')]
 \label{eq.seq.IIa}
\ee
and hence
\be
   P_n(x) \;=\; x^n \, \prod_{j=1}^{n} [\gamma' + j (\alpha'+\beta')]  \;,
 \label{eq.Pn.IIa}
\ee
and for family 2b
\be
   T(n,k) \;=\; \delta_{k0}\, \prod_{j=1}^{n} (\gamma + j\alpha)
 \label{eq.seq.IIb}
\ee
and hence
\be
   P_n(x) \;=\; \prod_{j=1}^{n} (\gamma + j\alpha)   \;.
 \label{eq.Pn.IIb}
\ee
In these formulae we see explicitly the parametric ambiguities mentioned above.
The S-fractions \reff{eq.ak.IIa} and \reff{eq.ak.IIb}
are then simply the S-fraction \reff{eq.stirlingcycle.Sfrac}
for the homogenized Stirling cycle polynomials
with the variables $w$ and $y$ in that formula
replaced by suitable linear combinations of the parameters $\bmu$.

Alternatively, we can deduce the S-fractions directly from the
exponential generating functions.
Family~2a is Spivey's \cite{Spivey_11} case (S1) $\beta = -\alpha$
specialized to $\gamma = -\alpha$,
and its egf is \cite[eq.~(A4)]{BSV}
\be
F(x,t) \;=\; \bigl[ 1 -(\alpha' + \beta') x t 
             \bigr]^{-(\alpha'+\beta'+\gamma')/(\alpha'+\beta')}\,. 
 \label{eq.egf.IIa}
\ee
This is the special case of \reff{eq.twovar.Ma.Salas.egf_u=y}
with $w = (\alpha'+\beta'+\gamma') x$, $y = (\alpha' + \beta') x$,
and inserting these parameters into \reff{eq.stirlingcycle.Sfrac}
gives \reff{eq.ak.IIa}.
Family~2b is Neuwirth's \cite{Neuwirth} case $\alpha'=0$
specialized to $\gamma' = -\beta'$,
and its egf is \cite[eq.~(A8)]{BSV}
\be
   F(x,t) \;=\;  (1 - \alpha t)^{-(\alpha+\gamma)/\alpha} \,.
\ee
This is the special case of \reff{eq.twovar.Ma.Salas.egf_u=y}
with $w = \alpha+\gamma$, $y = \alpha$,
and inserting these parameters into \reff{eq.stirlingcycle.Sfrac}
gives \reff{eq.ak.IIb}.

\subsubsection*{Particular cases}

These two families contain several specific cases of combinatorial interest:
\begin{itemize}
\item Multifactorials: They appear in both families. To make the story 
      short, let us consider family~2b. It is obvious that 
      $\bmu=(0,0,1,0,0,0)$ leads to $T(n,k)=\delta_{k0}$. If we apply 
      Corollary~\ref{cor.scaling}(c) with $\kappa=\nu$ and 
      $\lambda=-\rho$, we get that the triangular-array entries  
\be 
      T'(n,k) \;=\; \left( \prod_{j=0}^{n-1} (n\nu-\rho-j\nu)\right) \, 
                    \delta_{k0}
\ee
      satisfy the GKP recurrence with $\bmu'=(\nu,0,-\rho,0,0,0)$.
      For instance:
      \begin{itemize}
        \item[$\circ$]
              If $\nu=1$ and $\rho=0$, then $\bmu'=(1,0,0,0,0,0)$ leads to 
              the factorials $T'(n,k)= n! \, \delta_{k0}$. In this case, 
              we recover Euler's S-fraction \reff{eq.nfact.contfrac}. 
        \item[$\circ$]
              If $\nu=2$ and $\rho=-1$, then $\bmu'=(2,0,-1,0,0,0)$ leads to
              the semi-factorials $T'(n,k)= (2n-1)!! \, \delta_{k0}$, or
              the double factorial of odd numbers
              \cite[\seqnum{A001147}]{OEIS}. The corresponding 
              S-fraction was also found by Euler \cite[section~29]{Euler_1760}. 
        \item[$\circ$]
              If $\nu=3$ and $\rho=-1$, then $\bmu'=(3,0,-1,0,0,0)$ leads to
              the triple factorials $T'(n,k)= (3n-1)!!\,! \, \delta_{k0}$
              \cite[\seqnum{A007661}]{OEIS}.
        \item[$\circ$]
              If $\nu=4$ and $\rho=-1$, then $\bmu'=(4,0,-1,0,0,0)$ leads to
              the quadruple factorials $T'(n,k)= (4n-1)!!\,!! \, \delta_{k0}$
              \cite[\seqnum{A007662}]{OEIS}.
        \item[$\circ$]
              In general, if $\nu\ge 2$ is an integer, then 
              $\bmu'=(\nu,0,-1,0,0,0)$ leads to the $\nu$-th factorials
              $T'(n,k)= \prod_{j=0}^{n-1} (n\nu-j\nu-1) \, \delta_{k0}$. 
      \end{itemize} 
\item $\bmu=(0,1,s,0,1,-1)$ leads to
      $T(n,k)=\delta_{k0} s^n$. These numbers are the
      inverse pairs of the generalized Eulerian numbers $\euler{n}{k}_{(s,-s)}$ 
      \cite{BSV2} given by $\bmu=(0,1,s,1,-1,-s)$. (See the end of  
      Section~\ref{sec.family.VI}.)  
\end{itemize}

%
%
\subsubsection{Families 3a and 3b} \label{sec.family.III}

\begin{proposition}[S-fraction for families 3a and 3b]
   \label{prop.III}
The ogf $f(x,t;\bmu)$ for the recurrence \eqref{eq_binomvert} with 
$\bmu=(0,\beta,0,0,\beta',\gamma')$ has an S-type 
continued fraction representation in the ring $\Z[x,\beta,\beta',\gamma'][[t]]$ 
with coefficients 
\be
c_{2k-1} \;=\; (\gamma' \,+\, k\beta')\, x\,, \quad
c_{2k}   \;=\; k (\beta + \beta' x)   \;.
\label{eq.ak.IIIa}
\ee
Similarly, the ogf $f(x,t;\bmu)$ for the recurrence \eqref{eq_binomvert} with
$\bmu=(\alpha,-\alpha,\gamma,\alpha',-\alpha',0)$ 
has an S-type continued fraction representation in the ring
$\Z[x,\alpha,\gamma,\alpha'][[t]]$
with coefficients 
\be
c_{2k-1} \;=\; \gamma \,+\, k\alpha \,, \quad
c_{2k}   \;=\; k (\alpha + \alpha'  x)  \;.
\label{eq.ak.IIIb}
\ee
\end{proposition}   

\subsubsection*{Family 3a: $\bm{\mu=(0,\beta,0;0,\beta',\gamma')}$}

Family~3a is Neuwirth's \cite{Neuwirth} case $\alpha'=0$
specialized to $\alpha = \gamma = 0$.
Its egf is \cite[eq.~(A8)]{BSV}
\be
F(x,t) \;=\; \left[ 1 + \frac{\beta'\, x}{\beta} \, 
                    \left(1 - e^{\beta\, t}\right) 
             \right]^{-(\beta'+\gamma')/\beta'} \,. 
\label{eq.egf.IIIa}
\ee
This is the special case of \reff{eq.twovar.Ma.Salas.egf}
with $w = (\beta' + \gamma') x$, $u = \beta' x$, $y = \beta + \beta' x$,
so the ogf is given by the S-fraction with coefficients \reff{eq.ak.IIIa}.

\subsubsection*{Family 3b: 
           $\bm{\mu=(\alpha,-\alpha,\gamma,\alpha',-\alpha',0)}$}

Family~3b is Spivey's \cite{Spivey_11} case (S1) $\beta = -\alpha$
specialized to $\beta' = -\alpha'$ and $\gamma' = 0$,
and its egf is \cite[eq.~(A4)]{BSV}
\be
   F(x,t) \;=\; \left[ 1 + \frac{\alpha}{\alpha'\, x} \,
                    \bigl(1 - e^{\alpha' x t} \bigr)
             \right]^{-(\alpha+\gamma)/\alpha}
   \,.
\label{eq.egf.IIIb}
\ee
This is the special case of \reff{eq.twovar.Ma.Salas.egf}
with $w = \alpha + \gamma$, $u = \alpha$, $y = \alpha + \alpha' x$,
so the ogf is given by the S-fraction with coefficients \reff{eq.ak.IIIb}.
Of course, this result can also be obtained by applying duality to family~3a.

\subsubsection*{Particular cases}

These two families contain several interesting special cases,
although some of them have already appeared in the previous two subsections: 
\begin{itemize}
\item The Stirling subset numbers $\stirlingsubset{n}{k}$ have
      $\bmu=(0,1,0, 0,0,1)$ and belong to family~3a (as well as 1a).
\item The ordered Stirling subset numbers
      (also called surjective numbers \cite{Fekete_94})
      $k! \stirlingsubset{n}{k}$
      have $\bmu=(0,1,0, 0,1,0)$ and belong to family~3a.
\item Multifactorials: we now apply Corollary~\ref{cor.scaling}(b) to 
      $\bmu=(0,0,1,0,0,0)$ with $\kappa=\nu$ and $\lambda=-\rho$. Then,
      the numbers 
\be 
      T'(n,k) \;=\; \left( \prod_{j=1}^n (j\nu - \rho) \right) \, \delta_{k0}
\ee
      satisfy the GKP recurrence with $\bmu'=(\nu,-\nu,-\rho,0,0,0)$ 
      which belongs to family~3b. By choosing $\nu$ and $\rho$ in the same 
      way as we did at the end of Section~\ref{sec.family.II}, we obtain 
      factorials, double factorials, triple factorials, etc.
\end{itemize}

%
%
\subsubsection{Families 4a and 4b} \label{sec.family.IV}

\begin{proposition}[S-fraction for families 4a and 4b]
   \label{prop.IV}
The ogf $f(x,t;\bmu)$ for the recurrence \eqref{eq_binomvert} with 
$\bmu=(0,\kappa\beta',\kappa(\beta'+\gamma'),0,\beta',\gamma')$ 
has an S-type continued fraction representation in the ring 
$\Z[x;\beta',\gamma',\kappa][[t]]$ with coefficients 
\be
c_{2k-1} \;=\; (\gamma' \,+\, k\beta')\, (\kappa + x)\,, \quad
c_{2k}   \;=\; k \beta' x  \;.
\label{eq.ak.IVa}
\ee
Similarly, the ogf $f(x,t;\bmu)$ for the recurrence \eqref{eq_binomvert} with
$\bmu=(\alpha,-\alpha,\gamma,
       \kappa\alpha, -\kappa\alpha, \kappa(\alpha+\gamma))$
has an S-type continued fraction representation in the ring
$\Z[x;\alpha,\gamma,\kappa][[t]]$
with coefficients 
\be
c_{2k-1} \;=\; (\gamma \,+\, k\alpha)\, (1+\kappa x)\,, \quad
c_{2k}   \;=\; k \alpha   \;.
\label{eq.ak.IVb}
\ee
\end{proposition}

\subsubsection*{Family 4a: $\bm{\mu=(0,\kappa\beta',\kappa(\beta'+\gamma'), 
                                     0,\beta',\gamma')}$}

Family~4a is Neuwirth's \cite{Neuwirth} case $\alpha'=0$
specialized to $\alpha = 0$ and $\beta/\beta' = (\gamma-\beta)/\gamma'$.
Its egf is \cite[eq.~(A8)]{BSV}
\be
F(x,t) \;=\; \left[ 1 - \frac{\kappa+ x}{\kappa} \, 
                    \bigl(1 - e^{-\kappa\beta' t} \bigr) 
             \right]^{-(\beta'+\gamma')/\beta'} \,. 
 \label{eq.egf.IVa}
\ee
This is the special case of \reff{eq.twovar.Ma.Salas.egf}
with $w = (\beta' + \gamma') (\kappa + x)$,
$u = \beta' (\kappa + x)$, $y = \beta' x$,
so the ogf is given by the S-fraction with coefficients \reff{eq.ak.IVa}.

\subsubsection*{Family 4b: $\bm{\mu=(\alpha,-\alpha,\gamma, \kappa\alpha,
                                    -\kappa\alpha, \kappa(\alpha+\gamma))}$}

Family~4b is Spivey's \cite{Spivey_11} case (S1) $\beta = -\alpha$
specialized to $\beta' = -\alpha'$ and $\gamma'/\alpha' = 1 + \gamma/\alpha$,
and its egf is \cite[eq.~(A4)]{BSV}
\be
   F(x,t)
   \;=\;
   \left[ 1 - \frac{1+\kappa x}{\kappa x} \,
                    \left(1 - e^{-\kappa\alpha x t}\right)
             \right]^{-(\alpha+\gamma)/\alpha} 
   \,.
\ee
This is the special case of \reff{eq.twovar.Ma.Salas.egf}
with $w = (\alpha + \gamma) (1 + \kappa x)$,
$u = \alpha (1 + \kappa x)$, $y = \alpha$,
so the ogf is given by the S-fraction with coefficients \reff{eq.ak.IVb}.

\subsubsection*{Particular cases}

\begin{itemize}
\item Family~4a with $\gamma' = 0, \beta=1$, i.e.\ 
$\bmu=(0,1,1, 0,1,0)$, yields a sequence \cite[\seqnum{A028246}]{OEIS}
closely related to the ordered Stirling subset numbers:
\be
T(n,k) \;=\; k!\, \stirlingsubset{n+1}{k+1}
   \;.
\ee
Note that these numbers 
and the Eulerian numbers $\euler{n}{k}$ with $\bmu=(0,1,1,1,-1,0)$
(cf.~family~1b), are related by the Riordan involution $R$.
\end{itemize}

%
%
\subsubsection{Family 5} \label{sec.family.V} 

This self-dual family is given by $\bmu=(\alpha,0,\gamma, \alpha',0,\gamma')$.
It corresponds to the case of the GKP recurrence \reff{eq_binomvert}
in which the coefficients depend only on $n$,
and is Spivey's \cite{Spivey_11} case (S2) $\beta=\beta'=0$.
It is the Type-IV case of Problem~\ref{problem.GKP}
as defined in Ref.~\cite{BSV}, and we can read all the
relevant information from \cite[eqs.~(2.15), (A20), (A21)]{BSV}:
\begin{eqnarray}
\label{eq.egf.V}
F(x,t) &=& \bigl[ 1 - (\alpha + \alpha' x) \, t 
           \bigr]^{-[(\alpha+\gamma) + (\alpha'+\gamma')x] /
                          (\alpha + \alpha' x)} \\[2mm] 
\label{eq.Pn.V}
P_n(x) &=& \prod_{k=1}^n \bigl[ (\gamma + \gamma' x) \,+\, k(\alpha + \alpha' x)
                         \bigr]   \\[2mm]
\label{eq.seq.V}
T(n,k) &=& \sum\limits_{t=0}^n \sum\limits_{s=0}^t 
\stirlingcycle{n}{t} \binom{t}{s} \binom{n-t}{k-s} 
(\alpha+\gamma)^{t-s} (\alpha'+\gamma')^s \alpha^{n-k+s-t} (\alpha')^{k-s}
\qquad
\end{eqnarray}
So these polynomials are just translates and rescalings of the
homogenized Stirling cycle polynomials \reff{def_SZ_Pn_formula_v=y=u}.
Using \reff{eq.stirlingcycle.Sfrac}
with $w = (\alpha+\gamma) + (\alpha'+\gamma')x$
and $y = (\alpha + \alpha' x)$,
we obtain:

\begin{proposition}[S-fraction for family 5]
    \label{prop.V}
The ogf $f(x,t;\bmu)$ for the recurrence \eqref{eq_binomvert} with 
$\bmu=(\alpha,0,\gamma, \alpha',0,\gamma')$ 
has an S-type continued fraction representation in the ring 
$\Z[x;\alpha,\gamma,\alpha',\gamma'][[t]]$ with coefficients 
\be
c_{2k-1} \;=\;  (\gamma + \gamma' x) \,+\, k(\alpha + \alpha' x) \,, \quad  
c_{2k}   \;=\; k (\alpha + \alpha' x)   \;.
\label{eq.ak.V}
\ee
\end{proposition}   

As previously mentioned, family~5 is self-dual.
Duality \reff{eq.duality.mu} acts on family~5 by interchanging
$(\alpha,\gamma) \leftrightarrow (\alpha',\gamma')$.
   
\subsubsection*{Particular cases}

\begin{itemize}
\item The binomial coefficients $\binom{n}{k}$ have $\bmu=(0,0,1,0,0,1)$.
      In this case the ogf is a rational function $f(x,t)=1/[1-(1+x)t]$,
      and the S-fraction terminates at the first level (that is, $c_2 = 0$).
\item The Stirling cycle numbers $\stirlingcycle{n}{k}$ correspond to
      $\bmu=(1,0,-1,0,0,1)$.
\item The reversed Stirling cycle numbers $\stirlingcycle{n}{n-k}$ correspond to
      $\bmu=(0,0,1,1,0,-1)$.
\end{itemize}

%
%
\subsubsection{Family 6} \label{sec.family.VI} 

This self-dual family is given by  
$\bmu=(\kappa(\alpha'+\beta'), \kappa\beta', \kappa\gamma',
       \alpha', \beta', \gamma')$.
It is Spivey's \cite{Spivey_11} case (S3) 
$\alpha/\beta = \alpha'/\beta' + 1$
specialized to $\beta/\beta' = \gamma/\gamma'$.
We can deduce all the relevant information from \cite[eq.~(A2)]{BSV}:
\begin{eqnarray}
F(x,t) & = & \bigl[ 1 - (\alpha'+\beta')\, (\kappa + x)\, t  
             \bigr]^{-(\alpha'+\beta'+\gamma') / (\alpha'+\beta')}
  \label{eq.egf.VI}  \\[2mm]
P_n(x) & = &
    (\kappa+ x)^n \, \prod_{j=1}^n \bigl[ \gamma' + j (\alpha'+ \beta') \bigr]
  \label{eq.Pn.VI}   \\[2mm]
T(n,k) & = &   \kappa^{n-k}\, \binom{n}{k}\,  
   \prod_{j=1}^{n} \bigl[ \gamma' + j (\alpha' + \beta') \bigr]
  \label{eq.seq.VI}
\end{eqnarray}
Like family 5, these polynomials are translates and rescalings of the
homogenized Stirling cycle polynomials \reff{def_SZ_Pn_formula_v=y=u},
but now in the parameters rather than in $x$;
the $x$-dependence is a trivial power $(\kappa+ x)^n$.
Using \reff{eq.stirlingcycle.Sfrac}
with $w = (\alpha' + \beta' + \gamma') (\kappa+ x)$
and $y = (\alpha' + \beta') (\kappa+ x)$
we obtain:

\begin{proposition}[S-fraction for family 6]
    \label{prop.VI}
The ogf $f(x,t;\bmu)$ for the recurrence \eqref{eq_binomvert} with 
$\bmu=(\kappa(\alpha'+\beta'), \kappa\beta', \kappa\gamma',
       \alpha', \beta', \gamma')$.
has an S-type continued fraction representation in the ring 
$\Z[x,\alpha',\beta',\gamma',\kappa][[t]]$ with coefficients 
\be
c_{2k-1} \;=\; \bigl[ \gamma' + k(\alpha'+\beta') \bigr] \, (\kappa + x)
   \,, \quad
c_{2k}   \;=\; k (\alpha'+\beta') (\kappa + x)
\label{eq.ak.VI}
\ee
\end{proposition}   

Please note that the coefficients \reff{eq.ak.VI}
for family~6 depend on $\alpha'$ and $\beta'$ only via their sum.
So family~6, which appears to be four-dimensional,
is in fact only three-dimensional.
Moreover, the parameter $\kappa$ acts only by translation of $x$.
In fact, family~6 is derived from family~2a
by the substitution $x \mapsto x+\kappa$.
It follows that the matrix $\bT(\bmu)$ for family~6 arises from the 
one for family~2a by multiplication on the right by the 
$\kappa$-binomial matrix $B_\kappa$ [cf.\ \eqref{def_Bxi}]:
\be
   \bT(\kappa(\alpha'+\beta'), \kappa\beta', \kappa\gamma',
       \alpha', \beta', \gamma')
   \;=\;
   \bT(\alpha,-\alpha,-\alpha,\alpha', \beta', \gamma') \: B_\kappa
   \;.
 \label{eq.family6.family2a}
\ee
Indeed, this follows immediately by comparing \reff{eq.seq.IIa}
with \reff{eq.seq.VI}.

\medskip

{\bf Remark.}
By applying Corollary~\ref{cor6.prop.spivey.corollary5} to family~2a
and using \reff{eq.family6.family2a}, we see that family~6 {\em also}\/
satisfies the recurrence
\begin{eqnarray}
   T(n,k)
   & = &
   \bigl[ (\alpha + \kappa \alpha') n \,+\,
          (-\alpha + 2\kappa \beta') k  \,-\,
          \alpha \,+\, \kappa(\beta' +\gamma')
   \bigr] \, T(n-1,k)
           \nonumber \\[2mm]
   & & 
   \quad +\;
   (\alpha' n + \beta' k + \gamma') \, T(n-1,k-1)
           \nonumber \\[2mm]
   & & 
   \quad +\;
   \kappa \, (-\alpha + \kappa\beta') \, (k+1) \, T(n-1,k+1)
\end{eqnarray}
for $n \ge 1$.
(This is a recurrence of GKPZ form: see Section~\ref{subsec.GKPZ} below.)
It is far from obvious (at least to us)
that the solution of this recurrence is independent of $\alpha$
or depends on $\alpha',\beta'$ only via their sum,
much less that it satisfies the GKP recurrence \reff{eq_binomvert}
for family~6.  But it does.
\myendremark

\bigskip

As previously mentioned, family~6 is self-dual.
The action of duality on family~6 can be seen most clearly
by employing the parameter $\delta' \eqdef \alpha' + \beta'$
in place of $\alpha'$,
thereby parametrizing family~6 by the quadruplet
$(\delta',\beta',\gamma',\kappa)$.
Then duality \reff{eq.duality.mu} acts by taking
$(\delta', \beta', \gamma',\kappa) \mapsto
 (\kappa\delta', -\kappa\beta', \kappa\gamma', 1/\kappa)$.
   
\subsubsection*{Particular cases}

\begin{itemize}
\item The generalized $(s,-s)$-Eulerian numbers $\euler{n}{k}_{(s,-s)}$
      \cite{BSV2} correspond to $\bmu=(0,1,s$, $1,-1,-s)$.
      They are trivial rescalings of the binomial coefficients:
\be
\euler{n}{k}_{(s,-s)} \;=\; (-1)^k \, s^n \, \binom{n}{k}  \;.
\ee 
\item The numbers $T(n,k)=n! \binom{n}{k}$ 
      \cite[\seqnum{A196347}]{OEIS}
      can be obtained from $\bmu=(1,0,0,1,0$,$0)$. This can be seen by applying 
      Corollary~\ref{cor.scaling}(c) with $\kappa=1,\lambda=0$ to the 
      binomial coefficients. 
\end{itemize}

%
%
\section{Preliminaries for T-fractions and J-fractions}
   \label{sec.prelim.TJ}

In this section we review some transformation formulae for continued fractions
that will be needed for our discussion in Section~\ref{sec.mainT}
of T-fractions and J-fractions for the GKP recurrence.
The plan of this section is as follows:
First we review the formula for the contraction of S-fractions
or special T-fractions to J-fractions (Section~\ref{subsec.contraction}).
Then we review the theory of the binomial transform
and its action on S-fractions, special T-fractions, and J-fractions
(Section~\ref{subsec.binomialtrans}).

\subsection{Contraction}
   \label{subsec.contraction}

It is a classical result, going back to the middle of the nineteenth century,
that any S-fraction \reff{def_Stype.one}
can be re-expressed as a J-fraction \reff{def_Jtype.one}.
This operation, which is known as {\em contraction}\/, is given as follows:

\begin{proposition}[Even contraction for S-fractions]
   \label{prop.contraction_even}
We have
\be
   \cfrac{1}{1 - \cfrac{c_1 t}{1 - \cfrac{c_2 t}{1 -  \cfrac{c_3 t}{1- \cdots}}}}
   \;\;=\;\;
   \cfrac{1}{1 - c_1 t - \cfrac{c_1 c_2 t^2}{1 - (c_2 + c_3) t - \cfrac{c_3 c_4 t^2}{1 - (c_4 + c_5) t - \cfrac{c_5 c_6 t^2}{1- \cdots}}}}
   \;,
 \label{eq.contraction_even}
\ee
so that the J-fraction on the right-hand side has coefficients
\begin{subeqnarray}
   e_0  & = &  c_1
       \slabel{eq.contraction_even.coeffs.a}   \\
   e_n  & = &  c_{2n} + c_{2n+1}  \qquad\hbox{for $n \ge 1$}
       \slabel{eq.contraction_even.coeffs.b}   \\
   f_n  & = &  c_{2n-1} c_{2n}
       \slabel{eq.contraction_even.coeffs.c}
 \label{eq.contraction_even.coeffs}
\end{subeqnarray}
\end{proposition}

The classic algebraic proof of this result,
using the convergents of the continued fraction,
can be found in the book of Wall \cite[p.~21]{Wall_48}.
A very simple and elegant algebraic proof can be found in
\cite[Lemmas~1 and 2]{Dumont_94b}
\cite[proof of Lemma~1]{Dumont_95}
\cite[Lemma~4.5]{DiFrancesco_10}.
A beautiful and enlightening combinatorial proof,
using Flajolet's \cite{Flajolet_80} representation
of S-fractions (resp.\ J-fractions) in terms of Dyck (resp.\ Motzkin) paths,
can be found in the lectures of Viennot \cite[pp.~V-31--V-32]{Viennot_83}.

There is also \cite{Sokal_totalpos}
a generalization of Proposition~\ref{prop.contraction_even}
to a subclass of T-fractions,
namely, those with $d_i = 0$ for all {\em even}\/ levels $i$:

\begin{proposition}[Even contraction for T-fractions with $d_i = 0$ at all even levels~$i$]
   \label{prop.contraction_even.Ttype}
We have
\begin{eqnarray}
    & &  \hspace*{-1.5cm}
   \cfrac{1}{1 - d_1 t - \cfrac{c_1 t}{1 - \cfrac{c_2 t}{1 - d_3 t -  \cfrac{c_3 t}{1- \cdots}}}}
  \;\;=\;  \nonumber \\[2mm]
  & &  \hspace*{-7mm}
   \cfrac{1}{1 - (c_1 + d_1) t - \cfrac{c_1 c_2 t^2}{1 - (c_2 + c_3 + d_3) t - \cfrac{c_3 c_4 t^2}{1 - (c_4 + c_5 + d_5) t - \cfrac{c_5 c_6 t^2}{1- \cdots}}}}
   \;,
 \label{eq.contraction_even.Ttype}
\end{eqnarray}
so that the J-fraction on the right-hand side has coefficients
\begin{subeqnarray}
   e_0  & = &  c_1 + d_1   \\
   e_n  & = &  c_{2n} + c_{2n+1} + d_{2n+1}
                            \qquad\hbox{for $n \ge 1$} \\
   f_n  & = &  c_{2n-1} c_{2n}
 \label{eq.contraction_even.Ttype.coeffs}
\end{subeqnarray}
\end{proposition}

\noindent
Proposition~\ref{prop.contraction_even.Ttype} can be proven
\cite{Sokal_totalpos} either algebraically by manipulating
the continued fraction, or combinatorially by using the
representation of T-fractions (resp.\ J-fractions)
in terms of Schr\"oder (resp.\ Motzkin) paths.

\subsection{Binomial transform}   \label{subsec.binomialtrans}

Let $\ba = (a_n)_{n \ge 0}$ be a sequence
with values in a commutative ring (with identity element) $R$,
and let $\xi$ be an element of $R$ (or an indeterminate).
Then the {\em $\xi$-binomial transform}\/ of $\ba$
is the sequence $\bb = (b_n)_{n \ge 0}$ defined by
\be
   b_n  \;=\;  \sum_{k=0}^n \binom{n}{k} \, a_k \, \xi^{n-k}
   \;,
\ee
or in other words $\bb = B_\xi \, \ba$
where $B_\xi$ is the (unit-lower-triangular) $\xi$-binomial matrix
[cf., \eqref{def_Bxi}]
\be
   (B_\xi)_{nk}  \;=\;  \binom{n}{k} \: \xi^{n-k}
   \;.
 \label{def.Bxi}
\ee

If the ring $R$ contains the rationals, then we can form the
exponential generating functions
\be
   A(t) \;\eqdef\; \sum_{n=0}^\infty a_n \: \frac{t^n}{n!} \;,\qquad
   B(t) \;\eqdef\; \sum_{n=0}^\infty b_n \: \frac{t^n}{n!}
\ee
(considered as formal power series in the indeterminate $t$),
and an easy computation shows that they are related by
\be
   B(t)  \;=\; e^{\xi t} \, A(t)
   \;.
 \label{eq.binomialtrans.egf}
\ee
Therefore, an egf exhibiting a prefactor $e^{\xi t}$ is the
necessary and sufficient signal of a $\xi$-binomial transform.

On the other hand, if we look at the {\em ordinary}\/ generating functions
\be
   a(t) \;\eqdef\; \sum_{n=0}^\infty a_n \, t^n \;,\qquad
   b(t) \;\eqdef\; \sum_{n=0}^\infty b_n \, t^n 
\ee
(again considered as formal power series in the indeterminate $t$),
Euler \cite{Euler_1755_0} showed that\footnote{
   The identity \reff{eq.binomialtrans.euler} is a special case of
   the fundamental theorem of Riordan arrays (FTRA)
   \cite[pp.~137--144]{Barry_16}.
}
\be
   b(t)
   \;=\;
   \frac{1}{1-\xi t} \; a\Bigl( \frac{t}{1-\xi t} \Bigr)
   \;.
 \label{eq.binomialtrans.euler}
\ee
Furthermore, it turns out that whenever
$a(t)$ can be expressed as an S-fraction (resp.\ J-fraction),
then $b(t)$ can be expressed as a T-fraction (resp.\ J-fraction),
as we shall now show.

Let us start with the case of an S-fraction.
Let $\bc = (c_i)_{i \ge 1}$ be indeterminates,
and define the Stieltjes--Rogers polynomials $S_n(\bc) \in \Z[\bc]$ by
\be
   \sum_{n=0}^\infty S_n(c_1,\ldots,c_n) \, t^n
   \;=\;
   \cfrac{1}{1 - \cfrac{c_1 t}{1 - \cfrac{c_2 t}{1- \cdots}}}
   \;.
 \label{eq.Stype.cfrac.0}
\ee
Similarly,
let $\bc = (c_i)_{i \ge 1}$ and $\bd = (d_i)_{i \ge 1}$ be indeterminates,
and define the Thron--Rogers polynomials $T_n(\bc,\bd) \in \Z[\bc,\bd]$ by
\be
   \sum_{n=0}^\infty T_n(c_1,\ldots,c_n,d_1,\ldots,d_n) \, t^n
       \;=\; \cfrac{1}{1 - d_1 t - \cfrac{c_1 t}{ 1 - d_2 t -
                           \cfrac{c_2 t}{ 1 - \cdots}}}
   \;.
 \label{eq.Ttype.cfrac.0}
\ee
Barry \cite[Proposition~3]{Barry_09} proved the following:

\begin{proposition}[$\xi$-binomial transform of S-fraction as T-fraction]
    \label{prop.xi_binomial}
We have
\be
\sum\limits_{k=0}^n \binom{n}{k} \, S_k(\bc) \, \xi^{n-k} \;=\;
  T_n(\bc, \xi \bone_{\rm odd})
\label{eq_xi_binomial}
\ee
as an identity in $\Z[\bc,\xi]$, where $\bone_{\rm odd} = (1,0,1,0,\ldots)$.
\end{proposition}

\noindent
So the $\xi$-binomial transform of an S-fraction with coefficients
$\bc = (c_i)_{i \ge 1}$ is a T-fraction with the same coefficients $\bc$
and in which $d_i = \xi$ (resp.\ 0) for odd (resp.\ even) $i$.

In fact, this result can be generalized \cite{Sokal_totalpos}
from S-fractions to a subclass of T-fractions,
namely, those in which $d_i = 0$ at all even levels $i$:

\begin{proposition}[$\xi$-binomial transform of a subclass of T-fractions]
   \label{prop.Tn.x-binomial}
We have
\be
   \sum_{k=0}^n \binom{n}{k} \, T_k(\bc,\bd_{\rm odd}) \, \xi^{n-k}
   \;=\;
   T_n(\bc,\bd_{\rm odd}+\xi\bone_{\rm odd})
 \label{eq.prop.Tn.x-binomial}
\ee
as an identity in $\Z[\bc,\bd_{\rm odd},\xi]$,
where $\bd_{\rm odd} = (d_1,0,d_3,0,\ldots)$
and $\bone_{\rm odd} = (1,0,1,0,\ldots)$.
\end{proposition}

Finally, let us consider the case in which $a(t)$ can be expressed
as a J-fraction.
Let $\bee = (e_i)_{i \ge 0}$ and $\bff = (f_i)_{i \ge 1}$ be indeterminates,
and define the Jacobi--Rogers polynomials $J_n(\bee,\bff) \in \Z[\bee,\bff]$ by
\be
   \sum_{n=0}^\infty J_n(\bee,\bff) \, t^n
       \;=\; \cfrac{1}{1 - e_0 t - \cfrac{f_1 t^2}{ 1 - e_1 t -
                           \cfrac{f_2 t^2}{ 1 - \cdots}}}
   \;.
 \label{eq.Jtype.cfrac.0}
\ee
We then have:

\begin{proposition}[$\xi$-binomial transform of J-fraction]
   \label{prop.Jn.x-binomial}
We have
\be
   \sum_{k=0}^n \binom{n}{k} \, J_k(\bee,\bff) \, \xi^{n-k}
   \;=\;
   J_n(\bee +  \xi\bone, \bff)
 \label{eq.prop.Jn.x-binomial}
\ee
as an identity in $\Z[\bee,\bff,\xi]$, where $\bone = (1,1,1,1,\ldots)$.
\end{proposition}

\noindent
In other words, the $\xi$-binomial transform simply adds a constant $\xi$
to all the coefficients $e_i$.
This result can be found in Aigner \cite[eq.~(6.15)]{Aigner_01a}
and Barry \cite[Proposition~4]{Barry_09},
but it may well go back to the late nineteenth century.

Let us remark that all three Propositions can be proven either
algebraically by manipulating the continued fraction using Euler's
formula \reff{eq.binomialtrans.euler},
or combinatorially by using the expressions for the Stieltjes--Rogers,
Jacobi--Rogers and Thron--Rogers polynomials in terms of Dyck, Motzkin
and Schr\"oder paths, respectively.
See \cite{Sokal_totalpos} for details.

By combining the contraction formula (Proposition~\ref{prop.contraction_even})
with Proposition~\ref{prop.Jn.x-binomial},
we can alternatively write the $\xi$-binomial transform of an S-fraction
as a J-fraction:

\begin{corollary}[$\xi$-binomial transform of S-fraction as J-fraction]
    \label{cor.xi_binomial.StoJ}
We have
\be
\sum\limits_{k=0}^n \binom{n}{k} \, S_k(\bc) \, \xi^{n-k} \;=\;
  J_n(\bee,\bff)
\label{eq_xi_binomial.StoJ}
\ee
where
\begin{subeqnarray}
   e_0  & = &  c_1 \,+\, \xi
       \slabel{eq.contraction_even.coeffs.a.xi}   \\
   e_n  & = &  c_{2n} + c_{2n+1} \,+\, \xi  \qquad\hbox{for $n \ge 1$}
       \slabel{eq.contraction_even.coeffs.b.xi}   \\
   f_n  & = &  c_{2n-1} c_{2n}
       \slabel{eq.contraction_even.coeffs.c.xi}
\end{subeqnarray}
\end{corollary}

So the $\xi$-binomial transform of an S-fraction
can be written either as a T-fraction (Proposition~\ref{prop.xi_binomial})
or as a J-fraction (Corollary~\ref{cor.xi_binomial.StoJ}).
Of course, the T-fraction and the J-fraction are equivalent
by virtue of Proposition~\ref{prop.contraction_even.Ttype};
otherwise put, the $\xi$-binomial transform commutes with contraction
(as it must).

Similarly, by combining the contraction formula
for special T-fractions (Proposition \ref{prop.contraction_even.Ttype})
with Proposition~\ref{prop.Jn.x-binomial},
we can write the $\xi$-binomial transform of a special T-fraction
as a J-fraction:

\begin{corollary}[$\xi$-binomial transform of special T-fraction as J-fraction]
    \label{cor.xi_binomial.TtoJ}
We have
\be
\sum\limits_{k=0}^n \binom{n}{k} \, T_n(\bc, \bd_{\rm odd}) \, \xi^{n-k}
   \;=\; J_n(\bee,\bff)
\label{eq_xi_binomial.TtoJ}
\ee
where $\bd_{\rm odd} = (d_1,0,d_3,0,\ldots)$ and
\begin{subeqnarray}
   e_0  & = &  c_1 \,+\, d_1 \,+\, \xi
       \slabel{eq.contraction_evenT.coeffs.a.xi}   \\
   e_n  & = &  c_{2n} + c_{2n+1} \,+\, d_{2n+1} \,+\, \xi
       \qquad\hbox{for $n \ge 1$}
       \slabel{eq.contraction_evenT.coeffs.b.xi}   \\
   f_n  & = &  c_{2n-1} c_{2n}
       \slabel{eq.contraction_evenT.coeffs.c.xi}
\end{subeqnarray}
\end{corollary}

\noindent
Once again, the $\xi$-binomial transform commutes with contraction.

%
%
\section{Some GKP recurrences with a T-fraction or J-fraction representation}
      \label{sec.mainT}

All of the S-fractions found in Propositions~\ref{prop.I}--\ref{prop.VI}
trivially give rise to T-fractions by setting all $d_i = 0$
and to J-fractions by contraction (Proposition~\ref{prop.contraction_even}).
In this section we will show some less trivial cases of the GKP recurrence
in which the ogf has a representation
as a T-fraction or a J-fraction.
However, unlike what we did for S-fractions,
we shall not attempt a complete determination.
One obstacle to a complete analysis
is that T-fractions are highly underdetermined:
corresponding to the $n$ inputs $a_1,\ldots,a_n$
there are $2n$ parameters $c_1,\ldots,c_n$ and $d_1,\ldots,d_n$.
For this reason, a systematic search for T-fractions, such as we did
Section~\ref{sec.main} for the S-fractions,
does not seem at present to be feasible.  
A systematic search seems more feasible for the J-fractions,
since they, like the S-fractions, are fully determined
(modulo degenerate cases):
the $2n$ inputs $a_1,\ldots,a_{2n}$ give rise to
$2n$ parameters $e_0,\ldots,e_{n-1}$ and $f_1,\ldots,f_n$.
But this search is more involved than what we have done for S-fractions,
and we have not yet carried it to completion.
So the examples given in this section are surely only the tip of the iceberg.

We begin by presenting some examples of GKP recurrences
whose ogf can be expressed as a T-fraction or J-fraction
arising from a binomial transform (Section~\ref{subsec.TJ}).
We then present some examples whose ogf can be expressed as a
T-fraction not arising from a binomial transform (Section~\ref{subsec.T2})
or a
J-fraction not arising from a binomial transform (Section~\ref{subsec.J2}).

\subsection{Examples of GKP recurrences with T-fraction or J-fraction
      arising from the binomial transform}
   \label{subsec.TJ}

We will now use the $\xi$-binomial transform to find some cases
of the GKP recurrence \eqref{eq_binomvert}
in which the ogf has a representation as a T-fraction or J-fraction.
More specifically,
the row-generating polynomials $P_n(x)$ of our families~7a and 7b
will be the $\xi$-binomial transforms, for suitable $\xi$,
of the row-generating polynomials for families~3a and 3b:
that is, we will have
\be
   P_n(x;\bmu') 
   \;=\;
   \sum_{k=0}^n \binom{n}{k} \, P_k(x;\bmu) \, \xi^{n-k}
\ee
for suitable $\bmu$, $\bmu'$ and $\xi$.
The T-fractions for families~7a and 7b
will then be deduced from the S-fractions for families~3a and 3b
(Proposition~\ref{prop.III})
by using Proposition~\ref{prop.xi_binomial},
and the J-fractions will be found in a similar way
by using Corollary~\ref{cor.xi_binomial.StoJ}.

Let us observe that, under duality, the coefficients $\bc$ and $\bd$
in the T-fraction behave in the same way as the coefficient $\bc$
of the S-fraction:
namely, $c_i(x) \to x \, c_i(1/x)$ and $d_i(x) \to x \, d_i(1/x)$.  
Therefore, if $f(x,t;\bmu)$ has a T-fraction representation
with coefficients that are affine in $x$,
then $f(x,t;D\bmu)$ has also a T-fraction representation
with coefficients that are also affine in $x$:
namely, $a+bx \to b+ax$.

%
%
\subsubsection{Families 7a and 7b} \label{sec.family.VII}

Families 7a and 7b are defined as follows:
\begin{quote}
\begin{itemize}
\item[F7a.\ ] $\bmu=(0,\beta,\gamma, 0,\beta',\gamma')$
\item[F7b.\ ] $\bmu=(\alpha,-\alpha,\gamma, \alpha',-\alpha',\gamma')$
\end{itemize}
\end{quote}

\noindent
Please note that family~7a 
is a generalization of family~3a (which corresponds to $\gamma=0$),
while family~7b
is a generalization of family~3b (which corresponds to $\gamma'=0$).
Furthermore, family~7a
is also a generalization of family~4a
(which corresponds to $\gamma/\beta = 1 + \gamma'/\beta'$),
while family~7b
is a generalization of family~4b
(which corresponds to $\gamma'/\alpha' = 1 + \gamma/\alpha$).
Family~7a corresponds to the case of the GKP recurrence \reff{eq_binomvert}
in which the coefficients depend only on $k$,
while family~7b is the dual case
in which the coefficients depend only on $n-k$.

The key fact is the following:

\begin{proposition}[Families~7a and 7b as binomial transforms]
                         \label{prop.VII.binomialtrans}
The row-generating polynomials of family~7a
are the $\gamma$-binomial transform of those of family~3a:
\be
   P_n(x;0,\beta,\gamma, 0,\beta',\gamma')
   \;=\;
   \sum_{k=0}^n \binom{n}{k} \, P_n(x;0,\beta,0, 0,\beta',\gamma')
                             \, \gamma^{n-k}
   \;.
 \label{eq.prop.VII.binomialtrans.7a}
\ee
Similarly, the row-generating polynomials of family~7b
are the $(\gamma' x)$-binomial transform of those of family~3b:
\be
   P_n(x; \alpha,-\alpha,\gamma, \alpha',-\alpha',\gamma')
   \;=\;
   \sum_{k=0}^n \binom{n}{k} \,
      P_n(x; \alpha,-\alpha,\gamma, \alpha',-\alpha',0) \, (\gamma' x)^{n-k}
   \;.
 \label{eq.prop.VII.binomialtrans.7b}
\ee
\end{proposition}

\begin{proof}
Family~7a is Neuwirth's \cite[Theorem~18]{Neuwirth} case $\alpha' = 0$
specialized to $\alpha = 0$.
Its egf is given by \cite[eq.~(A8)]{BSV}:
\be
F(x,t) \;=\; e^{\gamma t}  \left[
  1 + \frac{\beta'\, x}{\beta} \, \left(1 - e^{\beta\, t}\right)
  \right]^{-(\beta'+\gamma')/\beta'} \,.
 \label{eq.egf.VII.a}
\ee
Comparing this with \reff{eq.egf.IIIa},
we see that family~7a is the $\gamma$-binomial transform of family~3a.
(We will give an alternate proof of this result in
Appendix~\ref{app.spivey-zhu}:
see the Remark after Corollary~\ref{cor2.prop.spivey.corollary5}.)

Similarly, family~7b is Spivey's case (S1) $\beta = -\alpha$
specialized to $\beta' = -\alpha'$ \cite{Spivey_11}.
Its egf is given by \cite[eq.~(A4)]{BSV}:
\be
F(x,t) \;=\; e^{\gamma' x t} \,  \left[ 1 + \frac{\alpha}{\alpha'\, x} \,
                    \bigl(1 - e^{\alpha' x t} \bigr)
             \right]^{-(\alpha+\gamma)/\alpha}   \,.
 \label{eq.egf.VII.b}
\ee
Comparing this with \reff{eq.egf.IIIb},
we see that family~7b is the $(\gamma' x)$-binomial transform of family~3b.
\end{proof}

{\bf Remark.}
Since $B_{\xi_1} B_{\xi_2} = B_{\xi_1 + \xi_2}$,
it follows that the $\xi$-binomial transform of family~7a
is simply family~7a with $\gamma \to \gamma + \xi$,
and the $(\xi x)$-binomial transform of family~7b
is family~7b with $\gamma' \to \gamma' + \xi$.
\myendremark

\medskip

Combining this with Propositions~\ref{prop.III} and \ref{prop.xi_binomial},
we obtain:

\begin{corollary}[T-fraction for families 7a and 7b]
    \label{cor.VII}
The ogf $f(x,t;\bmu)$ for the recurrence \eqref{eq_binomvert} with 
$\bmu=(0,\beta,\gamma,0,\beta',\gamma')$ 
has a T-type continued fraction representation in the ring 
$\Z[x;\beta,\gamma,\beta',\gamma'][[t]]$ with coefficients 
\begin{subeqnarray}
\label{eq.coef.VIIa}
\slabel{eq.ak.VIIa}
c_{2k-1} &=& (\gamma' + k\beta')\, x\,,\qquad\,\, 
c_{2k} \;=\; k\, (\beta+\beta' x) \,, \\ 
d_{2k-1} &=& \gamma \,, \qquad\qquad\qquad \;\; d_{2k} \;=\; 0
\end{subeqnarray}
Moreover, an analogous T-type representation exists in
$\Z[x;\alpha,\gamma,\alpha',\gamma'][[t]]$ for the dual parameter
$\bmu=(\alpha,-\alpha,\gamma, \alpha',
         -\alpha', \gamma')$ with coefficients 
\begin{subeqnarray}
\label{eq.coef.VIIb}
\slabel{eq.ak.VIIb}
c_{2k-1} &=&  \gamma + k\alpha  \,,\qquad\,\, 
c_{2k} \;=\; k (\alpha+\alpha'\, x) \,, \\ 
d_{2k-1} &=& \gamma' x \,, \quad\qquad \quad d_{2k} \;=\; 0
\end{subeqnarray}
\end{corollary}   

Similarly, combining Proposition~\ref{prop.VII.binomialtrans}
with Proposition~\ref{prop.III} and Corollary~\ref{cor.xi_binomial.StoJ},
we obtain:

\begin{corollary}[J-fraction for families 7a and 7b]
    \label{cor.VII.J}
The ogf $f(x,t;\bmu)$ for the recurrence \eqref{eq_binomvert} with 
$\bmu=(0,\beta,\gamma,0,\beta',\gamma')$ 
has a J-type continued fraction representation in the ring 
$\Z[x;\beta,\gamma,\beta',\gamma'][[t]]$ with coefficients 
\begin{subeqnarray}
   e_n  & = &  [\gamma + (\beta' + \gamma') x]  \,+\,
                  n (\beta + 2\beta' x)   \\[1mm]
   f_n  & = &  n (\gamma' + n\beta') x (\beta + \beta' x)
 \label{eq.Jfrac.family7a}
\end{subeqnarray}
Moreover, an analogous J-type representation exists in
$\Z[x;\alpha,\gamma,\alpha',\gamma'][[t]]$ for the dual parameter
$\bmu=(\alpha,-\alpha,\gamma, \alpha',
         -\alpha', \gamma')$ with coefficients 
\begin{subeqnarray}
   e_n  & = &  (\alpha + \gamma + \gamma' x) \,+\,
                  n (2\alpha + \alpha' x)   \\[1mm]
   f_n  & = &  n (\gamma + n\alpha) (\alpha + \alpha' x) 
 \label{eq.Jfrac.family7b}
\end{subeqnarray}
\end{corollary}   

\bigskip

{\bf Remarks.}
1.  The dual of $\bmu=(0,\beta,\gamma, 0,\beta',\gamma')$
is $D\bmu = (\beta',-\beta',\gamma';\beta,-\beta,\gamma)$,
so family~7b is obtained from family~7a
by applying duality followed by the map
$(\beta',\gamma',\beta,\gamma)\mapsto (\alpha,\gamma,\alpha',\gamma')$.
The T-fraction for family~7b can then be deduced by duality 
from the one for family~7a, by using the observations made at
the beginning of this section.

2. The special case $\bmu=(0,0,1;0,1,0)$ of family~7a
leads to the injective numbers $\text{Inj}(n,k) = n!/(n-k)!$
\cite{Fekete_94} \cite[\seqnum{A008279}]{OEIS}.
The dual array $T(n,k) = n!/k!$ \cite[\seqnum{A094587}]{OEIS}
is the special case $\bmu=(1,-1,0;0,0,1)$ of family~7b.

3. As noted earlier, families~2a and 2b are equivalent,
at the level of the matrices $\bT(\bmu)$,
to special cases of families~3a and 3b, respectively.
Therefore, the $\gamma$-binomial transform of family~2a
is a special case of family~7a,
and the $(\gamma' x)$-binomial transform of family~2b
is a special case of family~7b.
\myendremark
   
\subsection{Examples of GKP recurrences with T-fraction 
      not arising from the binomial transform}
   \label{subsec.T2}

\subsubsection{Families 8a and 8b: Generalized Ward polynomials and their dual}

In a very recent paper \cite{wardpoly},
Elvey Price and Sokal have given \cite[Theorem~1.2]{wardpoly}
a T-fraction for the ordinary generating function
of some polynomials $W_n(x,u,z,w)$ that generalize the Ward polynomials
\cite[\seqnum{A134991/A181996/A269939}]{OEIS}:
the coefficients of this T-fraction are
\be
   c_i \,=\, x + (i-1)u  \,,\quad  d_i \,=\, z + (i-1)w
   \,.
 \label{eq.wardpoly.Tfrac}
\ee
Moreover, in the special case $u=x$, these polynomials satisfy a
linear recurrence of GKP form \cite[Corollary~B.2]{wardpoly}:
setting $W_n(x,x,z,w) = \sum\limits_{k=0}^n W_{n,k}(z,w) \, x^k$,
the triangular array $\bigl( W_{n,k}(z,w) \bigr)_{0 \le k \le n}$ satisfies
\be
   W_{n,k}  \;=\;  (wk+z) \, W_{n-1,k}  \:+\: (n+k-1) \, W_{n-1,k-1}
   \quad\hbox{for $n \ge 1$}
 \label{eq.ward.recurrence.zw}
\ee
with initial condition $W_{0,k} = \delta_{k0}$.
So the ogf of the GKP recurrence with $\bmu = (0,\beta,\gamma,1,1$, $-1)$
has a T-fraction with coefficients $c_i = ix$, $d_i = \gamma + (i-1)\beta$.
And by applying the scaling $S_{1,\lambda}$ [cf.\ \reff{def.Skl}]
to this T-fraction, we deduce that the ogf of the GKP recurrence with
$\bmu = (0,\beta,\gamma,\alpha',\alpha',-\alpha')$
has a T-fraction with coefficients $c_i = i\alpha' x$,
$d_i = \gamma + (i-1)\beta$.
Finally, by applying duality to this result,
we deduce that the ogf of the GKP recurrence with
$\bmu = (2\widehat{\alpha},-\widehat{\alpha},-\widehat{\alpha},\alpha',-\alpha',\gamma')$
has a T-fraction with coefficients $c_i = i\widehat{\alpha}$,
$d_i = \gamma' + (i-1)\alpha'$.
(We use the change of parameters $\alpha \eqdef 2\widehat{\alpha}$
 to avoid fractions.)

We summarize this by defining
family 8a (the rescaled generalized Ward polynomials at $u=x$)
and family 8b (their dual):
\begin{quote}
\begin{itemize}
\item[F8a.\ ] $\bmu = (0,\beta,\gamma,\alpha',\alpha',-\alpha')$
\item[F8b.\ ] $\bmu = (2\widehat{\alpha},-\widehat{\alpha},-\widehat{\alpha},\alpha',-\alpha',\gamma')$
\end{itemize}
\end{quote}
We then have:

\begin{proposition}[T-fraction for families 8a and 8b]
   \label{prop.wardpoly}
The ogf $f(x,t;\bmu)$ for the recurrence \eqref{eq_binomvert} with
$\bmu= (0,\beta,\gamma,\alpha',\alpha',-\alpha')$
has a T-type continued fraction representation in the ring
$\Z[x;\beta,\gamma,\alpha'][[t]]$ with coefficients
\be
         c_i \,=\, i\alpha' x \,,\quad
         d_i \,=\, \gamma + (i-1)\beta
         \,.
 \label{eq.prop.wardpoly.1}
\ee
Moreover, an analogous T-type representation exists in
$\Z[x;\widehat{\alpha},\alpha',\gamma'][[t]]$ for the dual parameter
$\bmu = (2\widehat{\alpha},-\widehat{\alpha},-\widehat{\alpha},\alpha',-\alpha',\gamma')$
with coefficients
\be
         c_i \,=\, i\widehat{\alpha}  \,,\qquad
         d_i \,=\, [\gamma' + (i-1)\alpha'] \, x
         \,.
 \label{eq.prop.wardpoly.2}
\ee
\end{proposition}

\medskip

Let us remark that the proof of this result in \cite{wardpoly}
is rather indirect.
The polynomials $W_n(x,u,z,w)$ are defined combinatorially,
as generating polynomials for ``super-augmented perfect matchings''
of $[2n]$ with suitable weights.
The T-fraction \reff{eq.wardpoly.Tfrac}
and the recurrence \reff{eq.ward.recurrence.zw}
are then proven by combinatorial arguments (of very different forms).
It is an interesting open problem to prove Proposition~\ref{prop.wardpoly}
by direct arguments leading from the recurrence to the T-fraction
(or vice versa).

\subsubsection{Families 9a and 9b: Conjectured T- and J-fractions}

We have found empirically a T-fraction for the family
\begin{quote}
\begin{itemize}
\item[F9a.\ ] $\bmu = (0, \beta, (\kappa+1)\beta, -\widehat{\alpha}', 2 \widehat{\alpha}', \kappa\widehat{\alpha}')$
\end{itemize}
\end{quote}
and its dual
\begin{quote}
\begin{itemize}
\item[F9b.\ ] $\bmu = (\alpha, -2\alpha, \kappa\alpha, \alpha', -\alpha',
                            (\kappa+1)\alpha')$
\end{itemize}
\end{quote}
(In family 9a we have introduced the change of parameters
   $\alpha' \eqdef -\widehat{\alpha}'$ in order to
   have plus signs in the T-fraction.)

\begin{conjecture}[T-fraction for families 9a and 9b]
   \label{conj.9a9b}
The ogf $f(x,t;\bmu)$ for the recurrence \eqref{eq_binomvert} with
$\bmu = (0, \beta, (\kappa+1)\beta, -\widehat{\alpha}', 2 \widehat{\alpha}', \kappa\widehat{\alpha}')$
has a T-type continued fraction representation in the ring
$\Z[x;\beta,\widehat{\alpha}',\kappa][[t]]$ with coefficients
\begin{subeqnarray}
   c_{2k-1} &=&  (\kappa+k) \widehat{\alpha}' x  \,,\qquad\,\, 
   c_{2k} \;=\;  k \widehat{\alpha}' x \,, \\ 
   d_{2k-1} &=&  (\kappa+2k-1) \beta \,, \quad
   d_{2k} \;=\; 0
\end{subeqnarray}
Moreover, an analogous T-type representation exists in
$\Z[x;\alpha,\alpha',\kappa][[t]]$ for the dual parameter
$\bmu = (\alpha, -2\alpha, \kappa\alpha, \alpha', -\alpha', (\kappa+1)\alpha')$
with coefficients
\begin{subeqnarray}
   c_{2k-1} &=&  (\kappa+k) \alpha  \,,\qquad\qquad
   c_{2k} \;=\;  k\alpha \,, \\ 
   d_{2k-1} &=&  (\kappa+2k-1) \alpha' x \,, \quad\, d_{2k} \;=\; 0
\end{subeqnarray}
\end{conjecture}

\noindent
We have verified Conjecture~\ref{conj.9a9b} for $0 \le n \le 20$.

It will be observed that the T-fractions of Conjecture~\ref{conj.9a9b}
have $d_i = 0$ at all even levels $i$.
Therefore, Proposition~\ref{prop.contraction_even.Ttype}
implies that Conjecture~\ref{conj.9a9b} is equivalent to:

%
%
\begin{conjecture}[%
                   J-fraction for families 9a and 9b]
The ogf $f(x,t;\bmu)$ for the recurrence \eqref{eq_binomvert} with
$\bmu = (0, \beta, (\kappa+1)\beta, \alpha', -2 \alpha', -\kappa\alpha')$
has a J-type continued fraction representation in the ring
$\Z[x;\beta,\widehat{\alpha}',\kappa][[t]]$ with coefficients
\begin{subeqnarray}
    e_n  & = &  (\kappa+2n+1) (\beta + \widehat{\alpha}' x)  \\[1mm]
    f_n  & = &  n (\kappa + n) (\widehat{\alpha}' x)^2
\end{subeqnarray}
Moreover, an analogous J-type representation exists in
$\Z[x;\alpha,\alpha',\kappa][[t]]$ for the dual parameter
$\bmu = (\alpha, -2\alpha, \kappa\alpha, \alpha', -\alpha', (\kappa+1)\alpha')$
with coefficients
\begin{subeqnarray}
    e_n  & = &  (\kappa+2n+1) (\alpha + \alpha' x)   \\[1mm]
    f_n  & = &  n (\kappa + n) \alpha^2
\end{subeqnarray}
\end{conjecture}

{\bf Remark.}
Since the recurrence for family~9a has a term
$(-n+2k+\kappa)\widehat{\alpha}' \, T(n-1,k-1)$ on the right-hand side,
and the recurrence for family~9b has a term $(n-2k)\alpha \, T(n-1,k)$
--- neither of which is nonnegative for all $0 \le k \le n$ ---
it is far from obvious that the resulting coefficients $T(n,k)$ will be
polynomials with nonnegative coefficients in $x$
and $\beta,\kappa,\widehat{\alpha}'$ (or $x$ and $\alpha,\kappa,\alpha'$).
But this property is an immediate consequence of the T-fractions or J-fractions.
Indeed, the T-fractions imply the much stronger property
of coefficientwise Hankel-total positivity
(see Section~\ref{sec.hankel} below).
It would be good to understand, directly from the recurrence,
why one gets polynomials with nonnegative coefficients.
\myendremark

\subsection{Examples of GKP recurrences with J-fraction 
      not arising from the binomial transform}
   \label{subsec.J2}

\subsubsection{Family 1c}   \label{subsec.family1c}

Family 1c is defined as follows:
\begin{quote}
\begin{itemize}
\item[F1c.\ ] $\bmu=(0,\beta,\gamma, \alpha',-\alpha',\gamma')$
\end{itemize}
\end{quote}

\noindent
It is a simultaneous generalization
of family~1a (which corresponds to $\gamma=0$)
and family~1b (which corresponds to $\gamma'=0$).
Family~1c is self-dual: duality acts by interchanging
$(\beta,\gamma) \leftrightarrow (\alpha',\gamma')$.

Family 1c is Spivey's \cite{Spivey_11} case (S3)
$\alpha/\beta = \alpha'/\beta' + 1$
specialized to $\alpha=0$.
Its egf is \cite[eq.~(A2)]{BSV}
\be
   F(x,t)
   \;=\;
   e^{(\gamma/\beta)(\beta - \alpha' x)t}
   \left( \frac{\beta - \alpha' x e^{(\beta - \alpha' x)t}}
           {\beta - \alpha' x}
   \right) ^{\! -\gamma/\beta - \gamma'/\alpha'}
   \,.
  \label{eq.egf.Ic}
\ee
For family~1c we have the following J-fraction:

\begin{proposition}[J-fraction for family 1c]
    \label{prop.family1c}
The ogf $f(x,t;\bmu)$ for the recurrence \eqref{eq_binomvert} with
$\bmu=(0,\beta,\gamma, \alpha',-\alpha',\gamma')$
has a J-type continued fraction representation in the ring
$\Z[x;\beta,\gamma,\alpha',\gamma'][[t]]$ with coefficients
\begin{subeqnarray}
   e_n  & = &  (\gamma + n\beta) \,+\, (\gamma' + n\alpha') x  \\[1mm]
   f_n  & = &  n [\beta\gamma' + \gamma\alpha' + (n-1)\beta\alpha'] x
 \label{eq.Jfrac.family1c}
\end{subeqnarray}
\end{proposition}

\noindent
Proposition~\ref{prop.family1c} is in fact a special case
of a recent result of Zhu \cite{Zhu_20},
as will be discussed in Section~\ref{subsec.GKPZ}.

%
%

\section{Open questions}   \label{sec.open}

We conclude this paper by proposing some open problems that arise naturally
from our work.

\subsection{T-fractions, J-fractions, and binomial transforms}
   \label{sec.open1}

The main result of the present paper, Theorem~\ref{theor.main},
is a complete list of all families of parameters $\bmu \in \C^6$
for which the ordinary generating function \eqref{def_ogf}
of the GKP recurrence has an S-fraction expansion \reff{def_Stype.one}
with coefficients $c_1,c_2,\ldots$
that are {\em polynomials}\/ in~$x$ (rather than rational functions).
In Section~\ref{sec.mainT} we exhibited some additional families
where the ogf \eqref{def_ogf} has an expansion
as a T-fraction \reff{def_Ttype.one} or a J-fraction \reff{def_Jtype.one}
with polynomial coefficients, but this list was not systematic.
Some of these families arose from the $\xi$-binomial transform,
but even this sublist was not systematic.
This naturally suggests the following open problems:

\begin{problem}[Classification of T-fractions]
   \label{prob.Tfrac}
Determine all parameters $\bmu \in \C^6$
for which the ordinary generating function \eqref{def_ogf}
has a T-fraction expansion \reff{def_Ttype.one}
with coefficients $c_1,c_2,\ldots$ and $d_1,d_2,\ldots$
that are {\em polynomials}\/ in~$x$.
\end{problem}

\begin{problem}[Classification of J-fractions]
   \label{prob.Jfrac}
Determine all parameters $\bmu \in \C^6$
for which the ordinary generating function \eqref{def_ogf}
has a J-fraction expansion \reff{def_Jtype.one}
with coefficients $e_0,e_1,\ldots$ and $f_1,f_2,\ldots$
that are {\em polynomials}\/ in~$x$.
\end{problem}

\begin{problem}[Classification of $\xi$-binomial transforms]
   \label{prob.binomial1}
Determine all parameter sets $(\bmu,\bmu',\xi_0,\xi_1) \in \C^{14}$ for which
\be
   P_n(x;\bmu')
   \;=\;
   \sum_{k=0}^n \binom{n}{k} \, P_k(x;\bmu) \: (\xi_0 + \xi_1 x)^{n-k}
   \;.
 \label{eq.prob.binomial1.1bis}
\ee

In the special case $\xi_1 = 0$, this problem becomes:
Determine all parameter sets $(\bmu,\bmu',\xi)$ $\in \C^{13}$ for which
\be
   P_n(x;\bmu')
   \;=\;
   \sum_{k=0}^n \binom{n}{k} \, P_k(x;\bmu) \: \xi^{n-k}
 \label{eq.prob.binomial1.1}
\ee
or equivalently
\be
   \bT(\bmu')  \;=\;  B_\xi \: \bT(\bmu)
 \label{eq.prob.binomial1.2}
\ee
where $B_\xi$ is the $\xi$-binomial matrix \reff{def.Bxi}.
\end{problem}

\bigskip

\noindent
{\bf Remark.}  The special case $\xi_0 = 0$ can also be written
by a formula similar to \eqref{eq.prob.binomial1.2}.
Indeed, \reff{eq.prob.binomial1.1bis} reads
\be
   \sum_{j=0}^n T(n,j;\bmu') \, x^j
   \;=\;
   \sum_{k=0}^n \binom{n}{k} \sum_{j=0}^k T(k,j;\bmu) \, x^j \, (\xi x)^{n-k}
   \;.
\ee
Extracting the coefficient of $x^\ell$
and then substituting $\ell = n-k+m$, we obtain
\be
   T(n,n-m;\bmu')
   \;=\;
   \sum_{k=0}^n \binom{n}{k} \, \xi^{n-k} \, T(k,k-m;\bmu)
\ee
or in other words
\be
   T(n,m;D\bmu')
   \;=\;
   \sum_{k=0}^n \binom{n}{k} \, \xi^{n-k} \, T(k,m;D\bmu)
   \;,
\ee
i.e.
\be
\bT(D\bmu') \;=\; B_\xi \: \bT(D\bmu) \,,
\ee
which is is just \eqref{eq.prob.binomial1.2} applied to the duals
$D\bmu'$ and $D\bmu$.
So the two special cases $\xi_0 = 0$ and $\xi_1 = 0$ are related by duality,
and we obtain an essentially new problem only when
both $\xi_0$ and $\xi_1$ are nonzero.
\myendremark 

\bigskip

A related problem, in which the $\xi$-binomial matrix acts on the right, is:

\begin{problem}[Classification of $x$-shift transforms]
   \label{prob.binomial2}
Determine all parameter sets $(\bmu,\bmu'$, $\xi)$ $\in \C^{13}$ for which
\be
   P_n(x;\bmu')
   \;=\;
   P_n(x+\xi;\bmu)
 \label{eq.prob.binomial2.1}
\ee
or equivalently
\be
   \bT(\bmu')  \;=\;  \bT(\bmu) \: B_\xi
   \;.
 \label{eq.prob.binomial2.2}
\ee
\end{problem}

\noindent
In Section~\ref{sec.sym} we mentioned that
a brute-force computation using $n=0,1,2,3$
showed that the only solutions to the equations
$P_n(x;\bmu') = P_n(x + \xi; \bmu)$
valid for {\em generic}\/ parameters $\bmu$
are the identity map ($\xi=0$, $\bmu' = \bmu$)
and the Zhu involution
($\xi = -\beta/\beta'$ with $\bmu'$ given by \reff{def_transinv}).
However, there almost certainly do exist solutions
that are valid for lower-dimensional sets of $\bmu$;
the goal of Problem~\ref{prob.binomial2} is to find them all.

Finally, Problems~\ref{prob.binomial1} and \ref{prob.binomial2}
can be amalgamated as follows:

\begin{problem}[Classification of two-sided $\xi$-binomial transforms]
   \label{prob.binomial3}
Determine all parameter sets $(\bmu,\bmu',\xi_0,\xi_1,\xi_2) \in \C^{15}$
for which
\be
   P_n(x;\bmu')
   \;=\;
   \sum_{k=0}^n \binom{n}{k} \, P_k(x+\xi_2;\bmu) \: (\xi_0 + \xi_1 x)^{n-k}
   \;.
 \label{eq.prob.binomial3}
\ee
\end{problem}

\subsection{Coefficientwise Hankel-total positivity} \label{sec.hankel}

We begin by recalling
that a finite or infinite matrix of real numbers is called
{\em totally positive}\/ (TP) if all its minors are nonnegative,
and {\em totally positive of order~$r$} (TP${}_r$)
if all its minors of size $\le r$ are nonnegative.
Background information on totally positive matrices can be found
in \cite{Karlin_68,Gantmacher_02,Pinkus_10,Fallat_11};
they have application to many fields of pure and applied mathematics.
In particular, it is known
\cite[Th\'eor\`eme~9]{Gantmakher_37} \cite[section~4.6]{Pinkus_10}
that an infinite Hankel matrix $(a_{i+j})_{i,j \ge 0}$
of real numbers is totally positive if and only~if the underlying sequence
$(a_n)_{n \ge 0}$ is a Stieltjes moment sequence,
i.e.\ the moments of a positive measure on $[0,\infty)$.

Let us now consider sequences and matrices, not of real numbers,
but of polynomials (with integer or real coefficients)
in one or more indeterminates~$\bfx$:
in applications they will often be generating polynomials that enumerate
some combinatorial objects with respect to one or more statistics.
We equip the polynomial ring $\R[\bfx]$ with the coefficientwise
partial order:  that is, we say that $P$ is nonnegative
(and write $P \myge 0$)
in case $P$ is a polynomial with nonnegative coefficients.
We then say that a matrix with entries in $\R[\bfx]$ is
\emph{coefficientwise totally positive}
if all its minors are polynomials with nonnegative coefficients;
and analogously for coefficientwise total positivity of order~$r$.
We say that a sequence $\ba = (a_n)_{n \ge 0}$ with entries in $\R[\bfx]$
is \emph{coefficientwise Hankel-totally positive}
if its associated infinite Hankel matrix is coefficientwise totally positive;
and likewise for the version of order $r$.
Coefficientwise Hankel-total positivity of a sequence of polynomials
$(P_n(\bfx))_{n \ge 0}$ obviously {\em implies}\/
the pointwise Hankel-total positivity (i.e.\ the Stieltjes moment property)
for all $\bfx \ge \bzero$, but it is vastly stronger.

The key fact connecting S-fractions to coefficientwise Hankel-total positivity
is the following result \cite{Sokal_flajolet,Sokal_totalpos},
which is an immediate consequence of old ideas of Viennot
\cite[pp.~IV-13--IV-15]{Viennot_83}:

\begin{theorem}[Total positivity of S-fractions]
   \label{thm.TP.Sfrac}
Let $\bc = (c_i)_{i \ge 1}$ be indeterminates,
and define the Stieltjes--Rogers polynomials $S_n(\bc) \in \Z[\bc]$ by
\be
   \sum_{n=0}^\infty S_n(c_1,\ldots,c_n) \, t^n
   \;=\;
   \cfrac{1}{1 - \cfrac{c_1 t}{1 - \cfrac{c_2 t}{1- \cdots}}}
   \;.
 \label{eq.Stype.cfrac}
\ee
Then the sequence $\bS = ( S_n(\bc) )_{n \ge 0}$
is coefficientwise Hankel-totally positive in the indeterminates~$\bc$.

In particular, if we specialize the $c_i$ to be polynomials with
nonnegative real coefficients in some indeterminates~$\bfx$,
then the specialized sequence $\bS = ( S_n(\bc) )_{n \ge 0}$
is coefficientwise Hankel-totally positive in the indeterminates~$\bfx$.
\end{theorem}

Moreover, this result generalizes to T-fractions
\cite{Sokal_flajolet,Sokal_totalpos}:

\begin{theorem}[Total positivity of T-fractions]
   \label{thm.TP.Tfrac}
Let $\bc = (c_i)_{i \ge 1}$ and $\bd = (d_i)_{i \ge 1}$ be indeterminates,
and define the Thron--Rogers polynomials $T_n(\bc,\bd) \in \Z[\bc,\bd]$ by
\be
   \sum_{n=0}^\infty T_n(c_1,\ldots,c_n,d_1,\ldots,d_n) \, t^n
       \;=\; \cfrac{1}{1 - d_1 t - \cfrac{c_1 t}{ 1 - d_2 t -  
                           \cfrac{c_2 t}{ 1 - \cdots}}}
   \;.
 \label{eq.Ttype.cfrac}
\ee
Then the sequence $\bT = ( T_n(\bc,\bd) )_{n \ge 0}$
is coefficientwise Hankel-totally positive in the indeterminates
$\bc$ and $\bd$.

In particular, if we specialize the $c_i$ and $d_i$ to be polynomials with
nonnegative real coefficients in some indeterminates~$\bfx$,
then the specialized sequence $\bT = ( T_n(\bc,\bd) )_{n \ge 0}$
is coefficientwise Hankel-totally positive in the indeterminates~$\bfx$.
\end{theorem}

There is also a result for J-fractions \cite{Sokal_flajolet,Sokal_totalpos},
but it is more delicate;
Hankel-total positivity does {\em not}\/ hold coefficientwise
in the parameters $\bee = (e_i)_{i \ge 0}$ and $\bff = (f_i)_{i \ge 1}$,
but only when those parameters satisfy suitable inequalities.
The paper \cite{Sokal_totalpos} describing this general theory
for S-, T- and J-fractions is not yet publicly available,
but the foregoing results (as well as some much stronger ones)
can be found (with proofs) in \cite[section~9]{latpath_SRTR}.

Now, simple inspection of the results summarized in 
Propositions~\ref{prop.I}--\ref{prop.VI} shows that
all the coefficients $c_i$ are polynomials
with {\em nonnegative}\/ integer coeffcients in the variable $x$ and 
the parameters $\alpha,\beta,\gamma,\alpha',\beta',\gamma',\kappa$
(or a subset of them).
Therefore, Theorem~\ref{thm.TP.Sfrac} implies that
the corresponding sequence $\bP = (P_n(x;\bmu))_{n \ge 0}$
of row-generating polynomials is coefficientwise Hankel-totally positive,
jointly in all these indeterminates.
Similarly, in Corollary~\ref{cor.VII} and Proposition~\ref{prop.wardpoly},
all the coefficients $c_i$ and $d_i$ are polynomials
with nonnegative integer coeffcients in $x$ and the parameters,
so Theorem~\ref{thm.TP.Tfrac} implies coefficientwise Hankel-total positivity.
Our results in this paper therefore imply:

\begin{corollary}
In all the families~1a--6 of Theorem~\ref{theor.main},
as well as families~7a and 7b of Corollary~\ref{cor.VII}
and families~8a and 8b of Proposition~\ref{prop.wardpoly},
the sequence $\bP = (P_n(x;\bmu))_{n \ge 0}$
of row-generating polynomials is coefficientwise Hankel-totally positive,
jointly in all the indeterminates
$x$ and $\alpha,\beta,\gamma,\alpha',\beta',\gamma',\kappa$.
\end{corollary}  

But vastly more appears to be true:  it seems that
coefficientwise Hankel-total positivity is a {\em general}\/ property
of the GKP recurrence, not just of the special families 1a--8b studied here.
Namely, we conjecture:

\begin{conjecture}[Coefficientwise Hankel-total positivity of the GKP recurrence]
   \label{conj.hankel}
The sequence $\bP = (P_n(x;\bmu))_{n \ge 0}$
of row-generating polynomials of the GKP recurrence
is coefficientwise Hankel-totally positive,
jointly in all seven indeterminates
$x$ and $\alpha,\beta,\gamma,\alpha',\beta',\gamma'$.
\end{conjecture}  

\noindent
This conjecture was made a few years ago by one of us \cite{Sokal_unpub}
and was confirmed at that time up to the $8 \times 8$ Hankel matrix
$(P_{i+j}(x;\bmu))_{0 \le i,j \le 7}$.\footnote{
   This computation was performed in {\sc Mathematica};
   it required approximately $1.2 \times 10^7$ seconds CPU time
   and 587 GB memory,
   on a system using Intel Xeon E7-8837 processors running at 2.67 GHz.
}

In fact, a slightly stronger conjecture appears to be true:
in place of the usual GKP recurrence
\begin{equation}
  T(n,k)
  \;=\;
  (\alpha n + \beta k + \gamma)    \, T(n-1,k)
  \:+\:
  (\alpha' n + \beta' k + \gamma') \, T(n-1,k-1)
   \;,
  \label{eq_binomvert2}
\end{equation}
we can write instead
\begin{equation}
  T(n,k)
  \;=\;
  [\widetilde{\alpha} (n-1) + \widetilde{\beta} k + \widetilde{\gamma}]
       \, T(n-1,k)
  \:+\:
  [\widetilde{\alpha}' (n-1) + \widetilde{\beta}' (k-1) + \widetilde{\gamma}']
       \, T(n-1,k-1)
   \;,
  \label{eq_binomvert.stronger}
\end{equation}
which is equivalent to \reff{eq_binomvert2} with
\be
   \alpha = \widetilde{\alpha} ,\quad
   \beta = \widetilde{\beta} ,\quad
   \gamma = \widetilde{\gamma} - \widetilde{\alpha} ,\quad
   \alpha' = \widetilde{\alpha}' ,\quad
   \beta' = \widetilde{\beta}' ,\quad
   \gamma' = \widetilde{\gamma}' - \widetilde{\alpha}' - \widetilde{\beta}'
   \;.
\ee
Even in this new parametrization,
we apparently still get coefficientwise Hankel-total positivity:

\begin{conjecture}[Coefficientwise Hankel-total positivity of the GKP recurrence, strong version]
   \label{conj.hankel.strong}
The sequence $\bP = (P_n(x;\widetilde{\bmu}))_{n \ge 0}$
of row-generating polynomials of the recurrence \reff{eq_binomvert.stronger}
is coefficientwise Hankel-totally positive,
jointly in all seven indeterminates $x$ and
$\widetilde{\alpha},\widetilde{\beta},\widetilde{\gamma},
 \widetilde{\alpha}',\widetilde{\beta}',\widetilde{\gamma}'$.
\end{conjecture}  

\noindent
It was in fact this stronger conjecture that was confirmed
\cite{Sokal_unpub} up to the $8 \times 8$ Hankel matrix.
We have now also confirmed the coefficientwise Hankel-total positivity
of orders~2 and 3 up to the $21 \times 21$ and $9 \times 9$ Hankel matrices,
respectively.\footnote{
   For order~2, this computation was done by checking
   the strong log-convexity \reff{eq.SLC},
   coefficientwise in all seven indeterminates,
   up to $n = 38$.
   This is equivalent \cite{Sokal_totalpos}
   to the coefficientwise total positivity of order~2
   of the $21 \times 21$ Hankel matrix.
   This computation took $1.19 \times 10^7$ seconds CPU time and 19 GB memory.

   For order~3, this computation was done by the direct method:
   it took 656116 seconds CPU time and 157 GB memory.
}

Please observe that \reff{eq_binomvert.stronger} is exactly what we need
to ensure that the {\em matrix elements}\/ $T(n,k)$ are
polynomials with nonnegative coefficients in the parameters
$\widetilde{\alpha},\widetilde{\beta},\widetilde{\gamma},
 \widetilde{\alpha}',\widetilde{\beta}',\widetilde{\gamma}'$,
since we have written $n-1$ (resp.\ $k-1$) precisely in those terms
where we know that $n-1$ (resp.\ $k-1$) must be nonnegative
if the corresponding term is to make a nonzero contribution.
In other words, it is immediate from \reff{eq_binomvert.stronger}
that the sequence $\bP = (P_n(x;\widetilde{\bmu}))_{n \ge 0}$
is coefficientwise Hankel-totally positive {\em of order~1}\/
in $x$ and
$\widetilde{\alpha},\widetilde{\beta},\widetilde{\gamma},
 \widetilde{\alpha}',\widetilde{\beta}',\widetilde{\gamma}'$.
But the coefficientwise Hankel-total positivity of higher order
is decidedly nontrivial!

Some very weak versions of
Conjectures~\ref{conj.hankel} and \ref{conj.hankel.strong}
have been proven.
Liu and Wang \cite[Theorem~4.1 and Remark~4.2]{Liu_07}
showed that if $\alpha,\beta,\gamma,\alpha',\beta',\gamma'$
are real numbers satisfying
\begin{subeqnarray}
   & &
   \alpha \,\ge\,0 \,,\quad
   \alpha + \beta \,\ge\,0 \,,\quad
   \alpha + \gamma \,\ge\,0
         \\[2mm]
   & &
   \alpha' \,\ge\,0 \,,\quad
   \alpha' + \beta' \,\ge\,0 \,,\quad
   \alpha' + \beta' + \gamma' \,\ge\,0
         \\[2mm]
   & &
   \beta \alpha' - \alpha \beta' \,\ge\, 0
         \\[2mm]
   & &
   \beta (\alpha' + \beta') - \alpha \beta' \,\ge\, 0
         \\[2mm]
   & &
   \beta (\alpha' + \beta' + \gamma') - (\alpha + \gamma) \beta' \,\ge\, 0
 \label{eq.liu-wang}
\end{subeqnarray}
then the sequence $\bP = (P_n(x;\widetilde{\bmu}))_{n \ge 0}$
is coefficientwise log-convex in~$x$, i.e.\
\be
   P_n(x) \, P_{n+2}(x) \:-\: P_{n+1}(x)^2
   \; \myge_x \; 0
\ee
for all $n \ge 0$.
This is equivalent to the coefficientwise nonnegativity (in~$x$)
of all the {\em contiguous}\/ $2 \times 2$ minors
of the Hankel matrix $(P_{i+j}(x;\bmu))_{i,j \ge 0}$;
it is weaker than the full coefficientwise Hankel-total positivity of order~2.

Chen, Wang and Yang \cite[Theorem~2.4]{Chen_11} showed that
under the slightly different hypotheses
\begin{subeqnarray}
   & &
   \alpha \,\ge\,0 \,,\quad
   \beta \,\ge\,0 \,,\quad
   \alpha + \beta + \gamma \,\ge\,0
         \\[2mm]
   & &
   \alpha' \,\ge\,0 \,,\quad
   \beta' \,\ge\,0 \,,\quad
   \alpha' + \beta' + \gamma' \,\ge\,0
 \label{eq.chen-wang-yang}
\end{subeqnarray}
--- which neither imply nor are implied by \reff{eq.liu-wang} ---
the sequence $\bP = (P_n(x;\widetilde{\bmu}))_{n \ge 0}$
is coefficientwise strongly log-convex in~$x$, i.e.\
\be
   P_m(x) \, P_{n+2}(x) \:-\: P_{m+1}(x) \, P_{n+1}(x)
   \; \myge_x \; 0
 \label{eq.SLC}
\ee
for all $n \ge m \ge 0$.
It can be shown \cite{Sokal_totalpos}
that the coefficientwise strong log-convexity (in~$x$)
is equivalent to the coefficientwise (in~$x$)
Hankel-total positivity of order~2.

But these results are very far from proving even the
coefficientwise Hankel-total positivity of order~2
--- much less the coefficientwise Hankel-total positivity of all orders ---
in the seven indeterminates jointly.

\subsection{A generalization: The Graham--Knuth--Patashnik--Zhu recurrence}
   \label{subsec.GKPZ}

Inspired by a very recent paper of Zhu \cite{Zhu_20},
we would like to propose the following generalization of the GKP recurrence,
which we shall call the
{\em Graham--Knuth--Patashnik--Zhu}\/ (GKPZ) {\em recurrence}\/:
\begin{eqnarray}
  T(n,k)
  & = &
  (\alpha n + \beta k + \gamma)    \, T(n-1,k)
  \:+\:
  (\alpha' n + \beta' k + \gamma') \, T(n-1,k-1)
  \qquad
         \nonumber \\[1mm]
  & & \;\;
  \;+\;
  \sigma \, (n-k+1) \, T(n-1,k-2)
  \:+\:
  \tau \, (k+1) \, T(n-1,k+1)
  \label{eq.GKPZ}
\end{eqnarray}
for $n \ge 1$ and $k \in \Z$, with initial condition $T(0,k) = \delta_{k0}$
and parameters
$\bmu = (\alpha,\beta,\gamma,$ $\alpha',\beta',\gamma',\sigma,\tau)$.
Please note that because the coefficient $n-k+1$
in the ${T(n-1,k-2)}$ term vanishes when $k=n+1$,
it follows by induction on~$n$ that $T(n,k) = 0$ when $k > n$.
Similarly, because the coefficient $k+1$ in the ${T(n-1,k+1)}$ term
vanishes when $k=-1$, it follows that $T(n,k) = 0$ when $k < 0$.
So the matrix $\bT(\bmu) = \bigl( T(n,k;\bmu) \bigr)_{n \ge 0,\, k \in \Z}$
remains lower-triangular even in the presence of the two new terms.

A simple computation shows that the dual array $T^*(n,k) \eqdef T(n,n-k)$
satisfies a GKPZ recurrence with parameters
\be
   D\bmu 
   \;\eqdef\; 
   (\alpha' + \beta',-\beta',\gamma', \alpha + \beta, -\beta, \gamma,
    \tau,\sigma) \,.
 \label{eq.duality.GKPZ}
\ee
This is simply the GKP duality \reff{eq.duality.mu}
together with an interchange of $\sigma$ and $\tau$.

For the special case \cite[eq.~(1.3)]{Zhu_20}\footnote{
   The connection with Zhu's \cite[eq.~(1.3)]{Zhu_20}
   variables $\lambda, d, a_1, a_2, b_1, b_2$
   is $\kappa = d/\lambda$ and then
   $a_1 = \beta/\lambda$, $a_2 = \gamma/\lambda$,
   $b_1 = \alpha' + \kappa\beta$,
   $b_2 = \gamma' + \kappa(\beta-\gamma)$.
}
\be
   \bmu  \;=\; (0, \beta, \gamma, \alpha', -\alpha' + \kappa\beta, \gamma',
                \kappa\alpha',0)
   \;,
 \label{eq.GKPZ.zhu}
\ee
Zhu \cite{Zhu_20} proved
a Dobi\'nski-type formula for the row-generating polynomials $P_n(x)$
\cite[Theorem~2.2(i)]{Zhu_20},
an explicit formula for the exponential generating function
\cite[Theorem~2.2(ii)]{Zhu_20},
a J-fraction for the ordinary generating function
\cite[Theorem~2.7(i)]{Zhu_20},
and coefficientwise Hankel-total positivity (in~$x$)
for various further-specialized cases of the parameters $\bmu$
\cite[Theorem~2.7(iii)]{Zhu_20}.
The egf and J-fraction can be stated as follows:

\begin{proposition}[Zhu \cite{Zhu_20}]
   \label{prop.zhu}
For the GKPZ recurrence \eqref{eq.GKPZ} with parameters
$\bmu = (0, \beta, \gamma, \alpha', -\alpha' + \kappa\beta, \gamma',
                \kappa\alpha',0)$:
\begin{itemize}
   \item[(a)]
The egf $F(x,t;\bmu)$ is of the form
\be
   F(x,t)  \;=\; e^{at} \, \bigl[ 1 - b (e^{ct} - 1) \bigr]^{-\Delta}
 \label{eq.egf.zhu0}
\ee
where either
\begin{subeqnarray}
   a  & = & \frac{\gamma}{\beta} \, (\beta - \alpha' x)  \\[2mm]
   b  & = & {(\alpha' + \kappa\beta) x}{\beta - \alpha' x}   \\[2mm]
   c  & = & \beta - \alpha' x  \\[2mm]
   \Delta & = &  \frac{\gamma}{\beta}  \:+\:
                 \frac{\gamma' + \kappa(\beta-\gamma)}{\alpha' + \kappa\beta}
 \slabel{eq.Delta.zhu1}
 \label{eq.egf.zhu1}
\end{subeqnarray}
or alternatively
\begin{subeqnarray}
   a  & = & - \, \frac{\gamma' + \kappa(\beta-\gamma)}{\alpha' + \kappa\beta}
              \, (\beta - \alpha' x)  \\[2mm]
   b  & = & - \, \frac{\beta (1 + \kappa x)}{\beta - \alpha' x}  \\[2mm]
   c  & = & - (\beta - \alpha' x)  \\[2mm]
   \Delta & = &  \frac{\gamma}{\beta}  \:+\:
                 \frac{\gamma' + \kappa(\beta-\gamma)}{\alpha' + \kappa\beta}
 \slabel{eq.Delta.zhu2}
 \label{eq.egf.zhu2}
\end{subeqnarray}
   \item[(b)]
The ogf $f(x,t;\bmu)$ has a J-type continued fraction representation
in the ring $\Z[x,\beta,\gamma$, $\alpha',\kappa][[t]]$ with coefficients 
\begin{subeqnarray}
   e_n  & = &  (\gamma + n\beta)(1 + \kappa x)  \:+\: 
      [\gamma' + \kappa(\beta-\gamma) + n(\alpha' + \kappa\beta)] x
    \\[1mm]
   f_n & = & n [\beta\gamma' + \gamma\alpha' + \kappa\beta^2 +
                (n-1) \beta (\alpha' + \kappa\beta)] x (1 + \kappa x)
 \label{eq.Jfrac.zhu}
\end{subeqnarray}
\end{itemize}
\end{proposition}

\noindent
There is of course also a dual result to Proposition~\ref{prop.zhu},
whose statement we leave to the reader.

We remark that Zhu \cite{Zhu_20}
proves the Dobi\'nski-type formula for $P_n(x)$ \cite[Theorem~2.2(i)]{Zhu_20}
by induction on $n$;
from this he deduces the formula \reff{eq.egf.zhu0}--\reff{eq.egf.zhu2}
for the egf \cite[Theorem~2.2(ii)]{Zhu_20} by a straightforward computation.
Finally, he deduces the J-fraction \reff{eq.Jfrac.zhu}
\cite[Theorem~2.7(i)]{Zhu_20}
from the egf by the Stieltjes--Rogers addition-formula method
\cite[pp.~203--207]{Wall_48}.\footnote{
   Let us remark that the argument from this egf to the J-fraction
   for the corresponding ogf was essentially already known to Stieltjes.
   More precisely, as noted already in Section~\ref{sec.SZ},
   Stieltjes \cite[section~81]{Stieltjes_1894} observed that
   the S-fraction \reff{eq.eulerian.fourvar.contfrac} with $v=y$ and $u=1$
   is the formal Laplace transform of the exponential generating function
   \reff{def.Anyw};
   or in other words, the ogf corresponding to the egf \reff{eq.egf.zhu0}
   specialized to $a=0$ has an S-fraction with coefficients
   $c_{2k-1} = bc (\Delta+k-1)$, $c_{2k} = k (b+1)c$.
   This leads by contraction (Proposition~\ref{prop.contraction_even})
   to a J-fraction with coefficients
   $e_n = bc\Delta + nc(2b+1)$, $f_n = nb(b+1)(\Delta+n-1)c^2$.
   Then one restores $a$ by applying an $a$-binomial transform
   (Proposition~\ref{prop.Jn.x-binomial}), leading to a J-fraction with
   $e_n = a + bc\Delta + nc(2b+1)$, $f_n = nb(b+1)(\Delta+n-1)c^2$.
   Inserting either \reff{eq.egf.zhu1} or \reff{eq.egf.zhu2}
   then yields \reff{eq.Jfrac.zhu}.
}

In two special cases, the GKPZ parameters \reff{eq.GKPZ.zhu}
have $\sigma = \tau = 0$ and thus reduce to a GKP recurrence:
\begin{itemize}
   \item[1)] When $\kappa = 0$, the GKPZ parameters \reff{eq.GKPZ.zhu}
reduce to our family~1c (Section~\ref{subsec.family1c}),
and the J-fraction \reff{eq.Jfrac.zhu} reduces to \reff{eq.Jfrac.family1c}.
   \item[2)] When $\alpha' = 0$, the GKPZ parameters \reff{eq.GKPZ.zhu}
reduce to our family~7a (Section~\ref{sec.family.VII}),
and the J-fraction \reff{eq.Jfrac.zhu} reduces to \reff{eq.Jfrac.family7a}.
\end{itemize}

It is a very interesting open problem to generalize our work
in the present paper, and the proposed work in the preceding two subsections,
to the GKPZ recurrence.

\section*{Acknowledgements}

We acknowledge useful discussions with 
Eduardo J.S.~Villa\-se\~nor and J.~Fernando Barbero. 
We also thank Bao-Xuan Zhu for making available a copy
of his unpublished manuscript \cite{Zhu_unpub},
and Frank Johnson for very helpful information
concerning the discrete group $\scrg$.
One of us (J.S.)\ is grateful for the hospitality
of University College London,
where part of this work was done.
This research was supported in part
by the Spanish MINECO grant FIS2014-57387-C3-3-P;
by the 
FEDER/Ministerio de Ciencia, Innovaci\'on y Universidades--Agencia Estatal
de Investigaci\'on grant FIS2017-84440-C2-2-P;
by the Madrid Government
(Comunidad de Madrid--Spain) under the Multiannual Agreement with UC3M in the
line of Excellence of University Professors (EPUC3M23) and in the context
of the V-PRICIT (Regional Plan for Scientific Research and Technological
Innovation);
and by U.K.~Engineering and Physical Sciences Research Council
grant EP/N025636/1.

\appendix

\section{Matrix product of two GKP arrays}  \label{app.spivey-zhu}

In this appendix we provide more details concerning the matrix product
$\bT(\bmu) \, \bT(\widehat{\bmu})$ of two GKP arrays
(as well as some more general arrays),
with special attention to the case in which one of the two arrays
is the $\xi$-binomial matrix $B_\xi$ defined in \reff{def.Bxi}.
Some of these results can be found in the very interesting
(but apparently little-known) thesis of Th\'eor\^et \cite{Theoret_94},
whose approach we largely follow in Section~\ref{app.spivey-zhu.1}.
A few of our results are also in Spivey \cite{Spivey_11},
but it seems to us that our formulations and proofs are simpler.

\subsection{General results}   \label{app.spivey-zhu.1}

It is convenient to start from the more general ``binomial-like''
recurrence introduced in Section~\ref{sec.scaling}.  We have
\cite[p.~11, proof of Proposition~1.1.2]{Theoret_94}
\cite[p.~199]{Theoret_95b}:

%
%
\begin{proposition}[Matrix product of two ``binomial-like'' arrays]
   \label{prop.app.1}
\hspace*{-0.7mm}Let $\bA = \bigl( A(n,k) \bigr)_{0 \le k \le n}$ 
and $\bB = \bigl( B(n,k) \bigr)_{0 \le k \le n}$
be triangular arrays defined by the recurrences
\begin{subeqnarray}
  A(n,k)  & = &  a_{n,k} \, A(n-1,k) \:+\: a'_{n,k} \, A(n-1,k-1)
     \slabel{eq0.prop.app.1.a} \\[1mm]
  B(n,k)  & = &  b_{n,k} \, B(n-1,k) \:+\: b'_{n,k} \, B(n-1,k-1)
     \slabel{eq0.prop.app.1.b}
 \label{eq0.prop.app.1}
\end{subeqnarray}
for $n \ge 1$, with initial conditions $A(0,k) = B(0,k) = \delta_{k0}$.
Then
\begin{eqnarray}
   \sum_{j=0}^n A(n,j) \, B(j,k)
   & = &
   \sum_{j=0}^{n-1} (a_{n,j} + a'_{n,j+1} b_{j+1,k}) \, A(n-1,j) \, B(j,k)
         \nonumber \\[-1mm]
   &  &
   +\;  \sum_{j=0}^{n-1} a'_{n,j+1} b'_{j+1,k} \, A(n-1,j) \, B(j,k-1)
   \;+\;  \delta_{n0} \delta_{k0}
   \;.
   \qquad
 \label{eq.prop.app.1}
\end{eqnarray}
\end{proposition}

\begin{proof}
The formula obviously holds for $n=0$.
For $n \ge 1$ we have
\begin{subeqnarray}
   & &
   \!\!\!\!
   \sum_{j=0}^n A(n,j) \, B(j,k)
        \nonumber \\
   & &
   \;=\;
   \sum_{j=0}^{n-1} a_{n,j} \, A(n-1,j) \, B(j,k)
          \nonumber \\[-2mm]
   & & \qquad
   \;+\;  \sum_{j=1}^{n} a'_{n,j} \, A(n-1,j-1) \,
       \bigl[ b_{j,k} \, B(j-1,k) \:+\: b'_{j,k} \, B(j-1,k-1) \bigr]
       \qquad
         \\
   & &
   \;=\;
   \sum_{j=0}^{n-1} a_{n,j} \, A(n-1,j) \, B(j,k)
          \nonumber \\[-2mm]
   & & \qquad
   \;+\;  \sum_{j=0}^{n-1} a'_{n,j+1} \, A(n-1,j) \,
       \bigl[ b_{j+1,k} \, B(j,k) \:+\: b'_{j+1,k} \, B(j,k-1) \bigr]
   \;.
\end{subeqnarray}
Grouping terms gives \reff{eq.prop.app.1}.
\end{proof}

Proposition~\ref{prop.app.1}
does not in general give a recurrence for the matrix product $\bC = \bA \bB$,
because the coefficients on the right-hand side of \reff{eq.prop.app.1}
depend in general on $j$ (not just on $n$ and $k$).
However, this dependence is eliminated whenever
$a_{n,k}$ and $a'_{n,k}$ depend only on $n$
and also $b_{n,k}$ and $b'_{n,k}$ depend only on $k$.
In this case we have \cite[p.~11, Exemple~1]{Theoret_94}
\cite[proof of Theorem~4.8]{Brenti_95}
\cite[Theorem~10]{Neuwirth}:

\begin{corollary}
   \label{cor.prop.app.1}
Let $\bA = \bigl( A(n,k) \bigr)_{0 \le k \le n}$
and $\bB = \bigl( B(n,k) \bigr)_{0 \le k \le n}$
be triangular arrays defined by the recurrences
\begin{subeqnarray}
  A(n,k)  & = &  a_{n} \, A(n-1,k) \:+\: a'_{n} \, A(n-1,k-1)
     \\[1mm]
  B(n,k)  & = &  b_{k} \, B(n-1,k) \:+\: b'_{k} \, B(n-1,k-1)
\end{subeqnarray}
for $n \ge 1$, with initial conditions $A(0,k) = B(0,k) = \delta_{k0}$.
Then the matrix product $\bC = \bA \bB$ satisfies the recurrence
\be
   C(n,k)
   \;=\;
   (a_{n} + a'_{n} b_{k}) \, C(n-1,k)  \:+\:  a'_{n} b'_{k} \, C(n-1,k-1)
 \label{eq.cor.prop.app.1}
\ee
for $n \ge 1$, with initial condition $C(0,k) = \delta_{k0}$.
\end{corollary}

In fact, Th\'eor\^{e}t \cite[p.~13, Corollaire~1.1.3]{Theoret_94}
exhibits a more general situation in which
the coefficients on the right-hand side of \reff{eq.prop.app.1}
are independent of $j$:
   
\begin{corollary}
   \label{cor.prop.theoret}
Let $\bA = \bigl( A(n,k) \bigr)_{0 \le k \le n}$
and $\bB = \bigl( B(n,k) \bigr)_{0 \le k \le n}$
be triangular arrays defined by the recurrences
\begin{subeqnarray}
  A(n,k)  & = &  (\alpha_n + \beta_n \gamma_k) \, A(n-1,k) \:+\:
                  \beta_n \delta_k \, A(n-1,k-1)
     \\[1mm]
  B(n,k)  & = &  \frac{\phi_k - \gamma_{n-1}}{\delta_n} \, B(n-1,k)
                     \:+\: \frac{\psi_k}{\delta_n} \, B(n-1,k-1)
 \label{eq.cor.prop.theoret.1}
\end{subeqnarray}
for $n \ge 1$, with initial conditions $A(0,k) = B(0,k) = \delta_{k0}$.
Then the matrix product $\bC = \bA \bB$ satisfies the recurrence
\be
   C(n,k)
   \;=\;
   (\alpha_n + \beta_n \phi_k) \, C(n-1,k)  \:+\:  \beta_n \psi_k \, C(n-1,k-1)
 \label{eq.cor.prop.theoret.2}
\ee
for $n \ge 1$, with initial condition $C(0,k) = \delta_{k0}$.
\end{corollary}

\begin{proof}
Apply Proposition~\ref{prop.app.1}.
\end{proof}

{\bf Remarks.}
1.  Note that the coefficients $\gamma$ and $\delta$ occur with both
$n$ (or $n-1$) and $k$ as subscripts,
while $\alpha$ and $\beta$ occur only with $n$,
and $\phi$ and $\psi$ only with $k$.
Note also that $\bdelta$ is trivially eliminated
from the matrix product $\bC = \bA \bB$
as a consequence of the rescaling lemma
(Lemma~\ref{lemma.recurrence.rescaling}).
The elimination of $\bgamma$ is, however, less obvious.

2. Th\'eor\^{e}t proves \cite[pp.~13--14, Remarque~1]{Theoret_94}
that \reff{eq.cor.prop.theoret.1} is {\em the most general}\/
pair of ``binomial-like'' recurrences
(with nonvanishing coefficients $a_{n,k}, a'_{n,k}, b_{n,k}, b'_{n,k}$)
in which the coefficients on the right-hand side of \reff{eq.prop.app.1}
are independent of $j$.
However, he also shows later \cite[section~2.6]{Theoret_94}
some other cases in which the matrix product of two binomial-like arrays
is binomial-like, using a sufficient condition that is different from
Proposition~\ref{prop.app.1}.
\myendremark

Specializing Proposition~\ref{prop.app.1} to the GKP recurrence
\begin{subeqnarray}
   & &
   a_{n,k} \:=\: \alpha n + \beta k + \gamma \;,\qquad
   a'_{n,k} \:=\: \alpha' n + \beta' k + \gamma'
         \\[1mm]
   & &
   b_{n,k} \:=\: \widehat{\alpha} n + \widehat{\beta} k + \widehat{\gamma}
       \;,\qquad
   b'_{n,k} \:=\: \widehat{\alpha}' n + \widehat{\beta}' k + \widehat{\gamma}'
\end{subeqnarray}
we obtain \cite[Corollary~5]{Spivey_11}:
   
\begin{proposition}[Matrix product of two GKP arrays]
   \label{prop.spivey.corollary5}
Let $\bA = \bigl( A(n,k) \bigr)_{0 \le k \le n}$
and $\bB = \bigl( B(n,k) \bigr)_{0 \le k \le n}$
be triangular arrays defined by the recurrences
\begin{subeqnarray}
  A(n,k)
  & = &
  (\alpha n + \beta k + \gamma)    \, A(n-1,k)
  \:+\:
  (\alpha' n + \beta' k + \gamma') \, A(n-1,k-1)
  \qquad
     \\[1mm]
  B(n,k)
  & = &
  (\widehat{\alpha} n + \widehat{\beta} k + \widehat{\gamma})    \, B(n-1,k)
  \:+\:
  (\widehat{\alpha}' n + \widehat{\beta}' k + \widehat{\gamma}') \, B(n-1,k-1)
  \qquad
 \slabel{eq0.prop.spivey.corollary5.B}
\end{subeqnarray}
for $n \ge 1$, with initial conditions $A(0,k) = B(0,k) = \delta_{k0}$.
Then
\begin{multline}
\sum_{j=0}^n A(n,j) \, B(j,k)
   \;=\;
   \sum_{j=0}^{n-1} A(n-1,j) \, B(j,k) \\ \,
      \times \left[ (\alpha n + \beta j + \gamma) \:+\:
             (\alpha' n + \beta' (j+1) + \gamma')
             (\widehat{\alpha} (j+1) + \widehat{\beta} k + \widehat{\gamma})
      \right]
           \\
   \qquad\qquad\quad   \;+\;
   \sum_{j=0}^{n-1} A(n-1,j) \, B(j,k-1) \,
       (\alpha' n + \beta' (j+1) + \gamma')
       (\widehat{\alpha}' (j+1) + \widehat{\beta}' k + \widehat{\gamma}')
           \\[2mm]
  \qquad\qquad\qquad \;+\; \delta_{n0} \delta_{k0} \;. \hfill
\label{eq.prop.spivey.corollary5}
\end{multline}
\end{proposition}

\noindent
Let us remark that Spivey \cite{Spivey_11}
used the variant notation \reff{eq_binomvert.stronger} for the GKP recurrence;
restating his result \cite[Corollary~5]{Spivey_11}
in terms of the standard notation \reff{eq_binomvert}
and simplifying the expression slightly yields
\reff{eq.prop.spivey.corollary5}.

And specializing further to
$\beta = \beta' = \widehat{\alpha} = \widehat{\alpha}' = 0$,
we obtain the GKP special case of Corollary~\ref{cor.prop.app.1}:

\begin{corollary}
   \label{cor.prop.spivey.corollary5}
Let $\bA = \bigl( A(n,k) \bigr)_{0 \le k \le n}$
and $\bB = \bigl( B(n,k) \bigr)_{0 \le k \le n}$
be triangular arrays defined by the recurrences
\begin{subeqnarray}
  A(n,k)
  & = &
  (\alpha n + \gamma)    \, A(n-1,k)
  \:+\:
  (\alpha' n + \gamma') \, A(n-1,k-1)
  \qquad
     \\[1mm]
  B(n,k)
  & = &
  (\widehat{\beta} k + \widehat{\gamma})    \, B(n-1,k)
  \:+\:
  (\widehat{\beta}' k + \widehat{\gamma}') \, B(n-1,k-1)
  \qquad
\end{subeqnarray}
for $n \ge 1$, with initial conditions $A(0,k) = B(0,k) = \delta_{k0}$.
Then the matrix product $\bC = \bA \bB$ satisfies the recurrence
\begin{eqnarray}
   C(n,k)
   & = &
      \bigl[ (\alpha n + \gamma) \:+\:
             (\alpha' n + \gamma') (\widehat{\beta} k + \widehat{\gamma})
      \bigr] \,
      C(n-1,k)
          \nonumber \\[1mm]
   & & \qquad\quad
   +\:
       (\alpha' n + \gamma') (\widehat{\beta}' k + \widehat{\gamma}')
         \, C(n-1,k-1)
 \label{eq.cor.prop.spivey.corollary5}
\end{eqnarray}
for $n \ge 1$, with initial condition $C(0,k) = \delta_{k0}$.
\end{corollary}

Of course, the binomial-like recurrence \reff{eq.cor.prop.spivey.corollary5}
need not in general be of GKP form, due to the possibility of coefficients
involving the product $nk$.
But these terms vanish if, in addition,
either $\alpha' = 0$ or $\widehat{\beta} = \widehat{\beta}' = 0$:

\begin{corollary}
   \label{cor2.prop.spivey.corollary5}
Let $\bA = \bigl( A(n,k) \bigr)_{0 \le k \le n}$
and $\bB = \bigl( B(n,k) \bigr)_{0 \le k \le n}$
be triangular arrays defined by the recurrences
\begin{subeqnarray}
  A(n,k)
  & = &
  (\alpha n + \gamma)    \, A(n-1,k)
  \:+\:
  (\alpha' n + \gamma') \, A(n-1,k-1)
  \qquad
     \\[1mm]
  B(n,k)
  & = &
  (\widehat{\beta} k + \widehat{\gamma})    \, B(n-1,k)
  \:+\:
  (\widehat{\beta}' k + \widehat{\gamma}') \, B(n-1,k-1)
  \qquad
 \slabel{eq1.cor2.prop.spivey.corollary5.B}
\end{subeqnarray}
for $n \ge 1$, with initial conditions $A(0,k) = B(0,k) = \delta_{k0}$.
Assume further that
either $\alpha' = 0$ or $\widehat{\beta} = \widehat{\beta}' = 0$.
Then the matrix product $\bC = \bA \bB$ satisfies the GKP recurrence
\begin{eqnarray}
   C(n,k)
   & = &
      \bigl[ (\alpha + \alpha' \widehat{\gamma}) n +
             \gamma' \widehat{\beta} k + (\gamma + \gamma' \widehat{\gamma})
      \bigr] \,
      C(n-1,k)
          \nonumber \\[1mm]
   & & \qquad\quad
   +\:
       (\alpha' \widehat{\gamma}' n + \gamma' \widehat{\beta}' k
          + \gamma' \widehat{\gamma}')
         \, C(n-1,k-1)
 \label{eq.cor2.prop.spivey.corollary5}
\end{eqnarray}
for $n \ge 1$, with initial condition $C(0,k) = \delta_{k0}$.
\end{corollary}

\medskip

{\bf Remark.}
In Corollary~\ref{cor2.prop.spivey.corollary5},
the matrix $\bB$ is precisely the matrix $\bT$ of family~7a
(with hats inserted on the parameters);
and if we further specialize to
$(\alpha,\gamma,\alpha',\gamma') = (0,\xi,0,1)$,
then $\bA$ is the $\xi$-binomial matrix $B_\xi$.
In this case the matrix product $B_\xi \bT$
satisfies a GKP recurrence \reff{eq.cor2.prop.spivey.corollary5}
of the same form \reff{eq1.cor2.prop.spivey.corollary5.B}
but with $\widehat{\gamma} \to \widehat{\gamma} + \xi$.
This gives the promised direct proof
that family~7a is the $\gamma$-binomial transform of family~3a
(Proposition~\ref{prop.VII.binomialtrans}).
\hspace*{-2mm}
\myendremark

\medskip

Alternatively, we can obtain GKP recurrences from
Corollary~\ref{cor.prop.theoret} by setting
\be
  \alpha_n \:=\: \alpha n + \gamma \,,\quad
  \beta_n \:=\: \gamma'  \,,\quad
  \gamma_k \:=\: (\beta/\gamma') k  \,,\quad
  \delta_k \:=\: 1  \,,\quad
  \phi_k \:=\: \widehat{\beta} k   \,,\quad
  \psi_k \:=\: \widehat{\beta}' k + \widehat{\gamma}'
   \:.
\ee
We then have:

\begin{corollary}
   \label{cor.prop.theoret.2}
Let $\bA = \bigl( A(n,k) \bigr)_{0 \le k \le n}$
and $\bB = \bigl( B(n,k) \bigr)_{0 \le k \le n}$
be triangular arrays defined by the recurrences
\begin{subeqnarray}
  A(n,k)  & = &  (\alpha n + \beta k + \gamma) \, A(n-1,k) \:+\:
                  \gamma' \, A(n-1,k-1)
     \\[2mm]
  B(n,k)  & = &  \Bigl( - \frac{\beta}{\gamma'} n + \widehat{\beta} k +
                          \frac{\beta}{\gamma'} \Bigr) \, B(n-1,k)
                  \:+\: (\widehat{\beta}' k + \widehat{\gamma}') \, B(n-1,k-1)
\end{subeqnarray}
for $n \ge 1$, with initial conditions $A(0,k) = B(0,k) = \delta_{k0}$.
Then the matrix product $\bC = \bA \bB$ satisfies the recurrence
\be
   C(n,k)
   \;=\;
   (\alpha n + \gamma' \widehat{\beta} k + \gamma) \, C(n-1,k)
      \:+\:  \gamma' (\widehat{\beta}' k + \widehat{\gamma}') \, C(n-1,k-1)
\ee
for $n \ge 1$, with initial condition $C(0,k) = \delta_{k0}$.
\end{corollary}

\bigskip

Going back to Proposition~\ref{prop.spivey.corollary5},
another way to handle the $j$-dependent terms
--- rather than assuming that they are nonexistent,
i.e.\ $\beta = \beta' = \widehat{\alpha} = \widehat{\alpha}' = 0$ ---
is to transform them away by employing identities
to convert the $j$-dependence to $n$- or $k$-dependence.
This works in some special cases, as a result of the following simple fact:

\begin{lemma}[Identities for nearly-binomial matrices]
   \label{lemma.binomial}
\hfill\break
\vspace*{-5mm}
\begin{itemize}
   \item[(a)]  If the matrix $\bT$ is of the form
\be
   T(n,k)  \;=\;  \binom{n}{k} \, \gamma^{n-k} \, f(k)
 \label{eq.lemma.binomial.1a}
\ee
for some function $f(k)$, then for every integer $r \ge 0$,
\be
   (n-k)^{\underline{r}} \; T(n,k)
   \;=\;
   \gamma^r \, n^{\underline{r}} \; T(n-r,k)
   \;.
 \label{eq.lemma.binomial.2a}
\ee
[Here $x^{\underline{r}} \,\eqdef\, x(x-1) \cdots (x-r+1)$.]
This holds in particular whenever $\bT = \bT(\bmu)$ with
$\bmu = (0,0,\gamma,0,\beta',\gamma')$, in~which case
\be
   T(n,k)
   \;=\;
   \binom{n}{k} \, \gamma^{n-k} \, \prod_{j=1}^k (\gamma' + j\beta')
   \;.
 \label{eq.lemma.binomial.3a}
\ee
   \item[(b)]  If the matrix $\bT$ is of the form
\be
   T(n,k)  \;=\;  \binom{n}{k} \, (\gamma')^k \, g(n-k)
 \label{eq.lemma.binomial.1b}
\ee
for some function $g(n-k)$, then for every integer $r \ge 0$,
\be
   k^{\underline{r}} \; T(n,k)
   \;=\;
   (\gamma')^r \, n^{\underline{r}} \; T(n-r,k-r)
   \;.
 \label{eq.lemma.binomial.2b}
\ee
This holds in particular whenever $\bT = \bT(\bmu)$ with
$\bmu = (\alpha,-\alpha,\gamma,0,0,\gamma')$, in~which case
\be
   T(n,k)
   \;=\;
   \binom{n}{k} \, (\gamma')^k \, \prod_{j=1}^{n-k} (\gamma + j\alpha)
   \;.
 \label{eq.lemma.binomial.3b}
\ee
\end{itemize}
\end{lemma}

\begin{proof}
An easy computation.
\end{proof}

We next discuss some special cases of Proposition~\ref{prop.spivey.corollary5}
(or more generally, Proposition~\ref{prop.app.1})
where Lemma~\ref{lemma.binomial} can be used to remove the $j$-dependence;
we divide these according as the special matrix
(i.e.\ the one to which Lemma~\ref{lemma.binomial} is applied)
is on the left (Section~\ref{app.spivey-zhu.left})
or on the right (Section~\ref{app.spivey-zhu.right}).

\subsection{Special matrix on the left}  \label{app.spivey-zhu.left}

We distinguish two cases, according as we apply part~(a) or (b)
of Lemma~\ref{lemma.binomial} to the matrix $\bA$:

\paragraph{Case 1: $\bm{\alpha = \beta = \alpha' = 0}$.}

When $\bA$ is given by a GKP recurrence with $\alpha = \beta = \alpha' = 0$,
we can handle a matrix $\bB$ that is more general than a GKP recurrence:
namely, it can be given by a binomial-like recurrence \reff{eq0.prop.app.1.b}
in which $b_{n,k}$ and $b'_{n,k}$ are affine in $n$
but with coefficients depending in an arbitrary way on $k$:
\be
   a_{n,k} \,=\, \gamma \,,\quad
   a'_{n,k} \,=\, \beta' k + \gamma' \,,\quad
   b_{n,k} \,=\, \widehat{\alpha}_k n + \widehat{\gamma}_k \,,\quad
   b'_{n,k} \,=\, \widehat{\alpha}'_k n + \widehat{\gamma}'_k
   \;.
 \label{eq.case1.specializations}
\ee
We then start from Proposition~\ref{prop.app.1}
with the specializations \reff{eq.case1.specializations},
and isolate the terms proportional to $j$ or $j^2$:
the matrix $\bC = \bA \bB$ satisfies
\begin{eqnarray}
   C(n,k)
   & \!=\! &
   \bigl[ (\beta' + \gamma')
          (\widehat{\alpha}_k + \widehat{\gamma}_k)
          \,+\, \gamma
   \bigr] \,  C(n-1,k)
          \nonumber \\[2mm]
   & & 
   \;+\;
    (\beta' + \gamma')
    (\widehat{\alpha}'_k + \widehat{\gamma}'_k)
          \, C(n-1,k-1)
    \hspace*{-2cm}
           \nonumber \\
   & & 
   \;+\;
   \bigl[ \beta' (\widehat{\alpha}_k + \widehat{\gamma}_k)
          \,+\,
          (\beta' + \gamma') \widehat{\alpha}_k
   \bigr]
   \sum_{j=1}^{n-1} j \, A(n-1,j) \, B(j,k)
           \nonumber \\[-2mm]
   & & 
   \;+\;
   \beta' \widehat{\alpha}_k
   \sum_{j=1}^{n-1} j^2 \, A(n-1,j) \, B(j,k)
           \nonumber \\[-2mm]
   & & 
   \;+\;
   \bigl[ \beta' (\widehat{\alpha}'_k + \widehat{\gamma}'_k)
          \,+\,
          (\beta' + \gamma') \widehat{\alpha}'_k
   \bigr]
   \sum_{j=1}^{n-1} j \, A(n-1,j) \, B(j,k-1)
           \nonumber \\[-2mm]
   & & 
   \;+\;
   \beta' \widehat{\alpha}'_k
   \sum_{j=1}^{n-1} j^2 \, A(n-1,j) \, B(j,k-1)
 \label{eq1.case1}
\end{eqnarray}
for $n \ge 1$.
Now apply Lemma~\ref{lemma.binomial}(a) to the matrix $\bA$:
from \reff{eq.lemma.binomial.2a} with $r=1,2$ we deduce
\begin{subeqnarray}
   j \, A(n,j)
   & = &
   n \, A(n,j) \:-\: \gamma n \, A(n-1,j)
      \\[2mm]
   j^2 \, A(n,j)
   & = &
   n^2 \, A(n,j) \:-\: \gamma n(2n-1) \, A(n-1,j)
                 \:+\: \gamma^2 n(n-1) \, A(n-2,j)
         \nonumber \\
 \label{eq.jton.1}
\end{subeqnarray}
Inserting (\ref{eq.jton.1}a,b) with $n \to n-1$
on the right-hand side of \reff{eq1.case1},
we obtain:

\begin{corollary}
   \label{cor3.prop.spivey.corollary5}
Let $\bA = \bigl( A(n,k) \bigr)_{0 \le k \le n}$
and $\bB = \bigl( B(n,k) \bigr)_{0 \le k \le n}$
be triangular arrays defined by the recurrences
\begin{subeqnarray}
  A(n,k)
  & = &
   \gamma   \, A(n-1,k)
  \:+\:
  (\beta' k + \gamma') \, A(n-1,k-1)
  \qquad
     \\[1mm]
  B(n,k)
  & = &
  (\widehat{\alpha}_k n + \widehat{\gamma}_k)    \, B(n-1,k)
  \:+\:
  (\widehat{\alpha}'_k n + \widehat{\gamma}'_k) \, B(n-1,k-1)
  \qquad\quad
\end{subeqnarray}
for $n \ge 1$, with initial conditions $A(0,k) = B(0,k) = \delta_{k0}$.
Then the matrix product $\bC = \bA \bB$ satisfies the recurrence
\begin{eqnarray}
   C(n,k)
   & = &
      c_1 \, C(n-1,k) \:+\: c_2 \, C(n-1,k-1)
          \nonumber \\[1mm]
   & & 
   \;+\;
      c_3 \, C(n-2,k) \:+\: c_4 \, C(n-2,k-1)
          \nonumber \\[1mm]
   & &
   \;+\;
      c_5 \, C(n-3,k) \:+\: c_6 \, C(n-3,k-1)
 \label{eq.cor3.prop.spivey.corollary5}
\end{eqnarray}
for $n \ge 1$, with initial condition $C(0,k) = \delta_{k0}$,
where
\begin{subeqnarray}
   c_1  & = &  (\beta' n + \gamma') \,
               (\widehat{\alpha}_k n + \widehat{\gamma}_k)
                 \:+\: \gamma  \\[1mm]
   c_2  & = &  (\beta' n + \gamma') \,
               (\widehat{\alpha}'_k n + \widehat{\gamma}'_k)
                   \\[1mm]
   c_3  & = &  - (n-1) \gamma 
      \bigl[\gamma' \widehat{\alpha}_k +
            \beta' (\widehat{\gamma}_k + (2n-1) \widehat{\alpha}_k)
      \bigr]   \\[1mm]
   c_4  & = &  - (n-1) \gamma 
      \bigl[\gamma' \widehat{\alpha}'_k +
            \beta' (\widehat{\gamma}'_k + (2n-1) \widehat{\alpha}'_k)
      \bigr]   \\[1mm]
   c_5  & = &  (n-1)(n-2) \gamma^2 \beta' \widehat{\alpha}_k \\[1mm]
   c_6  & = &  (n-1)(n-2) \gamma^2 \beta' \widehat{\alpha}'_k 
\end{subeqnarray}
\end{corollary}

This is rather complicated, but it simplifies significantly
when also $\beta' = 0$, in which case $\bA$ is a rescaled binomial:
$A(n,k) = \binom{n}{k} \gamma^{n-k} (\gamma')^k$.
Normalizing to $\gamma' = 1$ and writing $\gamma = \xi$,
then $\bA$ becomes the $\xi$-binomial matrix $B_\xi$ defined in \reff{def.Bxi},
and Corollary~\ref{cor3.prop.spivey.corollary5} becomes:

\begin{corollary}
   \label{cor4.prop.spivey.corollary5}
Let $\bB = \bigl( B(n,k) \bigr)_{0 \le k \le n}$
be a triangular array defined by the recurrence
\be
  B(n,k)
  \;=\;
  (\widehat{\alpha}_k n + \widehat{\gamma}_k)    \, B(n-1,k)
  \:+\:
  (\widehat{\alpha}'_k n + \widehat{\gamma}'_k) \, B(n-1,k-1)
\ee
for $n \ge 1$, with initial condition $B(0,k) = \delta_{k0}$.
Then the matrix product $\bC = B_\xi \, \bB$ satisfies the recurrence
\begin{eqnarray}
   C(n,k)
   & = &
   (\widehat{\alpha}_k n + \widehat{\gamma}_k + \xi) \, C(n-1,k)
   \:+\:
   (\widehat{\alpha}'_k n + \widehat{\gamma}'_k) \, C(n-1,k-1)
          \nonumber \\[0.5mm]
   & & 
   \;-\; (n-1)\xi \bigl[\widehat{\alpha}_k \, C(n-2,k) \:+\:
                        \widehat{\alpha}'_k \, C(n-2,k-1) \bigr]
 \label{eq.cor4.prop.spivey.corollary5}
\end{eqnarray}
for $n \ge 1$, with initial condition $C(0,k) = \delta_{k0}$.
\end{corollary}

\noindent
When also $\widehat{\alpha}_k = \widehat{\alpha}'_k = 0$ for all $k$,
the terms $C(n-2,k)$ and $C(n-2,k-1)$ disappear,
and we have a special case of Corollary~\ref{cor.prop.app.1}.

Specializing Corollary~\ref{cor4.prop.spivey.corollary5} to the GKP case
\be
   \widehat{\alpha}_k \,=\, \widehat{\alpha} \,,\quad
   \widehat{\gamma}_k \,=\, \widehat{\beta} k + \widehat{\gamma} \,,\quad
   \widehat{\alpha}'_k \,=\, \widehat{\alpha}' \,,\quad
   \widehat{\gamma}'_k \,=\, \widehat{\beta}' k + \widehat{\gamma}'
   \;,
 \label{eq.alphagamma.GKP}
\ee
we get:

\begin{corollary}
   \label{cor4bis.prop.spivey.corollary5}
Let $\bB = \bigl( B(n,k) \bigr)_{0 \le k \le n}$
be a triangular array defined by the recurrence
\be
  B(n,k)
  \;=\;
  (\widehat{\alpha} n + \widehat{\beta} k + \widehat{\gamma})    \, B(n-1,k)
  \:+\:
  (\widehat{\alpha}' n + \widehat{\beta}' k + \widehat{\gamma}') \, B(n-1,k-1)
\ee
for $n \ge 1$, with initial condition $B(0,k) = \delta_{k0}$.
Then the matrix product $\bC = B_\xi \, \bB$ satisfies the recurrence
\begin{eqnarray}
   C(n,k)
   & = &
   (\widehat{\alpha} n + \widehat{\beta} k + \widehat{\gamma} + \xi) \, C(n-1,k)
   \:+\:
   (\widehat{\alpha}' n + \widehat{\beta}' k + \widehat{\gamma}') \, C(n-1,k-1)
          \nonumber \\[0.5mm]
   & &
   \;-\; (n-1)\xi \bigl[\widehat{\alpha} \, C(n-2,k) \:+\: \widehat{\alpha}' \, C(n-2,k-1) \bigr]
 \label{eq.cor4bis.prop.spivey.corollary5}
\end{eqnarray}
for $n \ge 1$, with initial condition $C(0,k) = \delta_{k0}$.
\end{corollary}

\noindent
When also $\widehat{\alpha} = \widehat{\alpha}' = 0$,
the terms $C(n-2,k)$ and $C(n-2,k-1)$ disappear,
and we are back in the situation discussed in the Remark
after Corollary~\ref{cor2.prop.spivey.corollary5}.

The form \reff{eq.cor4bis.prop.spivey.corollary5}
of the output matrix in Corollary~\ref{cor4.prop.spivey.corollary5}
suggests that one consider, as input,
a triangular array $\bB = \bigl( B(n,k) \bigr)_{0 \le k \le n}$
defined by a similar recurrence. Then we can obtain, using essentially
the same proof as for Corollary~\ref{cor4.prop.spivey.corollary5}, the
following more general result:

\begin{corollary} \label{cor8.gen}
Let $\bB = \bigl( B(n,k) \bigr)_{0 \le k \le n}$
be a triangular array defined by the recurrence
\begin{multline}
 B(n,k)
 \;=\;
 (\widehat{\alpha}_k n + \widehat{\gamma}_k) \, B(n-1,k)
 \;+\;
 (\widehat{\alpha}'_k n + \widehat{\gamma}'_k) \, B(n-1,k-1) \\ \quad
 \;+\;
 (n-1) \, \widehat{\delta}_k \, B(n-2,k)
 \;+\;
 (n-1) \, \widehat{\delta}'_k \, B(n-2,k-1)
   \qquad
 \label{eq.cor8.gen.0}
\end{multline}
for $n \ge 1$, with initial condition $B(0,k) = \delta_{k0}$.
Then the matrix product $\bC = B_\xi \, \bB$ satisfies the recurrence
\begin{multline}
  C(n,k)
  \;=\;
  (\widehat{\alpha}_k n + \widehat{\gamma}_k + \xi) \, C(n-1,k)
  \;+\;
  (\widehat{\alpha}'_k n + \widehat{\gamma}'_k) \, C(n-1,k-1) \\ \qquad
  \;+\; (n-1)\, (\widehat{\delta}_k - \xi \widehat{\alpha}_k)
  \, C(n-2,k)
 \;+\; (n-1)\, (\widehat{\delta}'_k - \xi \widehat{\alpha}'_k)
 \, C(n-2,k-1)
 \label{eq.cor8.gen}
\end{multline}
for $n \ge 1$, with initial condition $C(0,k) = \delta_{k0}$.
\end{corollary}

\begin{proof}
The matrix $\bC = B_\xi\, \bB$ satisfies
\begin{subeqnarray}
C(n,k) &=& \sum\limits_{j=0}^n B_\xi(n,j)\, B(j,k) \\
 &=& \sum\limits_{j=0}^n \bigl[ \xi\, B_\xi(n-1,j) + B_\xi(n-1,j-1)
  \bigr] \, B(j,k) \\
 &=& \xi \, C(n-1,k) + \sum\limits_{j=1}^n B_\xi(n-1,j-1) \, 
B(j,k) \,.
\end{subeqnarray}
Now insert the recurrence \reff{eq.cor8.gen.0} for $B(j,k)$.
The first two terms on the right-hand side produce the terms
given in Corollary~\ref{cor4.prop.spivey.corollary5}.
So we only need to deal with the last two terms on the right-hand side.
The contribution of the first one is
\begin{subeqnarray}
   & &
   \widehat{\delta}_k \, \sum\limits_{j=1}^n (j-1)\, B_\xi(n-1,j-1) \, B(j-2,k)
            \nonumber \\[-2mm]
   & & \qquad\qquad
   =\;
   \widehat{\delta}_k \, \sum\limits_{j=0}^{n-1} j\, B_\xi(n-1,j) \, B(j-1,k)
           \\[1mm]
   & & \qquad\qquad
   =\;
   \widehat{\delta}_k \, (n-1) \, C(n-2,k) \qquad
      \hbox{[using \reff{eq.lemma.binomial.2b} with $r=1$]} \,.
\end{subeqnarray}
Similarly, the contribution of the second term is
$\widehat{\delta}'_k \, (n-1) \, C(n-2,k-1)$.
These two contributions plus those given by
Corollary~\ref{cor4.prop.spivey.corollary5} give the claimed formula
\eqref{eq.cor8.gen}.
\end{proof}

\paragraph{Case 2: $\bm{\beta = -\alpha}$, $\bm{\alpha' = \beta' = 0}$.}

As in Case 1, we can handle a matrix $\bB$
given by a binomial-like recurrence \reff{eq0.prop.app.1.b}
in which $b_{n,k}$ and $b'_{n,k}$ are affine in $n$
but with coefficients depending in an arbitrary way on $k$:
\be
   a_{n,k} \,=\, \alpha (n-k) + \gamma \,,\quad
   a'_{n,k} \,=\, \gamma' \,,\quad
   b_{n,k} \,=\, \widehat{\alpha}_k n + \widehat{\gamma}_k \,,\quad
   b'_{n,k} \,=\, \widehat{\alpha}'_k n + \widehat{\gamma}'_k
   \;.
 \label{eq.case2.specializations}
\ee
We then insert the specializations \reff{eq.case2.specializations}
into Proposition~\ref{prop.app.1}
and isolate the terms proportional to $j$
(note that the terms proportional to $j^2$ vanish in this case):
the matrix $\bC = \bA \bB$ satisfies
\begin{eqnarray}
   C(n,k)
   & \!=\! &
   \bigl[ \alpha n + \gamma
                   + \gamma' (\widehat{\alpha}_k + \widehat{\gamma}_k)
   \bigr] \,  C(n-1,k)
   \:+\:
    \gamma' (\widehat{\alpha}'_k + \widehat{\gamma}'_k)
          \, C(n-1,k-1)
    \hspace*{-2cm}
           \nonumber \\
   & & 
   \;+\;
   (\gamma' \widehat{\alpha}_k - \alpha)
   \sum_{j=1}^{n-1} j \, A(n-1,j) \, B(j,k)
   \;+\;
   \gamma' \widehat{\alpha}'_k
   \sum_{j=1}^{n-1} j \, A(n-1,j) \, B(j,k-1)
           \nonumber \\[-3mm]
 \label{eq1.case2}
\end{eqnarray}
for $n \ge 1$.
Now apply Lemma~\ref{lemma.binomial}(b) with $r=1$ to the matrix $\bA$:
\be
   j \, A(n-1,j)  \;=\; \gamma' \, (n-1) \, A(n-2,j-1)
   \;.
\ee
And apply the recurrence \reff{eq0.prop.spivey.corollary5.B}
to $B(j,k)$ and $B(j,k-1)$ in order to convert $j$ to $j-1$
(so that we can again get a matrix product $\bA \bB$).
The last two terms of \reff{eq1.case2} become
\begin{eqnarray}
   & &
   \hspace*{-1cm}
   \gamma' \, (n-1) (\gamma' \widehat{\alpha}_k - \alpha)
   \sum_{j=1}^{n-1} A(n-2,j-1) \,\times\,
            \nonumber \\
   & &
      \bigl[ (\widehat{\alpha}_k j + \widehat{\gamma}_k)
                 \, B(j-1,k)
             \:+\:
             (\widehat{\alpha}'_k j + \widehat{\gamma}'_k)
                 \, B(j-1,k-1)
      \bigr]
   \,.
            \nonumber \\[2mm]
   & &
   \hspace*{-1cm}
   \;+\;
   (\gamma')^2 \, (n-1)  \, \widehat{\alpha}'_k
   \sum_{j=1}^{n-1} A(n-2,j-1) \,\times\,
            \nonumber \\
   & &
      \bigl[ (\widehat{\alpha}_{k-1} j + \widehat{\gamma}_{k-1})
                 \, B(j-1,k-1)
             \:+\:
             (\widehat{\alpha}'_{k-1} j + \widehat{\gamma}'_{k-1})
                 \, B(j-1,k-2)
      \bigr]
   \,.
    \qquad
\end{eqnarray}
The appearance here of the terms
$\widehat{\alpha}_i j$ and $\widehat{\alpha}'_i j$ (for $i = k, k-1$)
would lead in general to an infinite regress of $j$-dependence,
but this $j$-dependence disappears if $\widehat{\alpha}'_k = 0$ for all $k$
and moreover either (a) $\gamma' \widehat{\alpha}_k = 0$ for all $k$
or (b) $\gamma' \widehat{\alpha}_k = \alpha$ for all $k$.
We then have:

\begin{corollary}
   \label{cor5.prop.spivey.corollary5}
Let $\bA = \bigl( A(n,k) \bigr)_{0 \le k \le n}$
and $\bB = \bigl( B(n,k) \bigr)_{0 \le k \le n}$
be triangular arrays defined by the recurrences
\begin{subeqnarray}
  A(n,k)
  & = &
   (\alpha (n-k) + \gamma)  \, A(n-1,k)
  \:+\:
   \gamma' \, A(n-1,k-1)
  \qquad
     \\[1mm]
  B(n,k)
  & = &
  (\widehat{\alpha}_k n + \widehat{\gamma}_k)    \, B(n-1,k)
  \:+\:
  \widehat{\gamma}'_k \, B(n-1,k-1)
  \qquad\quad
\end{subeqnarray}
for $n \ge 1$, with initial conditions $A(0,k) = B(0,k) = \delta_{k0}$.
Suppose further that either
$\gamma' \widehat{\alpha}_k = 0$ for all $k$
or $\gamma' \widehat{\alpha}_k = \alpha$ for all $k$.
Then the matrix product $\bC = \bA \bB$ satisfies the recurrence
\begin{eqnarray}
   C(n,k)
   & = &
      c'_1 \, C(n-1,k) \:+\: c'_2 \, C(n-1,k-1)
          \nonumber \\[1mm]
   & & 
   \;+\;
      c'_3 \, C(n-2,k) \:+\: c'_4 \, C(n-2,k-1)
 \label{eq.cor5.prop.spivey.corollary5}
\end{eqnarray}
for $n \ge 1$, with initial condition $C(0,k) = \delta_{k0}$,
where
\begin{subeqnarray}
   c'_1  & = &  \alpha n \,+\, \gamma \,+\,
               \gamma' (\widehat{\alpha}_k + \widehat{\gamma}_k)
           \\[1mm]
   c'_2  & = &  \gamma' \, \widehat{\gamma}'_k
           \\[1mm]
   c'_3  & = &  (n-1) \, \gamma' \, (\gamma' \widehat{\alpha}_k - \alpha)
                             \, \widehat{\gamma}_k
           \\[1mm]
   c'_4  & = &  (n-1) \, \gamma' \, (\gamma' \widehat{\alpha}_k - \alpha) 
                 \, \widehat{\gamma}'_k
\end{subeqnarray}
In particular, when $\gamma' \widehat{\alpha}_k = \alpha$ for all $k$,
we have $c'_3 = c'_4 = 0$ and hence a binomial-like recurrence for $\bC$.
\end{corollary}

The further specialization of Corollary~\ref{cor5.prop.spivey.corollary5}
to the GKP case \reff{eq.alphagamma.GKP} is easily written out.
In particular, when $\widehat{\alpha}' = 0$
and $\gamma' \widehat{\alpha} = \alpha$,
then the recurrence for $\bC$ is of GKP form.

\bigskip

{\bf Remarks.}
1.  The ``suppose further'' hypotheses of
Corollary~\ref{cor5.prop.spivey.corollary5}
are satisfied when $\gamma'  = 0$, but this is a trivial case:
the matrix $\bA$ vanishes outside the zeroth column, so that $\bA \bB = \bA$.

2.  Suppose that $\widehat{\alpha}_k$ has a constant value
$\widehat{\alpha}$ (independent of $k$)
but that we suppress the hypothesis on $\gamma' \widehat{\alpha}$.
Then the difference between the left-hand and right-hand sides of
\reff{eq.cor5.prop.spivey.corollary5} is, for each pair $n,k$,
a polynomial in the parameters that has
$\gamma' \widehat{\alpha} \, (\gamma' \widehat{\alpha} - \alpha)$
as an overall factor.
\myendremark

\paragraph{The intersection of cases~1 and 2:
$\bm{\alpha = \beta = \alpha' = \beta' =
     \widehat{\alpha}_k = \widehat{\alpha}'_k = 0}$.}

Let us now compare Corollary~\ref{cor3.prop.spivey.corollary5}
specialized to $\beta' = \widehat{\alpha}_k = \widehat{\alpha}'_k = 0$
with Corollary~\ref{cor5.prop.spivey.corollary5}
specialized to $\alpha = \widehat{\alpha}_k = 0$.
We obtain a recurrence of GKP type with
\begin{subeqnarray}
   c_1 \;=\; c'_1  & = &  \gamma' \, \widehat{\gamma}_k \:+\: \gamma
                   \\[1mm]
   c_2 \;=\; c'_2  & = &  \gamma' \, \widehat{\gamma}'_k
                   \\[1mm]
   c_3 \;=\; c'_3  & = &  0
                   \\[1mm]
   c_4 \;=\; c'_4  & = &  0
                   \\[1mm]
   c_5 & = &  0
                   \\[1mm]
   c_6 & = &  0
\end{subeqnarray}

\subsection{Special matrix on the right}  \label{app.spivey-zhu.right}

Once again we distinguish two cases, according as we apply part~(a) or (b)
of Lemma \ref{lemma.binomial} to the matrix $\bB$:

\paragraph{Case 3: $\bm{\widehat{\alpha} = \widehat{\beta} = \widehat{\alpha}'
   = 0}$.}

We start from Proposition~\ref{prop.spivey.corollary5}
specialized to $\widehat{\alpha} = \widehat{\beta} = \widehat{\alpha}' = 0$,
and isolate the terms proportional to $j$
(note that the terms proportional to $j^2$ vanish in this case):
the matrix $\bC = \bA \bB$ satisfies
\begin{eqnarray}
   C(n,k)
   & \!=\! &
   \bigl[ (\alpha n + \gamma) \,+\,
          \widehat{\gamma} (\alpha' n + \beta' + \gamma')
   \bigr] \,  C(n-1,k)
         \nonumber \\[2mm]
   & & 
   \;+\;
    (\alpha' n + \beta' + \gamma') \, (\widehat{\beta}' k + \widehat{\gamma}')
          \, C(n-1,k-1)
    \hspace*{-2cm}
           \nonumber \\
   & & 
   \;+\;
   (\beta + \beta' \widehat{\gamma})
   \sum_{j=1}^{n-1} j \, A(n-1,j) \, B(j,k)
         \nonumber \\[-2mm]
   & & 
   \;+\;
   \beta' (\widehat{\beta}' k + \widehat{\gamma}')
   \sum_{j=1}^{n-1} j \, A(n-1,j) \, B(j,k-1)
 \label{eq1.case3}
\end{eqnarray}
for $n \ge 1$.
The next step would be to apply
Lemma~\ref{lemma.binomial}(a) with $r=1$ to the matrix $\bB$:
\be
   j \, B(j,k-1)
   \;=\;
   (k-1) \, B(j,k-1) \:+\: \widehat{\gamma} j \, B(j-1,k-1)
   \;.
\ee
But this leads to an infinite regress of $j$-dependence
(except in the degenerate case $\widehat{\gamma} = 0$,
 which leads to a matrix $\bB$ that vanishes outside the zeroth column).
So we do not know how to obtain a recurrence for $\bC = \bA \bB$
in this case.

\paragraph{Case 4: $\bm{\widehat{\beta} = -\widehat{\alpha}}$,
                   $\bm{\widehat{\alpha}' = \widehat{\beta}' = 0}$.}

Similarly to what was done in Cases 1 and 2
--- but now with the roles of $\bA$ and $\bB$,
and also $n$ and $k$, reversed ---
we can handle a matrix $\bA$ given by
a binomial-like recurrence \reff{eq0.prop.app.1.a}
in which $a_{n,k}$ and $a'_{n,k}$ are affine in $k$
but with coefficients depending in an arbitrary way on $n$:
\be
   a_{n,k} \,=\, \beta_n k + \gamma_n \,,\quad
   a'_{n,k} \,=\, \beta'_n k +\gamma'_n \,,\quad
   b_{n,k} \,=\, \widehat{\alpha} (n-k) + \widehat{\gamma}  \,,\quad
   b'_{n,k} \,=\, \widehat{\gamma}'
   \;.
 \label{eq.case4.specializations}
\ee
We insert the specializations \reff{eq.case4.specializations}
into Proposition~\ref{prop.app.1}
and isolate the terms proportional to $j$ or $j^2$:
the matrix $\bC = \bA \bB$ satisfies
\begin{eqnarray}
   C(n,k)
   & \!=\! &
   \bigl[ \gamma_n \,+\,
          (\beta'_n + \gamma'_n)
          (-\widehat{\alpha} k + \widehat{\alpha} + \widehat{\gamma})
   \bigr] \,  C(n-1,k)
           \nonumber \\[2mm]
   & & 
   \;+\;
    \widehat{\gamma}' (\beta'_n + \gamma'_n)
          \, C(n-1,k-1)
    \hspace*{-2cm}
           \nonumber \\[-1mm]
   & & 
   \;+\;
   \bigl[ \beta_n
          \,+\,
          \beta'_n (-\widehat{\alpha} k + \widehat{\alpha} + \widehat{\gamma})
          \,+\,
          \widehat{\alpha} (\beta'_n + \gamma'_n)
   \bigr]
   \sum_{j=1}^{n-1} j \, A(n-1,j) \, B(j,k)
           \nonumber \\[-2mm]
   & & 
   \;+\;
   \beta'_n \widehat{\alpha}
   \sum_{j=1}^{n-1} j^2 \, A(n-1,j) \, B(j,k)
           \nonumber \\[-1mm]
   & & 
   \;+\;
   \beta'_n \widehat{\gamma}'
   \sum_{j=1}^{n-1} j \, A(n-1,j) \, B(j,k-1)
 \label{eq1.case4}
\end{eqnarray}
for $n \ge 1$.
Now apply Lemma~\ref{lemma.binomial}(b) in reverse to the matrix $\bB$:
from \reff{eq.lemma.binomial.2b} with $r=1,2$ we deduce
\be
   j \, B(j,k)
   \;=\;
   \frac{1}{\widehat{\gamma}'} \, (k+1) \, B(j+1,k+1)  \:-\: B(j,k)
 \label{eq.jtok.4.a}
\ee
and a similar but more complicated formula for $j^2 B(j,k)$.
Applying the recurrence \reff{eq0.prop.spivey.corollary5.B}
to $B(j+1,k+1)$ in order to convert $j+1$ to $j$,
we can rewrite \reff{eq.jtok.4.a} as
\be
   j \, B(j,k)
   \;=\;
   k \, B(j,k)
     \:+\:
   \frac{1}{\widehat{\gamma}'} \, (k+1) \,
         [\widehat{\alpha} (j-k) + \widehat{\gamma}] \, B(j,k+1)
   \;.
\ee
But this leads to an infinite regress of $j$-dependence
unless we assume $\widehat{\alpha} = 0$, so we do this henceforth;
this also kills the $j^2$ term in \reff{eq1.case4}.
Then $\bB$ is a rescaled binomial,
$B(n,k) = \binom{n}{k} \widehat{\gamma}^{n-k} (\widehat{\gamma}')^k$,
and we have
\be
   j \, B(j,k)
   \;=\;
   k \, B(j,k)
     \:+\:
     \frac{\widehat{\gamma}}{\widehat{\gamma}'} \, (k+1) \, B(j,k+1)
   \;.
 \label{eq.case4.jB}
\ee
Now insert \reff{eq.case4.jB} into \reff{eq1.case4}.
Normalizing to $\widehat{\gamma}' = 1$ and writing $\widehat{\gamma} = \xi$,
then $\bB$ becomes the $\xi$-binomial matrix $B_\xi$ defined in \reff{def.Bxi},
and we obtain:

\begin{corollary}
   \label{cor.case4bis}
Let $\bA = \bigl( A(n,k) \bigr)_{0 \le k \le n}$
be a triangular array defined by the recurrence
\be
  A(n,k)  \;=\;  (\beta_n k + \gamma_n) \, A(n-1,k) \:+\:
                  (\beta'_n k +\gamma'_n) \, A(n-1,k-1)
 \label{eq.cor.case4bis.0}
\ee
for $n \ge 1$, with initial condition $A(0,k) = \delta_{k0}$.
Then the matrix product $\bC = \bA \, B_\xi$ satisfies the recurrence
\begin{eqnarray}
   C(n,k)
   & = &
   \bigl[ (\beta_n + 2\xi \beta'_n) k \,+\,
          \gamma_n + \xi(\beta'_n +\gamma'_n)
   \bigr] \, C(n-1,k)
           \nonumber \\[2mm]
   & &
   \quad +\;
   (\beta'_n k + \gamma'_n) \, C(n-1,k-1)
           \nonumber \\[2mm]
   & &
   \quad +\;
   \xi \, (\beta_n + \xi\beta'_n) \, (k+1) \, C(n-1,k+1)
 \label{eq.cor.case4bis}
\end{eqnarray}
for $n \ge 1$, with initial condition $C(0,k) = \delta_{k0}$.
\end{corollary}

Specializing Corollary~\ref{cor.case4bis} to the GKP case
\be
   \beta_n \,=\, \beta \,,\quad
   \gamma_n \,=\, \alpha n + \gamma \,,\quad
   \beta'_n \,=\, \beta' \,,\quad
   \gamma'_n \,=\, \alpha' n + \gamma' \,,
 \label{eq.betagamma.GKP}
\ee
we get:

\begin{corollary}
   \label{cor6.prop.spivey.corollary5}
Let $\bA = \bigl( A(n,k) \bigr)_{0 \le k \le n}$
be a triangular array defined by the GKP recurrence
\be
  A(n,k)
  \;=\;
  (\alpha n + \beta k + \gamma)    \, A(n-1,k)
  \:+\:
  (\alpha' n + \beta' k + \gamma') \, A(n-1,k-1)
\ee
for $n \ge 1$, with initial condition $A(0,k) = \delta_{k0}$.
Then the matrix product $\bC = \bA \, B_\xi$ satisfies the recurrence
\begin{eqnarray}
   C(n,k)
   & = &
   \bigl[ (\alpha + \xi \alpha') n \,+\,
          (\beta + 2\xi \beta') k  \,+\,
          \gamma + \xi(\beta' +\gamma')
   \bigr] \, C(n-1,k)
           \nonumber \\[2mm]
   & & 
   \quad +\;
   (\alpha' n + \beta' k + \gamma') \, C(n-1,k-1)
           \nonumber \\[2mm]
   & & 
   \quad +\;
   \xi \, (\beta + \xi\beta') \, (k+1) \, C(n-1,k+1)
 \label{eq.cor6.prop.spivey.corollary5}
\end{eqnarray}
for $n \ge 1$, with initial condition $C(0,k) = \delta_{k0}$.
\end{corollary}

\noindent
The $\xi = 1$ special case of this result was obtained by
Spivey \cite[Theorem~10]{Spivey_11}.
In the special case $\xi = 1$ and $\alpha' = \beta' = 0$,
a combinatorial proof was given by
Mansour and Shattuck \cite[Theorem~2.3]{Mansour_13}.

Let us observe that the recurrence \reff{eq.cor6.prop.spivey.corollary5}
is a special case of
the Graham--Knuth--Patashnik--Zhu recurrence \reff{eq.GKPZ}.
It reduces to a GKP recurrence in two cases:
the trivial case $\xi = 0$,
and the case $\xi = -\beta/\beta'$
corresponding to the Zhu involution \reff{def_transinv}.
This suggests (but does not prove)
that the matrix product $\bC = \bT(\bmu) \, B_\xi$
satisfies a GKP recurrence for {\em generic}\/ parameters $\bmu$
{\em only if}\/ $\xi = 0$ or $\xi = -\beta/\beta'$.
We have verified, by a brute-force computation using $n=0,1,2,3$,
that this is indeed the case:
the only solutions to the equations $\bT(\bmu') = \bT(\bmu) \, B_\xi$
valid for generic parameters $\bmu$
are the identity map $\xi = 0$ and the Zhu involution $\xi = -\beta/\beta'$.
Of course, there are additional solutions valid on subvarieties
in $\bmu$-space; the goal of Problem~\ref{prob.binomial2} is to find them all.

The observation that \reff{eq.cor6.prop.spivey.corollary5}
is a special case of the GKPZ recurrence suggests that we can go farther,
and generalize Corollary~\ref{cor6.prop.spivey.corollary5}
by {\em starting}\/ from a matrix $\bA$ satisfying the GKPZ recurrence
--- or even more strongly, generalize Corollary~\ref{cor.case4bis}
by starting from a matrix $\bA$ that satisfies
an amalgamation of \reff{eq.cor.case4bis.0} and the GKPZ recurrence.
This idea leads to the following:

\begin{corollary}
 \label{cor.case4.GKPZ}
Let $\bA = \bigl( A(n,k) \bigr)_{0 \le k \le n}$
be a triangular array defined by the recurrence
\begin{eqnarray}
 A(n,k)
 & = &
 (\beta_n k + \gamma_n) \, A(n-1,k)
 \:+\:
 (\beta'_n k +\gamma'_n) \, A(n-1,k-1)
 \qquad
 \nonumber \\[1mm]
 & &
   \hspace*{-1cm}
 \;+\;
 \sigma_n \, (n-k+1) \, A(n-1,k-2)
 \:+\:
 \tau_n \, (k+1) \, A(n-1,k+1)
   \qquad
 \label{eq.case4.GKPZ}
\end{eqnarray}
for $n \ge 1$, with initial condition $A(0,k) = \delta_{k0}$.
Then the matrix product $\bC = \bA \, B_\xi$ satisfies the recurrence
\begin{eqnarray}
 C(n,k)
 & = &
 \bigl[ (\beta_n + 2\xi \beta'_n - 3\xi^2 \sigma_n) k \,+\,
 \gamma_n + \xi(\beta'_n +\gamma'_n) + \xi^2 \sigma_n (n-1)
 \bigr] \, C(n-1,k)
 \nonumber \\[2mm]
 & &
 \;+\;
 \bigl[ (\beta'_n - 3 \xi \sigma_n) k \,+\,
 \gamma'_n + \xi\sigma_n(2n+1) \bigr] \, C(n-1,k-1)
 \nonumber \\[2mm]
 & &
 \;+\;
 \sigma_n \, (n-k+1) \, C(n-1,k-2)
 \nonumber \\[2mm]
 & &
 \;+\;
 (\tau_n + \xi\beta_n + \xi^2 \beta'_n - \xi^3\sigma_n) \, (k+1) \,
 C(n-1,k+1)
 \label{eq.cor.case4.GKPZ}
\end{eqnarray}
for $n \ge 1$, with initial condition $C(0,k) = \delta_{k0}$.
\end{corollary}

\begin{proof}
The matrix $\bC = \bA\, B_\xi$ satisfies
\begin{eqnarray}
C(n,k)
   & = &
 \sum\limits_{j=0}^n (\beta_n j + \gamma_n) \, A(n-1,j) \, B_\xi(j,k)
      \;+\;
   \sum\limits_{j=0}^n (\beta'_n j + \gamma'_n)  \, A(n-1,j-1) \, B_\xi(j,k)
        \nonumber \\
 & &
   \hspace*{-2cm}
   \;+\; \sigma_n \, \sum\limits_{j=0}^n (n-j+1) \, A(n-1,j-2) \, B_\xi(j,k)
 \;+\; \tau_n \, \sum\limits_{j=0}^n (j+1) \, A(n-1,j+1) \, B_\xi(j,k)
        \nonumber \\[-3mm]
 \label{eq1.coroA13}
\end{eqnarray}
for $n \ge 1$.
The first term on the right-hand side of \eqref{eq1.coroA13}
can be handled by using the identity
\be
   j \, B_\xi(j,k)
   \;=\;
   k B_\xi(j,k)  \:+\: \xi\, (k+1) \, B_\xi(j,k+1)
   \;.
\ee
This yields
\be
   (\beta_n k + \gamma_n) \, C(n-1,k) \;+\;
 \xi \beta_n \, (k+1) \, C(n-1,k+1) \,.
\label{line1.coroA13}
\ee
The second term on the right-hand side of \eqref{eq1.coroA13}
can be dealt with by using the identities
\begin{subeqnarray}
   B_\xi(j,k)
   & = &
   B_\xi(j-1,k-1)  \:+\: \xi \, B_\xi(j-1,k)
      \\[2mm]
   j \, B_\xi(j,k)
   & = &
   k \, B_\xi(j-1,k-1)  \:+\: (2k+1) \xi \, B_\xi(j-1,k)
                        \:+\: (k+1) \xi^2 B_\xi(j-1,k+1)
        \nonumber \\
\end{subeqnarray}
This yields
\begin{multline}
\xi \, (2 \beta'_n k + \beta'_n + \gamma'_n) \, C(n-1,k) \;+\;
 (\beta'_n k +\gamma'_n) \, C(n-1,k-1) \\ \qquad
  \;+\; \xi^2 \, (k+1) \beta'_n \, C(n-1,k+1) \,.
\label{line2.coroA13}
\end{multline}
In the third term on the right-hand side of \eqref{eq1.coroA13},
we use the identities
\begin{subeqnarray}
   B_\xi(j,k)
   & = &
   B_\xi(j-2,k-2)  \:+\: 2\xi \, B_\xi(j-2,k-1)  \:+\: \xi^2 \, B_\xi(j-2,k)
      \\[2mm]
   j \, B_\xi(j,k)
   & = &
   k \, B_\xi(j-2,k-2)  \:+\: (3k+1) \xi \, B_\xi(j-2,k-1)
         \nonumber \\
   & & \quad
                        \:+\: (3k+2) \xi^2 B_\xi(j-2,k)
                        \:+\: (k+1) \xi^3 B_\xi(j-2,k+1)
     \qquad
\end{subeqnarray}
This yields
\begin{multline}
\sigma_n \, \bigl[ \xi^2 \, (n-3k-1)\, C(n-1,k) \;+\;
 \xi \, (2n-3k+1)\, C(n-1,k-1)
 \\ \qquad \;+\;
 (n-k+1) \, C(n-1,k-2) - \xi^3\, (k+1) \, C(n-1,k+1) \bigr] \,.
\label{line3.coroA13}
\end{multline}
And finally, the fourth term on the right-hand side of \eqref{eq1.coroA13}
is handled by using the absorption/extraction identity
$(j+1) \, B_\xi(j,k) = (k+1) \, B_\xi(j+1,k+1)$, yielding
\be
\tau_n \, (k+1) \, C(n-1,k+1) \,.
\label{line4.coroA13}
\ee
Putting together the partial results
\eqref{line1.coroA13}/\eqref{line2.coroA13}/%
\eqref{line3.coroA13}/\eqref{line4.coroA13},
we arrive at the claimed result \eqref{eq.cor.case4.GKPZ}.
\end{proof}

Specializing Corollary~\ref{cor.case4.GKPZ} to the GKPZ case
\be
   \beta_n \,=\, \beta \,,\quad
   \gamma_n \,=\, \alpha n + \gamma \,,\quad
   \beta'_n \,=\, \beta' \,,\quad
   \gamma'_n \,=\, \alpha' n + \gamma' \,,\quad
   \sigma_n \,=\, \sigma \,,\quad
   \tau_n \,=\, \tau \,,
 \label{eq.betagamma.GKPZ}
\ee
we get:

\begin{corollary}
 \label{cor.case4.GKPZ2}
Let $\bA = \bigl( A(n,k) \bigr)_{0 \le k \le n}$
be a triangular array defined by the GKPZ recurrence
\begin{eqnarray}
  A(n,k)
  & = &
  (\alpha n + \beta k + \gamma)    \, A(n-1,k)
  \:+\:
  (\alpha' n + \beta' k + \gamma') \, A(n-1,k-1)
  \qquad
         \nonumber \\[1mm]
  & & \;\;
  \;+\;
  \sigma \, (n-k+1) \, A(n-1,k-2)
  \:+\:
  \tau \, (k+1) \, A(n-1,k+1)
 \label{eq.case4.GKPZ2}
\end{eqnarray}
for $n \ge 1$, with initial condition $A(0,k) = \delta_{k0}$.
Then the matrix product $\bC = \bA \, B_\xi$ satisfies the GKPZ recurrence
\begin{eqnarray}
 C(n,k)
 & = &
 \bigl[ (\alpha + \xi^2 \sigma) n
           \,+\, (\beta + 2\xi \beta' - 3\xi^2 \sigma) k \,+\,
 \gamma + \xi(\beta' +\gamma') - \xi^2 \sigma
 \bigr] \, C(n-1,k)
 \nonumber \\[2mm]
 & &
 \;+\;
 \bigl[ (\alpha' + 2\xi \sigma) n
           \,+\, (\beta' - 3 \xi \sigma) k \,+\,
 \gamma' + \xi\sigma \bigr] \, C(n-1,k-1)
 \nonumber \\[2mm]
 & &
 \;+\;
 \sigma \, (n-k+1) \, C(n-1,k-2)
 \nonumber \\[2mm]
 & &
 \;+\;
 (\tau_n + \xi\beta + \xi^2 \beta' - \xi^3\sigma) \, (k+1) \,
 C(n-1,k+1)
 \label{eq.cor.case4.GKPZ2}
\end{eqnarray}
for $n \ge 1$, with initial condition $C(0,k) = \delta_{k0}$.
\end{corollary}

\section{Inverse pairs of lower-triangular arrays}   \label{app.inverse}

Let $A = (a_{nk})_{n \ge k \ge 0}$ be a lower-triangular array
with entries in a commutative ring~$R$,
and define as usual the row-generating polynomials
\be
   A_n(x)  \;=\;  \sum_{k=0}^n a_{nk} \, x^k  \;\in\; R[x]
\ee
and the reversed row-generating polynomials
\be
   \overline{A}_n(x)
   \;=\;
   x^n \, A_n(1/x)
   \;=\;
   \sum_{k=0}^n a_{nk} \, x^{n-k}
   \;\in\; R[x]
   \;.
\ee

\begin{lemma}[Inverse pairs of lower-triangular arrays]
   \label{lemma.inverse}
Let $A = (a_{nk})_{n \ge k \ge 0}$
and $B = (b_{nk})_{n \ge k \ge 0} \vphantom{\displaystyle\sum\limits}$
be lower-triangular arrays
with entries in a commutative ring $R$,
and let $A_n(x)$, $\overline{A}_n(x)$ and $B_n(x)$, $\overline{B}_n(x)$
be their row-generating polynomials.
Let $\alpha \in R$.
Then the following are equivalent:
\begin{itemize}
   \item[(a)]
      $\displaystyle A_n(x)
       \;=\;
       (1 + \alpha x)^n \, B_n \Bigl( \frac{x}{1 + \alpha x} \Bigr) \;.$
   \item[(b)]
      $\displaystyle B_n(x)
       \;=\;
       (1 - \alpha x)^n \, A_n \Bigl( \frac{x}{1 - \alpha x} \Bigr) \;.$
   \item[(c)]
      $\displaystyle \overline{A}_n(x)  \;=\; \overline{B}_n(x + \alpha) \;.$
   \item[(d)]
      $\displaystyle \overline{B}_n(x)  \;=\; \overline{A}_n(x - \alpha) \;.$
   \item[(e)]
      $\displaystyle a_{nk}
       \;=\;
       \sum_{j=0}^k \alpha^{k-j} \, \binom{n-j}{k-j} \, b_{nj}  \;.$
   \item[(f)]
      $\displaystyle b_{nk}
       \;=\;
       \sum_{j=0}^k (-\alpha)^{k-j} \, \binom{n-j}{k-j} \, a_{nj}  \;.$
   \item[(g)]
      $\displaystyle a_{n,n-k}
       \;=\;
       \sum_{j=0}^k b_{n,n-j} \, \binom{j}{k} \, \alpha^{j-k}  \;.$
   \item[(h)]
      $\displaystyle b_{n,n-k}
       \;=\;
       \sum_{j=0}^k a_{n,n-j} \, \binom{j}{k} \, (-\alpha)^{j-k}  \;.$
\end{itemize}
\end{lemma}

\begin{proof}
The equivalence of (a)--(d) is an easy manipulation of generating functions;
extracting the coefficient of $x^k$ yields (e)--(h).
\end{proof}

Pairs $(A,B)$ satisfying the conditions of Lemma~\ref{lemma.inverse}
arise frequently in combinatorial applications,
most often with $\alpha = 1$.

Let us remark that
$\displaystyle (1 + \alpha x)^n \, B_n \Bigl( \frac{x}{1 + \alpha x} \Bigr)$
is here simply an abbreviation for
\linebreak
${\sum\limits_{k=0}^n b_{nk} \, x^k \, (1 + \alpha x)^{n-k}}$.
So there is no need to work in a field of rational functions;
everything can be done in the ring $R[x]$ of polynomials,
and moreover the ring $R$ is not required to be an integral domain.

Let us also observe that, in Lemma~\ref{lemma.inverse},
the labels $n$ ``go for the ride'':
the statements for different $n$ are completely unrelated.
So we can just rename $A_n,B_n$ as $f,g$.
Furthermore, in statements (a,b,e,f), $f$ and $g$ need not be polynomials;
they can be general formal power series.
And the power $n$ in $(1 \pm \alpha x)^n$ need not be a positive integer
(provided that the ring $R$ contains the rationals);
it can be an indeterminate, call it $p$.
So what (a)$\iff$(b) really asserts is the easily verified equivalence
\be
   f(x) \:=\: (1 + \alpha x)^p \, g \Bigl( \frac{x}{1 + \alpha x} \Bigr)
   \quad\Longleftrightarrow\quad
   g(x) \:=\: (1 - \alpha x)^p \, f \Bigl( \frac{x}{1 - \alpha x} \Bigr)
   \;,
\ee
in which the coefficient of each power of $x$
is a polynomial (with coefficients in $R \supseteq \Q$)
in the indeterminates $p$ and $\alpha$.
Similarly, (e)$\iff$(f) is the binomial identity
\be
   \sum_{j=0}^k \alpha^{k-j} \, \binom{p-j}{k-j} \,
                (-\alpha)^{j-\ell} \, \binom{p-\ell}{j-\ell}
   \;=\;
   \delta_{k\ell}
   \;,
 \label{lemma.inverse.identity}
\ee
in which both sides are polynomials (with coefficients in $\Q$)
in the indeterminates $p$ and $\alpha$.
The identity \reff{lemma.inverse.identity} can be found in Riordan
\cite[p.~49, Table~2.1, item~3]{Riordan_68}.\footnote{
   However, his $\binom{p-k}{p-n}$ should be rewritten as
   $\binom{p-k}{n-k}$ in order to bring out more clearly that
   $p$ can be an indeterminate.
}
It can also be proven without using generating functions:
First rewrite the Chu--Vandermonde identity
$\displaystyle{ \sum\limits_{i=0}^m \binom{-(x+1)}{i} \binom{x+y+1}{m-i}
                = \binom{y}{m}}$
as
$\displaystyle{ \sum\limits_{i=0}^m (-1)^i \binom{x+i}{i} \binom{x+y+1}{m-i}                    = \binom{y}{m}}$.
Specializing to $y = m-1$ we have
$\displaystyle{ \sum\limits_{i=0}^m (-1)^i \binom{x+i}{i} \binom{x+m}{m-i}}$
                $\displaystyle{= \binom{m-1}{m} = \delta_{m0}}$.\footnote{
   We have not been able to find either of these two latter identities
   in any of the standard tables of binomial identities,
   such as Gould's \cite[Table~3]{Gould_72}.
   But an essentially equivalent identity can be found in
   Riordan \cite[p.~8, eq.~(5)]{Riordan_68},
   and all such identities can in any case be considered as
   simple rephrasings and/or specializations of Chu--Vandermonde.
}
Substitute $m = k-\ell$ and $x = p-k$,
and then rewrite the summation in terms of $j = k-i$:
this yields \reff{lemma.inverse.identity}.

%
%

\end{document}


\bibliographystyle{plain}
\maketitle

%
%
%
%
\newcommand{\be}{\begin{equation}}
\newcommand{\ee}{\end{equation}}
\newcommand{\<}{\langle}
\renewcommand{\>}{\rangle}
\newcommand{\widebar}{\overline}
\def\reff#1{(\protect\ref{#1})}
\def\spose#1{\hbox to 0pt{#1\hss}}
\def\ltapprox{\mathrel{\spose{\lower 3pt\hbox{$\mathchar"218$}}
 \raise 2.0pt\hbox{$\mathchar"13C$}}}
\def\gtapprox{\mathrel{\spose{\lower 3pt\hbox{$\mathchar"218$}}
 \raise 2.0pt\hbox{$\mathchar"13E$}}}
\def\textprime{${}^\prime$}
\def\half{\frac{1}{2}}
\def\third{\frac{1}{3}}
\def\twothird{\frac{2}{3}}
\def\smfrac#1#2{\textstyle \frac{#1}{#2}}
\def\smhalf{\smfrac{1}{2} }

%
%
\newcommand{\restrict}{\upharpoonright}
\newcommand{\drop}{\setminus}
\renewcommand{\emptyset}{\varnothing}
\newcommand{\eqdef}{\stackrel{\rm def}{=}}
\newcommand{\rem}{\textrm{rem}}
\newcommand{\Sym}{{\mathfrak{S}}}
\def\proof{\par\medskip\noindent{\sc Proof.\ }}
\def\qed{ $\square$ \bigskip}
\newcommand{\myendremark}{ $\blacksquare$ \bigskip}
\def\proofof#1{\bigskip\noindent{\sc Proof of #1.\ }}

%
%
\newcommand{\cyc}{{\rm cyc}}
\newcommand{\exc}{{\rm exc}}
\newcommand{\rec}{{\rm rec}}
\newcommand{\arec}{{\rm arec}}
\newcommand{\erec}{{\rm erec}}

%
%
\newcommand{\C}{{\mathbb C}}
\newcommand{\D}{{\mathbb D}}
\newcommand{\Z}{{\mathbb Z}}
\newcommand{\N}{{\mathbb N}}
\newcommand{\R}{{\mathbb R}}
\newcommand{\Q}{{\mathbb Q}}

%
%
\newcommand{\TT}{{\mathsf T}}
\newcommand{\HH}{{\mathsf H}}
\newcommand{\VV}{{\mathsf V}}
\newcommand{\JJ}{{\mathsf J}}
\newcommand{\PP}{{\mathsf P}}
\newcommand{\DD}{{\mathsf D}}
\newcommand{\QQ}{{\mathsf Q}}
\newcommand{\RR}{{\mathsf R}}

%
%
\newcommand{\bmu}{{\bm{\mu}}}
\newcommand{\bsigma}{{\bm{\sigma}}}
\newcommand{\vecbsigma}{{\vec{\bm{\sigma}}}}
\newcommand{\bpi}{{\bm{\pi}}}
\newcommand{\vecbpi}{{\vec{\bm{\pi}}}}
\newcommand{\btau}{{\bm{\tau}}}
\newcommand{\bphi}{{\bm{\phi}}}
\newcommand{\bvarphi}{{\bm{\varphi}}}
\newcommand{\bGamma}{{\bm{\Gamma}}}

%
%
\newcommand{\psibar}{ {\bar{\psi}} }
\newcommand{\varphibar}{ {\bar{\varphi}} }

%
%
\newcommand{\bfa}{ {\bf a} }
\newcommand{\bfb}{ {\bf b} }
\newcommand{\bfc}{ {\bf c} }
\newcommand{\bfp}{ {\bf p} }
\newcommand{\bfr}{ {\bf r} }
\newcommand{\bfs}{ {\bf s} }
\newcommand{\bft}{ {\bf t} }
\newcommand{\bfu}{ {\bf u} }
\newcommand{\bfv}{ {\bf v} }
\newcommand{\bfw}{ {\bf w} }
\newcommand{\bfx}{ {\bf x} }
\newcommand{\bfy}{ {\bf y} }
\newcommand{\bfz}{ {\bf z} }
\newcommand{\bfT}{ {\bf T} }
\newcommand{\bone}{ {\mathbf 1} }

%
%
\newcommand{\ba}{{\bm{a}}}
\newcommand{\bb}{{\bm{b}}}
\newcommand{\bc}{{\bm{c}}}
\newcommand{\bd}{{\bm{d}}}
\newcommand{\bP}{{\bm{P}}}
\newcommand{\bT}{{\bm{T}}}

%
%
\newcommand{\scra}{{\mathcal{A}}}
\newcommand{\scrb}{{\mathcal{B}}}
\newcommand{\scrc}{{\mathcal{C}}}
\newcommand{\scrd}{{\mathcal{D}}}
\newcommand{\scre}{{\mathcal{E}}}
\newcommand{\scrf}{{\mathcal{F}}}
\newcommand{\scrg}{{\mathcal{G}}}
\newcommand{\scrh}{{\mathcal{H}}}
\newcommand{\scri}{{\mathcal{I}}}
\newcommand{\scrj}{{\mathcal{J}}}
\newcommand{\scrk}{{\mathcal{K}}}
\newcommand{\scrl}{{\mathcal{L}}}
\newcommand{\scrm}{{\mathcal{M}}}
\newcommand{\scrn}{{\mathcal{N}}}
\newcommand{\scro}{{\mathcal{O}}}
\newcommand{\scrp}{{\mathcal{P}}}
\newcommand{\scrq}{{\mathcal{Q}}}
\newcommand{\scrr}{{\mathcal{R}}}
\newcommand{\scrs}{{\mathcal{S}}}
\newcommand{\scrt}{{\mathcal{T}}}
\newcommand{\scru}{{\mathcal{U}}}
\newcommand{\scrv}{{\mathcal{V}}}
\newcommand{\scrw}{{\mathcal{W}}}
\newcommand{\scrx}{{\mathcal{X}}}
\newcommand{\scry}{{\mathcal{Y}}}
\newcommand{\scrz}{{\mathcal{Z}}}

%
%
\newcommand{\ofo}{ {{}_1 \! F_1} }
\newcommand{\tfo}{ {{}_2 F_1} }

%
%
\def\hboxrm#1{ {\hbox{\scriptsize\rm #1}} }
\def\hboxsans#1{ {\hbox{\scriptsize\sf #1}} }
\def\hboxscript#1{ {\hbox{\scriptsize\it #1}} }

%
%

%
%
\newcommand{\stirlingsubset}[2]{\genfrac{\{}{\}}{0pt}{}{#1}{#2}}
\newcommand{\stirlingcycle}[2]{\genfrac{[}{]}{0pt}{}{#1}{#2}}
\newcommand{\associatedstirlingsubset}[2]%
      {\left\{\!\! \stirlingsubset{#1}{#2} \!\! \right\}}
\newcommand{\assocstirlingsubset}[3]%
      {{\genfrac{\{}{\}}{0pt}{}{#1}{#2}}_{\! \ge #3}}
\newcommand{\assocstirlingcycle}[3]{{\genfrac{[}{]}{0pt}{}{#1}{#2}}_{\ge #3}}
\newcommand{\associatedstirlingcycle}[2]{\left[\!\!%
            \stirlingcycle{#1}{#2} \!\! \right]}
\newcommand{\euler}[2]{\genfrac{\langle}{\rangle}{0pt}{}{#1}{#2}}
\newcommand{\eulergen}[3]{{\genfrac{\langle}{\rangle}{0pt}{}{#1}{#2}}_{\! #3}}
\newcommand{\eulersecond}[2]{\left\langle\!\! \euler{#1}{#2} \!\!\right\rangle}
\newcommand{\eulersecondBis}[2]{\big\langle\!\! \euler{#1}{#2} \!\!\big\rangle}
\newcommand{\eulersecondgen}[3]%
     {{\left\langle\!\! \euler{#1}{#2} \!\!\right\rangle}_{\! #3}}
\newcommand{\associatedstirlingcycleBis}[2]{\big[ \!\!%
            \stirlingcycle{#1}{#2} \!\! \big]}
\newcommand{\binomvert}[2]{\genfrac{\vert}{\vert}{0pt}{}{#1}{#2}}
\newcommand{\doublebinom}[2]{\left(\!\! \binom{#1}{#2} \!\!\right)}
\newcommand{\nueuler}[3]{{\genfrac{\langle}{\rangle}{0pt}{}{#1}{#2}}^{\! #3}}
\newcommand{\nueulergen}[4]%
{{\genfrac{\langle}{\rangle}{0pt}{}{#1}{#2}}^{\! #3}_{\! #4}}

%
%
\newenvironment{sarray}{
          \textfont0=\scriptfont0
          \scriptfont0=\scriptscriptfont0
          \textfont1=\scriptfont1
          \scriptfont1=\scriptscriptfont1
          \textfont2=\scriptfont2
          \scriptfont2=\scriptscriptfont2
          \textfont3=\scriptfont3
          \scriptfont3=\scriptscriptfont3
        \renewcommand{\arraystretch}{0.7}
        \begin{array}{l}}{\end{array}}

\newenvironment{scarray}{
          \textfont0=\scriptfont0
          \scriptfont0=\scriptscriptfont0
          \textfont1=\scriptfont1
          \scriptfont1=\scriptscriptfont1
          \textfont2=\scriptfont2
          \scriptfont2=\scriptscriptfont2
          \textfont3=\scriptfont3
          \scriptfont3=\scriptscriptfont3
        \renewcommand{\arraystretch}{0.7}
        \begin{array}{c}}{\end{array}}

\newcommand{\seqnum}[1]{\href{http://oeis.org/#1}{#1}}

\def\k{\phantom{ \!}}

%
%
\section{Introduction} \label{sec.intro} 

We consider the Graham--Knuth--Patashnik recurrence, namely
\begin{equation}
  T(n,k)
  \;=\;
  (\alpha n + \beta k + \gamma)    \, T(n-1,k)
  \:+\:
  (\alpha' n + \beta' k + \gamma') \, T(n-1,k-1)
  \label{eq_binomvert}
\end{equation}
for $n \ge 1$, with initial condition $T(0,k) = \delta_{k0}$.
For each choice of the parameters
$\bmu = (\alpha,\beta,\gamma, \alpha',\beta',\gamma')$,
we obtain a unique solution $T(n,k) = T(n,k;\bmu)$,
forming a triangular array
$\bT(\bmu) = \bigl( T(n,k;\bmu) \bigr)_{0 \le k \le n}$.
Given a triangular array $\bT = \bigl( T(n,k) \bigr)_{0 \le k \le n}$,
we define its row-generating polynomials
\be
P_n(x) \;\eqdef\;
       \sum\limits_{k=0}^n T(n,k) \, x^k 
\label{def_Pn}
\ee
and its ordinary generating function (ogf)
\be
f(x,t) \;\eqdef\;
       \sum\limits_{n\ge 0} P_n(x) \, t^n \,.
\label{def_ogf}
\ee

We are looking for all the submanifolds in $\bmu \in \C^6$  
such that the corresponding ogf \eqref{def_ogf} has an
S-type continued fraction 
\be
f(x,t; \bmu) \;=\; \sum\limits_{n\ge 0} P_n(x;\bmu) \, t^n
       \;=\; \cfrac{c_0}{1 - \cfrac{c_1 t}{ 1 -
                           \cfrac{c_2 t}{ 1 - \cdots}}}
   \;,
\label{def_Stype.one}
\ee
where the coefficients $c_i$ are {\em polynomials}\/
(and not just rational functions) in $x$
that moreover are not identically zero.
The final result of this investigation is stated as Theorem~3.1
of the main paper.
The first step in the proof of Theorem~3.1
is to make a computer-assisted search using {\sc Mathematica}
to find a finite set of viable candidates for such an S-fraction.
In this Supplementary Material we will give all the details of this search,
which gives rise to the decision tree shown in Figure~\ref{decision_tree}.

Each node in this decision tree will be labeled
by a finite-length vector of the type 
(b${}_0$x${}_0$, b${}_1$x${}_1$, b${}_2$x${}_2$, \ldots , b${}_n$x${}_n$),
where in each entry b${}_i \in \{0,1\}$ and x${}_i\in \{$a, b, c, \ldots $\}$. 
Here b${}_i$ denotes the degree of the denominator polynomial $R(x)$,
as explained below, while the letter x${}_i$ labels the various
cases that can occur at that stage.
When there is only a single case, we can drop the letter x${}_i$
and write simply b${}_i$x${}_i$ = b${}_i$a = b${}_i$.

The reason behind this representation is that,
as explained in Section~3.1 of the accompanying article,
at the $k$-th level of this decision tree
we always find that the corresponding coefficient $c_k$ can be written as 
\be
c_k(x) \;=\; \frac{Q(x)}{R(x)}
\ee
where the denominator $R(x)$ is a polynomial of degree $\le 1$ in $x$
and the numerator $Q(x)$ is a polynomial of degree $\le 2$ in $x$,
and the coefficients in these polynomials are polynomial expressions in $\bmu$. 
Furthermore, the decisions taken at earlier stages will ensure
that the denominator polynomial $R(x)$ is not identically zero;
indeed, in all cases $R(x)$ turns out to be a nonzero multiple
of the preceding coefficient $c_{k-1}(x)$.
So $c_k(x) = Q(x)/R(x)$ is a polynomial in $x$ if and only if either
\begin{itemize}
   \item[(0)]  the denominator polynomial has degree~$0$ in $x$:
      that is, its leading coefficient $[x^1] R(x)$ is zero
      and its constant term $[x^0] R(x)$ is nonzero; or
   \item[(1)]  the denominator polynomial has degree~$1$ in $x$
      --- that is, its leading coefficient $[x^1] R(x)$ is nonzero ---
      and the remainder $\rem(Q(x),R(x))$ is zero.
\end{itemize}
Therefore, b${}_k=0$ (resp.\ 1) if we are in the first (resp.\ second) case,
and the letter x${}_k$ labels the different sub-cases that we encounter in 
each case (if more than one).
In~all these cases, it turns out that $c_k(x)$ is a polynomial in $x$
of degree at most~1, i.e.\ it is of the form $c_k(x) = A + Bx$.

A branch in this decision tree will be terminated if
\begin{itemize}
  \item We find a viable candidate family for a non-terminating S-fraction. 
        We consider that a submanifold $\scrm$ of $\bmu \in \C^6$
        gives rise to a viable candidate for a non-terminating S-fraction
        if $c_i$ is a polynomial in $x$ that is generically nonvanishing
        on $\scrm$, for $0\le i \le 10$.
        These are the red nodes in Figure~\ref{decision_tree}.
\end{itemize} 
In all other cases --- i.e.\ those for which the most recent $c_k$
is a polynomial in $x$ that is generically nonvanishing on $\scrm$,
but we have not yet obtained a viable candidate family ---
we proceed as follows.
With $c_k(x) = A+Bx$, the manifold $\scrm$ divides into two parts:
\begin{itemize}
   \item[(a)] $A = B = 0$.  Then $c_k(x)$ is the zero polynomial,
and we have a terminating S-fraction.
Although these cases are irrelevant for the proof of Theorem~3.1
(which concerns only non-terminating S-fractions),
they are nevertheless of some interest in their own right;
we will consider them in Section~\ref{sec.rational.ogf} below.
   \item[(b)] $A \neq 0$ or $B \neq 0$.  This disjunction of inequations
corresponds to the statement that $c_k(x)$ is not the zero polynomial.
Going forward, we will impose this disjunction of inequations
on all children of the given node;
that will ensure that their denominator polynomials $R(x)$
are not identically zero.
\end{itemize}
A branch in this decision tree will also be terminated if
\begin{itemize}
  \item The given node has no children: that is, the set of solutions $\bmu$
        is empty for both b${}_{k+1} = 0$ and b${}_{k+1} = 1$. 
        These are the gray leaf nodes in Figure~\ref{decision_tree}. 
\end{itemize}

We will also encounter cases that are either
(i) included in a previously found case,
or (ii) have a nonempty intersection with a previously found case
but without being included in it.
In what follows, we will mention these cases explicitly
where it seems relevant;
in situation~(i) we will discard them as redundant,
while in situation~(ii) we will sometimes (but not always)
impose additional inequations to remove the overlap
with the previously found case.

The full decision tree for this proof is shown in Figure~\ref{decision_tree}.

%
%
\section{Computer-assisted search} \label{sec.search} 

In this section we will describe in great detail the computer-assisted search
using {\sc Mathematica} that led to the 10 candidate families presented in
Theorem~3.1 of the main paper. 
Let us stress at the outset that our search was ``computer-assisted''
{\em only}\/ in the sense that it used {\sc Mathematica}
to perform elementary algebraic operations such as manipulation of
polynomials and rational functions, series expansions,
and {\tt Polynomial Remainder}.
We did {\em not}\/ need to use any more advanced algebraic functions
such as {\tt Solve} or {\tt Reduce}.
If one assumes the correctness of {\sc Mathematica}'s algebraic manipulations,
the rest of the proof is easily human-verifiable (though tedious)
and is explained in detail in Sections~\ref{sec.c0}--\ref{sec.c7} herein.

%
%
\subsection{\texorpdfstring{Step 0: Coefficient $\bm{c_0}$}{Coefficient c0}}
\label{sec.c0}

This coefficient is always $c_0=1$, so it is a polynomial of degree $0$ 
for any choice of $\bmu\in \C^6$. The denominator is thus $R(x)=1$,
a polynomial of degree~$0$.
This is the root vertex of our decision tree
(i.e., the top black node in Figure~\ref{decision_tree}). 

\bigskip

\noindent
{\bf Case (0)}: 
     $\bmu = (\alpha,\beta,\gamma, \alpha',\beta',\gamma')$.  

%
%
\subsection{\texorpdfstring{Step 1: Coefficient $\bm{c_1}$}{Coefficient c1}}
\label{sec.c1}
                        
The coefficient $c_1$ is also a polynomial in $x$,
\be
c_1 \;=\; (\alpha+\gamma) + (\alpha'+\beta'+\gamma')\, x
   \;,
\label{def.c1}
\ee
so the denominator is again $R(x) = 1$. 
This is equal to the value of $c_0$ for the Case~(0).
We thus have

\bigskip

\noindent
{\bf Case (0,\k0)}: 
$\bmu = (\alpha,\beta,\gamma, \alpha',\beta',\gamma')$.
In this case, $c_1$ is the polynomial \eqref{def.c1}.
Going forward, we impose the disjunction of inequations
\be
\alpha+\gamma  \:\neq\: 0 
\quad \text{ or } \quad 
\alpha'+\beta'+\gamma' \:\neq\: 0
\label{eq.t0.0}
\ee
in order to ensure that $c_1 \neq 0$.

%
%
\subsection{\texorpdfstring{Step 2: Coefficient $\bm{c_2}$}{Coefficient c2}}
\label{sec.c2}

Imposing now the conditions \eqref{eq.t0.0} of the Case~(0,\k0) --- namely,
$\alpha+\gamma \neq 0$ \emph{or} $\alpha'+\beta'+\gamma' \neq 0$ ---
the coefficient $c_2$ is a rational function of $x$:
$c_2 = Q(x)/R(x)$ with   
\begin{subeqnarray}
Q(x) &=& \alpha (\alpha+\gamma) \nonumber \\
     & & \quad \;+\; 
         (2 \alpha \alpha' + \alpha \beta' + \alpha \gamma' + \beta \beta' + 
            \beta \gamma'  + \beta \alpha' + \gamma \alpha') \, x \nonumber \\ 
     & & \quad \;+\; (\alpha'+\beta') (\alpha'+\beta'+\gamma')\, x^2 \\[2mm] 
R(x) &=& \alpha+\gamma \;+\; (\alpha'+\beta'+\gamma')\, x
\end{subeqnarray}
Note that the polynomial $R(x)$ is equal to the value \eqref{def.c1} of $c_1$ 
for the Case~(0,\k0).

The conditions \eqref{eq.t0.0} imply that
the polynomial $R(x)$ is {\em not}\/ identically zero.
It follows that there are two cases:
if $[x^1] \, R(x)= \alpha'+\beta'+\gamma' \neq 0$, then $\deg R = 1$ and
\be
\rem(Q(x),R(x)) \;=\; \frac{(\alpha + \gamma)[\beta' (\alpha+\gamma) - 
                           \beta (\alpha'+\beta'+\gamma')]}
                         {\alpha'+\beta'+\gamma'}  
   \;;
\label{eq.rem.c2}
\ee
while if $[x^1] \, R(x)= \alpha'+\beta'+\gamma' = 0$,
then $\deg R = 0$ and $\rem(Q(x),R(x))$ is identically zero.
[Henceforth, when we state a formula for $\rem(Q(x),R(x))$,
it will be tacitly understood that it applies when and only when
$[x^1] \, R(x) \neq 0$.]

We thus have:

\bigskip

\begin{itemize}
\item
$\bm{\deg R = 0\colon}$
$[x^1] \, R(x)= \alpha'+\beta'+\gamma'=0$
and $[x^0] \, R(x)= \alpha+\gamma\neq 0$.
This gives:

\bigskip

%
%
{\bf Case (0,\k0,\k0)}: 
$\bmu = (\alpha,\beta,\gamma, \alpha',\beta',-\alpha'-\beta')$ 
with $\alpha+\gamma \neq 0$. In this case, $c_2 = \alpha + \alpha' x$. 
Going forward, we impose the disjunction of inequations
\be
\alpha \:\neq\: 0 
\quad \text{ or } \quad 
\alpha' \:\neq\: 0
\ee
in order to ensure that $c_2 \neq 0$.

\bigskip

\item
$\bm{\deg R = 1\colon}$
$[x^1] \, R(x)= \alpha'+\beta'+\gamma'\neq 0$.
We now need to impose the condition $\rem(Q(x),R(x))=0$.
Since the numerator of \eqref{eq.rem.c2} is explicitly factorized,
each new case will correspond to one of its two factors being zero.
Therefore we find two possibilities: either $\alpha + \gamma=0$ or
$\beta' (\alpha+\gamma) - \beta (\alpha'+\beta'+\gamma')=0$. 
In the second case we can assume without loss of generality
that we are not also in the first case,
i.e.\ we can assume that $\alpha + \gamma \neq 0$.

\bigskip

%
%
\noindent
{\bf Case (0,\k0,\k1a)}: We have $\alpha + \gamma=0$ and hence
\be
\bmu \;=\; \bigl(\alpha,\beta,-\alpha, \alpha',\beta',\gamma'\bigr) 
\ee
with $\alpha'+\beta'+\gamma' \neq 0$. 
In this case, $c_2 = (\alpha+\beta) + (\alpha'+\beta') x$. 
Going forward, we impose the disjunction of inequations
\be
\alpha+\beta \:\neq\: 0 
\quad \text{ or } \quad 
\alpha'+\beta' \:\neq\: 0
\ee
in order to ensure that $c_2 \neq 0$.

\bigskip

%
%
\noindent
{\bf Case (0,\k0,\k1b)}: We have
$\beta' (\alpha+\gamma) - \beta (\alpha'+\beta'+\gamma')=0$ with 
$\alpha + \gamma \neq 0$ and $\alpha'+\beta'+\gamma' \neq 0$.
 Here it is convenient to express 
$\beta$ as a rational function of the other variables in $\bmu$: 
\be 
\bmu \;=\; \left(\alpha,\frac{(\alpha+\gamma) \beta'}{\alpha'+\beta'+\gamma'},
           \gamma,\alpha',\beta',\gamma'\right)
\ee
with $\alpha+\gamma \neq 0$ and $\alpha'+\beta'+\gamma' \neq 0$.
In this case, $c_2 = \alpha + (\alpha'+\beta') x$. 
Going forward, we impose the disjunction of inequations
\be
\alpha \:\neq\: 0 
\quad \text{ or } \quad 
\alpha'+\beta' \:\neq\: 0
\ee
in order to ensure that $c_2 \neq 0$.

\qquad
Let us again stress that in Case~(0,\k0,\k1b)
we have imposed the extra condition ${\alpha+\gamma\neq 0}$
in order to remove the overlap with Case~(0,\k0,\k1a).
\end{itemize}

%
%
\subsection{\texorpdfstring{Step 3: Coefficient $\bm{c_3}$}{Coefficient c3}}
\label{sec.c3}

We consider the three cases (i.e., vertices of the decision tree) obtained
in Section~\ref{sec.c2};
for each one, we compute the branches stemming from it
by using the next coefficient $c_3$.

%
%
\subsubsection{\texorpdfstring{Branches starting from Case~(0,\! 0,\! 0)}
                              {Branches starting from Case (0,0,0)}} 

We have $c_3 = Q(x)/R(x)$ with  
\begin{subeqnarray}
Q(x) &=& \alpha (2\alpha+\gamma) \;+\;
         \alpha' (3\alpha + \beta + \gamma) \, x \;+\; 
         \alpha' (\alpha'+\beta')\, x^2  \\[2mm]
R(x) &=& \alpha \;+\; \alpha'\, x  \\[2mm]
\rem(Q(x),R(x)) &=& \frac{\alpha (\alpha \beta' - \beta \alpha')}{\alpha'}
\end{subeqnarray}
and the conditions $\alpha+\gamma \neq 0$
and ($\alpha \neq 0$ or $\alpha' \neq 0$)  of Case~(0,\k0,\k0). 
The polynomial $R(x)$ is equal to the coefficient $c_2$ for
the Case~(0,\k0,\k0).

\bigskip

\begin{itemize}
\item
$\bm{\deg R = 0\colon}$
If $[x^1]R(x)= \alpha'=0$ and $[x^0]R(x)= \alpha \neq 0$, we get 

\bigskip

%
%
\noindent
{\bf Case (0,\k0,\k0,\k0)}: $\bmu = (\alpha,\beta,\gamma, 0,\beta',-\beta')$ 
with $\alpha+\gamma \neq 0$ and $\alpha \neq 0$. 
In this case, $c_3 = \gamma + 2\alpha$.
Pursuing the computation of this family to higher order,
we find that all the coefficients $c_i$ for $1 \le i \le 10$
are polynomials in $x$ that are generically nonzero:
\be
c_{2k-1} \;=\; \gamma \,+\, k\alpha \,, \quad
c_{2k}   \;=\; k\alpha   \;.
\label{eq.ak.IIb}
\ee
So this is a viable candidate, corresponding to Family~2b in the main text, 
and this branch is terminated.

\bigskip

\item
$\bm{\deg R = 1\colon}$
If $[x^1]R(x)= \alpha'\neq 0$, the condition $\rem(Q(x),R(x))=0$ leads to two 
new cases: $\alpha=0$ or $\alpha \beta' - \beta \alpha'=0$. 

\bigskip

%
%
\noindent
{\bf Case (0,\k0,\k0,\k1a)}: We have $\alpha=0$, so that  
\be
\bmu \;=\; \bigl(0,\beta,\gamma, \alpha',\beta',-\alpha'-\beta'\bigr) 
\ee 
with $\gamma \neq 0$ and $\alpha' \neq 0$. 
In this case, $c_3 = (\beta+\gamma) + (\alpha'+\beta')\, x$. 
Going forward, we impose the disjunction of inequations
\be
\beta+\gamma \:\neq\: 0 
\quad \text{ or } \quad 
\alpha'+\beta' \:\neq\: 0
\ee
in order to ensure that $c_3 \neq 0$.

\bigskip

%
%
\noindent
{\bf Case (0,\k0,\k0,\k1b)}: We have 
$\alpha \beta' - \beta \alpha'=0$ with 
$\alpha' \neq 0$. In this case it is convenient to express $\beta$ as 
a rational function of the other variables in $\bmu$: 
\be 
\bmu \;=\; \left(\alpha,\frac{\alpha\beta'}{\alpha'},\gamma,
                 \alpha',\beta',-\alpha'-\beta' \right)
\ee
with $\alpha+\gamma \neq 0$ and $\alpha' \neq 0$.
In this case, $c_3 = (2\alpha + \gamma) + (\alpha'+\beta')\, x$. 
Going forward, we impose the disjunction of inequations
\be
2\alpha+\gamma \:\neq\: 0 
\quad \text{ or } \quad 
\alpha'+\beta' \:\neq\: 0
\ee
in order to ensure that $c_3 \neq 0$.

\qquad
{\bf Remark.}  In Case~(0,\k0,\k0,\k1b) we could, if we wish,
impose the extra condition $\alpha \neq 0$ in order to remove the overlap
with Case~(0,\k0,\k0,\k1a).  But we choose not to do so.
\end{itemize}

%
%
\subsubsection{\texorpdfstring{Branches starting from Case~(0,\! 0,\! 1a)}
                              {Branches starting from Case (0,b,1a)}}   

We have $c_3 = Q(x)/R(x)$ with  
\begin{subeqnarray}
Q(x) &=& \alpha (\alpha+\beta) \;+\;
         (\alpha+\beta)(3\alpha' + 2\beta' + \gamma') \, x \nonumber \\
     & & \quad \;+\; 
         (\alpha'+\beta') (2\alpha'+2\beta'+\gamma')\, x^2     \\[2mm]
R(x) &=& \alpha+\beta \;+\; (\alpha' + \beta')\, x   \\[2mm]
\rem(Q(x),R(x)) &=& \frac{(\alpha+\beta) (\alpha \beta' - \beta \alpha')}
                          {\alpha' + \beta'} 
\label{eq.c3.a}
\end{subeqnarray}
and the conditions $\alpha'+\beta'+\gamma' \neq 0$
and ($\alpha+\beta \neq 0$ or $\alpha' + \beta' \neq 0$)
of the Case~(0,\k0,\k1a). 
The polynomial $R(x)$ is equal to the coefficient $c_2$ for
the Case~(0,\k0,\k1a).

\bigskip

\begin{itemize}
\item
$\bm{\deg R = 0\colon}$
If $[x^1]R(x)= \alpha'+\beta'=0$ and $[x^0]R(x)= \alpha + \beta\neq 0$, we get 

\bigskip

%
%
\noindent
{\bf Case (0,\k0,\k1a,\k0)}: 
$\bmu = (\alpha,\beta,-\alpha,\alpha',-\alpha',\gamma')$ with $\gamma' \neq 0$
and $\alpha + \beta\neq 0$. 
In this case, $c_3 = \alpha + (\alpha'+\gamma')x$. 
Going forward, we impose the disjunction of inequations
\be
\alpha \:\neq\: 0 
\quad \text{ or } \quad 
\alpha'+\gamma' \:\neq\: 0
\ee
in order to ensure that $c_3 \neq 0$.

\bigskip

\bigskip

\item
$\bm{\deg R = 1\colon}$
If $[x^1]R(x)= \alpha'+\beta'\neq 0$,
then $\rem(Q(x),R(x))=0$ gives two other cases: 
$\alpha+\beta=0$ or $\alpha \beta' - \beta \alpha'=0$. 
We divide this latter case into two,
according as $\alpha' \neq 0$ or $\alpha' = 0$.

\bigskip

%
%
\noindent
{\bf Case (0,\k0,\k1a,\k1a)}: We have $\alpha+\beta=0$, 
so that  
\be 
\bmu \;=\; \bigl(\alpha,-\alpha,-\alpha, \alpha',\beta',\gamma'\bigr) 
\ee
with $\alpha'+ \beta' \neq 0$ and $\alpha'+\beta'+\gamma'\neq 0$. 
In this case, $c_3 = (2\alpha'+2\beta'+\gamma') x$. 
We~find in fact that all coefficients $c_i$ for $1 \le i \le 10$ are 
polynomials in $x$:
\be
c_{2k-1} \;=\; [\gamma' \,+\, k(\alpha'+\beta')]\, x\,, \quad
c_{2k}   \;=\; k (\alpha'+\beta')\, x   \;.
\label{eq.ak.IIa}
\ee
So this is a viable candidate, corresponding to Family~2a in the main text,
and this branch is terminated.

\bigskip

%
%
\noindent
{\bf Case (0,\k0,\k1a,\k1b)}: We have 
$\alpha \beta' - \beta \alpha'=0$ with 
$\alpha'+\beta' \neq 0$ and $\alpha'\neq 0$.
The latter inequation allows us to express $\beta$ as a rational function 
of the other variables in $\bmu$: 
\be 
\bmu \;=\; \left(\alpha,\frac{\alpha\beta'}{\alpha'},-\alpha,
                 \alpha',\beta',\gamma' \right)
\ee
with $\alpha' \neq 0$, $\alpha' + \beta' \neq 0$, and 
$\alpha' + \beta' + \gamma' \neq 0$.
In this case, $c_3 = \alpha + (2\alpha'+2\beta'+\gamma') x$. 
Going forward, we impose the disjunction of inequations
\be
\alpha \:\neq\: 0 
\quad \text{ or } \quad 
2\alpha'+2\beta'+\gamma' \:\neq\: 0
\ee
in order to ensure that $c_3 \neq 0$.

\bigskip

%
%
\noindent
{\bf Case (0,\k0,\k1a,\k1c)}:
We have $\alpha \beta' - \beta \alpha'=0$ and $\alpha' = 0$
with $\alpha'+\beta' \neq 0$.
This implies (and is equivalent to)
$\alpha = \alpha' = 0$ and $\beta' \neq 0$.
Hence
\be 
\bmu \;=\; \bigl(0,\beta,0, 0,\beta',\gamma'\bigr) 
\ee
with $\beta' \neq 0$ and $\beta'+\gamma'\neq 0$. 
In this case, $c_3 = (2\beta'+\gamma') x$. 
We~find in fact that all coefficients $c_i$ for $1 \le i \le 10$ are 
polynomials in $x$:
\be
c_{2k-1} \;=\; (\gamma' \,+\, k\beta')\, x\,, \quad
c_{2k}   \;=\; k (\beta + \beta' x)   \;.
\label{eq.ak.IIIa}
\ee
So this is a viable candidate, corresponding to Family~3a in the main text,
and this branch is terminated.
\end{itemize}

%
%
\subsubsection{\texorpdfstring{Branches starting from Case~(0,\! 0,\! 1b)}
                              {Branches starting from Case (0,b,1b)}}

We have $c_3 = Q(x)/R(x)$ with  
\begin{subeqnarray}
Q(x) &=& \alpha (2\alpha+\gamma)(\alpha'+\beta'+\gamma') \nonumber \\ 
     & & \quad  \;+\;
         \Bigl(\alpha \bigl( 4(\alpha')^2 + 4(\beta')^2 + 4 (\gamma')^2 + 
                  8 \alpha'\beta' + 5\alpha' \gamma' + 4 \beta' \gamma'\bigr) 
         \nonumber \\
     & & \quad \quad \quad  
          + \gamma (\alpha'+\beta')(\alpha'+2\beta'+\gamma') \Bigr) \, x 
         \nonumber \\
     & & \quad \;+\; 
         (\alpha'+\beta') (\alpha'+\beta'+\gamma')(2\alpha'+2\beta'+\gamma')
         \, x^2   \qquad \qquad \\[2mm]
R(x) &=& (\alpha' + \beta' + \gamma') \, \bigl( \alpha + 
         (\alpha' + \beta')\, x\bigr)  \\[2mm]
\rem(Q(x),R(x)) &=& \frac{\alpha \beta' 
                          (\alpha \gamma' - \gamma \alpha'-\gamma \beta')}
                          {\alpha' + \beta'}  
\end{subeqnarray}
and the conditions $\alpha+\gamma\neq 0$, $\alpha'+\beta'+\gamma' \neq 0$
and ($\alpha \neq 0$ or $\alpha'+\beta' \neq 0$)
of the Case~(0,\k0,\k1b). 
The polynomial $R(x)$ is equal to the coefficient $c_2$ for
the Case~(0,\k0,\k1b) multiplied by the factor 
$\alpha'+\beta'+\gamma'\neq 0$.

\bigskip

\begin{itemize}
\item
$\bm{\deg R = 0\colon}$
If $[x^1]R(x)= (\alpha'+\beta')(\alpha' + \beta' + \gamma')=0$ and 
$[x^0]R(x)= \alpha (\alpha' + \beta' + \gamma') \neq 0$, 
these conditions reduce to $\alpha'+\beta'=0$
(since $\alpha' + \beta' + \gamma'\neq 0$) and $\alpha\neq 0$.
We get

\bigskip

%
%
\noindent
{\bf Case (0,\k0,\k1b,\k0)}: We have
\be 
\bmu \;=\;  \left(\alpha,-\frac{(\alpha+\gamma)\alpha'}{\gamma'},\gamma, 
                  \alpha',-\alpha',\gamma'\right)
\ee
with $\alpha \neq 0$ and $\gamma' \neq 0$ and $\alpha + \gamma \neq 0$. 
In this case, $c_3 = (2\alpha + \gamma) + (\alpha'+\gamma')\, x$. 
Going forward, we impose the disjunction of inequations
\be
2\alpha + \gamma \:\neq\: 0 
\quad \text{ or } \quad 
\alpha'+\gamma' \:\neq\: 0
\ee
in order to ensure that $c_3 \neq 0$.

\bigskip

\item
$\bm{\deg R = 1\colon}$
If $[x^1]R(x)= (\alpha'+\beta')(\alpha'+\beta'+\gamma')\neq 0$,
then $\rem(Q(x),R(x)) = 0$ gives three cases:
$\alpha=0$, $\beta'=0$, and $\alpha \gamma' - \gamma(\alpha'+\beta')=0$. 

\bigskip

%
%
\noindent
{\bf Case (0,\k0,\k1b,\k1a)}: We have $\alpha=0$, so that  
\be 
\bmu \;=\; \left(0,\frac{\gamma\beta'}{\alpha'+\beta'+\gamma'},\gamma, 
            \alpha',\beta',\gamma'\right) 
\ee
with $\alpha'+ \beta' \neq 0$, $\alpha'+\beta'+\gamma'\neq 0$
and $\gamma \neq 0$. 
In this case, 
\be
c_3 \;=\; \frac{\gamma(\alpha'+2\beta'+\gamma')}{\alpha'+\beta'+\gamma'} \;+\; 
         (2\alpha'+2\beta'+\gamma')\, x \,.  
\ee
Going forward, we impose the disjunction of inequations
\be
\gamma(\alpha'+2\beta'+\gamma') \:\neq\: 0 
\quad \text{ or } \quad 
2\alpha'+2\beta'+\gamma' \:\neq\: 0
\ee
in order to ensure that $c_3 \neq 0$.

\bigskip

%
%
\noindent
{\bf Case (0,\k0,\k1b,\k1b)}: We have $\beta' =0$, so that
\be 
\bmu \;=\; \bigl(\alpha,0,\gamma, \alpha',0,\gamma'\bigr) 
\ee
with $\alpha' + \gamma' \neq 0$, $\alpha'\neq 0$ and 
$\alpha + \gamma \neq 0$. In this case, $c_3 = (2\alpha + \gamma) + 
(2\alpha' + \gamma')\, x$.  
We~find in fact that all coefficients $c_i$ for $1 \le i \le 10$ are 
polynomials in $x$:
\be
c_{2k-1} \;=\;  (\gamma + \gamma' x) \,+\, k(\alpha + \alpha' x) \,, \quad
c_{2k}   \;=\; k (\alpha + \alpha' x)   \;.
\label{eq.ak.V}
\ee 
So this is a viable candidate, corresponding to Family~$5$ in the main text,
and this branch is terminated.

\qquad
{\bf Remark.} Here we could impose the additional condition $\alpha \neq 0$
to avoid overlap with Case~(0,\k0,\k1b,\k1a), but we choose not to,
since this is a perfectly good candidate irrespective of whether
$\alpha$ is zero or nonzero.

\bigskip

%
%
\noindent
{\bf Case (0,\k0,\k1b,\k1c)}: We have 
$\alpha\gamma' - \gamma(\alpha' + \beta') =0$
with $\alpha' + \beta' \neq 0$, $\alpha' + \beta' + \gamma' \neq 0$, and 
$\alpha+\gamma\neq 0$.
Here it is convenient to express $\gamma$ as a rational 
function of the other variables in $\bmu$: 
\be 
\bmu \;=\; \left(\alpha,\frac{\alpha\beta'}{\alpha'+\beta'},
                        \frac{\alpha\gamma'}{\alpha'+\beta'},
                 \alpha',\beta',\gamma'\right)
\label{eq.bmu.VI.old}
\ee
with $\alpha' + \beta' \neq 0$, $\alpha'+\beta'+\gamma' \neq 0$, and 
$\alpha\neq 0$
(where this last inequation comes from $\alpha+\gamma\neq 0$).
In this case, 
\be
c_3 \;=\; \bigl(2\alpha' + 2\beta' + \gamma'\bigr)\,  
          \left(\frac{\alpha}{\alpha'+\beta'} + x \right) \,.  
\ee
We~find in fact that all coefficients $c_i$ for $1 \le i \le 10$ are 
polynomials in $x$:
\be
c_{2k-1} \;=\; \bigl[\gamma' + k (\alpha'+\beta')\bigr]\, 
               \left(\frac{\alpha}{\alpha' + \beta'} + x\right)   \quad
c_{2k}   \;=\; k\, \bigl[\alpha + (\alpha' + \beta')x \bigr] \,.
\label{eq.ak.VI.old}
\ee 
The coefficients $c_i$ \eqref{eq.ak.VI.old} are \emph{not} polynomials 
in the parameters $\bmu$, but if we perform the change of parameters 
$\alpha \mapsto \kappa (\alpha'+\beta')$, then \eqref{eq.bmu.VI.old}
reduces to
\be 
\bmu \;=\; \bigl(\kappa (\alpha'+\beta'),\kappa \beta',\kappa \gamma',
                 \alpha',\beta',\gamma'\bigr)
\label{eq.bmu.VI}
\ee
with $\alpha' + \beta' \neq 0$ and $\alpha'+\beta'+\gamma' \neq 0$, and 
$\kappa \neq 0$.
The coefficients \eqref{eq.ak.VI.old} then become
\be
c_{2k-1} \;=\; \bigl[\gamma' + k(\alpha'+\beta') \bigr] \, (\kappa + x)
   \,, \quad
c_{2k}   \;=\; k (\alpha'+\beta') (\kappa + x)
   \;,
\label{eq.ak.VI}
\ee 
which are polynomials jointly in $x,\alpha',\beta',\gamma',\kappa$.
So this is a viable candidate, corresponding to Family~$6$ in the main text,
and this branch is terminated.

\qquad
{\bf Remark.} Here we could impose the additional condition
$\beta' \neq 0$ to avoid overlap with Case~(0,\k0,\k1b, \k1b);
but we choose not to, since this is a perfectly good candidate
irrespective of whether $\beta'$ is zero or nonzero.
Moreover, in obtaining this solution we have imposed
the condition $\alpha \neq 0$;
but this condition is actually superfluous,
since this is a perfectly good candidate
irrespective of whether $\alpha$ is zero or nonzero.
\end{itemize}

%
%
\subsection{\texorpdfstring{Step 4: Coefficient $\bm{c_4}$}{Coefficient c4}}
\label{sec.c4}

We consider the six cases obtained in Section~\ref{sec.c3} that require 
additional conditions to be terminated. For each of these vertices, we 
now compute the branches stemming from them by using the next 
coefficient $c_4$. To save space, from this section on, we will not make 
further remarks about the possibility of imposing additional conditions to
avoid overlaps with other cases. 

%
%
\subsubsection{\texorpdfstring{Branches starting from Case~(0,\! 0,\! 0,\! 1a)}
                              {Branches starting from Case (0,0,0,1a)}} 

We have $c_4 = Q(x)/R(x)$ with
\begin{subeqnarray}
Q(x) &=& (3\beta\alpha' + \beta\beta' + 2\gamma\alpha')\, x \;+\;
         (\alpha'+\beta') (2\alpha'+\beta')\, x^2   \qquad \qquad \\[2mm]
R(x) &=& \beta+\gamma \;+\; (\alpha' + \beta')\, x  \\[2mm]
\rem(Q(x),R(x)) &=& -\frac{(\beta+\gamma)(\beta \alpha' - \gamma \beta')}
                          {\alpha' + \beta'}  
\end{subeqnarray}
and the conditions $\gamma\alpha'\neq 0$ and ($\beta+\gamma\neq 0$ or 
$\alpha'+\beta'\neq 0$) of the Case~(0,\k0,\k0,\k1a). The polynomial $R(x)$ 
is equal to the coefficient $c_3$ for the Case~(0,\k0,\k0,\k1a).

\bigskip

\begin{itemize}
\item
$\bm{\deg R = 0\colon}$
If $[x^1]R(x)= \alpha'+\beta'=0$ and
$[x^0]R(x)= \beta+\gamma  \neq 0$, we get 

\bigskip

%
%
\noindent
{\bf Case (0,\k0,\k0,\k1a,\k0)}: 
$\bmu =  \bigl(0,\beta,\gamma,\alpha',-\alpha',0\bigr)$ 
with $\gamma\alpha' \neq 0$ and $\beta+\gamma\neq 0$. 
In this case, $c_4 = 2\alpha' x$. 
Pursuing the computation of this family to higher order,
we find that all the coefficients $c_i$ for $1 \le i \le 10$ are 
polynomials in $x$ that are generically nonzero: 
\be
c_{2k-1} \;=\; \gamma \;+\; (k-1) \beta \,, \quad
c_{2k}   \;=\; k \alpha'\, x   \;.
\label{eq.ak.Ib}
\ee 
So this is a viable candidate, corresponding to Family~1b in the main text,
and this branch is terminated.

\bigskip

\item
$\bm{\deg R = 1\colon}$
If $[x^1]R(x)= \alpha'+\beta'\neq 0$,
then $\rem(Q(x),R(x)) = 0$ gives two cases:
$\beta+\gamma=0$, and  $\beta \alpha' - \gamma \beta'=0$. 

\bigskip

%
%
\noindent
{\bf Case (0,\k0,\k0,\k1a,\k1a)}: We have $\beta+\gamma=0$, so that  
\be 
\bmu \;=\; \bigl(0,\beta,-\beta,\alpha',\beta',-\alpha'-\beta'\bigr)
\ee
with $\alpha'+ \beta' \neq 0$ and $\beta\alpha'\neq 0$. 
In this case, $c_4 = \beta + (2\alpha'+\beta')\, x \neq 0$, as $\beta\neq 0$. 

\bigskip

%
%
\noindent
{\bf Case (0,\k0,\k0,\k1a,\k1b)}: We have 
$\beta \alpha' - \gamma \beta'$. As $\gamma\neq 0$, the solution is
$\beta' = \beta \alpha'/\gamma$, so that 
\be 
\bmu \;=\; \left(0,\beta,\gamma, \alpha',\frac{\beta\alpha'}{\gamma},
                 -\frac{(\beta+\gamma)\alpha'}{\gamma} \right) 
\ee
with $\gamma\alpha'\neq 0$ and $\beta + \gamma \neq 0$. 
In this case, 
\be
c_4 \;=\; \frac{\alpha'}{\gamma} \, (\beta+2\gamma)\, x \,.  
\ee
Going forward, we impose the inequation
$\beta+2\gamma \neq 0$ in order to ensure that $c_4\neq 0$.
\end{itemize}

%
%
\subsubsection{\texorpdfstring{Branches starting from Case~(0,\! 0,\! 0,\! 1b)}
                              {Branches starting from Case (0,0,0,1b)}}

We have that $c_4 = Q(x)/R(x)$ with
\begin{subeqnarray}
Q(x) &=& 2\alpha\alpha' (2\alpha+\gamma) \nonumber \\
     & & \quad \;+\;  \bigl(6 \alpha (\alpha')^2 + 3\alpha\alpha'\beta' + 
                   \alpha(\beta')^2 + 2\gamma(\alpha')^2 \bigr)\, x 
         \nonumber \\
     & & \quad \;+\;  \alpha' (\alpha'+\beta') (2\alpha'+\beta')\, x^2   
     \\[2mm]
R(x) &=& \alpha' \, \bigl(2\alpha+ \gamma \;+\; (\alpha' + \beta')\, x\bigr)  
     \\[2mm]
\rem(Q(x),R(x)) &=& \frac{\beta' (2\alpha +\gamma)
                          (\alpha\alpha' - \alpha \beta'+\gamma \alpha')}
                          {\alpha' + \beta'}  
\end{subeqnarray}
and the conditions $\alpha'\neq 0$,  $\alpha+\gamma\neq 0$, and 
($2\alpha+\beta\neq 0$ or $\alpha'+\eta'\neq0$) 
of the Case~(0,\k0,\k0,\k1b). The polynomial $R(x)$ is equal to the 
coefficient $c_3$ for the Case~(0,\k0,\k0,\k1b) multiplied by the 
factor $\alpha'\neq 0$.

\bigskip

\begin{itemize}
\item
$\bm{\deg R = 0\colon}$
If $[x^1]R(x)= \alpha'(\alpha'+\beta')=0$ and
$[x^0]R(x)= \alpha'(2\alpha+\gamma) \neq 0$, these conditions reduce to
$\alpha'+\beta'=0$ and $2\alpha+\gamma\neq 0$, as $\alpha'\neq 0$. We get 

\bigskip

%
%
\noindent
{\bf Case (0,\k0,\k0,\k1b,\k0)}: 
$\bmu =  \bigl(\alpha,-\alpha,\gamma,\alpha',-\alpha',0\bigr)$ 
with $\alpha' \neq 0$, $\alpha+\gamma \neq 0$, and $2\alpha+\gamma\neq 0$. 
In this case, $c_4 = 2(\alpha+\alpha' x)$. 
Pursuing the computation of this family to higher order,
we find that all the coefficients $c_i$ for $1 \le i \le 10$ are
polynomials in $x$ that are generically nonzero:
\be
c_{2k-1} \;=\; \gamma \,+\, k\alpha \,, \quad
c_{2k}   \;=\; k (\alpha + \alpha'  x)  \;.
\label{eq.ak.IIIb}
\ee 
So this is a viable candidate, corresponding to Family~3b in the main text,
and this branch is terminated.

\bigskip

\item
$\bm{\deg R = 1\colon}$
If $[x^1]R(x)= \alpha'(\alpha'+\beta')\neq 0$, then $\alpha'+\beta'\neq 0$,
as $\alpha'\neq 0$. Then $\rem(Q(x),R(x)) = 0$ gives three cases:
$2\alpha+\gamma=0$, $\alpha \alpha' - \alpha \beta'+\gamma \alpha'=0$, and
$\beta'=0$. Notice that the latter one $\beta'=0$, leads to 
$\bmu = (\alpha,0,\gamma,\alpha',0,-\alpha')$, which is the specialization 
of Family~5 [i.e., Case~(0,\k0,\k1b,\k1b)] when $\gamma'=-\alpha'$. 
Therefore, it should be discarded. The other two solutions leads to:

\bigskip

%
%
\noindent
{\bf Case (0,\k0,\k0,\k1b,\k1a)}: We have $2\alpha+\gamma=0$, 
so that  
\be 
\bmu \;=\; \left(\alpha,\frac{\alpha\beta'}{\alpha'},-2\alpha,
                 \alpha',\beta',-\alpha'-\beta'\right) 
\ee
with $\alpha'+\beta'\neq 0$ and $\alpha\alpha'\neq 0$. In this case, 
\be
c_4 \;=\; \frac{1}{\alpha'}\, (2\alpha'+\beta')\, (\alpha+\alpha'\, x)\,.  
\ee
Going forward, we impose the inequation
$2\alpha'+\beta' \neq 0$ in order to ensure that $c_4\neq 0$.

\bigskip

%
%
\noindent
{\bf Case (0,\k0,\k0,\k1b,\k1b)}: We have 
$\alpha \alpha' - \alpha \beta'+\gamma \alpha'=0$. As $\alpha'\neq 0$, 
the solution is $\gamma = -\alpha (\alpha'-\beta')/\alpha'$, so that 
\be 
\bmu \;=\; \left(\alpha,\frac{\alpha\beta'}{\alpha'},
                -\frac{\alpha (\alpha'-\beta')}{\alpha'},
                \alpha',\beta',-\alpha'-\beta' \right) 
\ee
with $\alpha'+\beta'\neq 0$ and $\alpha\alpha'\beta'\neq 0$. 
In this case, $c_4 = 2\alpha + (2\alpha' + \beta')x \neq 0$, as $\alpha\neq 0$. 
\end{itemize}

%
%
\subsubsection{\texorpdfstring{Branches starting from Case~(0,\! 0,\! 1a,\! 0)}
                              {Branches starting from Case (0,0,1a,0)}} 

We have $c_4 = Q(x)/R(x)$ with
\begin{subeqnarray}
Q(x) &=& \alpha(2\alpha+\beta) 
     \;+\;  \bigl(3 \alpha \alpha' + 2\beta\alpha' + 
                   2\alpha\gamma' + 2\beta\gamma' \bigr)\, x \\[2mm]
R(x) &=& \alpha \;+\; (\alpha' + \gamma')\, x  
     \\[2mm]
\rem(Q(x),R(x)) &=& -\frac{\alpha 
                          (\alpha\alpha' + \beta\alpha'+\beta\gamma')}
                          {\alpha' + \gamma'}  
\end{subeqnarray}
and the conditions $\gamma'\neq 0$, $\alpha+\beta\neq 0$, and 
($\alpha\neq 0$ or $\alpha'+\gamma'\neq 0$) of the Case~(0,\k0,\k1a,\k0).
The polynomial $R(x)$ is equal to the coefficient $c_3$ for the
Case~(0,\k0,\k1a,\k0).

\bigskip

\begin{itemize}
\item
$\bm{\deg R = 0\colon}$
If $[x^1]R(x)= \alpha'+\gamma'=0$ and $[x^0]R(x)= \alpha \neq 0$, 
we get

\bigskip

%
%
\noindent
{\bf Case (0,\k0,\k1a,\k0,\k0)}: 
$\bmu =  \bigl(\alpha,\beta,-\alpha,\alpha',-\alpha',-\alpha'\bigr)$
with $\alpha+\beta \neq 0$ and $\alpha\alpha'\neq 0$. 
In this case, $c_4 = (2\alpha+ \beta) + \alpha' x \neq 0$, as $\alpha'\neq 0$. 

\bigskip

\item
$\bm{\deg R = 1\colon}$
If $[x^1]R(x)= \alpha'+\gamma'\neq 0$,
then $\rem(Q(x),R(x)) = 0$ gives two cases:
$\alpha=0$ and $\alpha\alpha' + \beta\alpha'+\beta\gamma'=0$. 

\bigskip

%
%
\noindent
{\bf Case (0,\k0,\k1a,\k0,\k1a)}: We have $\alpha=0$, so that  
\be
\bmu \;=\; \bigl(0,\beta,0, \alpha',-\alpha',\gamma'\bigr)  
\ee
with $\alpha'+\gamma'\neq 0$ and $\beta\gamma'\neq 0$. 
In this case, $c_4 = 2\beta$.  
Pursuing the computation of this family to higher order,
we find that all the coefficients $c_i$ for $1 \le i \le 10$ are
polynomials in $x$ that are generically nonzero:
\be
c_{2k-1} \;=\; [\gamma' \,+\, (k-1)\alpha']\, x\,, \quad
c_{2k}   \;=\; k \beta   \;.
\label{eq.ak.Ia}
\ee
So this is a viable candidate, corresponding to Family~1a in the main text,
and this branch is terminated.

\bigskip

%
%
\noindent
{\bf Case (0,\k0,\k1a,\k0,\k1b)}: We have 
$\alpha\alpha' + \beta\alpha'+\beta\gamma'=0$. As $\alpha'+\gamma'\neq 0$,
the solution is given by 
\be
\beta \;=\; - \frac{\alpha\alpha'}{\alpha'+\gamma'} \,,
\ee 
so that  
\be 
\bmu \;=\; \left(\alpha,-\frac{\alpha\alpha'}{\alpha'+\gamma'},-\alpha,
                 \alpha',-\alpha',\gamma'\right) 
\ee
with $\alpha'+\gamma'\neq 0$ and $\alpha\gamma'\neq 0$. In this case, 
\be
c_4 \;=\; \frac{\alpha}{\alpha'+\gamma'}\, \bigl(\alpha'+2\gamma'\bigr) \,.  
\ee
Going forward, we impose the inequation
$\alpha'+2\gamma' \neq 0$ in order to ensure that $c_4\neq 0$.
\end{itemize}

%
%
\subsubsection{\texorpdfstring{%
  Branches starting from Case~(0,\! 0,\! 1a,\! 1b)}
 {Branches starting from Case (0,b,1a,1b)}} 

We find that $c_4 = Q(x)/R(x)$ with
\begin{subeqnarray}
Q(x) &=& \alpha^2(2\alpha'+\beta') \nonumber \\
     & &  \;+\;  \alpha \bigl(6 (\alpha')^2 + 9\alpha'\beta' + 4(\beta')^2 +  
                   2\alpha'\gamma' + 2\beta'\gamma' \bigr)\, x \nonumber \\
     & &  \;+\;  2\alpha'  \bigl(\alpha' + \beta'\bigr)\, 
                           \bigl(2\alpha' + 2\beta' + \gamma'\bigr) \, x^2
    \\[2mm]
R(x) &=& \alpha'\, \bigl(\alpha \;+\; (2\alpha' + 2\beta'+\gamma')\, x 
                   \bigr) \\[2mm]
\rem(Q(x),R(x)) &=& -\frac{\alpha^2 \beta' (\alpha' + 2\beta'+\gamma')} 
                          {2\alpha' + 2\beta' + \gamma'}  
\end{subeqnarray}
and the conditions $\alpha'\neq 0$, $\alpha'+\beta'\neq 0$,  
$\alpha'+\beta'+\gamma'\neq 0$, and ($\alpha\neq 0$ or 
$2\alpha'+2\beta'\neq 0$) of the Case~(0,\k0,\k1a,\k1b).
The polynomial $R(x)$ is equal to the coefficient $c_3$ for the
Case~(0,\k0,\k1a,\k1b) multiplied by $\alpha'\neq 0$.

\bigskip

\begin{itemize}
\item
$\bm{\deg R = 0\colon}$
If $[x^1]R(x)= \alpha'(2\alpha'+2\beta' + \gamma')=0$ and 
$[x^0]R(x)= \alpha\alpha' \neq 0$, these conditions reduce to 
$2\alpha'+2\beta' + \gamma'=0$ and $\alpha\neq 0$, as $\alpha'\neq 0$. We get 

\bigskip

%
%
\noindent
{\bf Case (0,\k0,\k1a,\k1b,\k0)}: we have the family
\be 
\bmu \;=\;  \left(\alpha,\frac{\alpha\beta'}{\alpha'},-\alpha,
                  \alpha',\beta',-2\alpha'-2\beta' \right)
\ee
with $\alpha'+\beta' \neq 0$ and $\alpha\alpha'\neq 0$. In this case, 
\be
c_4 \;=\; \frac{1}{\alpha'}\, (2\alpha'+ \beta') \, (\alpha + \alpha' \, x)\,.
\ee
Going forward, we impose the inequation
$2\alpha'+\beta' \neq 0$ in order to ensure that $c_4\neq 0$.

\bigskip

\item
$\bm{\deg R = 1\colon}$
If $[x^1]R(x)= \alpha'(2\alpha'+2\beta' +\gamma')\neq 0$, this condition 
reduces to $2\alpha'+2\beta' +\gamma'\neq 0$, as $\alpha'\neq 0$.
Then  $\rem(Q(x),R(x)) = 0$ gives three cases:
$\alpha=0$, $\beta'=0$, and $\alpha' + 2\beta'+\gamma'=0$. Notice that the
first solution implies $\bmu = (0,0,0,\alpha',\beta',\gamma')$,
which is the specialization of Family~2a [i.e., Case~(0,\k0,\k1a,\k1a)]
when $\alpha=0$. Moreover, the second one leads to 
$\bmu = (\alpha,0,-\alpha,\alpha',0,\gamma')$, which is the specialization
of Family~5 [i.e., Case~(0,\k0,\k1b,\k1b)] when $\gamma=-\alpha$.
Therefore, we have only one case left:

\bigskip

%
%
\noindent
{\bf Case (0,\k0,\k1a,\k1b,\k1a)}: We have 
$\alpha' + 2\beta'+\gamma'=0$: 
\be 
\bmu \;=\; \left(\alpha,\frac{\alpha\beta'}{\alpha'},-\alpha,
                 \alpha',-\beta',-\alpha'-2\beta'\right) 
\ee
with $\alpha'+\beta'\neq 0$ and $\alpha'\beta'\neq 0$. In this case, 
\be
c_4 \;=\; \frac{\alpha (2\alpha' + \beta')}{\alpha'} \;+\; 
          2 \bigl(\alpha'+\beta'\bigr)\, x \;\neq\; 0 \,,  
\ee
as $\alpha'+\beta'\neq 0$.
\end{itemize}

%
%
\subsubsection{\texorpdfstring{%
  Branches starting from Case~(0,\! 0,\! 1b,\! 0)}
 {Branches starting from Case (0,0,1b,0)}} 

We find that $c_4 = Q(x)/R(x)$ with
\begin{subeqnarray}
Q(x) &=& 2\alpha(2\alpha+\gamma)\gamma' \nonumber \\[2mm]
     & & \qquad \;+\; 
     \bigl(3 \alpha\alpha'\gamma' + 2\alpha(\gamma')^2 - 
                   \alpha(\alpha')^2 - \gamma(\alpha')^2 \bigr)\, x 
     \qquad \qquad \\[2mm]
R(x) &=& \gamma' \bigl( 2\alpha+ \gamma \;+\; (\alpha' + \gamma')\, x \bigr) 
     \\[2mm]
\rem(Q(x),R(x)) &=& \frac{\alpha' (2\alpha +\gamma)
                          (\alpha\alpha' + \gamma\alpha'-\alpha\gamma')}
                          {\alpha' + \gamma'}  
\end{subeqnarray}
and the conditions $\alpha+\gamma\neq 0$, $\alpha\gamma'\neq 0$, and 
($2\alpha+\gamma\neq 0$ or $\alpha'+\gamma'\neq 0$) of the 
Case~(0,\k0,\k1b,\k0). The polynomial $R(x)$ is equal to the coefficient
$c_3$ for the Case~(0,\k0,\k1b,\k0) multiplied by $\gamma'\neq 0$.

\bigskip

\begin{itemize}
\item
$\bm{\deg R = 0\colon}$
If $[x^1]R(x)= \gamma'(\alpha'+\gamma')=0$ and
$[x^0]R(x)= \gamma'(2\alpha+ \gamma) \neq 0$, these conditions reduce to
$\alpha'+\gamma'=0$ and $2\alpha+ \gamma \neq 0$, as $\gamma'\neq 0$. 
We get

\bigskip

%
%
\noindent
{\bf Case (0,\k0,\k1b,\k0,\k0)}: 
$\bmu =  \bigl(\alpha,\alpha+\gamma,\gamma,\alpha',-\alpha',-\alpha'\bigr)$
with $\alpha+\gamma \neq 0$, $2\alpha+\gamma\neq 0$, and $\alpha\alpha'\neq 0$.
In this case, $c_4 = 2\alpha +\alpha'\, x \neq 0$, as $\alpha\alpha'\neq 0$. 

\bigskip

\item
$\bm{\deg R = 1\colon}$
If $[x^1]R(x)= \gamma'(\alpha'+\gamma')\neq 0$, this condition
reduces to $\alpha'+\gamma'\neq 0$, as $\gamma'\neq 0$.
Then  $\rem(Q(x),R(x)) = 0$ gives three cases:
$\alpha'=0$, $2\alpha+\gamma=0$, and 
$\alpha\alpha' + \gamma\alpha' - \alpha\gamma'=0$. Notice that the
first solution implies $\bmu = (\alpha,0,\gamma,0,0,\gamma')$,
which is the specialization of Family~5 [i.e., Case~(0,\k0,\k1b,\k1b)]
when $\alpha'=0$. The other two cases give: 

\bigskip

%
%
\noindent
{\bf Case (0,\k0,\k1b,\k0,\k1a)}: We have $2\alpha+\gamma=0$, 
so that  
\be 
\bmu \;=\; \left(\alpha,\frac{\alpha\alpha'}{\gamma'},-2\alpha,
                 \alpha',-\alpha',\gamma'\right) 
\ee
with $\alpha'+\gamma'\neq 0$ and $\alpha\gamma'\neq 0$. In this case, 
\be
c_4 \;=\; \frac{\alpha}{\gamma'}\, (\alpha'+2\gamma')\,.  
\ee
Going forward, we impose the inequation
$\alpha'+2\gamma' \neq 0$ in order to ensure that $c_4\neq 0$.

\bigskip

%
%
\noindent
{\bf Case (0,\k0,\k1b,\k0,\k1b)}: We have 
$\alpha\alpha' + \gamma\alpha'-\alpha\gamma'=0$. As $\alpha+\gamma\neq 0$, 
the solution is
\be
\alpha' \;=\; \frac{\alpha\gamma'}{\alpha+\gamma}\,,
\ee
so that 
\be 
\bmu \;=\; \left(\alpha,-\alpha,\gamma,
                 \frac{\alpha\gamma'}{\alpha+\gamma},  
                -\frac{\alpha\gamma'}{\alpha+\gamma}, \gamma' \right) 
\label{eq.bmu.IVb.old}
\ee
with $\alpha+\gamma\neq 0$, $2\alpha+\gamma\neq 0$, and $\alpha\gamma'\ne 0$.
In this case, $c_4 = 2\alpha$. 
Pursuing the computation of this family to higher order,
we find that all the coefficients $c_i$ for $1 \le i \le 10$ are
polynomials in $x$ that are generically nonzero:
\be
c_{2k-1} \;=\; (\gamma \,+\, k\alpha)\, \left( 1 + 
                \frac{\gamma'}{\alpha + \gamma}\, x \right) \,, \quad
c_{2k}   \;=\; k \alpha   \;.
\label{eq.ak.IVb.old}
\ee
The coefficients $c_i$ \eqref{eq.ak.IVb.old} are \emph{not} polynomials
in $\alpha$ or $\gamma$, but if we perform the change of parameters
$\gamma' \mapsto \kappa (\alpha+\gamma)$, then $\bmu$ \eqref{eq.bmu.IVb.old}
reduces to
\be
\bmu \;=\; \bigl(\alpha,-\alpha,\gamma, \kappa\alpha,-\kappa\alpha,
                 \kappa (\alpha+\gamma) \bigr)
\label{eq.bmu.IVb}
\ee
with $\alpha+\gamma \neq 0$, $2\alpha+\gamma\neq 0$, and $\alpha\kappa\neq 0$. 
The coefficients $c_i$ \eqref{eq.ak.IVb.old} transform to
\be
c_{2k-1} \;=\; (\gamma \,+\, k\alpha)\, (1+\kappa x)\,, \quad
c_{2k}   \;=\; k \alpha   \;.
\label{eq.ak.IVb}
\ee
where these coefficients are polynomials jointly in $x$, $\alpha$,
$\gamma$, and $\kappa$.
So this is a viable candidate, corresponding to Family~4b in the main text,
and this branch is terminated.
\end{itemize}

%
%
\subsubsection{\texorpdfstring{%
  Branches starting from Case~(0,\! 0,\! 1b,\! 1a)}
 {Branches starting from Case (0,0,1b,1a)}} 

We have $c_4 = Q(x)/R(x)$ with
\begin{subeqnarray}
Q(x) &=& 
     \gamma \bigl(2 (\alpha')^2 + 7\alpha'\beta' + 4 (\beta')^2 +  
                  2 \alpha'\gamma' + 2\beta'\gamma' \bigr)\, x \nonumber \\
     & & \quad \;+\; 2(\alpha'+\beta')(\alpha'+\beta'+\gamma')
                      (2\alpha'+2\beta'+\gamma') \, x^2  \\[2mm]
R(x) &=& \gamma (\alpha'+ 2\beta'+\gamma') \;+\; 
         (\alpha'+\beta'+\gamma') (2\alpha' + 2\beta'+\gamma')\, x  
     \qquad \qquad \\[2mm]
\rem(Q(x),R(x)) &=& -\frac{\alpha' \beta' \gamma^2 (\alpha' + 2\beta'+\gamma')}
                   {(\alpha'+\beta'+\gamma')(2\alpha' + 2\beta'+\gamma')}  
\end{subeqnarray}
and the conditions $\alpha'+\beta'+\gamma'\neq 0$, $\alpha'+\beta'\neq 0$, 
and $\gamma\neq 0$ of the Case~(0,\k0,\k1b,\k1a). The polynomial $R(x)$ is
the coefficient $c_3$ of the Case~(0,\k0,\k1b,\k1a) multiplied by the 
factor ${\alpha'+ 2\beta'+\gamma'\neq 0}$. 

\bigskip

\begin{itemize}
\item
$\bm{\deg R = 0\colon}$
If $[x^1]R(x)= (\alpha'+\beta'+\gamma')(2\alpha' + 2\beta'+\gamma')=0$ and
$[x^0]R(x)= {\gamma(\alpha'+ 2\beta'+\gamma')} \neq 0$, these conditions 
reduce to $2\alpha' + 2\beta'+\gamma'=0$ and $\alpha'+ 2\beta'+\gamma' \neq 0$,
as $\gamma\neq 0$ and $\alpha'+\beta'+\gamma'\neq 0$. We get

\bigskip

%
%
\noindent
{\bf Case (0,\k0,\k1b,\k1a,\k0)}: we have the family
\be 
\bmu \;=\;  \left(0,-\frac{\gamma\beta'}{\alpha'+\beta'},\gamma,
                 \alpha',\beta',-2\alpha'-2\beta' \right)
\ee
with $\alpha'+\beta' \neq 0$, $\alpha'+ 2\beta'+\gamma' \neq 0$,
and $\gamma\alpha'\neq 0$. In this case, $c_4 = (2\alpha' +\beta')\, x$. 
Going forward, we impose the inequation
$2\alpha'+\beta' \neq 0$ in order to ensure that $c_4\neq 0$.

\bigskip

\item
$\bm{\deg R = 1\colon}$
If $[x^1]R(x)= (\alpha'+\beta'+\gamma')(2\alpha' + 2\beta'+\gamma')\neq 0$,
it reduces to $2\alpha' + 2\beta'+\gamma'\neq 0$, as
$\alpha'+\beta'+\gamma'\neq 0$.
Then  $\rem(Q(x),R(x)) = 0$ gives three cases:
$\alpha'=0$, $\beta'=0$, and $\alpha' + 2\beta'+\gamma'=0$. Notice that the
first solution implies $\bmu = (0,0,\gamma,\alpha',0,\gamma')$,
which is the specialization of Family~5 [i.e., Case~(0,\k0,\k1b,\k1b)]
when $\alpha=0$. The other two cases give:

\bigskip

%
%
\noindent
{\bf Case (0,\k0,\k1b,\k1a,\k1a)}: We have $\alpha'=0$, 
so that  
\be 
\bmu \;=\; \left(0,\frac{\gamma\beta'}{\beta'+\gamma'},\gamma,
                 0,\beta',\gamma'\right) 
\label{eq.bmu.IVa.old}
\ee
with $\beta'+\gamma'\neq 0$, $2\beta'+\gamma'\neq 0$, and $\gamma\beta'\neq 0$.
In this case, $c_4=2\beta'\, x$.  
We~find in fact that all coefficients $c_i$ for $1 \le i \le 10$ are 
polynomials in $x$:
\be
c_{2k-1} \;=\; (\beta' k + \gamma') \, \left( 
      \frac{\gamma}{\beta'+\gamma'} \,+\, x\right)\,, \quad
c_{2k}   \;=\; k \beta'\, x   \;.
\label{eq.ak.IVa.old}
\ee
The coefficients $c_i$ \eqref{eq.ak.IVa.old} are \emph{not} polynomials
in $\beta'$ or $\gamma'$, but if we perform the change of parameters
$\gamma \mapsto \kappa (\beta'+\gamma')$, then $\bmu$ \eqref{eq.bmu.IVa.old}
reduces to
\be
\bmu \;=\; \bigl(0,\kappa\beta',\kappa(\beta'+\gamma'), 0, \beta', \gamma' 
                 \bigr)
\label{eq.bmu.IVa}
\ee
with $\beta'+\gamma' \neq 0$, $2\beta'+\gamma' \neq 0$,
and $\kappa\beta' \neq 0$. The coefficients $c_i$ \eqref{eq.ak.IVa.old} 
transform to
\be
c_{2k-1} \;=\; (\beta' k + \gamma')\, (\kappa \;+\; x) \,, \quad 
c_{2k}   \;=\; k \ \beta'\, x  \;.
\label{eq.ak.IVa}
\ee
where these coefficients are polynomials jointly in $x$, $\beta'$,
$\gamma'$, and $\kappa$.
So this is a viable candidate, corresponding to Family~4a in the main text,
and this branch is terminated.

\bigskip

%
%
\noindent
{\bf Case (0,\k0,\k1b,\k1a,\k1b)}: We have 
$\alpha' + 2\beta'+\gamma'=0$, so that 
\be 
\bmu \;=\; \bigl(0,\beta,-\beta,\alpha',\beta',-\alpha'-2 \beta' \bigr)  
\ee
with $\alpha'+\beta'\neq 0$ and $\beta\alpha'\beta'\neq 0$. 
In this case, $c_4 = \beta + 2(\alpha'+\beta')\, x \neq 0$, 
as $\beta(\alpha'+\beta')\neq 0$. 

\end{itemize}

%
%
\subsection{\texorpdfstring{Step 5: Coefficient $\bm{c_5}$}{Coefficient c5}}
\label{sec.c5}

We consider the 12 cases obtained in Section~\ref{sec.c4} that require 
additional conditions to be terminated. For each of these vertices, we 
now compute the branches stemming from them by using the next 
coefficient $c_5$. 

%
%
\subsubsection{%
\texorpdfstring{Branches starting from Case~(0,\! 0,\! 0,\! 1a,\! 1a)}
               {Branches starting from Case (0,0,0,1a,1a)}} 

We have $c_5 = Q(x)/R(x)$ with
\begin{subeqnarray}
Q(x) &=& 2\beta(2\alpha'+\beta')\, x \;+\;
         2(\alpha'+\beta') (2\alpha'+\beta')\, x^2   \qquad \qquad \\[2mm]
R(x) &=& \beta \;+\; (2\alpha' + \beta')\, x  \\[2mm]
\rem(Q(x),R(x)) &=& -\frac{2\beta^2\alpha'}{2\alpha' + \beta'}  
\end{subeqnarray}
and the conditions $\alpha'+\beta'\neq 0$ and $\beta\alpha'\neq 0$ 
of the Case~(0,\k0,\k0,\k1a,\k1a). The polynomial $R(x)$ is equal to the
coefficient $c_4$ for the Case~(0,\k0,\k0,\k1a,\k1a).

\begin{itemize}
\item
$\bm{\deg R = 0\colon}$
If $[x^1]R(x)= 2\alpha'+\beta'=0$ and $[x^0]R(x)= \beta \neq 0$, we get 
%
%
$\bmu =  \left(0,\beta,-\beta, \alpha',-2\alpha',\alpha' \right)$ 
with $\beta\alpha'\neq 0$. In this case, $c_5 = 0$, so this solution should
not be taken into account.

\item 
$\bm{\deg R = 1\colon}$
If $[x^1]R(x)= 2\alpha'+\beta'\neq 0$, then $\rem(Q(x),R(x))\neq 0$ 
as $\beta\alpha'\neq 0$. Therefore, there is no solution of this type.
\end{itemize}

This branch does not have any viable solution. Therefore, 
Case~(0,\k0,\k0,\k1a,\k1a) is a leaf of the decision tree.  
                          
%
%
\subsubsection{%
\texorpdfstring{Branches starting from Case~(0,\! 0,\! 0,\! 1a,\! 1b)}
               {Branches starting from Case (0,0,0,1a,1b)}} 

We find that  
\be
c_5 \;=\; 2\beta + \gamma \;+\; \frac{2(\beta+\gamma)\alpha'}{\gamma}\, x 
\ee
is a nonzero polynomial in $x$, due to the inequations coming from 
Case~(0,\k0,\k0,\k1a,\k1b): $\beta+\gamma\neq 0$ and 
$\gamma\alpha'\neq 0$. Then $R(x)=1$, and we have only one case: 

\begin{itemize}
\item
$\bm{\deg R = 0\colon}$
%
%
\noindent
{\bf Case (0,\k0,\k0,\k1a,\k1b,\k0)}: we have 
\be
\bmu \;=\; \left(0,\beta,\gamma, \alpha',\frac{\beta\alpha'}{\gamma},
                 -\frac{(\beta+\gamma)\alpha'}{\gamma} \right)
\ee
with $\gamma\alpha'\neq 0$ and $\beta + \gamma \neq 0$.
\end{itemize}

%
%
\subsubsection{%
\texorpdfstring{Branches starting from Case~(0,\! 0,\! 0,\! 1b,\! 1a)}
               {Branches starting from Case (0,0,0,1b,1a)}} 

We find that $c_5 = \alpha + 2 (\alpha'+ \beta')\, x$  
is a nonzero polynomial in $x$, due to the inequations $\alpha'+\beta'\neq 0$, 
$2\alpha'+\beta'\neq 0$, and $\alpha\alpha'\neq 0$ of the 
Case~(0,\k0,\k0,\k1b,\k1a). Then $R(x)=1$, and we have only one case: 

\begin{itemize}
\item
$\bm{\deg R = 0\colon}$
%
%
\noindent
{\bf Case (0,\k0,\k0,\k1b,\k1a,\k0)}: we have 
\be
\bmu \;=\; \left(\alpha,\frac{\alpha\beta'}{\alpha'},-2\alpha,
                 \alpha',\beta',-\alpha'-\beta'\right)
\ee
with $\alpha'+\beta'\neq 0$, $2\alpha'+\beta'\neq 0$, and 
$\alpha\alpha'\neq 0$. 
\end{itemize}

%
%
\subsubsection{%
\texorpdfstring{Branches starting from Case~(0,\! 0,\! 0,\! 1b,\! 1b)}
               {Branches starting from Case (0,0,0,1b,1b)}} 

We have $c_5 = Q(x)/R(x)$ with
\begin{subeqnarray}
Q(x) &=& 2\alpha^2(2\alpha'+\beta') \;+\; 2\alpha(2\alpha'+\beta')^2\, x 
         \nonumber \\
     & & \quad \;+\;
         2\alpha'(\alpha'+\beta') (2\alpha'+\beta')\, x^2   \\[2mm]
R(x) &=& \alpha'\, \bigl(2\alpha \;+\; (2\alpha' + \beta')\, x\bigr)  \\[2mm]
\rem(Q(x),R(x)) &=& -\frac{2\alpha^2(\beta')^2}{2\alpha' + \beta'}  
\end{subeqnarray}
and the conditions $\alpha'+\beta'\neq 0$ and $\alpha\alpha'\beta'\neq 0$ 
of the Case~(0,\k0,\k0,\k1b,\k1b). The polynomial $R(x)$ is equal to the 
coefficient $c_4$ of the Case~(0,\k0,\k0,\k1b,\k1b) multiplied by 
$\alpha'\neq 0$.

\begin{itemize}
\item
$\bm{\deg R = 0\colon}$
If $[x^1]R(x)= \alpha'(2\alpha' + \beta') =0$ and 
$[x^0]R(x) = 2\alpha\alpha'\neq 0$, they reduce to $2\alpha' + \beta'=0$, 
as $\alpha\alpha'\neq 0$.  Then there is a single solution when
$2\alpha' + \beta'=0$:  
%
%
$\bmu = \bigl(\alpha,-2\alpha,-3\alpha, \alpha',-2\alpha',\alpha' \bigr)$ 
with $\alpha\alpha'\neq 0$. In this case, $c_5 = 0$, so this solution should
not be taken into account.

\item
$\bm{\deg R = 1\colon}$
If $[x^1]R(x)= 2\alpha'+\beta'\neq 0$, then $\rem(Q(x),R(x))\neq 0$  
as $\alpha\beta'\neq 0$. So there is no solution of this type.
\end{itemize}

This branch does not have any viable solution. Therefore,
Case~(0,\k0,\k0,\k1b,\k1b) is a leaf of the decision tree.

%
%
\subsubsection{%
\texorpdfstring{Branches starting from Case~(0,\! 0,\! 1a,\! 0,\! 0)}
               {Branches starting from Case (0,0,1a,0,0)}} 

We have $c_5 = Q(x)/R(x)$ with
\begin{subeqnarray}
Q(x) &=& 2\alpha(2\alpha+\beta) \;+\; 2\alpha'(2\alpha+\beta)^2\, x   \\[2mm]
R(x) &=& 2\alpha+\beta \;+\; \alpha'\, x  \\[2mm]
\rem(Q(x),R(x)) &=& -2(\alpha+\beta)(2\alpha + \beta)  
\end{subeqnarray}
and the conditions $\alpha+\beta\neq 0$ and $\alpha\alpha'\neq 0$ 
of the Case~(0,\k0,\k1a,\k0,\k0). The polynomial $R(x)$ is the coefficient 
$c_4$ of the Case~(0,\k0,\k1a,\k0,\k0). 

\begin{itemize}
\item
$\bm{\deg R = 0\colon}$
As $[x^1]R(x)= \alpha \neq 0$, there is no solution of this type. 

\item
$\bm{\deg R = 1\colon}$
If $[x^1]R(x)= \alpha \neq 0$, then $\rem(Q(x),R(x))=0$ gives the case 
$2\alpha + \beta=0$:   
%
%
$\bmu =  \bigl(\alpha,-2\alpha,-\alpha, \alpha',-\alpha',-\alpha' \bigr)$
with $\alpha\alpha'\neq 0$. In this case, $c_5 = 0$, so this solution should
not be taken into account.
\end{itemize}

This branch does not have any viable solution. Therefore,
Case~(0,\k0,\k1a,\k0,\k0) is a leaf of the decision tree.

%
%
\subsubsection{%
\texorpdfstring{Branches starting from Case~(0,\! 0,\! 1a,\! 0,\! 1b)}
               {Branches starting from Case (0,0,1a,0,1b)}} 

We find that $c_5 = 2\alpha + (2 \alpha'+ \gamma')\, x$ 
is a nonzero polynomial in $x$, due to the inequations $\alpha\gamma'\neq 0$,
$\alpha'+\gamma'\neq 0$, and $\alpha'+2\gamma'\neq 0$ of the 
Case~(0,\k0,\k1a,\k0,\k1b). Then $R(x) = 1$, and we have only one case: 

\begin{itemize}
\item
$\bm{\deg R = 0\colon}$
%
%
\noindent
{\bf Case (0,\k0,\k1a,\k0,\k1b,\k0)}: we have
\be
\bmu \;=\; \left(\alpha,-\frac{\alpha\alpha'}{\alpha'+\gamma'},-\alpha,
                 \alpha',-\alpha',\gamma'\right)
\ee
with $\alpha\gamma'\neq 0$, $\alpha'+\gamma'\neq 0$, and 
$\alpha'+2\gamma'\neq 0$. 
\end{itemize}

%
%
\subsubsection{%
\texorpdfstring{Branches starting from Case~(0,\! 0,\! 1a,\! 1b,\! 0)}
               {Branches starting from Case (0,0,1a,1b,0)}} 

We find that $c_5 = 2\alpha + (\alpha'+ \beta')\, x$
is a nonzero polynomial in $x$, due to the inequations $\alpha\alpha'\neq 0$,
$\alpha'+\beta'\neq 0$, and $2\alpha'+\beta'\neq 0$ of the 
Case~(0,\k0,\k1a,\k1b,\k0). Then $R(x) = 1$, and we have only one case: 

\begin{itemize}
\item
$\bm{\deg R = 0\colon}$
%
%
\noindent
{\bf Case (0,\k0,\k1a,\k1b,\k0,\k0)}: we have
\be
\bmu \;=\;  \left(\alpha,\frac{\alpha\beta'}{\alpha'},-\alpha,
                  \alpha',\beta',-2\alpha'-2\beta' \right)
\ee
with $\alpha'+\beta' \neq 0$, $2\alpha'+\beta' \neq 0$, and 
$\alpha\alpha'\neq 0$. 
\end{itemize}

%
%
\subsubsection{%
\texorpdfstring{Branches starting from Case~(0,\! 0,\! 1a,\! 1b,\! 1a)}
               {Branches starting from Case (0,0,1a,1b,1a)}} 

We have $c_5 = Q(x)/R(x)$ with
\begin{subeqnarray}
Q(x) &=& 2\alpha^2(2\alpha'+\beta') \;+\; 2\alpha(2\alpha'+\beta')^2\, x 
         \nonumber \\
     & & \quad \;+\; 2\alpha' (\alpha'+\beta') (2\alpha'+\beta')\, x^2  
         \\[2mm]
R(x) &=& \alpha(2\alpha'+\beta') \;+\; 2\alpha'(\alpha'+\beta')\, x \\[2mm]
\rem(Q(x),R(x)) &=& -\frac{\alpha^2 (\beta')^2(2\alpha'+\beta')}
                          {2\alpha' (\alpha'+\beta')}  
\end{subeqnarray}
and the conditions $\alpha'+\beta'\neq 0$ and $\alpha'\beta'\neq 0$ 
of the Case~(0,\k0,\k1a,\k1b,\k1a). The polynomial $R(x)$ is equal to the
coefficient $c_4$ of the Case~(0,\k0,\k1a,\k1b,\k1a) multiplied by 
$\alpha'\neq 0$.

\begin{itemize}
\item
$\bm{\deg R = 0\colon}$ As $[x^1]R(x)= 2\alpha'(\alpha'+\beta') \neq 0$, 
there is no solution of this type.

\item
$\bm{\deg R = 0\colon}$ 
If $[x^1]R(x)= 2\alpha'(\alpha'+\beta') \neq 0$, then $\rem(Q(x),R(x))=0$  
gives two cases $\alpha=0$, and $2\alpha'+\beta'=0$. The first one gives
$\bmu=(0,0,0,\alpha',\beta',-\alpha'-2\beta')$. But this is the 
specialization of Family~2a [i.e., Case~(0,\k0,\k1a,\k1a)] when $\alpha=0$ 
and $\gamma'=-\alpha'-2\beta'$. Therefore, it should be discarded. The latter 
case gives:
%
%
$\bmu =  \left(\alpha,-2\alpha,-\alpha, \alpha',-2\alpha',-3\alpha' \right)$
with $\alpha'\beta'\neq 0$. In this case, $c_5 = 0$, so this solution should
not be taken into account.
\end{itemize}

This branch does not have any viable solution. Therefore,
Case~(0,\k0,\k1a,\k1b,\k1a) is a leaf of the decision tree.

%
%
\subsubsection{%
\texorpdfstring{Branches starting from Case~(0,\! 0,\! 1b,\! 0,\! 0)}
               {Branches starting from Case (0,0,1b,0,0)}} 

We have $c_5 = Q(x)/R(x)$ with
\begin{subeqnarray}
Q(x) &=& 2\alpha(3\alpha+\gamma) \;+\; 2\alpha'(3\alpha+\gamma)\, x  \\[2mm]
R(x) &=& 2\alpha \;+\; \alpha'\, x \\[2mm]
\rem(Q(x),R(x)) &=& -2\alpha (3\alpha+\gamma) 
\end{subeqnarray}
and the conditions $\alpha+\gamma\neq 0$, $2\alpha+\gamma\neq 0$, and 
$\alpha\alpha'\neq 0$ of the Case~(0,\k0,\k1b,\k0,\k0). The polynomial 
$R(x)$ is equal to the coefficient $c_4$ of the Case~(0,\k0,\k1b,\k0,\k0).

\begin{itemize}
\item
$\bm{\deg R = 0\colon}$ As  $[x^1]R(x)= \alpha' \neq 0$, there is no 
solution of this type.

\item
If $[x^1]R(x)= \alpha' \neq 0$, then $\rem(Q(x),R(x))=0$ gives the case 
$3\alpha+\gamma=0$: 
%
%
$\bmu = (\alpha,-2\alpha,-3\alpha, \alpha',-\alpha',-\alpha' )$ 
with $\alpha\alpha'\neq 0$. In this case, $c_5 = 0$, so this solution should
not be taken into account.
\end{itemize}

This branch does not have any viable solution. Therefore,
Case~(0,\k0,\k1b,\k0,\k0) is a leaf of the decision tree.
 
%
%
\subsubsection{%
\texorpdfstring{Branches starting from Case~(0,\! 0,\! 1b,\! 0,\! 1a)}
               {Branches starting from Case (0,0,1b,0,1a)}} 

We find that $c_5 = \alpha + (2\alpha'+ \gamma')\, x$
is a non-zero polynomial in $x$ due to the inequations $\alpha\gamma'\neq 0$,
$\alpha'+\gamma'\neq 0$, and $\alpha'+2\gamma'\neq 0$ of the
Case~(0,\k0,\k1b,\k0,\k1a). Then $R(x) = 1$, and we have only one case:

\begin{itemize}
\item
$\bm{\deg R = 0\colon}$
%
%
\noindent
{\bf Case (0,\k0,\k1b,\k0,\k1a,\k0)}: we have
\be
\bmu \;=\; \left(\alpha,\frac{\alpha\alpha'}{\gamma'},-2\alpha,
                 \alpha',-\alpha',\gamma'\right)
\ee
with $\alpha\gamma'\neq 0$, $\alpha'+\gamma'\neq 0$, and 
$\alpha'+2\gamma'\neq 0$ 
\end{itemize}

%
%
\subsubsection{%
\texorpdfstring{Branches starting from Case~(0,\! 0,\! 1b,\! 1a,\! 0)}
               {Branches starting from Case (0,0,1b,1a,0)}} 

We find that 
\be
c_5 \;=\; \frac{\gamma(\alpha-\beta')}{\alpha'+\beta'} \;+\; 
               (\alpha'+ \beta')\, x
\ee
is a nonzero polynomial in $x$, due to the conditions $\gamma\alpha'\neq 0$,
$\alpha'+\beta'\neq 0$, and $2\alpha'+\beta'\neq 0$ of the 
Case~(0,\k0,\k1b,\k1a,\k0). Then $R(x) = 1$, and we have only one case:

\begin{itemize}
\item
$\bm{\deg R = 0\colon}$
%
%
\noindent
{\bf Case (0,\k0,\k1b,\k1a,\k0,\k0)}: we have
\be
\bmu \;=\;  \left(0,-\frac{\gamma\beta'}{\alpha'+\beta'},\gamma,
                 \alpha',\beta',-2\alpha'-2\beta' \right)
\ee
with $\gamma\alpha'\neq 0$, $\alpha'+\beta'\neq 0$, and 
$2\alpha'+\beta'\neq 0$.  
\end{itemize}

%
%
\subsubsection{%
\texorpdfstring{Branches starting from Case~(0,\! 0,\! 1b,\! 1a,\! 1b)}
               {Branches starting from Case (0,0,1b,1a,1b)}} 

We have $c_5 = Q(x)/R(x)$ with
\begin{subeqnarray}
Q(x) &=& 2\beta(2\alpha'+\beta') \, x
         \;+\; 2(\alpha'+ \beta') (2\alpha'+\beta')\, x^2   \\[2mm]
R(x) &=& \beta \;+\; 2(\alpha' + \beta')\, x  \\[2mm]
\rem(Q(x),R(x)) &=& -\frac{\beta^2(2\alpha'+\beta')}%
                          {2(\alpha' + \beta')}  
\end{subeqnarray}
and the conditions $\alpha'+\beta'\neq 0$ and $\beta\alpha'\beta'\neq 0$ 
of the Case~(0,\k0,\k1b,\k1a,\k1b). The polynomial $R(x)$ is equal to the
coefficient $c_4$ of the Case~(0,\k0,\k1b,\k1a,\k1b).

\begin{itemize}
\item
$\bm{\deg R = 0\colon}$ As $[x^1]R(x)= 2(\alpha' + \gamma') \neq 0$, there
are no solution of this type. 
 
\item
$\bm{\deg R = 1\colon}$
If $[x^1]R(x)= 2(\alpha' + \gamma') \neq 0$, then $\rem(Q(x),R(x))=0$
gives the case $2\alpha'+\beta'=0$: 
%
%
$\bmu = (0,\beta,-\beta, \alpha',-2\alpha',3\alpha' )$
with $\beta\alpha'\neq 0$. In this case, $c_5 = 0$, so this solution should
not be taken into account.
\end{itemize}

This branch does not have any viable solution. Therefore,
Case~(0,\k0,\k1b,\k1a,\k1b) is a leaf of the decision tree.
 
%
%
\subsection{\texorpdfstring{Step 6: Coefficient $\bm{c_6}$}{Coefficient c6}}
\label{sec.c6}

We consider the six cases obtained in Section~\ref{sec.c5} that require  
additional conditions to be terminated. For each of these vertices, we
now compute the branches stemming from them by using the next
coefficient $c_6$.

%
%
\subsubsection{%
\texorpdfstring{Branches starting from Case~(0,\! 0,\! 0,\! 1a,\! 1b,\! 0)}
               {Branches starting from Case (0,0,0,1a,1b,0)}}

We have $c_6 = Q(x)/R(x)$ with
\begin{subeqnarray}
Q(x) &=& \gamma\alpha'(4\beta^2+9\beta\gamma + 3\gamma^2)\, x \;+\;
         2(\alpha')^2 (\beta+\gamma)(2\beta+3\gamma)\, x^2   
         \qquad \qquad \\[2mm]
R(x) &=& \gamma^2(2\beta +\gamma) \;+\; 2\gamma\alpha'(\beta + \gamma)\, x  
\\[2mm]
\rem(Q(x),R(x)) &=& -\frac{\beta\gamma^3(2\beta+\gamma)}{2(\beta+\gamma)} 
\end{subeqnarray}
with the conditions $\beta+\gamma\neq 0$ and $\gamma\alpha'\neq 0$ of  
the Case~(0,\k0,\k0,\k1a,\k1b,\k0). Note that $R(x)$ is equal to the
coefficient $c_5$ of Case~(0,\k0,\k0,\k1a,\k1b,\k0) multiplied by 
$\gamma^2 \neq 0$.  

\begin{itemize}
\item
$\bm{\deg R = 0\colon}$
As $[x^1]R(x)= \gamma \alpha'(\beta+\gamma) \neq 0$, there is no solution
of this type.

\item
$\bm{\deg R = 1\colon}$
If $[x^1]R(x)= \gamma \alpha'(\beta+\gamma) \neq 0$, then  
$\rem(Q(x),R(x))=0$ gives two cases $\beta=0$, or $2\beta+\gamma=0$.
In the former one, we obtain 
$\bmu=(0,0,\gamma,\alpha',0,-\alpha')$ which is the specialization of 
Family~5 [i.e., Case~(0,\k0,\k1b,\k1b)] when $\alpha=0$ and $\gamma'=-\alpha'$.
The latter case give:

\bigskip 

%
%
\noindent
{\bf Case (0,\k0,\k0,\k1a,\k1b,\k0,\k1a)}: we have (after some trivial scale 
transformation) the family
$\bmu =  (0,\beta,-2\beta,-2\beta',\beta',\beta')$ 
with $\beta\beta'\neq 0$. 
In this case, $c_6 = \beta - 4\beta' \, x \neq 0$. 
\end{itemize}

%
%
\subsubsection{%
\texorpdfstring{Branches starting from Case~(0,\! 0,\! 0,\! 1b,\! 1a,\! 0)}
               {Branches starting from Case (0,0,0,1b,1a,0)}}

We have $c_6 = Q(x)/R(x)$ with
\begin{subeqnarray}
Q(x) &=& \alpha^2(3\alpha'+\beta') \;+\; 
         \alpha(9 (\alpha')^2 + 11\alpha'\beta' + 4 (\beta')^2)\, x \nonumber
         \\
     & & \quad \;+\;
         2(\alpha')^2 (\alpha'+\beta')(3\alpha'+2\beta')\, x^2   
     \qquad \\[2mm]
R(x) &=& \alpha' \, \bigl( \alpha \;+\; 2(\alpha' + \beta')\, x \bigr)  
\\[2mm]
\rem(Q(x),R(x)) &=& -\frac{\alpha^2\beta'(\alpha'+2\beta')}{2(\alpha'+\beta')} 
\end{subeqnarray}
and the conditions $\alpha'+\beta'\neq 0$, $2\alpha'+\beta'\neq 0$, and 
$\alpha\alpha'\neq 0$ of the Case~(0,\k0,\k0,\k1b,\k1a,\k0). Note that $R(x)$ 
is equal to the coefficient $c_5$ of the Case~(0,\k0,\k0,\k1b,\k1a,\k0) 
multiplied by $\alpha'\neq 0$. 

\begin{itemize}
\item
$\bm{\deg R = 0\colon}$
As $[x^1]R(x)= 2\alpha'(\alpha' + \beta') \neq 0$, there is no solution of 
this type.

\item
$\bm{\deg R = 1\colon}$
If $[x^1]R(x)= 2\alpha'(\alpha' + \beta') \neq 0$, then 
$\rem(Q(x),R(x))=0$ gives two cases: $\beta'=0$, and $\alpha'+2\beta'=0$. 
The former one leads to 
$\bmu=(\alpha,0,-2\alpha,\alpha',0,-\alpha')$, 
which is a specialization of Family~5 [i.e., Case~(0,\k0,\k1b, \k1b)] when 
$\gamma=-2\alpha$ and $\gamma'=-\alpha'$; so it must be discarded. 
The other one gives:

\bigskip

%
%
\noindent
{\bf Case (0,\k0,\k0,\k1b,\k1a,\k0,\k1a)}: we have (after some trivial scale 
transformations) the family 
$\bmu = (-2\beta,\beta,4\beta, -2\beta',\beta',\beta')$ 
with $\beta\beta'\neq 0$. 
In this case, $c_6 = -5\beta - 4\beta' \, x \neq 0$. 
\end{itemize}
 
%
%
\subsubsection{%
\texorpdfstring{Branches starting from Case~(0,\! 0,\! 1a,\! 0,\! 1b,\! 0)}
               {Branches starting from Case (0,0,1a,0,1b,0)}}

We have $c_6 = Q(x)/R(x)$ with
\begin{subeqnarray}
Q(x) &=& 2\alpha^2(2\alpha'+3\gamma') \;+\; 
          \alpha\bigl(4(\alpha')^2 + 9 \alpha'\gamma' + 3(\gamma')^2 \bigr)\, x 
          \qquad  \\[2mm]
R(x) &=& (\alpha'+\gamma')\, \bigl( 2\alpha\;+\; 
         (2\alpha' + \gamma')\, x  \bigr) \\[2mm]
\rem(Q(x),R(x)) &=& -\frac{2\alpha^2\alpha'\gamma'}{2\alpha' + \gamma'}  
\end{subeqnarray}
with $\alpha\gamma'\neq 0$, $\alpha'+\gamma'\neq 0$, and
$\alpha'+2\gamma'\neq 0$. Note that $R(x)$ is equal to the coefficient of
Case~(0,\k0,\k1a,\k0,\k1b,\k0) multiplied by $\alpha'+\gamma'\neq 0$. 

\begin{itemize}
\item
$\bm{\deg R = 0\colon}$
If $[x^1]R(x)= (\alpha' + \gamma')(2\alpha' + \gamma') = 0$, and
$[x^0]R(x)= 2\alpha(\alpha'+\gamma') \neq 0$, they reduce to 
$2\alpha' + \gamma'=0$ if we take into account the above inequations. 
This solution leads to: 

\bigskip

%
%
\noindent
{\bf Case (0,\k0,\k1a,\k0,\k1b,\k0,\k0)}: we have the family
$\bmu = (\alpha,\alpha,-\alpha, \alpha',-\alpha',-2\alpha' )$
with $\alpha\alpha'\neq 0$. In this case, $c_6 = 4\alpha+\alpha'\, x \neq 0$. 

\item
$\bm{\deg R = 1\colon}$
If $[x^1]R(x)= (\alpha' + \gamma')(2\alpha' + \gamma') \neq 0$, this 
condition reduces to $2\alpha' + \gamma'\neq 0$. Then  
$\rem(Q(x),R(x))=0$ gives the case $\alpha'=0$. We obtain the family 
$\bmu = (\alpha,0,-\alpha,0,0,\gamma')$; but this is the specialization of 
Family~5 [i.e., Case~(0,\k0,\k1b,\k1b)] when $\gamma=-\alpha$ and
$\alpha'=0$; so it must be discarded.  
\end{itemize}

%
%
\subsubsection{%
\texorpdfstring{Branches starting from Case~(0,\! 0,\! 1a,\! 1b,\! 0,\! 0)}
               {Branches starting from Case (0,0,1a,1b,0,0)}}

We have $c_6 = Q(x)/R(x)$ with
\begin{subeqnarray}
Q(x) &=& 2\alpha^2(3\alpha'+\beta') \;+\; 
          \alpha\bigl(9(\alpha')^2 + 7 \alpha'\beta' + 2(\beta')^2 \bigr)\, x 
 \nonumber \\
    & & \quad \;+\; \alpha'(\alpha'+\beta')(3\alpha'+2\beta')\, x^2   \\[2mm]
R(x) &=& \alpha' \, \bigl( 2\alpha \;+\; 
         (\alpha' + \beta')\, x \bigr)  \\[2mm]
\rem(Q(x),R(x)) &=& \frac{2\alpha^2\beta'(\alpha'-\beta')}{\alpha' + \beta'} 
\end{subeqnarray}
and the conditions $\alpha'+\beta' \neq 0$, $2\alpha'+\beta' \neq 0$, and  
$\alpha\alpha'\neq 0$ of Case~(0,\k0,\k1a,\k1b,\k0,\k0). 
Note that $R(x)$ is equal to the coefficient $c_5$ of the 
Case~(0,\k0,\k1a,\k1b,\k0,\k0) multiplied by $\alpha' \neq 0$.

\begin{itemize}
\item
$\bm{\deg R = 0\colon}$
As $[x^1]R(x)= \alpha' (\alpha' + \beta') \neq 0$, there are no solution 
of this type. 

\item
$\bm{\deg R = 1\colon}$
If $[x^1]R(x)= \alpha' (\alpha' + \beta') \neq 0$, then $\rem(Q(x),R(x))=0$ 
gives two cases $\beta'=0$, and $\alpha'-\beta'=0$. 
The former one corresponds to  
$\bmu = (\alpha,0,-\alpha,\alpha',0,-2\alpha')$; but this is the 
specialization of Family~5 [i.e., Case~(0,\k0, \k1b,\k1b)] when 
$\gamma=-\alpha$ and $\gamma'=-2\alpha'$; so it must be discarded. 
The latter case gives:  

%
%
\noindent
{\bf Case (0,\k0,\k1a,\k1b,\k0,\k0,\k1a)}: we have the family
$\bmu = (\alpha,\alpha,-\alpha, \alpha',\alpha',-4\alpha')$ 
with $\alpha\alpha'\neq 0$. In this case, $c_6 = 4\alpha+5\alpha'\, x \neq 0$.
\end{itemize}

%
%
\subsubsection{%
\texorpdfstring{Branches starting from Case~(0,\! 0,\! 1b,\! 0,\! 1a,\! 0)}
               {Branches starting from Case (0,0,1b,0,1a,0)}}

We have $c_6 = Q(x)/R(x)$ with
\begin{subeqnarray}
Q(x) &=& \alpha^2(\alpha'+3\gamma') \;+\; 
         \alpha\bigl(4(\alpha')^2 + 9\alpha'\gamma' +2(\gamma')^2 \bigr)\, x 
    \\[2mm]
R(x) &=& \gamma' \, \bigl( \alpha \;+\; 
         (2\alpha' + \gamma')\, x  \bigr) \\[2mm]
\rem(Q(x),R(x)) &=& -\frac{2\alpha^2\alpha'(\alpha'+\beta')}%
                          {2\alpha' + \gamma'}  
\end{subeqnarray}
with the conditions $\alpha\gamma'\neq 0$, $\alpha'+\gamma'\neq 0$, and
$\alpha'+2\gamma'\neq 0$ of Case~(0,\k0,\k1b,\k0,\k1a,\k0). Note that 
$R(x)$ is equal to the coefficient $c_5$ of the 
Case~(0,\k0,\k1b,\k0,\k1a,\k0) multiplied by $\gamma'\neq 0$.

\begin{itemize}
\item
$\bm{\deg R = 0\colon}$ We have $[x^1]R(x)= \gamma'(2\alpha' + \gamma') = 0$ 
and $[x^0]R(x)=\alpha\gamma'\neq 0$, which reduce to $2\alpha' + \gamma'=0$. 
The single solution leads to: 

%
%
\noindent
{\bf Case (0,\k0,\k1b,\k0,\k1a,\k0,\k0)}: we have (after a trivial change
of parameters) the family
$\bmu = (-2\beta,\beta,4\beta,\alpha',-\alpha',-2\alpha')$
with $\beta\alpha'\neq 0$. In this case, $c_6 = -5\beta+\alpha'\, x\neq 0$. 

\item
$\bm{\deg R = 1\colon}$ 
If $[x^1]R(x)= \gamma'(2\alpha' + \gamma') \neq 0$, then $\rem(Q(x),R(x))=0$
gives the case $\alpha'=0$: $\bmu = (\alpha,0,-2\alpha,0,0,\gamma')$, 
but this is the specialization of Family~5 [i.e., Case~(0,\k0,\k1b,\k1b)] when
$\gamma=-2\alpha$ and $\alpha'=0$; so it must be discarded.
\end{itemize}

%
%
\subsubsection{%
\texorpdfstring{Branches starting from Case~(0,\! 0,\! 1b,\! 1a,\! 0,\! 0)}
               {Branches starting from Case (0,0,1b,1a,0,0)}}

We have $c_6 = Q(x)/R(x)$ with
\begin{subeqnarray}
Q(x) &=& \gamma \bigl(3(\alpha')^2 - 3\alpha'\beta' -2(\beta')^2 \bigr)\, x
  \;+\; (\alpha'+ \gamma')^2 (3\alpha'+2\beta')\, x^2 \qquad\quad   \\[2mm]
R(x) &=& \gamma(\alpha'-\beta') \;+\; (\alpha' + \beta')^2\, x  \\[2mm]
\rem(Q(x),R(x)) &=& \frac{2\gamma^2\alpha'\beta'(\alpha'-\beta')}%
                          {(\alpha' + \gamma')^2}  
\end{subeqnarray}
with the conditions $\gamma\alpha'\neq 0$, $\alpha'+\beta'\neq 0$, and
$2\alpha'+\beta'\neq 0$ of Case~(0,\k0,\k1b,\k1a,\k0,\k0). Note that 
$R(x)$ is equal to the coefficient $c_5$ of the 
Case~(0,\k0,\k1b,\k1a,\k0,\k0) multiplied by $\alpha'+\beta'\neq 0$. 

\begin{itemize}
\item
$\bm{\deg R = 0\colon}$ As $[x^1]R(x)= (\alpha' + \gamma')^2 \neq 0$, there is
no solution of this type. 

\item
$\bm{\deg R = 1\colon}$ 
If $[x^1]R(x)= (\alpha' + \gamma')^2 \neq 0$, then $\rem(Q(x),R(x))=0$ 
gives two cases $\beta'=0$ and $\alpha'-\beta'=0$. The former one leads to 
$\bmu=(0,0,\gamma,\alpha',0,-2\alpha')$, which corresponds to the 
specialization of Family~5 [i.e., Case~(0,\k0,\k1b,\k1b)] when $\alpha=0$ and 
$\gamma'=-2\alpha'$. Therefore, this solution should be discarded. 
The latter case gives: 

\bigskip 

%
%
\noindent
{\bf Case (0,\k0,\k1b,\k1a,\k0,\k0,\k1a)}: we have (after a trivial change
of parameters) the family
$\bmu = (0,\beta,-2\beta, \alpha',\alpha',-4\alpha')$
with $\beta\alpha'\neq 0$. In this case, $c_6 = \beta +5\alpha'\, x\neq 0$. 
\end{itemize}

%
%
\subsection{\texorpdfstring{Step 7: Coefficient $\bm{c_7}$}{Coefficient c7}}
\label{sec.c7}

We consider the six cases obtained in Section~\ref{sec.c6} that require 
additional conditions to be terminated. For each of these vertices, we 
now compute the branches stemming from them by using the next 
coefficient $c_7$.
It will turn out that none of the six cases have any viable solutions;
therefore, all six correspond to leaves of the decision tree,
and this is the end of the computation.

%
%
\subsubsection{\texorpdfstring{%
 Branches starting from Case~(0,\! 0,\! 0,\! 1a,\! 1b,\! 0,\! 1a)}
{Branches starting from Case (0,0,0,1a,1b,0,1a)}}

We find that $c_7 = Q(x)/R(x)$ with
\begin{subeqnarray}
Q(x) &=& -8\beta\beta'\, x \;+\; 12 (\beta')^2\, x^2 \\[2mm]
R(x) &=& \beta \;-\; 4\beta'\, x  
\\[2mm]
\rem(Q(x),R(x)) &=& -\frac{5\beta^2}{4}   
\end{subeqnarray}
and the conditions $\beta\beta'\neq 0$ of the 
Case~(0,\k0,\k0,\k1a,\k1b,\k0,\k1a). Note that $R(x)$ is equal to the 
coefficient $c_6$ of the Case~(0,\k0,\k0,\k1a,\k1b,\k0,\k1a). 

\begin{itemize}
\item
$\bm{\deg R = 0\colon}$ As $[x^1]R(x)= -4\beta' \neq 0$, there is no 
solution of this type. 

\item
$\bm{\deg R = 1\colon}$ In this case, $\rem(Q(x),R(x))\neq 0$, 
so there is no solution of this type.
\end{itemize}

This branch does not have any viable solution, so 
Case~(0,\k0,\k0,\k1a,\k1b,\k0,\k1a) is a leaf of the decision tree.

%
%
\subsubsection{\texorpdfstring{%
 Branches starting from Case~(0,\! 0,\! 0,\! 1b,\! 1a,\! 0,\! 1a)}
{Branches starting from Case (0,0,0,1b,1a,0,1a)}}

We find that $c_7 = Q(x)/R(x)$ with
\begin{subeqnarray}
Q(x) &=& -20 \beta^2\, \;-\; 32 \beta\beta'\, x \;-\; 12 (\beta')^2\, x^2 
\\[2mm]
R(x) &=& 5\beta \;+\; 4\beta'\, x  
\\[2mm]
\rem(Q(x),R(x)) &=& -\frac{5\beta^2}{4}   
\end{subeqnarray}
and the conditions $\beta\beta'\neq 0$ of the 
Case~(0,\k0,\k0,\k1b,\k1a,\k0,\k1a). Note that $-R(x)$ is equal to 
the coefficient $c_6$ of the Case~(0,\k0,\k0,\k1b,\k1a,\k0,\k1a).

\begin{itemize}
\item
$\bm{\deg R = 0\colon}$ As $[x^1]R(x)= -4\beta' \neq 0$, there is no
solution of this type.

\item
$\bm{\deg R = 1\colon}$ In this case, $\rem(Q(x),R(x))\neq 0$,
so there is no solution of this type.
 
\end{itemize}

This branch does not have any viable solution, so 
Case~(0,\k0,\k0,\k1b,\k1a,\k0,\k1a) is a leaf of the decision tree.

%
%
\subsubsection{\texorpdfstring{%
 Branches starting from Case~(0,\! 0,\! 1a,\! 0,\! 1b,\! 0,\! 0)}
{Branches starting from Case (0,0,1a,0,1b,0,0)}}

We find that $c_7 = Q(x)/R(x)$ with
\begin{subeqnarray}
Q(x) &=& 12 \alpha^2\, \;+\; 8 \alpha\alpha'\, x \\[2mm]
R(x) &=& 4\alpha \;+\; \alpha'\, x  
\\[2mm]
\rem(Q(x),R(x)) &=& -20 \alpha^2   
\end{subeqnarray}
and the conditions $\alpha\alpha'\neq 0$ of the 
Case~(0,\k0,\k1a,\k0,\k1a,\k0,\k0). Note that $R(x)$ is equal to the
coefficient $c_6$ of Case~(0,\k0,\k1a,\k0,\k1b,\k0,\k0).

\begin{itemize}
\item
$\bm{\deg R = 0\colon}$ As $[x^1]R(x)= \alpha' \neq 0$, there is no
solution of this type.

\item
$\bm{\deg R = 1\colon}$ In this case, $\rem(Q(x),R(x))\neq 0$,
so there is no solution of this type.
\end{itemize}

This branch does not have any viable solution, so 
Case~(0,\k0,\k1a,\k0,\k1b,\k0,\k0) is a leaf of the decision tree.

%
%
\subsubsection{\texorpdfstring{%
 Branches starting from Case~(0,\! 0,\! 1a,\! 1b,\! 0,\! 0,\! 1a)}
{Branches starting from Case (0,0,1a,1b,0,0,1a)}}

We find that $c_7 = Q(x)/R(x)$ with
\begin{subeqnarray}
Q(x) &=& 12 \alpha^2\, \;+\; 32 \alpha\alpha'\, x \;+\; 20 (\alpha')^2\, x^2 
   \\[2mm]
R(x) &=& 4\alpha \;+\; 5\alpha'\, x  
\\[2mm]
\rem(Q(x),R(x)) &=& -\frac{4 \alpha^2}{5}   
\end{subeqnarray}
and the conditions $\alpha\alpha'\neq 0$ of the 
Case~(0,\k0,\k1a,\k1b,\k0,\k0,\k1a). Note that $R(x)$ is equal to the
coefficient $c_6$ of Case~(0,\k0,\k1a,\k1b,\k0,\k0,\k1a). 

\begin{itemize}
\item
$\bm{\deg R = 0\colon}$ As $[x^1]R(x)= 5\alpha' \neq 0$, there is no
solution of this type.

\item
$\bm{\deg R = 1\colon}$ In this case, $\rem(Q(x),R(x))\neq 0$,
so there is no solution of this type.
\end{itemize}

This branch does not have any viable solution, so  
Case~(0,\k0,\k1a,\k1b,\k0,\k0,\k1a) is a leaf of the decision tree.

%
%
\subsubsection{\texorpdfstring{%
 Branches starting from Case~(0,\! 0,\! 1b,\! 0,\! 1a,\! 0,\! 0)}
{Branches starting from Case (0,0,1b,0,1a,0,0)}}

We find that $c_7 = Q(x)/R(x)$ with
\begin{subeqnarray}
Q(x) &=& -20 \beta^2\, \;+\; 8 \beta\alpha'\, x  \\[2mm]
R(x) &=& 5\beta \;-\; \alpha'\, x   \\[2mm]
\rem(Q(x),R(x)) &=& 20 \beta^2   
\end{subeqnarray}
and the conditions $\beta\alpha'\neq 0$ of the 
Case~(0,\k0,\k1b,\k0,\k1a,\k0,\k0). Note that $-R(x)$ is equal to the
coefficient $c_6$ of Case~(0,\k0,\k1b,\k0,\k1a,\k0,\k0). 

\begin{itemize}
\item
$\bm{\deg R = 0\colon}$ As $[x^1]R(x)= -\alpha' \neq 0$, there is no
solution of this type.

\item
$\bm{\deg R = 1\colon}$ In this case, $\rem(Q(x),R(x))\neq 0$,
so there is no solution of this type.
\end{itemize}

This branch does not have any viable solution, so 
Case~(0,\k0,\k1b,\k0,\k1a,\k0,\k0) is a leaf of the decision tree.

%
%
\subsubsection{\texorpdfstring{%
 Branches starting from Case~(0,\! 0,\! 1b,\! 1a,\! 0,\! 0,\! 1a)}
{Branches starting from Case (0,0,1b,1a,0,0,1a)}}

We find that $c_7 = Q(x)/R(x)$ with
\begin{subeqnarray}
Q(x) &=& 8\beta \alpha'\,x \;+\; 20 (\alpha')^2\, x^2  \\[2mm]
R(x) &=& \beta \;+\; 5 \alpha'\, x   \\[2mm]
\rem(Q(x),R(x)) &=& - \frac{4 \beta^2}{5}   
\end{subeqnarray}
and the conditions $\beta\alpha'\neq 0$ of the 
Case~(0,\k0,\k1b,\k1a,\k0,\k0,\k1a). Note that $R(x)$ is equal to the
coefficient $c_6$ of Case~(0,\k0,\k1b,\k1a,\k0,\k0,\k1a). 

\begin{itemize}
\item
$\bm{\deg R = 0\colon}$ As $[x^1]R(x)= 5\alpha' \neq 0$, there is no
solution of this type.

\item
$\bm{\deg R = 1\colon}$ In this case, $\rem(Q(x),R(x))\neq 0$,
so there is no solution of this type.
\end{itemize}

This branch does not have any viable solution, so 
Case~(0,\k0,\k1b,\k1a,\k0,\k0,\k1a) is a leaf of the decision tree.

%
%
\section{Non-terminating S-fractions} \label{sec.Sfractions.ogf} 

The final result of Section~\ref{sec.search} is that there are ten 
families that are viable candidates for having an ogf  
with a non-terminating S-fraction representation. These families are:

\begin{itemize}
\item[1a.] Case~(0,\k0,\k1a,\k0,\k1a) given by  
          $\bmu=(0,\beta,0,\alpha',-\alpha',\gamma')$ with 
          $\alpha'+\gamma'\neq 0$ and $\beta\gamma'\neq 0$. 
          The conjectured coefficients are 
          $c_{2k-1} = [\gamma' + (k-1)\alpha'] x$ and
          $c_{2k}   = k \beta$.

\item[1b.] Case~(0,\k0,\k0,\k1a,\k1b) given by  
          $\bmu=(0,\beta,\gamma,\alpha',-\alpha',0)$ with 
          $\beta+\gamma\neq 0$ and $\gamma\alpha'\neq 0$.
          The conjectured coefficients are 
          $c_{2k-1} = \gamma + (k-1)\beta$ and
          $c_{2k}   = k \alpha' x$.

\item[2a.] Case~(0,\k0,\k1a,\k1a) given by  
          $\bmu=(\alpha,-\alpha,-\alpha,\alpha',\beta',\gamma')$
          with $\alpha'+\beta'+\gamma'\neq 0$, $\alpha'+\beta'\neq 0$, 
          and $2\alpha'+2\beta'+\gamma'\neq 0$.
          The conjectured coefficients are 
          $c_{2k-1} = [\gamma' + k(\alpha' + \beta')] x$ and
          $c_{2k}   = k (\alpha'+\beta') x$.

\item[2b.] Case~(0,\k0,\k0,\k0) given by  
          $\bmu=(\alpha,\beta,\gamma, 0,\beta',-\beta')$ with 
          $\alpha+\gamma\neq 0$, $2\alpha+\gamma\neq 0$, and $\alpha\neq 0$. 
          The conjectured coefficients are
          $c_{2k-1} = k \alpha + \gamma$ and
          $c_{2k}   = k \alpha$.

\item[3a.] Case~(0,\k0,\k1a,\k1a) given by  
          $\bmu=(0,\beta,0, 0,\beta',\gamma')$ with 
          $\beta'+\gamma'\neq 0$, $2\beta'+\gamma'\neq 0$, and $\beta'\neq 0$.
          The conjectured coefficients are
          $c_{2k-1} = (\gamma' + k \beta') x$ and
          $c_{2k}   = k (\beta+\beta' x)$.

\item[3b.] Case~(0,\k0,\k0,\k1b,\k0) given by  
          $\bmu=(\alpha,-\alpha,\gamma, \alpha',-\alpha',0)$ with 
          $\alpha+\gamma\neq 0$, $\alpha+2\gamma\neq 0$, and $\alpha'\neq 0$.
          The conjectured coefficients are
          $c_{2k-1} = k \alpha + \gamma$ and
          $c_{2k}   = k (\alpha + \alpha' x)$.

\item[4a.] Case~(0,\k0,\k1b,\k1a,\k1a) given by  
          $\bmu=(0,\kappa\beta',\kappa(\beta'+\gamma'), 0,\beta', \gamma')$
          with $\beta'+\gamma'\neq 0$, $2\beta'+\gamma'\neq 0$, and 
          $\kappa\beta'\neq 0$. The conjectured coefficients are
          $c_{2k-1} = (k\beta' + \gamma')(\kappa+x)$ and
          $c_{2k}   = k \beta' x$.

\item[4b.] Case~(0,\k0,\k1b,\k0,\k1b) given by  
          $\bmu=(\alpha,-\alpha,\gamma,
                       \kappa\alpha,-\kappa\alpha,\kappa(\alpha+\gamma))$ with
          $\alpha+\gamma\neq 0$, $2\alpha+\gamma\neq 0$, and 
          $\alpha\kappa\neq 0$. The conjectured coefficients are
          $c_{2k-1} = (\gamma + k \alpha)(1+\kappa x)$ and
          $c_{2k}   = k \alpha$.

\item[5.] Case~(0,\k0,\k1b,\k1b) given by  
          $\bmu=(\alpha,0,\gamma, \alpha',0,\gamma')$ with
          $\alpha'+\gamma'\neq 0$, $\alpha+\gamma\neq 0$, $\alpha'\neq 0$,
          and ($2\alpha+\gamma\neq 0$ or $2\alpha'+\gamma'\neq 0$).
          The conjectured coefficients are
          $c_{2k-1} = (\gamma + \gamma' x) + k(\alpha + \alpha' x)$
          and $c_{2k}   = k (\alpha + \alpha' x)$.

\item[6.] Case~(0,\k0,\k1b,\k1c) given by  
          $\bmu=(\kappa(\alpha'+\beta'),\kappa\beta',
                       \kappa\gamma',\alpha',\beta',\gamma')$ with
          $\alpha'+\beta'+\gamma'\neq 0$, $\alpha'+\beta'\neq 0$, 
          $2\alpha'+2\beta'+\gamma'\neq 0$, and $\alpha'\neq 0$. 
          The conjectured coefficients are
          $c_{2k-1} = [\gamma' + k (\alpha' + \beta')] (\kappa+x)$
          and $c_{2k}   = k (\alpha' + \beta')(\kappa + x)$.
\end{itemize}

Please note that the inequations that accompany each family are needed only
to ensure that the first few coefficients $c_i$ are nonvanishing.
Thus, the families obtained in
Section~\ref{sec.c3} (i.e., 2a, 2b, 3a, 5, and 6) have inequations that
imply that $c_1,c_2,c_3 \neq 0$,
while the families obtained in Section~\ref{sec.c4}
(i.e., 1a, 1b, 3b, 4a, and 4b)
have inequations that imply that $c_1,c_2,c_3,c_4 \neq 0$.
Of course, an infinite set of inequations
(which are easily deducible from the given formulae)
will be needed to ensure that {\em all}\/ $c_i \neq 0$.
But since Theorem~3.1 of the main text explicitly allows that, in each family,
the continued fraction might be terminating in some degenerate cases,
these inequations can be discarded.
This is why they do not appear in the main text.

%
%
\section{Rational ordinary generating functions} \label{sec.rational.ogf} 

In Section~\ref{sec.search} we found the ten families 1a--6
of non-terminating S-fractions compiled in Section~\ref{sec.Sfractions.ogf}.
A side effect of this computation was to find also
several non-trivial rational ogfs (that is, families 
$\bmu\in \C^6$ that lead to terminating S-fractions). By non-trivial, we mean
that they are \emph{not} particular cases of any of the ten families 1a--6.

In Section~\ref{sec.c1}, we find the first non-trivial family with a 
rational ogf. The Case~(0,\k0) provides the coefficient 
$c_1 = \alpha + \gamma + (\alpha'+\beta'+\gamma')x$, which vanishes on the
submanifold $\alpha + \gamma = \alpha'+\beta'+\gamma' = 0$. This leads to  
the self-dual family:
\begin{itemize}
  \item[s0.] $\bmu = (\alpha,\beta,-\alpha,\alpha',\beta',-\alpha'-\beta')$
             corresponding to Case~(0,\k0). The coefficients are $c_0 =1$, 
             and $c_1 = 0$.
\end{itemize}
Of course, this case corresponds to the trivial matrix
$T(n,k) = \delta_{n0} \delta_{k0}$.
(We remark that although the parameters $\bmu$ corresponding to
Case~(0,\k0) are not a special case of any of the families 1a--6,
this trivial matrix {\em is}\/ a special case of {\em all}\/
the families 1a--6:  it suffices to specialize the parameters
so that $c_1 = 0$.  For instance, in the family~1a we take $\gamma' = 0$.
This illustrates the parametric ambiguities discussed in
[56, section~2.4] [2, section~3],
corresponding to the non-injectivity of the map $\bmu \mapsto \bT(\bmu)$.)

In Section~\ref{sec.c2}, we find that $c_2$ vanishes in three submanifolds
corresponding to the families  
$\bmu=(0,\beta,\gamma,0,\beta',-\beta')$ [Case~(0,\k0,\k0)], 
$\bmu=(\alpha,-\alpha,-\alpha,\alpha',-\alpha',\gamma')$ [Case~(0,\k0,\k1a)], 
and $\bmu=(0,-\gamma\alpha',\gamma\gamma',\alpha',-\alpha',\gamma')$
[Case~(0,\k0,\k1b)]. These
families are actually specializations of Family~2b (when $\alpha=0$),
Family~2a (when $\beta'=-\alpha'$), and Family~6 (when $\beta'=-\alpha'$ and
$\kappa=\gamma$), respectively. Therefore, they all are trivial families 
leading to a rational ogf. 

In Section~\ref{sec.c3}, we have six possible submanifolds where $c_3=0$. We 
find that all of them correspond to specializations of families 1a--6, so they 
correspond to trivial rational ogf. 

On the other hand, in Section~\ref{sec.c4}
we find six non-trivial rational ogfs. For instance,
in Case~(0,\k0,\k0, \k1a,\k1b), we found that 
$c_4=[\alpha'(\beta+2\gamma)/\gamma]\, x$. This coefficient can be made
equal to zero if $\beta+2\gamma=0$. This leads to 
$\bmu=(0,-2\gamma,\gamma,\alpha',-2\alpha',\alpha')$, 
which corresponds to the family~s1a below. We can work out the other five 
cases in a similar fashion. The results can be summarized in the following 
list: 

\begin{itemize}
  \item[s1a.] $\bmu=(0,-2\gamma,\gamma,\alpha',-2\alpha',\alpha')$ 
               corresponding to Case~(0,\k0,\k0,\k1a,\k1b). The 
               coefficients are $c_1 =\gamma$, $c_2 = \alpha' x$, 
               $c_3 =-\gamma-\alpha' x$, and $c_4 = 0$. 

  \item[s1b.] $\bmu=(\alpha,-2\alpha,-\alpha,-2\gamma',2\gamma',\gamma')$ 
               corresponding to Case~(0,\k0,\k1a,\k0,\k1b). The 
               coefficients are $c_1 = \gamma' x$, $c_2 = -\alpha$, 
               $c_3 = \alpha-\gamma' x$, and $c_4 = 0$. 

  \item[s2a.] $\bmu=(0,-2\gamma,\gamma,\alpha',-2\alpha',2\alpha')$ 
               corresponding to Case~(0,\k0,\k1b,\k1a,\k0). The 
               coefficients are $c_1 =\gamma + \alpha' x$, $c_2 = -\alpha' x$,
               $c_3 =-\gamma$, and $c_4 = 0$.

  \item[s2b.] $\bmu=(\alpha,-2\alpha,-2\alpha,-2\gamma',2\gamma',\gamma')$ 
               corresponding to Case~(0,\k0,\k1b,\k0,\k1a). The 
               coefficients are $c_1 =- \alpha + \gamma' x$, $c_2 = \alpha$,
               $c_3 =-\gamma' x$, and $c_4 = 0$. 

  \item[s3a.] $\bmu=(\alpha,-2\alpha,-2\alpha,\alpha',-2\alpha',\alpha')$ 
               corresponding to Case~(0,\k0,\k0,\k1b,\k1a). The 
               coefficients are $c_1 =-\alpha$, $c_2 = \alpha + \alpha' x$,
               $c_3 = -\alpha' x$, and $c_4 = 0$.

  \item[s3b.] $\bmu=(\alpha,-2\alpha,-\alpha,\alpha',-2\alpha',2\alpha')$ 
               corresponding to Case~(0,\k0,\k1a,\k1b,\k0). The 
               coefficients are $c_1 = \alpha' x$, $c_2 = -\alpha-\alpha' x$,
               $c_3 = \alpha$, and $c_4 = 0$. 
\end{itemize}
It is clear that under duality  
\hbox{s1a $\leftrightarrow$ s1b,}
\hbox{s2a $\leftrightarrow$ s2b,}
\hbox{s3a $\leftrightarrow$ s3b.}

In Section~\ref{sec.c5}, we find explicitly another six non-trivial
rational ogf's with $c_5=0$: 
\begin{itemize}
  \item[s4a.] $\bmu=(0,\beta,-\beta,\alpha',-2\alpha',\alpha')$
               corresponding to Case~(0,\k0,\k0,\k1a,\k1a,\k0). The
               coefficients are $c_1 =- \beta$, $c_2 = \alpha' x$,
               $c_3 = -\alpha' x$, $c_4 = \beta$, and $c_5=0$.

  \item[s4b.] $\bmu=(\alpha,-2\alpha,-\alpha,\alpha',-\alpha',-\alpha')$
               corresponding to Case~(0,\k0,\k1a,\k0,\k0,\k1a). The
               coefficients are $c_1 =- \alpha' x$, $c_2 = -\alpha$,
               $c_3 = \alpha$, $c_4 = \alpha' x$, and $c_5=0$.

  \item[s5a.] $\bmu=(\alpha,-2\alpha,-3\alpha,\alpha',-2\alpha',\alpha')$
               corresponding to Case~(0,\k0,\k0,\k1b,\k1b,\k0). The
               coefficients are $c_1 =-2\alpha$, $c_2 = \alpha+\alpha' x$,
               $c_3 = -\alpha-\alpha' x$, $c_4 = 2\alpha$, and $c_5=0$.

  \item[s5b.] $\bmu=(\alpha,-2\alpha,-\alpha,\alpha',-2\alpha',3\alpha')$
               corresponding to Case~(0,\k0,\k1a,\k1b,\k1a,\k1a). The
               coefficients are $c_1 =2\alpha' x$, $c_2 =-\alpha-\alpha' x$,
               $c_3 = \alpha+\alpha' x$, $c_4 = -2\alpha' x$, and $c_5=0$.

  \item[s6a.] $\bmu=(\alpha,-2\alpha,-3\alpha,\alpha',-\alpha',-\alpha')$
               corresponding to Case~(0,\k0,\k1b,\k0,\k0,\k1a). The
               coefficients are $c_1 =-2\alpha-\alpha' x$, $c_2 = \alpha$, 
               $c_3 = -\alpha$, $c_4 = 2\alpha+\alpha' x$, and $c_5=0$.

  \item[s6b.] $\bmu=(0,\beta,-\beta,\alpha',-2\alpha',3\alpha')$
               corresponding to Case~(0,\k0,\k1b,\k1a,\k1b,\k1a). The
               coefficients are $c_1 = -\beta+2\alpha' x$, $c_2 = -\alpha' x$,
               $c_3 = \alpha' x$, $c_4 = \beta-2\alpha' x$, and $c_5=0$.
\end{itemize}
It is clear that under duality
\hbox{s4a $\leftrightarrow$ s4b,}
\hbox{s5a $\leftrightarrow$ s5b,}
\hbox{s6a $\leftrightarrow$ s6b.}

Finally, we have not found any non-trivial rational ogf coming from $c_6=0$ or
$c_7=0$. Therefore, we have found \emph{all} the non-trivial families 
$\bmu\in \C^6$ whose ogf is a rational function in $t$
that can be expressed by a (terminating) S-fraction in which all the
coefficients $c_i$ are polynomials in $x$. 

\clearpage
\begin{sidewaysfigure}
\centering
\includegraphics[width=620pt]{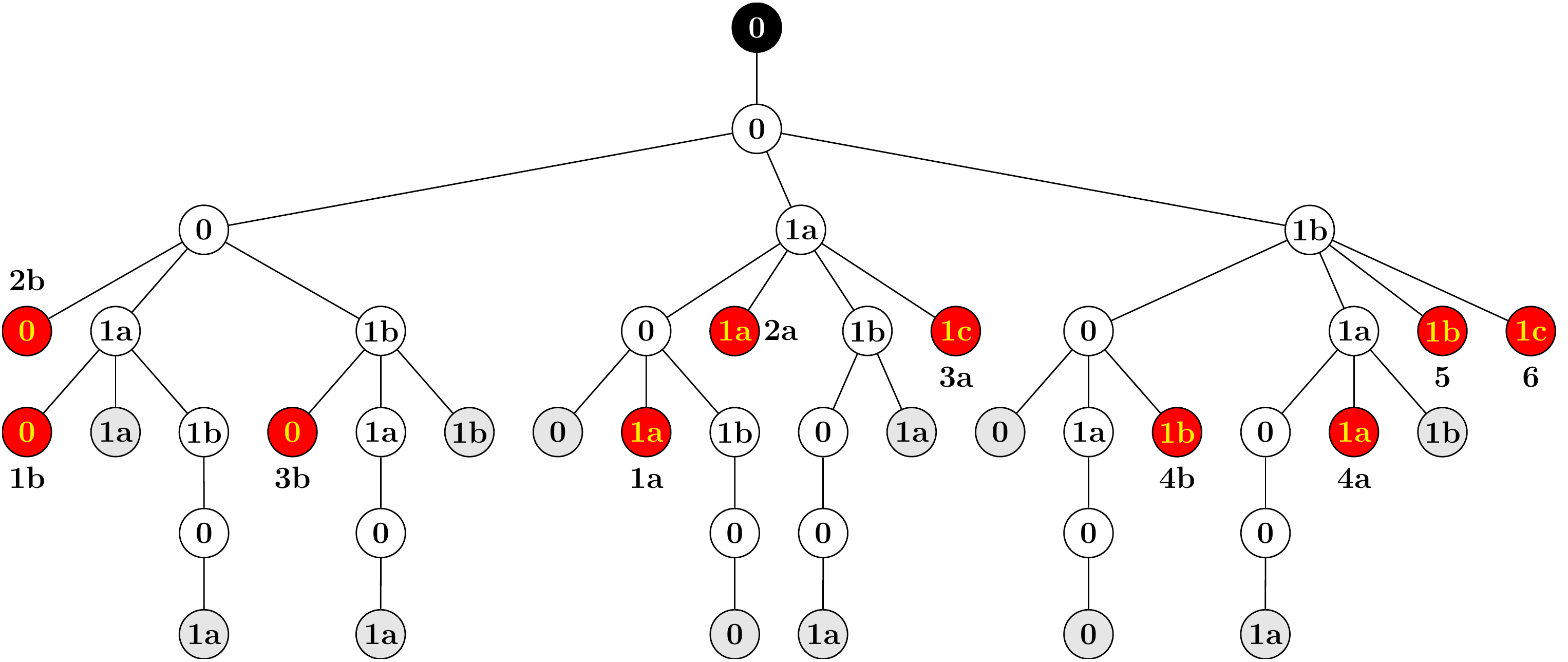}
\caption{\label{decision_tree}
Decision tree for the computation performed in Section~\ref{sec.search}.
The black node is the root vertex, corresponding to Case~(0).
The successive levels are based on analyzing the coefficients
$c_1,c_2,c_3,c_4,c_5,c_6,c_7$.
The 10 red nodes correspond to the candidate families
found in Section~\ref{sec.Sfractions.ogf};
they are labeled as in Theorem~3.1 of the main article.
Finally, the 12 gray leaves are the nodes with no children, and the white nodes
are the internal vertices of the tree.  
}
\end{sidewaysfigure}